\documentclass[book,11 pt]{article}
\usepackage{amsmath}
\usepackage{amstext}
\usepackage{mathtools}
\usepackage{amsthm}
\usepackage{amssymb}
\usepackage{amsfonts}
\usepackage{enumerate}
\usepackage{titling}
\usepackage{titlesec}
\usepackage[all]{xy}
\usepackage{setspace}
\usepackage{glossaries}
\usepackage{leftidx}
\makeglossaries

\glossarystyle{indexgroup}
\usepackage{tabularx}
\usepackage{bm}
\usepackage{upgreek}
\usepackage[bottom]{footmisc}

\usepackage{graphicx}
\usepackage{tocloft}

\usepackage{mathrsfs} 
\usepackage{amsbsy}

\numberwithin{equation}{section}

\usepackage{ulem}

\usepackage{fancyhdr}
\DeclareMathAlphabet{\mathpzc}{OT1}{pzc}{m}{it}

\titlespacing{\section}{0 cm }{1,5 cm}{0,2 cm}
\titlespacing{\subsection}{0 cm }{1 cm}{0,2 cm}

\newtheorem{prop}{Proposition}[section]
\newtheorem*{prop*}{Proposition}
\newtheorem*{agree}{General Agreement}
\newtheorem*{prop2.1}{Proposition 2.1}
\newtheorem*{prop1.1}{Proposition 1.1}
\newtheorem*{prop3.3}{Proposition 3.3}
\newtheorem{theo}{Theorem}[section]

\newtheorem*{theo_1*}{Theorem 1}
\newtheorem*{theo_2}{Theorem 2}
\newtheorem*{theo_3*}{Theorem 3}
\newtheorem*{theo_4*}{Theorem 4}
\newtheorem*{theo_5.a*}{Theorem 5.a}
\newtheorem*{theo_6.a*}{Theorem 6.a}
\newtheorem*{theo_5.b*}{Theorem 5.b}
\newtheorem*{theo_6.b*}{Theorem 6.b}

\newtheorem{lemma}{Lemma}[section]
\newtheorem{definition}{Definition}[section]
\newtheorem{remark}{Remark}[section]
\newtheorem{cor}{Corollary}[section]

\newtheorem{example}{Example}[section]

\usepackage{hyperref}
\begin{document}

\thispagestyle{empty}

  \setlength{\headsep}{2cm}
{{\centering

{\huge{{\textbf{Limiting Distributions of Scaled}}}

\vspace{0.2 cm}

{\huge{{\textbf{Eigensections in a GIT-Setting}}}}

\vspace{1 cm}

{\LARGE{\scshape{Daniel Berger}}}
\vspace{0.2 cm}

}}

\vspace{0.5 cm}

\begin{abstract}
Let $\mathrm{\mathbf{L}\rightarrow \mathbf{X}}$ be a base point free $\mathrm{\mathbb{T}=T^{\mathbb{C}}}$-linearized hermitian line bundle over a compact variety $\mathrm{\mathbf{X}}$ where $\mathrm{T=\left(S^{1}\right)^{m}}$ is a real torus. 
The main focus of this paper is to describe the asymptotic behavior of a certain class of sequences $\mathrm{\left(s_{n}\right)_{n}}$ of $\mathrm{\mathbb{T}}$-eigensections $\mathrm{s_{n}\in H^{0}\left(\mathbf{X},\mathbf{L}^{n}\right)}$ as $\mathrm{n\rightarrow \infty}$, introduced by {\scshape{Shiffman, Tate}} and {\scshape{Zelditch}}, and its connection to the geometry of the Hilbert quotient $\mathrm{\pi\!:\!\bf{X}^{ss}_{\xi}\rightarrow \mathbf{X}^{ss}_{\xi}/\!\!/\mathbb{T}}$ where $\mathrm{\xi\in \mathfrak{t}^{*}}$.

Using these sequences $\mathrm{\left(s_{n}\right)_{n}}$ we will first define a naturally associated sequence of probability measures $\mathrm{\left(\bm{\nu}_{n}\left(y\right)\right)_{n}}$ on each fiber $\mathrm{\pi^{-1}\left(y\right)}$ of the Hilbert quotient where $\mathrm{y}$ varies in a Zariski dense subset $\mathrm{\mathbf{Y}_{0}}$ of the base $\mathrm{\mathbf{X}^{ss}_{\xi}/\!\!/\mathbb{T}}$. In the main part of this paper we will then show that $\mathrm{\left(\bm{\nu}_{n}\left(y\right)\right)_{n}}$ converges uniformly over $\mathrm{\mathbf{Y}_{0}}$ to a Dirac fiber measure whose support is completely determined by the hermitian bundle metric $\mathrm{h}$ and the asymptotic geometry of the rescaled weight vectors $\mathrm{\left(n^{-1}\xi_{n}\right)_{n}}$ given by the initial sequence $\mathrm{\left(s_{n}\right)_{n}}$. 
 
The essential step of the proof is based on the work of {\scshape{D.\! Barlet}} whose results provide us with the construction of an equivariant, dimensional-theoretical {\itshape{flattening}} $\mathrm{\Pi\!:\!\widetilde{\bf{\,X\,}}\rightarrow \widetilde{\bf{\,Y\,}}}$ of the Hilbert quotient $\mathrm{\pi\!:\!\bf{X}^{ss}_{\xi}\rightarrow \mathbf{X}^{ss}_{\xi}/\!\!/\mathbb{T}}$ which turns out to be crucial in order to guarantee uniform estimates over all of $\mathrm{\mathbf{Y}_{0}}$.

\end{abstract}

\vspace{1cm}

  \setlength{\headsep}{0.8cm}
\begin{spacing}{1.2}

\setlength{\pdfpagewidth}{8.5in}
\setlength{\pdfpageheight}{11in}

\tableofcontents
\thispagestyle{empty}

\fontsize{11}{12}\selectfont 

\setcounter{page}{1}
\section{Introduction and Statement of Results}
The aim of this thesis is to examine the asymptotic geometry of a certain class of sequences of eigensections of a line bundle by describing the convergence properties of a naturally associated measure sequence. In this discussion, we let $\mathrm{\mathbf{X}}$ be a projective, normal, purely m-dimensional variety and $\mathrm{\mathbb{T}=T^{\mathbb{C}}}$ a complex torus with the unique maximal compact subgroup $\mathrm{T=\big(S^{1}\big)^{m}}$. We equip $\mathrm{\mathbf{X}}$ with an algebraic action $\mathrm{\phi\!:\!\mathbb{T}\times \mathbf{X}\rightarrow \mathbf{X}}$ of the complex torus $\mathrm{\mathbb{T}}$ and assume $\mathrm{\phi}$ to be compatible with the holomorphic structure of $\mathrm{\mathbf{X}}$. Furthermore, we will fix a based point free line bundle $\mathrm{p\!:\!\mathbf{L}\rightarrow \mathbf{X}}$ over $\mathrm{\mathbf{X}}$ and an algebraic $\mathrm{\mathbb{T}}$-action $\mathrm{\widehat{\phi}\!:\!\mathbb{T}\times \mathbf{L}\rightarrow \mathbf{L}}$ which projects down to $\mathrm{\phi}$ so that the corresponding morphisms of the fibers $\mathrm{\mathbf{L}_{\vert x}}$ are linear transformations. In the sequel, we will refer to $\mathrm{\mathbf{L}}$ as a $\mathrm{\mathbb{T}}$-{{linearized}} line bundle. Moreover, let $\mathrm{h}$ \label{Hermitian Bundle Metric} be a $\mathrm{T}$-invariant smooth, hermitian, positive bundle metric on $\mathrm{\mathbf{L}}$ and let $\mathrm{\mathfrak{t}}$ be the Lie algebra of the compact form $\mathrm{T\subset \mathbb{T}^{\mathbb{C}}}$. In this context, there exists ({cf.\! }\cite{Gu-St}) a naturally associated moment map $\mathrm{\mu\!:\!\mathbf{X}\rightarrow \mathfrak{t}^{*}}$ \label{Notation Moment Map} with respect to the K\"ahler form $\mathrm{p^{*}\omega=-\frac{i}{2}\partial \overline{\partial} \vert \cdot\vert_{h}^{2}}$ \label{Notation Kaehler form} given by the formula $$\mathrm{p^{*}\mu^{\xi}=-\frac{1}{4}d^{c}\mathbf{log}\,\vert \cdot\vert_{h}^{2}\,\widehat{X}_{\xi}}$$ where $\mathrm{\widehat{X}_{\xi}\label{Notation Fundamental Vector Field on L}}$ denotes the fundamental vector field of $\mathrm{\widehat{\phi}}$ on the total space $\mathrm{\mathbf{L}}$ (in the sequel we will use the notation $\mathrm{\vert \cdot \vert^{2}_{h}=\vert \cdot\vert^{2}}$\label{Notation Norm of Section with Respect to H}).

If $\mathrm{\xi\in Im\left(\mu\right)}$, it is possible to define an equivalence relation on the Zariski open, $\mathrm{\mathbb{T}}$-invariant set of semistable points $$\mathrm{\mathbf{X}^{ss}_{\xi}=\left\{x\in \mathbf{X}\!:\!\,cl\left(\mathbb{T}.x\right)\cap \mu^{-1}\left(\xi\right)\neq \varnothing\right\}\label{Notation Set of Semistable Points}}$$ given by $$\mathrm{z_{0}\sim z_{1}:\Leftrightarrow cl\left(\mathbb{T}.z_{0}\right)\cap cl\left(\mathbb{T}.z_{1}\right)\neq \varnothing.}$$ Having defined $\mathrm{\sim}$, it is possible (cf. \cite{He-Hu2}, {pp.\ }{\bf{\oldstylenums{310}-\oldstylenums{349}}}) to equip the induced quotient $\mathrm{\mathbf{X}^{ss}_{\xi}/\!\sim}$ with a unique, holomorphic structure of a complex space denoted by $\mathrm{\mathbf{X}^{ss}_{\xi}/\!\!/\mathbb{T}\label{Notation Hilbert Quotient}}$ so that the quotient map $\mathrm{\pi\!:\!\bf{X}^{ss}_{\xi}\rightarrow \mathbf{X}^{ss}_{\xi}/\!\!/\mathbb{T}}$ \label{Notation Hilbert Quotient Map} is holomorphic and has the following two characteristic properties: \vspace{-0.2 cm} \begin{enumerate}  
\item If $\mathrm{\mathcal{O}_{\mathbf{X}}}$ denotes the sheaf of holomorphic functions on $\mathrm{\mathbf{X}}$ and $\mathrm{\mathcal{O}_{\mathbf{X}^{ss}/\!\!/\mathbb{T}}}$ and is the sheaf of holomorphic functions on $\mathrm{\mathbf{X}^{ss}_{\xi}/\!\!/\mathbb{T}}$, then $\mathrm{(\pi_{*}\mathcal{O}_{\mathbf{X}^{ss}_{\xi}})^{\mathbb{T}}=\mathcal{O}_{\mathbf{X}^{ss}_{\xi}/\!\!/\mathbb{T}}}$. \vspace{-0.2 cm}
\item We have $\mathrm{cl\,\left(\mathbb{T}.x\right)\cap cl\,\left(\mathbb{T}.y\right)\neq\varnothing}$ if and only if $\mathrm{\pi\left(x\right)=\pi\left(y\right)}$.
\end{enumerate} As a generalization of the work of {\scshape{Shiffman, Tate}} and {\scshape{Zelditch}} (cf.\ \cite{S-T-Z}) and the results of {\scshape{Huckleberry}}, {\scshape{Sebert}} (cf.\ \cite{Hu-Se}), we link the geometry of the sequence of $\mathrm{\mathbb{T}}$-representation spaces, given by $\mathrm{H^{0}\left(\mathbf{X},\mathbf{L}^{n}\right)}$, as $\mathrm{n\rightarrow \infty}$ to the geometry of the quotient $\mathrm{\pi\!:\!\mathbf{X}^{ss}_{\xi}\rightarrow \mathbf{X}^{ss}_{\xi}/\!\!/\mathbb{T}\eqqcolon\mathbf{Y}\label{Notation Abbreviation for Hilbert Quotient}}$. To be more precise, we show that for each choice of $\mathrm{\xi\in Im\left(\mu\right)}$, it is possible to construct a convergent measure sequence which localizes along $\mathrm{\mu^{-1}\left(\xi\right)}$ by using sequences of $\mathrm{\mathbb{T}}$-eigensections $\mathrm{s_{n}\in H^{0}\left(\mathbf{X},\mathbf{L}^{n}\right)}$, i.e.\ $$\mathrm{{\bf{exp}}\left(\eta\right).s_{n}=e^{2\,\pi\,\sqrt{-1}\xi_{n}\left(\eta\right)}\,s_{n}\text{ where }\eta\in \mathfrak{t},\,\xi_{n}\in\mathfrak{t^{*}_{\mathbb{Z}}}}$$ whose rescaled weights $\mathrm{\frac{\xi_{n}}{n}}$ asymptotically approximate the chosen $\mathrm{\xi\in \mathfrak{t}^{*}}$ appropriately as $\mathrm{n\rightarrow \infty}$. 

As a starting point, we show (cf.\ {\bf{Theorem 1}}) that, given $\mathrm{\xi\in Im\left(\mu\right)}$, there exists a finite cover $\mathrm{\left\{\mathbf{X}^{i}\right\}_{i\in I}\text{ of }\mathbf{X}^{ss}_{\xi}}$ consisting of open, $\mathrm{\pi}$-saturated subsets and a finite collection $$\mathrm{\left\{\left( s_{n}^{i} \right)_{n}\right\}_{i\in I},\text{ where }s_{n}^{i}\in H^{0}(\mathbf{X},\mathbf{L}^{n})}$$ of sequences consisting of $\mathrm{\mathbb{T}}$-eigensections such that the following properties are fulfilled:
\begin{enumerate}  
\item $\mathrm{\frac{\xi^{i}_{n}}{n}\rightarrow \xi}$ where $\mathrm{\vert \xi^{i}_{n}-n\,\xi\vert\in \mathcal{O}\left(1\right)}$ for each $\mathrm{i}$.\vspace{-0.2 cm}
\item $\mathrm{\mathbf{X}^{i}\subset \mathbf{X}(s_{n}^{i})\coloneqq \left\{x\in \mathbf{X}:\, s_{n}^{i}\left(x\right)\neq 0\right\}\label{Notation GTZU}}$ for all $\mathrm{n}$ big enough.\vspace{-0.2 cm}
\item $\mathrm{\mathbf{X}^{i}\cap \mathbb{T}.\mu^{-1}\left(\xi\right)=\mathbf{X}^{i}\cap \mathbb{T}.\mu^{-1}\left(n^{-1}\xi^{i}_{n}\right)}$ for all $\mathrm{n\geq N_{0}}$.
\end{enumerate} The construction of such a {\itshape{tame collection}} will be the first step of the present work and involves the combinatorial analysis of the sets $\mathrm{\mu\left(cl\left(\mathbb{T}.x\right)\right)}$, $\mathrm{x\in \mathbf{X}^{ss}_{\xi}}$, which are known to be convex polytopes in $\mathrm{\mathfrak{t}^{*}}$ ({cf.\ }\cite{Ati}). The crucial step of the existence proof is to control the dependence of the geometry of $\mathrm{\mu\left(cl\left(\mathbb{T}.x\right)\right)}$ as $\mathrm{x}$ varies in $\mathrm{\mathbf{X}^{ss}_{\xi}}$. Even in the special case $\mathrm{\mathbf{X}^{s}_{\xi}=\mathbf{X}^{ss}_{\xi}}$, i.e.\ where every $\mathrm{\pi}$-fiber is given by a $\mathrm{\mathbb{T}}$-orbit, the  shape and position of $\mathrm{\mu\left(cl\left(\mathbb{T}.x\right)\right)}$ can in general vary considerably.

After having proven the existence of  $\mathrm{\left\{\left(s_{n}^{i}\right)_{n}\right\}_{i}}$ with the aforementioned properties, it is possible to define for each finite, open cover $\mathrm{\mathfrak{U}=\left\{\mathbf{U}_{i}\right\}_{i}}$ of $\mathrm{\mathbf{Y}}$ subordinate to $\mathrm{\left\{\pi\left(\mathbf{X}^{i}\right)\right\}_{i}}$, i.e.\ $\mathrm{\mathbf{U}_{i}\subset \pi\left(\mathbf{X}^{i}\right)}$, a finite collection $\mathrm{\left\{\bm{\nu}^{\mathfrak{U}}_{n}\right\}_{n}=\left\{\bm{\nu}^{i}_{n}\right\}_{i}}$ of sequences of $\mathrm{\pi}$-fiber probability measures on Zariski open subsets of $\mathrm{\mathbf{X}^{ss}_{\xi}}$ by using the corresponding norm functions $\mathrm{\vert s^{i}_{n}\vert^{2}}$ with respect to the hermitian bundle metric $\mathrm{h}$. The precise construction of the collection $\mathrm{\left\{\bm{\nu}^{\mathfrak{U}}_{n}\right\}_{n}}$ is based on the observation that the ambiguity of the norm sequence $\mathrm{\left(\vert s_{n}^{i}\vert^{2}\right)_{n}}$, which is only well-determined up to scalar multiplication, can be abolished if one considers the normalized sequence given by $\mathrm{\Vert s_{n}^{i}\Vert^{-2}\vert s_{n}^{i}\vert^{2}}$. Here $\mathrm{\Vert s_{n}^{i}\Vert^{2}}$ denotes the fiber integral of the function $\mathrm{\vert s_{n}^{i}\vert^{2}}$ over $\mathrm{\pi_{y}\coloneqq \pi^{-1}\left(y\right)}$. The collection $\mathrm{\left\{\bm{\nu}^{\mathfrak{U}}_{n}\right\}_{n}}$ associated to the tame collection $\mathrm{\left\{s_{n}^{i}\right\}_{i}}$ is then given by $$\mathrm{\bm{\nu}_{n}^{i}\left(y\right)\left(\mathbf{A}\right)\coloneqq \int_{\mathbf{A}}\Vert s_{n}^{i}\Vert^{-2}\vert s_{n}^{i}\vert^{2}\,d\,[\pi_{y}]\text{ for }\mathbf{A}\text{ measurable and }y\in \mathbf{U}_{i}\cap \mathbf{Y}_{0}}$$ where $\mathrm{\int_{\mathbf{A}}d\,[\pi_{y}]}$ is defined to be the integral over $\mathrm{\mathbf{A}\subset \pi^{-1}\left(y\right)}$ of the restriction $\mathrm{\omega^{dim_{\mathbb{C}}\,\pi_{y}}\vert\pi_{y}}$ with a certain multiplicity\footnote{All relevant details of the theory of fiber integration can be found in \cite{Kin}.}. Since the dimension $\mathrm{k_{y}=dim_{\mathbb{C}}\,\pi^{-1}\left(y\right)}$ of a fiber $\mathrm{\pi^{-1}\left(y\right)}$ for $\mathrm{y\in \mathbf{Y}}$ can change as $\mathrm{y}$ moves in $\mathrm{\mathbf{Y}}$ and since the construction of each $\mathrm{\bm{\nu}_{n}^{i}\left(y\right)}$ involves $\mathrm{k_{y}}$, one can not expect that $\mathrm{\bm{\nu}_{n}^{\mathfrak{U}}}$ defines a uniform object over the full base $\mathrm{\mathbf{Y}}$. However, it is possible to find a Zariski dense subset $\mathrm{\mathbf{Y}_{0}\subset \mathbf{Y}}$, so that $\mathrm{\pi^{-1}\left(y\right)}$ for $\mathrm{y\in \mathbf{Y}_{0}}$ are purely $\mathrm{k}$-dimensional complex varieties (set $\mathrm{\mathbf{X}_{0}\coloneqq \pi^{-1}\left(\mathbf{Y}_{0}\right)}$). Over this set, it is reasonable to examine the convergence properties of the measure sequence $\mathrm{\bm{\nu}_{n}^{i}\left(\cdot\right)}$ for each $\mathrm{i\in I}$. More precisely, if $\mathrm{f\in\mathcal{C}^{0}\!\left(\mathbf{X}\right)}$ is a continuous function and if $\mathrm{f_{red}}$ denotes the {\itshape{reduced function}} on the base $\mathrm{\mathbf{Y}}$ given by the restriction $\mathrm{\overline{f}\vert \mu^{-1}\left(\xi\right)}$ of the averaged function $\mathrm{\overline{f}\left(x\right)=\int_{T}f\left(t.x\right)\,d\nu_{T}}$ \label{Notation Averaged Function} with respect to the {\scshape{Haar}} measure $\mathrm{\nu_{T}}$, we prove the following theorem.

\begin{theo_3*}\label{Theorem Uniform Convergence of the Tame Measure Sequence} \textnormal{[\scshape{Uniform Convergence of the Tame Measure Sequence}]}\vspace{0.1 cm}\\
For for every tame collection $\mathrm{\left\{s_{n}^{i}\right\}_{i}}$ there exists a finite cover $\mathrm{\mathfrak{U}}$ of $\mathrm{\mathbf{Y}}$ with $\mathrm{\mathbf{U}_{i}\subset}$ $\mathrm{\pi\left(\mathbf{X}^{i}\right)}$ so that the collection of fiber probability measures $\mathrm{\left\{\bm{\nu}^{\mathfrak{U}}_{n}\right\}}$ associated to $\mathrm{\left\{s^{i}_{n}\right\}_{i}}$ converges uniformly on $\mathrm{\mathbf{Y}_{0}}$ to the fiber Dirac measure of $\mathrm{\mu^{-1}\left(\xi\right)\cap \mathbf{X}_{0}}$, i.e.\ for every $\mathrm{i\in I}$ and every $\mathrm{f\in \mathcal{C}^{0}\left(\mathbf{X}\right)}$, we have $$\begin{xy}
 \xymatrix{\mathrm{\left(y\mapsto \int_{\pi^{-1}\left(y\right)}f\, d\bm{\nu}_{n}^{i}\left(y\right)\right)}\ar[rr]^{\mathrm{\,\,\,\,\,}}&&\mathrm{f_{red}\text{ uniformly on }\mathbf{U}_{i}\cap \mathbf{Y}_{0}.\label{Notation Reduced Function}}
 }
\end{xy}$$
\end{theo_3*}

Furthermore, if $\mathrm{\left\{D^{\mathfrak{U}}_{n}\right\}=\left\{D_{n}^{i}\right\}_{i}}$ denotes the corresponding collection of cumulative fiber probability densities given by $$\mathrm{D^{i}_{n}\left(y,\cdot \right)\!:\!t\mapsto D^{i}_{n}\left(y,t \right)\coloneqq \int_{\left\{ \frac{\vert s_{n}^{i}\vert^{2}}{\left\Vert s_{n}^{i}\right\Vert ^{2}}\geq t\right\} \cap\pi^{-1}\left(y\right)}d\,[\pi_{y}]\text{ for }y\in\mathbf{ U}_{i}\cap \mathbf{Y}_{0}}$$ the following convergence result can be proved.
\begin{theo_4*}\label{Theorem Uniform Convergence of the Tame Distribution Sequence} \textnormal{[\scshape{Uniform Convergence of the Tame Distribution Sequence}]}\vspace{0.1 cm}\\
For every $\mathrm{t\in \mathbb{R}}$ and every tame collection $\mathrm{\left\{s_{n}^{i}\right\}_{i}}$ there exists a finite cover $\mathrm{\mathfrak{U}}$ of $\mathrm{\mathbf{Y}}$ with $\mathrm{\mathbf{U}_{i}\subset}$ $\mathrm{\pi\left(\mathbf{X}^{i}\right)}$ so that the collection of cumulative fiber probability densities $\mathrm{\left\{D^{\mathfrak{U}}_{n}\left(\cdot,t\right)\right\}}$ associated to $\mathrm{\left\{s^{i}_{n}\right\}_{i}}$ converges uniformly on $\mathrm{\mathbf{Y}_{0}}$ to the zero function on $\mathrm{\mathbf{Y}_{0}}$, i.e.\ for every $\mathrm{i\in I}$ we have $$\begin{xy}
 \xymatrix{\mathrm{\left(y\mapsto D_{n}^{i}\left(y,t\right)\right)}\ar[rr]^{\mathrm{\,\,\,\,\,}}&&\mathrm{0\text{ uniformly on }\mathbf{U}_{i}\cap \mathbf{Y}_{0}\subset \pi\left(\mathbf{X}^{i}\right).}
 }
\end{xy}$$
\end{theo_4*}

In this sense, we have shown that each tame sequence of eigensections $\mathrm{\left(s_{n}^{i}\right)_{n}}$, attached to a prescribed weight $\mathrm{\xi}$, gives rise to a sequence of fiber measures over $\mathrm{\mathbf{Y}_{0}}$, which independently of the choice of $\mathrm{\left(s_{n}^{i}\right)_{n}}$, localizes uniformly along the critical $\mathrm{\mu^{-1}\left(\xi\right)}$. The localization property of the measure sequence $\mathrm{\bm{\nu}_{n}^{\mathfrak{U}}\left(\cdot\right)}$ attached to the tame collection $\mathrm{\left\{\left(s^{i}_{n}\right)\right\}_{i}}$ is a consequence of the fact that the corresponding sequences of strictly plurisubharmonic functions $\mathrm{\varrho^{i}_{n}\!:\!\mathbf{X}^{i}\rightarrow \mathbb{R}}$ given by $$\mathrm{\varrho^{i}_{n}\coloneqq -\frac{1}{n}\mathbf{log}\,\vert s_{n}^{i}\vert^{2}}$$ converge (along with all derivatives) uniformly on compact sets to a strictly plurisubharmonic function $\mathrm{\varrho^{i}}$. It is crucial to note that the restriction $\mathrm{\varrho^{i}\vert \pi^{-1}\left(y\right)}$ of $\mathrm{\varrho^{i}}$ to each fiber of the projection $\mathrm{\pi\!:\!\mathbf{X}^{ss}_{\xi}\rightarrow \mathbf{Y}}$ takes on its $\mathrm{T}$-invariant minimum along the uniquely defined $\mathrm{T}$-orbit $\mathrm{T.x_{y}\subset \pi^{-1}\left(y\right)}$ given by $$\mathrm{T.x_{y}=\mu^{-1}\left(\xi\right)\cap \pi^{-1}\left(y\right).}$$ Using this observation, it is then possible to deduce estimates of the magnitude of $\mathrm{\varrho^{i}_{n}}$ and hence of $\mathrm{e^{-n\,\varrho^{i}_{n}}=\vert s_{n}^{i}\vert^{2}}$ outside a $\mathrm{T}$-invariant, relatively compact tube of $\mathrm{\mu^{-1}\left(\xi\right)}$ as $\mathrm{n}$ tends to infinity (cf.\ {\bf{Theorem 2}}).

Apart form the determination of the asymptotic behavior of $\mathrm{\varrho^{i}_{n}}$, which plays an essential role in the proof of {\bf{Theorem 3, 4}}, it is also necessary to deal with the following issue: The fact that $\mathrm{\bm{\nu}_{n}^{\mathfrak{U}}\left(\cdot\right)}$ is defined over a non-compact base makes a direct application of the standard convergence theorems of measure theory considerably more difficult. Therefore, a good portion of the proof of the above convergence theorem will be devoted to resolving this issue by constructing a new quotient $\mathrm{\Pi\!:\!\widetilde{\bf{\,X\,}}\rightarrow \widetilde{\bf{\,Y\,}}}$ which extends the restricted quotient $\mathrm{\pi\!:\!\mathbf{X}_{0}\rightarrow \mathbf{Y}_{0}\label{Notation Inverse Image of Y_0 Under The Projection Map}}$, so that the following diagram commutes
 
$$\begin{xy}
\xymatrix{
\mathrm{\mathbf{X}_{0}}\ar@{^{(}->}[rrr]\ar[d]_{\mathrm{\pi\vert \mathbf{X}_{0}}} &&& \mathrm{\widetilde{\bf{\,X\,}}}\ar[d]^{\mathrm{\Pi}}\\
\mathrm{\mathbf{Y}_{0}}\ar@{^{(}->}[rrr]&&& \mathrm{\widetilde{\bf{\,Y\,}}}\\
}
\end{xy}
$$ and so that all fibers of $\mathrm{\Pi}$ are compact and of pure dimension $\mathrm{k}$. 

The construction of this equivariant, dimensional-theoretical {\itshape{flattening}} $\mathrm{\Pi\!:\!\widetilde{\bf{\,X\,}}\rightarrow \widetilde{\bf{\,Y\,}}}$, which is based on results of {\scshape{D.\! Barlet}}, allows us to realize the fiber measure sequence $\mathrm{\bm{\nu}_{n}^{\mathfrak{U}}\left(\cdot\right)}$ as a restriction of a measure sequence defined on $\mathcal{\widetilde{\bf{\,X\,}}}$. Since $\mathrm{\widetilde{\bf{\,X\,}}}$ and all its fibers are compact, it is then possible to show the above mentioned convergence properties of $\mathrm{\left\{\bm{\nu}_{n}^{\mathfrak{U}}\right\}_{n}}$ by applying results concerning the continuity of fiber integration. 

In the last part of this work, we return to the initial sequence of $\mathrm{\mathbb{T}}$-eigensections $\mathrm{\left(s_{n}\right)_{n}}$ and examine the convergence properties of the fiber measure sequence induced by $\mathrm{\vert s_{n}\vert^{2}\Vert s_{n}\Vert^{-2}}$. Unlike in the tame case, we first have to face the problem that the function $\mathrm{\Vert s_{n}\Vert^{-2}}$ is only well defined for all $\mathrm{y\in \mathbf{Y}}$ with the property $\mathrm{s_{n}\vert\pi^{-1}\left(y\right)\not\equiv0}$. The task of defining a maximal, $\mathrm{n}$-stable set in $\mathrm{\mathbf{Y}}$, on which we can consistently write down a measure sequence $\mathrm{\left(\bm{\nu}_{n}\right)_{n}}$ for all $\mathrm{n\in \mathbb{N}}$ big enough attached to $\mathrm{\left(s_{n}\right)_{n}}$, naturally leads to the notion of a {\itshape{removable singularity}}. Having available this concept, which is based on the idea that certain singularities can be "divided out" by multiplying $\mathrm{s_{n}}$ with locally defined, invariant holomorphic functions, we are able to uniquely extend the measure sequence beyond its original set of definition for all $\mathrm{n\geq N_{0}}$ on an open subset of $\mathrm{\mathbf{Y}_{0}}$ which is given by $\mathrm{\mathbf{R}_{N_{0}}\cap \mathbf{Y}_{0}}$. Here, $\mathrm{\mathbf{R}_{N_{0}}}$ denotes the open subset of all singularities $\mathrm{y\in \mathbf{Y}}$ which are removable for all $\mathrm{n\geq N_{0}}$. Once the maximal set of definition given by $\mathrm{\mathbf{R}_{N_{0}}\cap \mathbf{Y}_{0}}$ is found, we continue our discussion by analyzing the convergence properties of the measure sequence $\mathrm{\left(\bm{\nu}_{n}\right)_{n}}$ over $\mathrm{\mathbf{Y}_{0}\cap \mathbf{R}_{N_{0}}}$ and obtain the following two results.

\begin{theo_5.b*}\label{theo_5.b*}\textnormal{[\scshape{Uniform Convergence of the Initial Distribution Sequence}]}\vspace{0.1 cm}\\
For fixed $\mathrm{t\in \mathbb{R}}$ the sequence $\mathrm{\left(D_{n}\left(\cdot,t\right)\right)_{n}}$ converges uniformly on $\mathrm{\mathbf{Y}_{0}\cap\mathbf{R}_{N_{0}}}$ to the zero function.
\end{theo_5.b*}

\begin{theo_6.b*}\label{theo_6.b*}\textnormal{[\scshape{Uniform Convergence of the Initial Measure Sequence}]}\vspace{0.10 cm}\\ 
For $\mathrm{f\in \mathcal{C}^{0}\!\left(\mathbf{X}\right)}$ the sequence $$\mathrm{\left(y\mapsto \int_{\pi^{-1}\left(y\right)}f\,d\bm{\nu}_{n}\left(y\right)\right)_{n}}$$ converges uniformly over $\mathrm{\mathbf{Y}_{0}\cap \mathbf{R}_{N_{0}}}$ to the reduced function $\mathrm{f_{red}}$.
\end{theo_6.b*} Since the deviation of the initial measure sequence $\mathrm{\left(\bm{\nu}_{n}\right)}$ induced by $\mathrm{\left(s_{n}\right)_{n}}$ from a tame, locally defined measure sequence $\mathrm{\left(\bm{\nu}^{i}_{n}\right)_{n}}$ is completely described by the locally defined sequence of functions $\mathrm{\left(\triangle^{i}_{n}\right)_{n}}$ given by $\mathrm{s_{n}=\triangle^{i}_{n}\cdot s^{i}_{n}}$, it is not surprising that the proof of both theorems is based on technics we already used before when proving {\bf{Theorem 3, 4}}. These are combined with certain analytic facts about the growth properties of $\mathrm{\triangle^{i}_{n}}$ as $\mathrm{\varrho^{i}\rightarrow \infty}$ on $\mathrm{\pi^{-1}\left(y\right)}$ and as $\mathrm{n\rightarrow \infty}$.

\vspace{3 cm}

{\setlength{\parindent}{0cm}{\bf{Acknowledgements.}} The author would like to express his gratitude to his advisor Professor {\scshape{Alan Huckleberry}} who initiated this project and provided permanent support and advise whenever needed from its very beginning.} 

Furthermore, the author is indebted to the constructive remarks of Prof.\ Dr.\ {\scshape{Heinzner}}, Prof.\ Dr.\ {\scshape{Winkelmann}}, Prof.\ Dr.\ {\scshape{Wurzbacher}} and Dr.\ {\scshape{Tsanov}}.
 
Finally, the author would like to thank the {\scshape{Studienstiftung des Deutschen Volkes}} for financial support which included a scholarship for the author's years as a student and PhD candidate at {\scshape{Ruhr-Universit{\"a}t Bochum}} and for his time as a visiting student at {\scshape{Massachusetts Institute of Technology}}. Also in this regard, the author owes special thanks to Prof.\ {\scshape{Vogan}} who generously supported the research period at MIT.    \newpage

\newpage
\section{Existence of Tame Sequences}\label{Existence of Tame Sequences}

The aim of this section is to prove the following theorem.

\begin{theo_1*} \textnormal{[\scshape{Existence of Tame Sequences}]}\vspace{0.1 cm}\\ 
If $\mathrm{\mathbf{L}\rightarrow \mathbf{X}}$ is as above and $\mathrm{\xi\in Im\left(\mu\right)}$, then we can find a finite cover $$\mathrm{\left\{\mathbf{X}^{i}\right\}_{i\in I}\text{ of }\mathbf{X}^{ss}_{\xi}}$$ consisting of open, $\mathrm{\pi}$-saturated subsets and a finite collection $$\mathrm{\left\{\left( s_{n}^{i} \right)_{n}\right\}_{i\in I},\text{ where }s_{n}^{i}\in H^{0}(\mathbf{X},\mathbf{L}^{n})\label{Notation i-th tame sequence}}$$ of sequences consisting of $\mathrm{\mathbb{T}}$-eigensections such that the following properties are fulfilled:
\begin{enumerate}  
\item $\mathrm{\frac{\xi^{i}_{n}}{n}\rightarrow \xi}$ where $\mathrm{\vert \xi^{i}_{n}-n\,\xi\vert\in \mathcal{O}\left(1\right)}$ for each $\mathrm{i}$\label{Notation Approximating Sequence of Weight Vectors}.\vspace{-0.2 cm}
\item $\mathrm{\mathbf{X}^{i}\subset \mathbf{X}(s_{n}^{i})}$\label{Notation i-th set} for all $\mathrm{n}$ big enough.\vspace{-0.2 cm}
\item $\mathrm{\mathbf{X}^{i}\cap \mathbb{T}.\mu^{-1}\left(\xi\right)=\mathbf{X}^{i}\cap \mathbb{T}.\mu^{-1}\left(\frac{\xi^{i}_{n}}{n}\right)}$ for all $\mathrm{n\geq N_{0}}$.
\end{enumerate} 
\end{theo_1*}
In the sequel, we will refer to a collection $\mathrm{\left\{\left(s_{n}^{i}\right)\right\}_{i}}$ of sequences of $\mathrm{\mathbb{T}}$-eigensections $\mathrm{s_{n}^{i}\in H^{0}}$ $\mathrm{\left(\mathbf{X},\mathbf{L}^{n}\right)}$ with the aforementioned properties as {\itshape{tame}}. Before we proceed with the proof of {\bf{Theorem 1}}, let us consider an example of a tame collection.
\begin{example}\label{Example First Example}  \textnormal{Let $\mathrm{\mathbf{X}=\mathbb{C}\mathbb{P}^{1}\times \mathbb{C}\mathbb{P}^{1}}$ with K\"aher form $\mathrm{\omega=p_{1}^{*}\omega_{FS}+p_{2}^{*}\omega_{FS}}$ where $\mathrm{p_{i}\!:\!}$ $\mathrm{\mathbf{X}}$ $\mathrm{\rightarrow}$ $\mathrm{\mathbb{C}\mathbb{P}^{1}}$ denotes the $\mathrm{i}$-th projection, $\mathrm{i\in\left\{0,1\right\}}$. Equip $\mathrm{\mathbf{X}}$ with the $\mathrm{\mathbb{T}\cong\mathbb{C}^{*}}$-action given by $$\mathrm{t.\left(\left[z_{0}\!:\!z_{1}\right],\left[\zeta_{0}\!:\!\zeta_{1}\right]\right)=\left(\left[t\, z_{0}\!:\!z_{1}\right],\left[t\, \zeta_{0}\!:\!\zeta_{1}\right]\right)}$$ and consider the $\mathrm{\mathbb{T}}$-linearization on $\mathrm{\mathbf{L}=p_{0}^{*}\mathcal{O}_{\mathbb{C}\mathbb{P}^{1}}\left(1\right)\otimes p_{1}^{*}\mathcal{O}_{\mathbb{C}\mathbb{P}^{1}}\left(1\right)\rightarrow \mathbf{X}}$ induced by the $\mathrm{\mathbb{T}}$-action on $\mathrm{\mathbb{C}\mathbb{P}^{4}}$ given by $$\mathrm{t.[u_{0}\!:\!u_{1}\!:\!u_{2}\!:\!u_{3}\!:\!v]=[t\,u_{0}\!:\!u_{1}\!:\!u_{2}\!:\!t^{-1}\,u_{3}\!:\!v].}$$ Here we view $\mathrm{\mathbf{L}}$ embedded in $\mathrm{\mathcal{O}_{\mathbb{C}\mathbb{P}^{3}}\left(1\right)\cong\mathbb{C}\mathbb{P}^{4}\setminus \{[0\!:\!0\!:\!0\!:\!0\!:\!1]\}}$ realized as the cone over $$\mathrm{\mathbf{X}\cong\left\{[u]\in \mathbb{C}\mathbb{P}^{4}\!:\!\, u_{0}u_{3}-u_{1}u_{2}=0, v=0\right\} \subset \left\{[u]\in \mathbb{C}\mathbb{P}^{4}\!:\!\,v=0\right\}\cong \mathbb{C}\mathbb{P}^{3}.}$$ Furthermore, let $\mathrm{H}$ denote the standard hermitian metric on $\mathrm{\mathcal{O}_{\mathbb{C}\mathbb{P}^{3}}\left(1\right)}$ which is given by $\mathrm{H([u\!:\!v])=\Vert u\Vert^{-2}\vert v\vert^{2}}$ for $\mathrm{u=\left(u_{0},u_{1},u_{2},u_{3}\right)\neq 0}$ and which is $\mathrm{T}$-invariant with respect to the above action. Define $\mathrm{h\coloneqq H\vert \mathbf{X}}$.}

\textnormal{If $\mathrm{\xi=0}$ and $\mathrm{\mu}$ denotes the moment map induced by the hermitian bundle metric $\mathrm{h}$, it follows that $$\mathrm{\mu^{-1}\left(0\right)=\left\{\left([z_{0}\!:\!z_{1}],[\zeta_{0}\!:\!\zeta_{1}]\right)\!:\!\,z_{0}\zeta_{0}-z_{1}\zeta_{1}=0\right\}\subset \mathbb{C}\mathbb{P}^{1}\times \mathbb{C}\mathbb{P}^{1}.}$$ A calculation shows that the only eigensections $\mathrm{s_{n}\in H^{0}\left(\mathbf{X},\mathbf{L}^{n}\right)}$ which can be part of a tame sequences in the sense of {\bf{Theorem 1}} are given by the collection $$\mathrm{\left\{ s_{n,k}=z_{0}^{n-k}z_{1}^{k}\zeta_{0}^{k}\zeta_{1}^{n-k},\,0\leq k\leq n\right\} \subset H^{0}\left(\mathbf{X},\mathbf{L}^{n}\right).}$$ A tame collection is for example given by $$\mathrm{\Big\{\!\left(z_{0}^{n}\zeta_{1}^{n}\right)_{n},\left(z_{1}^{n}\zeta_{0}^{n}\right)_{n}\!\Big\}.\,\, \boldsymbol{\Box}}$$}
\end{example}

For the proof of the existence of tame sequences, we will first restate an important fact about the geometry of an arbitrary $\mathrm{\mathbb{T}}$-orbit closure $\mathrm{cl\left(\mathbb{T}.x\right)}$ \label{Notation T-Orbit Closure} where $\mathrm{x\in\mathbf{X}}$:

Every image $\mathrm{\mu(cl(\mathbb{T}.x))\subset \mathfrak{t}^{*}}$ of an arbitrary orbit closure $\mathrm{ cl(\mathbb{T}.x)\subset \mathbf{X}}$ is the convex hull of the image of finitely many fixed points $\mathrm{x_{i}\in {\bf{Fix}}^{\mathbb{T}}\coloneqq \{x\in \mathbf{X}\!:\!\,t.x}$ \label{Notation Fix Point Set Of T-Action} $\mathrm{=x}$ $\mathrm{\text{for all }t\in \mathbb{T}\}}$ ({cf.\ }\cite{Ati}), i.e.\:
\begin{equation}\label{Equation Convex Hull}
\mathrm{\mu(cl(\mathbb{T}.x))=\mathfrak{Conv}\,\left\{\mu({\bf{S}}_{x})\right\}}
\end{equation} where $$\mathrm{{\bf{S}}_{x}=\left\{\sigma_{x,j_{x}}\in cl\left(\mathbb{T}.x\right)\cap {\bf{Fix}}^{\mathbb{T}},\, 1\leq j_{x}\leq m_{x}\right\}\subset {\bf{Fix}}^{\mathbb{T}}}$$ \label{Notation T-Fixed Points Contained In The Closure Of T.x} is a finite set and $\mathrm{\mathfrak{Conv}\,\left\{\mu\left({\bf{S}}_{x}\right)\right\}}$ \label{Notation Convex Hull Of A Subset} denotes the convex hull of the corresponding image $\mathrm{\mu\left({\bf{S}}_{x}\right)}$.

Consider the decomposition of the $\mathrm{\mathbb{T}}$-representation space $\mathrm{H^{0}(\mathbf{X},\mathbf{L})}$ in its eigenspaces $\mathrm{\mathbb{C}s_{k}}$ where $\mathrm{1\leq k\leq m\coloneqq dim_{\mathbb{C}}\,H^{0}(\mathbf{X},\mathbf{L})}$. For any non-empty $\mathrm{J\subset \left\{1,\dots, m\right\}}$ let $\mathrm{\mathcal{S}_{J}\coloneqq \left\{s_{j}\!:\!\,j\in J\right\}}$ $\mathrm{\subset H^{0}\left(\mathbf{X},\mathbf{L}\right)}$ \label{Notation Set Of T-Eigensections Indexed By J} and let $\mathrm{{\bm{\mathfrak{S}}}_{J}\subset \mathfrak{t}^{*}}$ \label{Notation Set Of All Characters Corresponding To S_x} be the set of all characters corresponding to the set $\mathrm{\mathcal{S}_{J}}$.

We introduce the following notation: For a non-empty $\mathrm{J\subset \{1,\dots,m\}}$, let $$\mathrm{\mathbf{M}_{J}\coloneqq\left\{x\in \mu^{-1}\left(\xi\right)\!:\!\, s_{j}\left(x\right)\neq 0\,\text{ for all } j\in J, s_{j}\left(x\right)= 0\,\text{ for all } j\in\complement\, J\right\}. \label{Notation Definition of M_J}}$$ Note that $\mathrm{\mathbf{M}_{J}\subset \mu^{-1}\left(\xi\right)}$ is a $\mathrm{T}$-invariant, open subset which might be empty for certain non-empty subindices $\mathrm{J\subset \{1,}$ $\mathrm{\dots,m\}}$. Furthermore, the collection $\mathrm{\left\{\mathbf{M}_{J}\right\}_{J\neq \varnothing}}$ is finite and cover $\mathrm{\mu^{-1}\left(\xi\right)}$. The first claim follows by the fact that $\mathrm{dim_{\mathbb{C}}\,H^{0}\left(\mathbf{X},\mathbf{L}\right)<\infty}$ and the second claim is a direct consequence of the assumption that $\mathrm{\mathbf{L}}$ is base point free.

The next lemma establishes a connection between the geometry of the image of $\mathrm{\mathbb{T}.x}$ under the moment map $\mathrm{\mu}$ where $\mathrm{x\in \mathbf{M}_{J}}$ and the convex set $\mathrm{\mathfrak{Conv}\,{\bm{\mathfrak{S}}}_{J}}$.

\begin{lemma}\label{Lemma Convex Hull Moment Map}
Let $\mathrm{x\in \mathbf{M}_{J}}$, then it follows that $$\mathrm{\mu(cl(\mathbb{T}.x))=\mathfrak{Conv}\,{\bm{\mathfrak{S}}}_{J}}$$ and, if $\mathrm{x}$ is not a $\mathrm{\mathbb{T}}$-fixed point, $$\mathrm{\xi\in\mu\left(\mathbb{T}.x\right)=Relint\,\mathfrak{Conv}\,{\bm{\mathfrak{S}}}_{J}}$$ where $\mathrm{Relint\,\mathfrak{Conv}\,{\bm{\mathfrak{S}}}_{J}}$ \label{Notation Relative Interior} denotes the relative interior of $\mathrm{{\bm{\mathfrak{S}}}_{J}\subset \mathfrak{t}^{*}}$.
\end{lemma}
\begin {proof} First of all note that since the image of $\mathrm{cl\left(\mathbb{T}.x\right)}$ is known to be the convex hull of the image $\mathrm{\mu\left({\bf{S}}_{x}\right)}$ where $\mathrm{{\bf{S}}_{x}}$ are the $\mathrm{\mathbb{T}}$-fixed points in $\mathrm{cl\left(\mathbb{T}.x\right)}$, the first claim follows as soon as we have shown that $\mathrm{\mu\left(\sigma_{x,j_{x}}\right)\in {\bm{\mathfrak{S}}}_{J}}$ for all $\mathrm{\sigma_{x,j_{x}}\in {\bf{S}}_{x}}$.

Let $\mathrm{\sigma_{x,j_{x}}\in {\bf{S}}_{x}}$. Since $\mathrm{\mathbf{L}}$ is assumed to be base point free there exists at least one $\mathrm{\mathbb{T}}$-eigensection $\mathrm{s}$ which does not vanish at $\mathrm{\sigma_{x,j_{x}}}$. As $\mathrm{x\in \mathbf{M}_{J}}$, such an $\mathrm{s}$ is necessarily given by $\mathrm{s=s_{j}}$ for $\mathrm{j\in J}$. If $\mathrm{\xi_{j}}$ denotes its corresponding character, we deduce $\mathrm{\mu\left(\sigma_{x,j_{j}}\right)=\xi_{j}}$ which is a direct consequence of the following reasoning: Let $\mathrm{D\!:\!\mathcal{A}^{0}\left(\mathbf{L}\right)\rightarrow\mathcal{A}^{1}\left(\mathbf{L}\right)}$ be the uniquely defined, hermitian connection associated to $\mathrm{h}$ and recall that we have the formula ({cf.\! }\cite{Gu-St})
\begin{equation}\label{Equation Formula Connection Moment Map}
\mathrm{D_{\mathbf{X}_{\eta}}s+2\,\sqrt{-1}\mu^{\eta}s=\frac{d}{dt}_{t=0}{\bf{exp}}(\sqrt{-1}\,t\,\eta).s.\text{ for all }\eta\in \mathfrak{t}.}
\end{equation} If we apply formula \ref{Equation Formula Connection Moment Map} to $\mathrm{s=s_{j}}$ at the fixed point $\mathrm{\sigma_{x,j_{x}}}$, we deduce $\mathrm{\mu^{\eta}(\sigma_{x,j_{x}})=\xi_{j}(\eta)}$ for all $\mathrm{\eta \in \mathfrak{t}}$. So it follows $\mathrm{\mu(\sigma_{x,j_{x}})=\xi_{j}}$ and hence $$\mathrm{\mu\left(cl\left(\mathbb{T}.x\right)\right)= \mathfrak{Conv}\,{\bm{\mathfrak{S}}}_{J}}$$ as claimed.

The second claim follows from the fact that $\mathrm{\xi=\mu\left(x\right)}$ for $\mathrm{x\in \mathbf{M}_{J}\subset \mu^{-1}\left(\xi\right)}$ and the general fact that $\mathrm{cl\left(f\left(\mathbf{A}\right)\right)=f\left(cl\left(\mathbf{A}\right)\right)}$ for a continuous map $\mathrm{f\!:\!\mathbf{X}\rightarrow \mathbf{Y}}$ where $\mathrm{\mathbf{X}}$ is compact and $\mathrm{\mathbf{A}\subset \mathbf{X}}$.
\end{proof}

For the proof of the next proposition, which will be crucial for the existence of tame sequences, we need the following technical lemma.

\begin{lemma}\label{Lemma Convex Polytope}
Let $\mathrm{{\bm{\mathfrak{P}}}=\mathfrak{Conv}\,\left\{(q_{1},\dots,q_{m})\right\}}$, $\mathrm{q_{1},\cdots,q_{m}\in \mathbb{R}^{n}}$ and let 
$$\mathrm{{\bm{\mathfrak{P}}}_{n^{-1}\mathbb{N}^{0}}\coloneqq \mathfrak{Conv}_{n^{-1}\mathbb{N}^{0}}\left\{q_{1},\dots,q_{m}\right\}\coloneqq \Bigg\{ p=\sum_{j}\nu_{j}q_{j}\!:\!\,\sum_{j}\nu_{j}=1,\,\nu_{j}\in n^{-1}\mathbb{N}^{0}\Bigg\}.}$$
If $\mathrm{\xi\in{\bm{\mathfrak{P}}}}$, then there exists a sequence $\mathrm{\xi_{n}=\sum_{j}\nu_{j,n}q_{j}}$ so that the following conditions are fulfilled:
\begin{enumerate}  
\item $\mathrm{\xi_{n}\rightarrow \xi}$\\[-0.6 cm]
\item $\mathrm{\mathrm{\left|\xi_{n}-n\,\xi\right|\in\mathcal{O}\left(1\right)}}$\\[-0.6 cm]
\item $\mathrm{\xi_{n}\subset {\bm{\mathfrak{P}}}_{n^{-1}\mathbb{N}^{0}}}$ for all $\mathrm{n}$ big enough.\\[-0.6 cm]
\item There exists $\mathrm{N_{0}\in \mathbb{N}}$ and a partition $\mathrm{J=J_{0}\cup J_{1}}$, $\mathrm{J_{1}\neq \varnothing}$ of $\mathrm{J=\left\{1,\dots,m\right\}}$ so that $\mathrm{\nu_{j,n}=0}$ for all $\mathrm{j\in J_{0}}$, $\mathrm{n\geq N_{0}}$ and $\mathrm{\nu_{j,n}>0}$ for all $\mathrm{j\in J_{1}}$, $\mathrm{n\geq N_{0}}$.\\[-0.6 cm]
\item The sequences $\mathrm{\left(\nu_{j,n}\right)_{n}}$ are convergent so that the limits are strictly positive for all $\mathrm{j\in J_{1}}$.
\end{enumerate}
\end{lemma}
\begin{proof} Let $\mathrm{{\bm{\mathfrak{P}}}_{0}\subset \mathbb{R}^{m}}$ denote the convex hull of the standard basis $\mathrm{\{e_{1},\dots,e_{m}\}}$ in $\mathrm{\mathbb{R}^{m}}$, then 
\begin{equation}\label{Equation Image Given By Matrix}
\mathrm{{\bm{\mathfrak{P}}}\coloneqq \mathfrak{Conv}\,\left\{q_{1},\dots,q_{m}\right\}={\bf{M}}({\bm{\mathfrak{P}}}_{0})}
\end{equation} for the matrix $\mathrm{\bf{M}}$ whose columns are given by $\mathrm{(q_{1},\dots,q_{m})}$. By \ref{Equation Image Given By Matrix} we find $\mathrm{\nu\in {\bm{\mathfrak{P}}}_{0}}$ with $\mathrm{\xi={\bf{M}}\left(\nu\right)}$. Without restriction of generality, we can assume that 
\begin{equation}\label{Equation In Order To Guarantee}
\mathrm{\nu\in Relint\,\mathfrak{Conv}\,\{e_{\ell},\dots,e_{m}\}},
\end{equation} where $\mathrm{1\leq \ell\leq m-1}$ because otherwise, $\mathrm{\xi=q_{j_{0}}}$ for one $\mathrm{1\leq j_{0}\leq m}$ and the claim follows immediately by choosing $\mathrm{\nu_{j_{0},n}=1}$ and $\mathrm{\nu_{j,n}=0}$ for all $\mathrm{j\neq j_{0}}$.

Now, it is direct to see that we can  choose a sequence $\mathrm{\left( \nu_{n}\right)_{n}}$ in $\mathrm{\mathfrak{Conv}_{n^{-1}\mathbb{N}^{0}}\,\{e_{\ell},}$ $\mathrm{\dots,e_{m}\}}$ so that $\mathrm{\nu_{n}\rightarrow \nu}$ where all limits are strictly positive. Furthermore, we can always guarantee that $\mathrm{\left|\nu_{n}-n\,\nu\right|}$ $\mathrm{\in\mathcal{O}\left(1\right)}$. Set $$\mathrm{\left(\xi_{n}\right)_{n}=\left( {\bf{M}}(\nu_{n})\right)_{n}\subset {\bm{\mathfrak{P}}}_{n^{-1}\mathbb{N}^{0}}.}$$ 

The first claim follows by $\mathrm{\xi_{n}=\mathbf{M}\left(\nu_{n}\right)\rightarrow \mathbf{M}\left(\nu\right)=\xi}$, the second claim is a direct consequence of $\mathrm{\left|\xi_{n}-n\xi\right|\leq\left\Vert \mathbf{M}\right\Vert _{\infty}\left|\nu_{n}-n\nu\right|\in\mathcal{O}\left(1\right)}$ and the third, resp.\ fifth claim follows by construction. The fourth claim results from equation \ref{Equation In Order To Guarantee}:\ We have $\mathrm{\nu_{j,n}=0}$ for all $\mathrm{j}$ with $\mathrm{1\leq j \leq \ell-1}$ and $\mathrm{\nu_{j,n}>0}$ for all $\mathrm{j}$ with $\mathrm{\ell\leq j\leq m}$ and $\mathrm{n}$ big enough which follows by $\mathrm{\nu_{n}\rightarrow \nu \in Relint\, \mathfrak{Conv}}$ $\mathrm{\{e_{\ell},\dots,e_{m}\}}$. Hence, we have $\mathrm{J_{0}=\left\{1,\dots,\ell-1\right\}}$ and $\mathrm{J_{1}=\left\{\ell,\dots,m\right\}}$.
\end{proof}

The following proposition will be the essential step in order to prove {\bf{Theorem 1}}.

\begin{prop}\label{Proposition Existence of Tame Sequences} If $\mathrm{\mathbf{L}\rightarrow \mathbf{X}}$ is as above and $\mathrm{\xi\in Im(\mu)}$, then for each $\mathrm{x \in \mathbf{M}_{J}}$ there exists a sequence 
$$\mathrm{\left(s^{J}_{n}\right)_{n}, s^{J}_{n}\in H^{0}(\mathbf{X},\mathbf{L}^{n})}$$
of $\mathrm{\xi^{J}_{n}}$-eigensections with the following characteristic properties:

\begin{enumerate}  
\item $\mathrm{\frac{\xi^{J}_{n}}{n}\rightarrow \xi}$ where $\mathrm{\vert\xi^{J}_{n}-n\,\xi\vert\in \mathcal{O}\left(1\right)}$\\[-0.6 cm]
\item The set $\mathrm{\mathbf{X}(s^{J}_{n})  \coloneqq\{x\in \mathbf{X}\!:\! s_{n}^{J}(x)\neq 0\}}$ is independent of $\mathrm{n}$ for $\mathrm{n}$ big enough. \\[-0.6 cm]
\item $\mathrm{\mathbf{X}^{J}\coloneqq\pi^{-1}\left(\pi\left(\mathbf{M}_{J}\right)\right)\subset \mathbf{X}(s^{J}_{n})}$.\\[-0.6 cm]
\item $\mathrm{\mathbf{X}^{J} \cap \mathbb{T}.\mu^{-1}\left(\xi\right)=\mathbf{X}^{J} \cap \mathbb{T}.\mu^{-1}\left(\frac{\xi^{J}_{n}}{n}\right)}$ for all $\mathrm{n}$ big enough.

\end{enumerate} 
\end{prop}
\begin{proof} By {\bf{Lemma \ref{Lemma Convex Hull Moment Map}}} we know that $\mathrm{\mu\left(cl\left(\mathbb{T}.x\right)\right)=\mathfrak{Conv}\, {\bm{\mathfrak{S}}}_{J}}$. By applying {\bf{Lemma \ref{Lemma Convex Polytope}}} we find a sequence $$\mathrm{\left(\xi_{n}^{J,\prime}\right)_{n}\subset\mathfrak{Conv}\,{\bm{\mathfrak{S}}}_{J}}$$ so that $$\mathrm{\xi^{J,\prime}_{n}=\sum_{j=1}^{Card\,J}\nu_{j,n}^{J}\xi_{j}}$$ where $\mathrm{n\cdot\nu_{j,n}^{J}\in \mathbb{N}\cup\left\{0\right\}}$ for all $\mathrm{n\in \mathbb{N}}$ and $\mathrm{\left(\xi^{\prime}_{n}\right)_{n}}$ converges to $\mathrm{\xi}$ so that $\mathrm{\vert \xi^{J,\prime}_{n}-n\,\xi\vert\in \mathcal{O}\left(1\right)}$. Note that we have $\mathrm{n\cdot\sum_{j=1}^{Card J}\nu_{j,n}^{J}=n}$ and that $$\mathrm{\xi^{J,\prime}_{n}=\sum_{j=1}^{Card\,J}\nu_{j,n}^{J}\xi_{j}=\sum_{j\in J_{1}}\nu_{j,n}^{J}\xi_{j}}$$ (recall that, according to {\bf{Lemma \ref{Lemma Convex Polytope}}}, we have a partition $\mathrm{J=J_{0}\cup J_{1}}$, $\mathrm{J_{0}\neq\varnothing}$, where $\mathrm{\nu_{j,n}^{J}=0}$ for $\mathrm{j\in J_{0}}$). Consider $$\mathrm{s_{n}^{J}\coloneqq\prod_{j=1}^{Card\, J}s_{j}^{n\cdot\nu_{j,n}^{J}}\in H^{0}(\mathbf{X},\mathbf{L}^{n\cdot\sum_{j=1}^{Card J}\nu_{j,n}^{J}})=H^{0}(\mathbf{X},\mathbf{L}^{n}),}$$ which defines a sequence of eigensections whose associated weight vectors $\mathrm{\xi^{J}_{n}}$ approximate $\mathrm{\xi}$:
$$\mathrm{\frac{\xi_{n}^{J}}{n}=\sum_{j=1}^{Card\,J}\nu_{j,n}^{J}\xi_{j}\rightarrow \xi.}$$ Furthermore, by the fourth claim of {\bf{Lemma \ref{Lemma Convex Polytope}}}, we know that $\mathrm{\nu_{j,n}=0}$ for all $\mathrm{j\in J_{0}}$ and $\mathrm{\nu_{j,n}>0}$ for all $\mathrm{j\in J_{1}}$ for all $\mathrm{n}$ big enough. Therefore we deduce $$\mathrm{\mathbf{X}\left(s^{J}_{n}\right)=\complement\bigcup_{j\in J_{1}}\left\{s_{j}=0\right\}}$$ for all $\mathrm{n}$ big enough which proves the second claim.

The third property can be proved as follows: If $\mathrm{z\in \pi^{-1}\left(\pi\left(\mathbf{M}_{J}\right)\right)}$, then we have $\mathrm{cl\left(\mathbb{T}.z\right)}$ $\mathrm{\cap \mathbb{T}.\mathbf{M}_{J}\neq \varnothing}$. Since $\mathrm{cl\left(\mathbb{T}.z\right)}$ is $\mathrm{\mathbb{T}}$-invariant, it follows that $\mathrm{\mathbb{T}.\mathbf{M}_{J}\subset cl\left(\mathbb{T}.z\right)}$. As $\mathrm{s_{n}^{J}\vert \mathbb{T}.\mathbf{M}_{J}\not\equiv 0}$ for all $\mathrm{n}$ big enough, it follows that $\mathrm{s_{n}^{J}\vert cl\left(\mathbb{T}.z\right)\not\equiv 0}$ and hence $\mathrm{s_{n}^{J}\left(z\right)\neq 0}$ for all $\mathrm{n}$ big enough.

It remains to verify the fourth claim. For this, let $\mathrm{x\in \mathbf{X}^{J}\cap \mathbb{T}.\mu^{-1}\left(\xi\right)}$. Note that we can assume that $\mathrm{x}$ is not a $\mathbb{T}$ fixed point: If $\mathrm{x}$ is a $\mathbb{T}$ fixed point, then it follows $\mathrm{x\in \mathbf{M}_{J}\subset \mu^{-1}\left(\xi\right)}$ and hence $\mathrm{\xi_{j}=\xi}$ for all $\mathrm{j\in J}$. In particular, we find $\mathrm{\xi=n^{-1}\xi^{J}_{n}}$ for all $\mathrm{n}$ big enough, so the fourth claim is immediate. In the sequel, let $\mathrm{x\in \mathbf{X}^{J}\cap \mathbb{T}.\mu^{-1}\left(\xi\right)}$ be not a $\mathrm{\mathbb{T}}$ fixed point. Observe that $\mathrm{x\in \mathbb{T}.\mathbf{M}_{J}}$ and hence $\mathrm{\xi\in\mu\left(\mathbb{T}.x\right)}$ or $$\mathrm{\xi\in Relint\,\mu\left(cl\left(\mathbb{T}.x\right)\right)=Relint\,\mathfrak{Conv}\,{\bm{\mathfrak{S}}}_{J}}$$ by {\bf{Lemma \ref{Lemma Convex Hull Moment Map}}}. Since $\mathrm{n^{-1}\xi^{J,\prime}_{n}\rightarrow \xi}$ and $\mathrm{\xi\in Relint\,\mathfrak{Conv}\,{\bm{\mathfrak{S}}}_{J}}$, we have $$\mathrm{n^{-1}\xi^{J,\prime}_{n}\in Relint\,\mu\left(cl\left(\mathbb{T}.x\right)\right)=Relint\,\mathfrak{Conv}\,{\bm{\mathfrak{S}}}_{J}}$$ for all $\mathrm{n}$ big enough as well. Therefore we deduce $$\mathrm{x\in \mathbb{T}.\mu^{-1}\left(n^{-1}\xi^{J,\prime}_{n}\right)}$$ for all $\mathrm{n}$ big enough, which proves  the inclusion $\mathrm{"\!\!\subset\!"}$. 

Now, let $\mathrm{x\in \mathbf{X}^{J} \cap \mathbb{T}.\mu^{-1}\left(n^{-1}\xi^{J}_{n}\right)}$. Note that, as in the previous case, we can assume that $\mathrm{x}$ is not a $\mathrm{\mathbb{T}}$-fixed point. Moreover, let $\mathrm{\mathbb{T}.z_{x}}$ be the unique closed orbit in $\mathrm{\pi^{-1}\left(\pi\left(x\right)\right)}$ i.e.\ $$\mathrm{\pi^{-1}\left(\pi\left(x\right)\right)\cap \mathbb{T}.\mu^{-1}\left(\xi\right)=\mathbb{T}.z_{x}.}$$ Note that $\mathrm{z_{x}\in \mathbb{T}.\mathbf{M}_{J}}$. If the claim were false, i.e.\ if $\mathrm{\mathbb{T}.x\neq\mathbb{T}.z_{x}}$, then we would deduce $$\mathrm{\xi\in \mathfrak{Conv}\,{\bm{\mathfrak{S}}}_{J}= \mu\left(cl\left(\mathbb{T}.z_{x}\right)\right)\subset  bd\,\mu\left(cl\left(\mathbb{T}.x\right)\right)}$$ where $\mathrm{n^{-1}\xi^{J,\prime}_{n}\in \mu\left(cl\left(\mathbb{T}.z_{x}\right)\right)}$ for all $\mathrm{n}$. In particular, by the above inclusion, it then follows that $\mathrm{n^{-1}\xi^{J,\prime}_{n}\notin \mu\left(\mathbb{T}.x\right)}$ for all $\mathrm{n}$ in contradiction to  $\mathrm{x\in \mathbf{X}{J} \cap \mathbb{T}.\mu^{-1}\left(n^{-1}\xi^{J}_{n}\right)}$ for all $\mathrm{n}$ big enough.
Therefore, the assumption is false and it follows that $\mathrm{x\in \mathbb{T}.\mathbf{M}_{J}}$ and hence $\mathrm{\xi\in\mu\left(\mathbb{T}.x\right)}$ or $\mathrm{x\in \mathbb{T}.\mu^{-1}\left(\xi\right)}$ which proves $\mathrm{"\!\!\supset\!"}$.
\end{proof}

After having shown {\bf{Proposition \ref{Proposition Existence of Tame Sequences}}} we can prove that the set of semistable points $\mathrm{\mathbf{X}^{ss}_{\xi}}$ can be covered by the $\mathrm{n}$-stable complements of finitely many sequences $\mathrm{\left(s_{n}^{i}\right)_{n}}$ of eigensections whose associated weight sequences $\mathrm{\left(\xi^{i}_{n}\right)_{n}}$ approximate the ray $\mathrm{\mathbb{R}^{\geq 0}\xi}$.

\begin{theo_1*} \textnormal{[\scshape{Existence of Tame Sequences}]}\vspace{0.1 cm}\\ 
If $\mathrm{\mathbf{L}\rightarrow \mathbf{X}}$ is as above and $\mathrm{\xi\in Im\left(\mu\right)}$, then we can find a finite cover $$\mathrm{\left\{\mathbf{X}^{i}\right\}_{i\in I}\text{ of }\mathbf{X}^{ss}_{\xi}}$$ consisting of open, $\mathrm{\pi}$-saturated subsets and a finite collection $$\mathrm{\left\{\left( s_{n}^{i} \right)_{n}\right\}_{i\in I},\text{ where }s_{n}^{i}\in H^{0}(\mathbf{X},\mathbf{L}^{n})}$$ of sequences consisting of $\mathrm{\mathbb{T}}$-eigensections such that the following properties are fulfilled:
\begin{enumerate}  
\item $\mathrm{\frac{\xi^{i}_{n}}{n}\rightarrow \xi}$ where $\mathrm{\vert \xi^{i}_{n}-n\,\xi\vert\in \mathcal{O}\left(1\right)}$ for each $\mathrm{i}$.\vspace{-0.2 cm}
\item $\mathrm{\mathbf{X}^{i}\subset \mathbf{X}(s_{n}^{i})}$ for all $\mathrm{n}$ big enough.\vspace{-0.2 cm}
\item $\mathrm{\mathbf{X}^{i}\cap \mathbb{T}.\mu^{-1}\left(\xi\right)=\mathbf{X}^{i}\cap \mathbb{T}.\mu^{-1}\left(\frac{\xi^{i}_{n}}{n}\right)}$ for all $\mathrm{n\geq N_{0}}$.
\end{enumerate} 
\end{theo_1*}
\begin{proof}  First of all choose an indexing $\mathrm{I}$ of all non-empty $\mathrm{\mathbf{M}_{J}}$ and recall that the collection $\mathrm{\left\{\mathbf{M}_{i}\right\}_{i\in I}}$ defines a finite cover $\mathrm{\mu^{-1}\left(\xi\right)}$. Hence, the corresponding collection $\mathrm{\left\{\mathbf{X}^{i}\right\}_{i\in I}}$ of open, $\mathrm{\pi}$-saturated subset $\mathrm{\mathbf{X}^{i}=\pi^{-1}\left(\pi\left(\mathbf{M}_{i}\right)\right)}$ defines a finite cover of $\mathrm{\mathbf{X}^{ss}_{\xi}}$ and the claim is a direct consequence of {\bf{Proposition \ref{Proposition Existence of Tame Sequences}}}. 
\end{proof}

Before we proceed with the proof of {\bf{Proposition \ref{Proposition Uniform Convergence Potential Functions in the Abelian Case}}}, we consider the following example of {\bf{Theorem 1}}.

\begin{example} \textnormal{Let $\mathrm{\mathbf{X}=\mathbf{\Sigma}_{m}}$, $\mathrm{m\in \mathbb{N}}$, be the $\mathrm{m}$-th {\scshape{Hirzebruch}}-Surface (cf. \cite{Hir}) which is defined as the projectivization $\mathrm{\mathbb{P}\left(\mathcal{O}_{\mathbb{C}\mathbb{P}^{1}}\left(1\right)\oplus \mathcal{O}_{\mathbb{C}\mathbb{P}^{1}}\left(m\right)\right)}$ and which is isomorphic to the hypersurface $\mathrm{\{z_{0}^{m}\zeta_{1}}$ $\mathrm{-z_{1}^{m}\zeta_{2}=0\}\subset \mathbb{C}\mathbb{P}^{1}\times \mathbb{C}\mathbb{P}^{2}}$.}

\textnormal{Consider the $\mathrm{\mathbb{T}=\mathbb{C}^{*}}$-action on $\mathrm{\mathbb{C}\mathbb{P}^{5}}$ given by $\mathrm{t.[u]=[u_{0}\!:\!t\,u_{1}\!:\!t\,u_{2}\!:\!u_{3}\!:\!t\,u_{4}\!:\!t\,u_{5}]}$ which pulls back to the $\mathrm{\mathbb{C}^{*}}$-action on $\mathrm{\mathbb{C}\mathbb{P}^{1}\times \mathbb{C}\mathbb{P}^{2}}$ given by $$\mathrm{t.\left(\left[z_{0}\!:\!z_{1}\right],\left[\zeta_{0}\!:\!\zeta_{1}\!:\!\zeta_{2}\right]\right)=([z_{0}\!:\!z_{1}],[\zeta_{0}\!:\!t\,\zeta_{1}\!:\!t\,\zeta_{2}])}$$ via the {\scshape{Segre}}-embedding $\mathrm{\sigma_{1,2}\!:\!\mathbb{C}\mathbb{P}^{1}\times \mathbb{C}\mathbb{P}^{2}\hookrightarrow \mathbb{C}\mathbb{P}^{5}}$. Moreover, fix the $\mathrm{\mathbb{C}^{*}}$-linearization on $\mathrm{\mathcal{O}_{\mathbb{C}\mathbb{P}^{5}}\left(1\right)}$ $\mathrm{\rightarrow \mathbb{C}\mathbb{P}^{5}}$ given by $\mathrm{t.[u,\zeta]=[t.u\!:\!\zeta]}$ where we have used the identification $$\mathrm{\mathcal{O}_{\mathbb{C}\mathbb{P}^{5}}\left(1\right)\cong \mathbb{C}\mathbb{P}^{6}\setminus\left\{[0\!:\!{\dots}\!:\!0\!:\!1]\right\}.}$$ It is direct to check that this $\mathrm{\mathbb{C}^{*}}$-linearization can be pulled back to a $\mathrm{\mathbb{C}^{*}}$-linearization on $\mathrm{\mathbf{L}\coloneqq \left(\sigma_{1,2}^{*}\mathcal{O}_{\mathbb{C}\mathbb{P}^{5}}\left(1\right)\right)\vert \mathbf{\Sigma}_{m}}$. Furthermore, the moment map of the $\mathrm{\mathbb{C}^{*}}$-linearization on $\mathrm{\mathcal{O}_{\mathbb{C}\mathbb{P}^{5}}\left(1\right)}$ $\mathrm{\rightarrow \mathbb{C}\mathbb{P}^{5}}$ associated to the standard hermitian metric $\mathrm{h}$ on $\mathrm{\mathcal{O}_{\mathbb{C}\mathbb{P}^{5}}\left(1\right)}$ defined by $$\mathrm{h\left(\left(u_{0},\zeta_{0}\right),\left(u_{0},\zeta_{0}\right)\right)=\left\Vert u\right\Vert ^{-2}\zeta_{0}\overline{\zeta}_{1}}$$ yields a moment map $\mathrm{\mu\!:\!\mathbf{X}\rightarrow \mathfrak{t}^{*}}$ on $\mathrm{X}$ which is given by $$\mathrm{\mu\left(\left[z_{0}\!:\!z_{1}\right],\left[\zeta_{0}\!:\!\zeta_{1}\!:\!\zeta_{2}\right]\right)=\frac{1}{2}\frac{\left|z_{0}\zeta_{1}\right|^{2}+\left|z_{0}\zeta_{2}\right|^{2}+\left|z_{1}\zeta_{1}\right|^{2}+\left|z_{1}\zeta_{2}\right|^{2}}{\left|z_{0}\zeta_{0}\right|^{2}+\left|z_{0}\zeta_{1}\right|^{2}+\left|z_{0}\zeta_{2}\right|^{2}+\left|z_{1}\zeta_{0}\right|^{2}+\left|z_{1}\zeta_{1}\right|^{2}+\left|z_{1}\zeta_{2}\right|^{2}}}.$$}
\textnormal{If $\mathrm{\xi=\frac{1}{2\,\sqrt{2}}\in Im\,\mu=[0,\frac{1}{2}]}$, then $$\mathrm{\mu^{-1}\left(\xi\right)=\left\{\left(\sqrt{2}-1\right)\left(\left|z_{0}\zeta_{1}\right|^{2}+\left|z_{0}\zeta_{2}\right|^{2}+\left|z_{1}\zeta_{1}\right|^{2}+\left|z_{1}\zeta_{2}\right|^{2}\right)-\left|z_{0}\zeta_{0}\right|^{2}-\left|z_{1}\zeta_{2}\right|^{2}=0\right\}\cap \mathbf{\Sigma}_{m}}$$ and $$\mathrm{\mathbf{X}^{ss}_{\xi}=\mathbf{X}\setminus\left(\left\{ \zeta_{0}=0\,\vee\,\zeta_{1}=\zeta_{2}=0\right\}\right) \subset\mathbb{C}\mathbb{P}^{1}\times \mathbb{C}\mathbb{P}^{2}}$$ where $\mathrm{\pi\!:\!\mathbf{X}^{ss}_{\xi}\rightarrow \mathbf{Y}\cong \mathbb{C}\mathbb{P}^{1}}$ can be identified with the restriction $\mathrm{p_{\mathbb{C}\mathbb{P}^{1}}\vert \mathbf{\Sigma}_{m}=\pi}$ of the projection map $\mathrm{p_{\mathbb{C}\mathbb{P}^{1}}\!:\!\mathbb{C}\mathbb{P}^{1}\times \mathbb{C}\mathbb{P}^{2}\rightarrow \mathbb{C}\mathbb{P}^{1}}$. A further analysis of the example shows that $\mathrm{6=dim_{\mathbb{C}}}$ $\mathrm{H^{0}\left(\mathbf{\Sigma}_{m},\mathbf{L}\right)}$ where $$\mathrm{\begin{array}{rclrclrclrclrclrcl}
\mathrm{s_{1}} & \mathrm{=} & \mathrm{z_{0}\zeta_{0},} & \mathrm{s_{2}} & \mathrm{=} & \mathrm{z_{0}\zeta_{1},} & \mathrm{s_{3}} & \mathrm{=} & \mathrm{z_{0}\zeta_{2},} & \mathrm{s_{4}} & \mathrm{=} & \mathrm{z_{1}\zeta_{0},} & \mathrm{s_{5}} & \mathrm{=} & \mathrm{z_{1}\zeta_{1},} & \mathrm{s_{6}} & \mathrm{=} & \mathrm{z_{1}\zeta_{2}}\\
\mathrm{\xi_{1}} & \mathrm{=} & \mathrm{0,} & \mathrm{\xi}_{2} & \mathrm{=} & \mathrm{1,} & \mathrm{\xi_{3}} & \mathrm{=} & \mathrm{1,} & \mathrm{\xi_{4}} & \mathrm{=} & \mathrm{0}, & \mathrm{\xi_{5}} & \mathrm{=} & \mathrm{1} & \mathrm{\xi_{6}} & \mathrm{=} & \mathrm{1}.\end{array}}$$ Moreover, it is direct to check that $\mathrm{\mathbf{M}_{J}\neq \varnothing}$ if and only if $\mathrm{J\in\left\{\left\{ 1,3\right\} ,\left\{ 4,5\right\},J\right\}}$. A tame collection is for example given by $$\mathrm{\left\{\mathbf{X}^{i}\right\}_{i=1,2}=\left\{\pi^{-1}\left(\mathbb{C}\mathbb{P}^{1}\setminus\left\{[1\!:\!0]\right\}\right),\,\pi^{-1}\left(\mathbb{C}\mathbb{P}^{1}\setminus\left\{[0\!:\!1]\right\}\right)\right\}}$$ where $$\mathrm{\left\{\left(s^{i}_{n}\right)_{n}\right\}_{i=1,2}=\left\{\Big(s_{1}^{n-\left\lfloor \frac{n}{2}\right\rfloor }s_{3}^{\left\lfloor \frac{n}{2}\right\rfloor }\Big)_{n},\,\Big(s_{4}^{n-\left\lfloor \frac{n}{2}\right\rfloor }s_{5}^{\left\lfloor \frac{n}{2}\right\rfloor }\Big)_{n}\right\}. \,\boldsymbol{\Box}}$$}
\end{example}

We complete this section by proving the following proposition.
\begin{prop}\label{Proposition Uniform Convergence Potential Functions in the Abelian Case} If $\mathrm{\left(s^{i}_{n}\right)_{n}}$ is a tame collection, then the associated collection of sequences of strictly plurisubharmonic potentials $\mathrm{\left(\varrho^{i}_{n}\right)_{n}}$ given by $$\mathrm{\varrho_{n}^{i}=-\frac{1}{n}\mathbf{log}\,\vert s_{n}^{i}\vert^{2}\label{Notation Sequence of S.p.s.h. Functions Associated to the Tame Sequence}}$$ converges uniformly on every compact set of $\mathrm{\mathbf{X}^{i}}$ to a smooth strictly plurisubharmonic function $\mathrm{\varrho^{i}\!:\!\mathbf{X}^{i}\rightarrow \mathbb{R}}$\label{Notation S.p.s.h. Limit Function Associated to the Tame Sequence}. Moreover, the same is true for all its derivatives. 
\end{prop}
\begin{proof} First of all recall (cf.\ proof of {\bf{Proposition \ref{Proposition Existence of Tame Sequences}}}) that each sequence $\mathrm{\left(s^{i}_{n}\right)_{n}}$ is given by $$\mathrm{s_{n}^{J}\coloneqq\prod_{j=1}^{Card J}s_{j}^{n\cdot\nu_{j,n}^{J}},}$$ where $\mathrm{J\subset \left\{1,\dots,m\right\}}$ is a finite, suitable index set and $\mathrm{\big(\nu_{j,n}^{J}\big)_{n}}$ a sequence of integers such that $$\mathrm{\sum_{j=1}^{Card\,J}\nu^{J}_{j,n}\xi _{j}=\sum_{j\in J_{1}}\nu_{j,n}^{J}\xi_{j}\rightarrow \xi}$$ where $\mathrm{J=J_{0}\cup J_{1}}$, $\mathrm{J_{1}\neq \varnothing}$, $\mathrm{\nu_{j}=0}$ for all $\mathrm{j\in J_{0}}$ and all $\mathrm{n}$ big enough, resp.\ $\mathrm{\nu_{j}>0}$ for all $\mathrm{j\in J_{1}}$ and all $\mathrm{n}$ big enough. Hence, it follows that $$\mathrm{\varrho^{i}_{n}=-\sum_{j\in J_{1}}\nu^{J}_{j,n}\mathbf{log}\,\vert s_{j}\vert^{2}, \nu_{j,n}^{J}>0\text{ for }j\in J_{1}}$$ for all $\mathrm{n}$ big enough. Recall that $\mathrm{\mathbf{X}\left(s^{J}_{n}\right)=\complement\bigcup_{j\in J_{1}}\left\{s_{j}=0\right\}}$ for all $\mathrm{n}$ big enough. As the sequences $\mathrm{\big(\nu_{j,n}^{J}\big)_{n}}$ are convergent with strictly positive limits for all $\mathrm{j\in J_{1}}$, it follows that $\mathrm{\varrho^{i}_{n}}$ converges uniformly on compact subsets of $\mathrm{\mathbf{X}^{i}}$ in all derivatives to the smooth s.p.s.h. ($\mathrm{\coloneqq}${\itshape{s}}trictly {\itshape{p}}luri{\itshape{s}}ub{\itshape{h}}armonic)\ function $\mathrm{\varrho^{i}\!:\!\mathbf{X}^{i}\rightarrow \mathbb{R}}$ given by $$\mathrm{\varrho^{i}=-\sum_{j\in J_{1}}\,\underset{n\rightarrow\infty}{lim}\nu_{j,n}^{J}\,\mathbf{log}\,\left|s_{j}\right|^{2}.}$$
\end{proof}

\section{Uniform Localization Proposition}\label{Uniform Localization Proposition}

In this section fix a tame sequence $\mathrm{\left(s_{n}^{i}\right)_{n}}$ and let $\mathrm{\left(\varrho_{n}^{i}\right)_{n}}$ be the associated sequence of strictly plurisubharmonic functions. Moreover, recall that by {\bf{Theorem 1}}, the subset $\mathrm{\mathbf{X}^{i}}$ is $\mathrm{\pi}$-saturated and contained in the complement of the zero set $\mathrm{\mathbf{X}\left(s_{n}^{i}\right)}$ of $\mathrm{s_{n}^{i}}$ for all $\mathrm{n}$ big enough. 

We start this section by recalling the following basic fact (cf.?\cite{He-Hu2}, {\bf{{pp.\ }\oldstylenums{310}-\oldstylenums{349}}}).

\begin{lemma}
Let $\mathrm{\varrho^{i}\!:\!\mathbf{X}^{i}\rightarrow \mathbb{R}^{\geq 0}}$ and let $\mathrm{\mathbf{Y}^{i}=\pi\left(\mathbf{X}^{i}\right)\label{Notation Image of X_i Under Projection Map}}$ be as above, then for each $\mathrm{y\in \mathbf{Y}^{i}}$ there exists a $\mathrm{\pi}$-saturated, open subset $\mathrm{\mathbf{V}^{i}=\pi^{-1}\left(\mathbf{W}^{i}\right)}$\label{Notation V^i}\label{Notation W^i} of $\mathrm{\mathbf{X}^{ss}_{\xi}}$ where $\mathrm{\mathbf{W}^{i}\subset \mathbf{Y}^{i}}$ is open so that $\mathrm{\left(\varrho^{i}\times \pi\right)\vert \mathbf{V}^{i}}$ is proper. 
\end{lemma}

Note that it is always possible to assume that $\mathrm{\mathbf{W}^{i}\subset \mathbf{Y}^{i}}$ is a compact neighborhood which we will do form now on. Moreover, since $\mathrm{\mathbf{Y}}$ is compact finitely many of those compact neighborhoods $\mathrm{\mathbf{W}^{i}}$ will cover $\mathrm{\mathbf{Y}}$.

In the sequel, we will work with the normalized s.p.s.h function $\mathrm{\hat{\varrho}^{i}}$, resp. $\mathrm{\hat{\varrho}^{i}_{n}}$ on $\mathrm{\mathbf{X}^{i}}$ defined by $$\mathrm{\hat{\varrho}^{i}\coloneqq \varrho^{i}-\pi^{*}\varrho^{i}_{red}\text{ resp.\! }\hat{\varrho}^{i}_{n}\coloneqq \varrho^{i}_{n}-\pi^{*}\varrho^{i}_{n,red}.}$$ Since $\mathrm{\pi^{*}\varrho^{i}}$ is continuous and since $\mathrm{\mathbf{W}^{i}\subset \mathbf{Y}^{i}}$ was chosen to be a compact neighborhood, the above lemma remains valid for $\mathrm{\hat{\varrho}^{i}}$, resp{.\! }for $\mathrm{\hat{\varrho}^{i}_{n}}$. In the sequel, we will set $\mathrm{\varrho^{i}=\hat{\varrho}^{i}}$ and $\mathrm{\hat{\varrho}^{i}_{n}=\varrho^{i}_{n}}$. 

We define $$\mathrm{T\left(\epsilon,\mathbf{W}^{i}\right)\coloneqq \left(\varrho^{i}\times \pi\right)^{-1}\left([0,\epsilon]\times \mathbf{W}^{i}\right)\label{Notation Compact Neighborhood Tube}.}$$ By the above lemma, $\mathrm{T\left(\epsilon,\mathbf{W}^{i}\right)}$ is a relatively compact subset of $\mathrm{\mathbf{X}^{i}\subset \mathbf{X}^{ss}_{\xi}}$ which is $\mathrm{T}$-invariant by its definition. Furthermore, define $$\mathrm{T^{c}\left(\epsilon,\mathbf{W}^{i}\right)\coloneqq \left(\varrho^{i}\times \pi\right)^{-1}\left([\epsilon,\infty)\times \mathbf{W}^{i}\right)\label{Notation Complement of Compact Neighborhood Tube}.}$$

\begin{remark} \label{Preparation Remark} Note that there exists $\mathrm{N_{0}\left(\epsilon\right)\in\mathbb{N}}$ so that $$\mathrm{\mu^{-1}\left(n^{-1}\xi^{i}_{n}\right)\cap\pi^{-1}\left(\mathbf{W}^{i}\right)\subset T\left(\epsilon,\mathbf{W}^{i}\right)\text{ for all }n\geq N_{0}\left(\epsilon\right).}$$ Otherwise there would exist a sequence $\mathrm{\left(x_{n}\right)_{n}}$ in $\mathrm{\mu^{-1}\left(n^{-1}\xi^{i}_{n}\right)\cap\pi^{-1}\left(\mathbf{W}^{i}\right)}$ so that $\mathrm{x_{n}}$ $\mathrm{\notin}$ $\mathrm{T(\epsilon,}$ $\mathrm{\mathbf{W}^{i})}$ for all $\mathrm{n\in \mathbb{N}}$ big enough. Since $\mathrm{\mathbf{X}}$ and $\mathrm{\mathbf{W}^{i}}$ are compact, we can find a convergent subsequence $\mathrm{\left(x_{n_{j}}\right)_{j}}$ so that $\mathrm{x_{n_{j}}\rightarrow x_{0}}$ where $\mathrm{\pi\left(x_{0}\right)\in \mathbf{W}^{i}}$ and $\mathrm{x_{0}\in \mu^{-1}\left(\xi\right)}$ which follows by $\mathrm{x_{n_{j}}\in \mu^{-1}(n^{-1}\xi^{i}_{n_{j}})}$ and $\mathrm{{n_{j}}^{-1}\xi^{i}_{n_{j}}\rightarrow \xi}$. Hence, we have $\mathrm{x_{0}\in \mu^{-1}\left(\xi\right)\cap \pi^{-1}\left(\mathbf{W}^{i}\right)}$. On the other hand, we have assumed that $\mathrm{x_{n_{j}}\notin T\left(\epsilon,\mathbf{W}^{i}\right)}$ for all $\mathrm{n}$ big enough. However, since $\mathrm{T\left(\epsilon,\mathbf{W}^{i}\right)}$ is an open neighborhood of $\mathrm{\mu^{-1}\left(\xi\right)\cap \pi^{-1}\left(\mathbf{W}^{i}\right)}$ in $\mathrm{\pi^{-1}\left(\mathbf{W}^{i}\right)}$ this yields a contradiction to $\mathrm{x_{0}\in \mu^{-1}\left(\xi\right)\cap \pi^{-1}\left(\mathbf{W}^{i}\right)}$.
\end{remark}

Before we prove {\bf{Theorem 2}}, we need the following preparation.

\begin{lemma} \label{Preparation Lemma} Let $\mathrm{\left(s_{n}^{i}\right)_{n}}$ be a tame sequence as above, $\mathrm{\pi^{-1}\left(y\right)\subset \mathbf{X}^{i}}$ and let $\mathrm{\mathbb{T}.z_{y}}$ be the unique closed orbit in $\mathrm{\pi^{-1}\left(y\right)}$ then it follows that $\mathrm{\varrho^{i}_{n}\vert \pi^{-1}\left(y\right)\cap \mathbb{T}.z_{y}}$ takes on a unique minimum along the set $$\mathrm{T.\mu^{-1}\left(n^{-1}\xi^{i}_{n}\right)\cap \pi^{-1}\left(y\right)}$$ which is contained in $\mathrm{\mathbb{T}.\mu^{-1}\left(\xi\right)\cap \pi^{-1}\left(y\right)}$.
\end{lemma}
\begin{proof} Note that since $\mathrm{\left(s_{n}^{i}\right)_{n}}$ is a tame sequence we have $$\mathrm{\mathbb{T}.\mu^{-1}\left(\xi\right)\cap \mathbf{X}^{i}=\mathbb{T}.\mu^{-1}\left(n^{-1}\xi^{i}\right)\cap \mathbf{X}^{i}}$$ by the third claim of {\bf{Theorem 1}}. This shows that the set $$\mathrm{T.\mu^{-1}\left(n^{-1}\xi^{i}_{n}\right)\cap \pi^{-1}\left(y\right)}$$ is contained in $\mathrm{\mathbb{T}.\mu^{-1}\left(\xi\right)\cap \pi^{-1}\left(y\right)}$ so the second claim is proved.

The first claim is a direct consequence of the fact that the unique minimum of $\mathrm{\varrho^{i}_{n}}$ on $$\mathrm{\mathbb{T}.\mu^{-1}\left(n^{-1}\xi^{i}_{n}\right)\cap \pi^{-1}\left(y\right)}$$ is known to be  equal to $\mathrm{T.\mu^{-1}\left(n^{-1}\xi^{i}_{n}\right)\cap \pi^{-1}\left(y\right)}$.
\end{proof}

We can now prove the uniform Localization Proposition.

\begin{theo_2} \label{Theorem Localization of the Sequence of Potential Functions} \textnormal{[\scshape{Uniform Localization of the Potential Functions}]}\vspace{0.1 cm}\\
Let $\mathrm{\left(s_{n}^{i}\right)_{n}}$ be a tame sequence and $\mathrm{\left(\varrho_{n}^{i}\right)_{n}}$ the associated sequence of strictly plurisubharmonic functions. Let $\mathrm{\mathbf{W}^{i}\subset\pi\left(\mathbf{X}^{i}\right)}$ and $\mathrm{\epsilon>0}$ be as above and let $\mathrm{\delta>0}$ be given. Then there exists $\mathrm{N_{0}\in\mathbb{N}}$ so that $$\mathrm{\varrho_{n}^{i}\left(x\right)\geq\epsilon-\delta}$$ for all $\mathrm{x\in T^{c}\left(\epsilon,\mathbf{W}^{i}\right)}$ and all $\mathrm{n\geq N_{0}}$.
\end{theo_2}
\begin{proof} First of all, by {\bf{Remark \ref{Preparation Remark}}}, we can assume that 
\begin{equation}
\mathrm{\mu^{-1}\left(n^{-1}\xi^{i}_{n}\right)\cap\pi^{-1}\left(\mathbf{W}^{i}\right)\subset T\left(\epsilon,\mathbf{W}^{i}\right)\text{ for all }n\geq N_{0}\left(\epsilon\right).}\label{Critical Subsets Contained in Tube}
\end{equation} Moreover, note that $\mathrm{\varrho^{i}_{n}}$ converges uniformly on relatively compact sets and hence on $\mathrm{T\left(\epsilon,\mathbf{W}^{i}\right)}$. Therefore, we can find an $\mathrm{N_{0}\in \mathbb{N}}$ so that 
\begin{equation}
\mathrm{\varrho^{i}_{n}\left(x\right)\geq\epsilon-\delta}\label{Convergence on Boundary}
\end{equation}
for all $\mathrm{x}$ in the compact subset $\mathrm{\left(\varrho^{i}\times \pi\right)^{-1}\left(\left\{\epsilon\right\}\times \mathbf{W}^{i}\right)}$ and all $\mathrm{n\geq N_{0}}$. We continue the proof by considering the following two cases:

1) First of all let, $\mathrm{x\in T^{c}\left(\epsilon,\mathbf{W}^{i}\right)\cap \mathbb{T}.z_{x}}$ where $\mathrm{\mathbb{T}.z_{x}}$ is the unique closed orbit in the fiber $\mathrm{{\pi^{-1}\left(\pi\left(x\right)\right)}}$. We have to show that $$\mathrm{\varrho^{i}_{n}\vert T^{c}\left(\epsilon,\mathbf{W}^{i}\right)\cap \mathbb{T}.z_{x}>\epsilon-\delta.}$$

Applying {\bf{Lemma \ref{Preparation Lemma}}}, we know that the restriction $\mathrm{\varrho^{i}_{n}\vert \mathbb{T}.z_{x}}$ of the strictly plurisubharmonic function $\mathrm{\varrho^{i}_{n}}$ on the unique closed orbit $\mathrm{\mathbb{T}.z_{x}}$ in the fiber $\mathrm{\pi^{-1}}$ $\mathrm{\left(\pi\left(x\right)\right)}$ $\mathrm{\subset \mathbf{X}^{i}}$ takes on its minimum along $$\mathrm{\mu^{-1}\left(n^{-1}\xi^{i}\right)\cap \pi^{-1}\left(\pi\left(x\right)\right)\subset \mathbb{T}.z_{x}.}$$ 

If $\mathrm{\mathfrak{m}_{x}\coloneqq \mathfrak{\bm{Ann}}\,\mathfrak{t}_{x}\subset \mathfrak{t}^{*}}$ denotes the annihilator $\mathrm{\left\{ \eta\in\mathfrak{t}^{*}:\,\left\langle \eta,\xi\right\rangle =0\text{ for all }\xi\in\mathfrak{t}_{x}\right\} }$ of the isotropy $\mathrm{\mathfrak{t}_{x}}$, then $\mathrm{\mathbb{T}.z_{x}}$ is isomorphic to the homogenous vector bundle with typical fiber $\mathrm{\mathfrak{m}_{x}}$: $$\mathrm{\mathbb{T}.z_{x}\cong T\times^{T_{x_{z}}}\mathfrak{m}_{z}}$$ Note that the zero section in $\mathrm{T\times^{T_{x_{z}}}\mathfrak{m}_{z}}$ is mapped under the above identification to the $\mathrm{T}$-orbit $$\mathrm{T.\mu^{-1}\left(\xi\right)\cap \mathbb{T}.z_{x}.}$$ Moreover, using this identification, it follows that the restriction  of $\mathrm{\varrho^{i}_{n}}$ on $\mathrm{\mathbb{T}.z_{x}}$ yields a smooth function on $\mathrm{T\times^{T_{x_{z}}}\mathfrak{m}_{z}}$ whose restriction on each fiber $\mathrm{[\left\{t\right\}\times \mathfrak{m}_{z}]}$, $\mathrm{t\in T}$ is a strictly convex function with a unique minimum given by $\mathrm{[\left\{t\right\}\times \left\{m_{n}\right\}]}$ for $\mathrm{m_{n}\in \mathfrak{m}_{z}}$. Note that the tube $\mathrm{T\left(\epsilon,\mathbf{W}^{i}\right)\cap \mathbb{T}.z_{x}}$ is isomorphic to a tube of the zero section in $\mathrm{T\times^{T_{x_{0}}}\mathfrak{m}_{x}}$ and also note that this tube contains $\mathrm{[T\times \left\{m_{n}\right\}]}$ for all $\mathrm{n}$ big enough by the remark at the beginning of this proof which is based on {\bf{Remark \ref{Preparation Remark}}} . 

We continue the proof by connecting $\mathrm{x=\left[\left\{ t\right\} \times\left\{ \eta_{x}\right\} \right]}$, for $\mathrm{\eta_{x}\in \mathfrak{m}_{z}}$ suitable, and the unique minimum $\mathrm{m_{n}=[\left\{t\right\}\times \left\{m_{n}\right\}]}$ of $\mathrm{\varrho^{i}_{n}}$ with a line $\mathrm{\Lambda\!:\!\mathbb{R}\rightarrow [\left\{t\right\}\times \mathfrak{m}_{x}]}$ in the vector space $\mathrm{[\left\{t\right\}\times \mathfrak{m}_{x}]}$ so that $\mathrm{\Lambda\left(0\right)=m_{n}}$ and $\mathrm{\Lambda\left(1\right)=x}$. Note that this line intersects $\mathrm{\varrho^{i,-1}\left(\epsilon\right)}$ in, say $\mathrm{\Lambda\left(\tau_{n}\right)=y_{n}}$, where $\mathrm{0<\tau_{n}<x}$, because the minimum of $\mathrm{\varrho^{i}_{n}}$ is contained in the tube $\mathrm{T\left(\epsilon,\mathbf{W}^{i}\right)}$ (for all $\mathrm{n}$ big enough) whereas $\mathrm{x}$ is not by our assumption. To sum up, we have a convex function $\mathrm{\Lambda^{*}\varrho^{i}_{n}}$ on $\mathbb{R}$ with a unique minimum at $\mathrm{0}$ so that $\mathrm{\Lambda^{*}\varrho^{i}_{n}\left(\tau_{n}\right)\geq \epsilon-\delta}$ for all $\mathrm{n}$ big enough by equation \ref{Convergence on Boundary}. Hence, it follows that $\mathrm{\varrho^{i}_{n}\left(x\right)\geq \epsilon-\delta}$ as well for all $\mathrm{n}$ big enough.

2) The next step is to show that the inequality $\mathrm{\varrho^{i}_{n}\left(x\right)\geq\epsilon-\delta }$ also holds for all $\mathrm{x\in T^{c}(\epsilon,}$ $\mathrm{\mathbf{W}^{i})}$ so that $\mathrm{x}$ is not contained in the unique closed orbit $\mathrm{\mathbb{T}.z_{x}}$ of the fiber $\mathrm{\pi^{-1}\left(\pi\left(x\right)\right)}$.

Let us assume that this is not true. By the {\scshape{Hilbert}} Lemma ({cf.\ }\cite{Kra}), we can find a one parameter group $\mathrm{\gamma\!:\!\mathbb{C}\rightarrow \mathbb{T}}$ so that $$\mathrm{\mathrm{\underset{\mathrm{t\rightarrow0}}{\mathrm{lim}}\,\gamma\left(t\right).x\in \mathbb{T}.z_{x}}.}$$ 

Note that neither the pull back $\mathrm{\gamma^{*}\varrho^{i}}$ nor the pull back $\mathrm{\gamma^{*}\varrho^{i}_{n}}$ attains its $\mathrm{S^{1}}$-invariant minimum on $\mathrm{\mathbb{C}^{*}}$ since otherwise it would follow that $\mathrm{\varrho^{i}\vert Im\left(\gamma\vert \mathbb{C}^{*}\right)}$ and $\mathrm{\varrho^{i}_{n}\vert Im\left(\gamma\vert \mathbb{C}^{*}\right)}$ attain their $\mathrm{S^{1}}$-invariant minimum on $\mathrm{\mathbb{T}.x}$. However, this would yield a contradiction to the assumption $\mathrm{\mathbb{T}.x\neq \mathbb{T}.z_{x}}$ and the claim of {\bf{Lemma} \ref{Preparation Lemma}}. Hence, it follows that $\mathrm{t\mapsto \gamma^{*}\varrho^{i}_{n}\left(t\right)}$ and $\mathrm{t\mapsto \gamma^{*}\varrho^{i}\left(t\right)}$ where $\mathrm{t\in \mathbb{C}^{*}}$ are strictly monotone decreasing when $\mathrm{\vert t\vert\rightarrow 0}$ (they can not be strictly increasing since $\mathrm{s^{i}_{n}}$ does not vanish on $\mathrm{\mathbb{T}.z_{x}}$). The present case can be subdivided into the following cases:

2.a) Assume that $$\mathrm{x_{0}=\mathrm{\underset{\mathrm{t\rightarrow0}}{\mathrm{lim}}\,\gamma\left(t\right).x\in \mathbb{T}.z_{x}\cap T^{c}\left(\epsilon,\mathbf{W}^{i}\right)},}$$ then we have $$\mathrm{\varrho^{i}_{n}\left(x_{0}\right)\geq\epsilon-\delta}$$ for all $\mathrm{n\geq N_{0}}$ by case 1). Since $\mathrm{\gamma^{*}\varrho^{i}_{n}}$ is monotone decreasing for $\mathrm{\vert t\vert\rightarrow 0}$ we deduce that $$\mathrm{\delta-\epsilon\geq\gamma^{*}\varrho^{i}_{n}\left(1\right)=\varrho^{i}_{n}\left(x\right)}$$  for all $\mathrm{n\geq N_{0}}$ as claimed.

2.b) Assume that $$\mathrm{x_{0}=\mathrm{\underset{\mathrm{t\rightarrow0}}{\mathrm{lim}}\,\gamma\left(t\right).x\in \mathbb{T}.z_{x}\cap T\left(\epsilon,\mathbf{W}^{i}\right)},}$$ then there exists $\mathrm{t\in \mathbb{C}^{*}}$ so that $\mathrm{\gamma\left(t\right).x\in \varrho^{i,-1}\left(\epsilon\right)}$ and hence $\mathrm{\varrho^{i}_{n}\left(\gamma\left(t\right).x\right)=\epsilon-\delta}$ for all $\mathrm{n\geq N_{0}}$ by \ref{Convergence on Boundary}. As in case 2.a), it then follows that $\mathrm{\varrho^{i}_{n}\left(x\right)\geq\delta-\epsilon}$  for all $\mathrm{n\geq N_{0}}$ as claimed.
\end{proof}

We close this section with the following corollary which slightly generalizes the claim of {\bf{Theorem \ref{Theorem Localization of the Sequence of Potential Functions}}}.

In the context of {\bf{Theorem \ref{Theorem Localization of the Sequence of Potential Functions}}}, fix $\mathrm{m_{0}\in \mathbb{N}}$ and consider the sequence $\mathrm{\big(s^{i}_{n}\cdot s^{i,-1}_{m_{0}}\big)_{n}}$ of meromorphic $\mathrm{\eta^{i,m_{0}}_{n}\coloneqq \xi^{i}_{n}-\xi^{i}_{m_{0}}}$-eigensections. Since we have $\mathrm{s^{i}_{n}\left(x\right)}$ $\mathrm{\neq 0}$ for all $\mathrm{x\in \mathbf{X}^{i}}$ and all $\mathrm{n\in \mathbb{N}}$ and hence in particular for $\mathrm{n=m_{0}}$, we can define $$\mathrm{\varrho^{i,m_{0}}_{n}\coloneqq -\frac{1}{n}\mathbf{log}\,\vert s^{i}_{n}\cdot s^{i,-1}_{m_{0}}\vert^{2}.}$$ It is direct to check that the above definition yields a sequence of smooth, strictly plurisubharmonic, $\mathrm{T}$-invariant functions for all $\mathrm{n\geq m_{0}}$ on the $\mathrm{\pi}$-saturated open set $\mathrm{\mathbf{X}^{i}}$. Furthermore, it is known that $\mathrm{\varrho^{i,m_{0}}_{n}\vert \pi^{-1}\left(y\right)}$ for $\mathrm{y\in \pi\left(\mathbf{X}^{i}\right)}$ takes on its unique minimum along the set $$\mathrm{\mu^{-1}\left(n^{-1}\eta^{i,m_{0}}_{n}\right)\cap \pi^{-1}\left(y\right).}$$ Therefore, after having applied the argumentation of {\bf{Remark \ref{Preparation Remark}}}, we can assume, as at the beginning of the proof of {\bf{Theorem \ref{Theorem Localization of the Sequence of Potential Functions}}}, that $$\mathrm{\mu^{-1}\left(n^{-1}\eta^{i,m_{0}}_{n}\right)\cap\pi^{-1}\left(\mathbf{W}^{i}\right)\subset T\left(\epsilon,\mathbf{W}^{i}\right)\text{ for all }n\text{ big enough.}}$$ The proof of {\bf{Theorem \ref{Theorem Localization of the Sequence of Potential Functions}}} now translates verbatim to the sequence $\mathrm{\left(\varrho^{m_{0}}_{n}\right)_{n}}$ and yields the following corollary.

\begin{cor}\label{Remark Convergence Theorem For The Shifted Sequence} If $\mathrm{\delta>0}$, $\mathrm{m_{0}\in \mathbb{N}}$ fixed, then it follows $$\mathrm{\varrho^{i,m_{0}}_{n}\left(x\right)\geq \epsilon-\delta}$$ for all $\mathrm{x\in T^{c}\left(\epsilon,\mathbf{W}^{i}\right)}$ and all $\mathrm{n}$ big enough.
\end{cor}

\section{Fiber Probability Measure Sequence}

\subsection{Definition of $\mathrm{\widehat{\pi}\!:\!\widehat{\mathbf{\,X\,}}\rightarrow \mathbf{Y}}$ and $\mathrm{\mathbf{Y}_{0}}$}\label{Definition of Y_{0}}

As before, let $\mathrm{\mathbf{X}}$ be a purely $\mathrm{m}$-dimensional, normal $\mathrm{\mathbb{T}}$-variety with K\"ahler structure $\mathrm{\omega}$ and let $\mathrm{\mathbf{Y}=\mathbf{X}^{ss}_{\xi}/\!\!/\mathbb{T}}$ be the associated Hilbert quotient. Note that we can always assume $\mathrm{\mathbf{Y}}$ to be purely dimensional, {i.e.\ }$\mathrm{n=}$ $\mathrm{dim_{\mathbb{C}}\,\mathbf{Y}}$: Since  $\mathrm{\mathbf{X}}$ is normal by our assumption, it follows (cf.\ \cite{He-Hu1}, {p.\ }{\bf{\oldstylenums{124}}}) that  $\mathrm{\mathbf{Y}}$ is normal as well. In particular, it follows that $\mathrm{\mathbf{Y}}$ is locally of pure dimension (cf.\ \cite{Gr-Re}, {p.\ }{\bf{\oldstylenums{125}}}) and by considering the connected components of $\mathrm{\mathbf{Y}}$, which are finite in number, we can confine ourselves to the case where $\mathrm{\mathbf{Y}}$ is of pure dimension $\mathrm{n=}$ $\mathrm{dim_{\mathbb{C}}\,\mathbf{Y}}$.

Now, recall that $\mathrm{\mathbf{X}^{ss}_{\xi}}$ is Zariski open and Zariski dense in $\mathrm{\mathbf{X}}$. Consider the compact variety $\mathrm{\widehat{\mathbf{\,X\,}}}$ \label{Notation X Hat} defined by the normalization of $\mathrm{cl\,{\bf{\Gamma_{\pi}}}}$, where $\mathrm{{\bf{\Gamma_{\pi}}}}$ denotes the graph of $\mathrm{\pi}$, i.e. $$\mathrm{{\bf{\Gamma}}_{\pi}\coloneqq \left\{\left(x,y\right)\in \mathbf{X}^{ss}_{\xi}\times \mathbf{Y}\!:\!\,\pi\left(x\right)=y\right\}\label{Notation Graph of Quotient Map}.}$$ Furthermore, we define the algebraic map $\mathrm{\widehat{\pi}\!:\!\widehat{\mathbf{\,X\,}}\rightarrow \mathbf{Y}}$ \label{Notation Regular Map between cal X and Y} by $\mathrm{\widehat{\pi}\coloneqq p_{\mathbf{Y}}\vert cl\,\mathbf{\Gamma}_{\pi}\circ\zeta}$, where $\mathrm{\zeta\!:\!\widehat{\,\mathbf{X}\,}=}$ $\mathrm{\left(cl\,\mathbf{\Gamma}_{\pi}\right)^{nor}\rightarrow cl\,\mathbf{\Gamma}_{\pi}}$ \label{Notation Normalization Map Zeta} denotes the normalization map and $\mathrm{p_{\mathbf{Y}}}$ the projection map $\mathrm{p_{\mathbf{Y}}\!:\!\mathbf{X}}$ $\mathrm{\times \mathbf{Y}\rightarrow}$ $\mathrm{\mathbf{Y}}$. Moreover, endow $\mathrm{\widehat{\mathbf{\,X\,}}}$ with the $\mathrm{\mathbb{T}}$-action induced by the lift (for its existence see \cite{Gr-Re}, {pp.\! }{\bf{\oldstylenums{164} f.}}) of the $\mathrm{\mathbb{T}}$-action on $\mathrm{cl\,\mathbf{\Gamma}_{\pi}}$, defined by $\mathrm{t.\left(x,y\right)=\left(t.x,y\right)}$ and equip $\mathrm{\widehat{\mathbf{\,X\,}}}$ with a smooth $\mathrm{\left(2,2\right)}$-form $\mathrm{\omega^{\prime}}$ \label{Notation Smooth (2,2,)-form} given by $\mathrm{\omega^{\prime}\coloneqq \left(p_{\mathbf{X}}\vert cl\left({\mathbf{\Gamma}}_{\pi}\right)\circ\zeta\right)^{*}\omega}$. For the sake of completeness, we note the following remark.

\begin{remark} \label{Remark T-Invariance of Fibers} The graph $\mathrm{{\bf{\Gamma}}_{\pi}\mathrm{\subset \mathbf{X}\times \mathbf{Y}}}$ is $\mathrm{\mathbb{T}}$-invariant with respect to the action on $\mathrm{\mathbf{X}\times \mathbf{Y}}$ given by $\mathrm{t.\left(x,y\right)=\left(t.x,y\right)}$. Since $\mathrm{{\bf{\Gamma}}_{\pi}}$ is a Zariski open and Zariski dense subset of $\mathrm{\mathbf{X}}$, it follows that $\mathrm{cl\,{\bf{\Gamma}}_{\pi}}$ is $\mathrm{\mathbb{T}}$-invariant. Moreover, as $\mathrm{\zeta}$ is $\mathrm{\mathbb{T}}$-equivariant, all fibers $\mathrm{\widehat{\pi}^{-1}\left(y\right)}$, $\mathrm{y\in \mathbf{Y}}$ are $\mathrm{\mathbb{T}}$-invariant as well.
\end{remark}

Note that since $\mathrm{\mathbf{X}}$ is a assumed to be of pure dimension $\mathrm{m}$, it follows that $\mathrm{cl\,\mathbf{\Gamma}_{\pi}}$ is likewise a purely $\mathrm{m}$-dimensional subvariety of $\mathrm{\mathbf{X}\times \mathbf{Y}}$. As $\mathrm{\zeta}$ is finite, we deduce that $\mathrm{\widehat{\mathbf{\,X\,}}}$ is of pure dimension $\mathrm{m}$ too.
 
The next step is to find a Zariski open subset  $\mathrm{\mathbf{Y}_{0}\subset \mathbf{Y}}$ so that the fibers of the restricted projection $\mathrm{\widehat{\pi}\vert \widehat{\pi}^{-1}\left(\mathbf{Y}_{0}\right)\!:\! \widehat{\pi}^{-1}\left(\mathbf{Y}_{0}\right)\rightarrow \mathbf{Y}_{0}}$ are all purely $\mathrm{k}$-dimensional varieties. The existence of $\mathrm{{\bf{Y}}_{0}}$ is a direct consequence of known facts in complex analysis and algebraic geometry: By a theorem of {\scshape{Cartan}} and {\scshape{Remmert}} ({cf.\! }\cite{Loj}, {p.\ }{\bf{\oldstylenums{271} f.}}), it follows that the subset $\mathrm{{\bf{E}}\subset \widehat{\mathbf{\,X\,}}}$ defined by $$\mathrm{{\bf{E}}\coloneqq \left\{x\in\widehat{\mathbf{\,X\,}}\!:\, k<dim_{\mathbb{C},x}\,\widehat{\pi}^{-1}\left(\widehat{\pi}\left(x\right)\right)\right\}\subset \widehat{\mathbf{\,X\,}}}$$ is a proper analytic subset of $\mathrm{\widehat{\,\mathbf{X}\,}}$ where $\mathrm{k=m-n=dim_{\mathbb{C}}\widehat{\mathbf{\,X\,}}-dim_{\mathbb{C}}\,\mathbf{Y}}$. Hence, by {\scshape{Chow}}'s Theorem it follows that $\mathrm{\mathbf{E}}$ is a proper algebraic subset. Applying the Direct Image Theorem  (cf.\ \cite{Gr-Re}, {p.\ }{\bf{\oldstylenums{207}}}), one deduces that the image $\mathrm{\widehat{\pi}\left(\mathbf{E}\right)}$ is a proper analytic subset of $\mathrm{\mathbf{Y}}$. In particular it is a proper algebraic subvariety of $\mathrm{\mathbf{Y}}$. Now, set $\mathrm{\mathbf{Y}_{0}\coloneqq \complement\, \widehat{\pi}\left(\mathbf{E}\right)}$ \label{Notation Subset Y_0} and note that all fibers of $\mathrm{\widehat{\pi}}$ over $\mathrm{\mathbf{Y}_{0}}$ are purely $\mathrm{k}$-dimensional by construction.

For later use, we introduce the following notation: Let $\mathrm{\mathbf{X},\mathbf{Y}}$ be purely dimensional complex spaces ($\mathrm{m}$ $\mathrm{=dim_{\mathbb{C}}\,\mathbf{X}}$, $\mathrm{n}$ $\mathrm{=dim_{\mathbb{C}}\,\mathbf{Y}}$) where $\mathrm{\mathbf{Y}}$ is assumed to be normal and let $\mathrm{F\!:\!\mathbf{X}\rightarrow \mathbf{Y}}$ be a holomorphic map so that all non-empty fibers $\mathrm{F^{-1}\left(y\right)}$ are of pure dimension $\mathrm{k=m-n}$. Then $\mathrm{F}$ is called a $\mathrm{k}$-{\itshape{fibering}}. Note that $\mathrm{\pi\vert \pi^{-1}\left(\mathbf{Y}_{0}\right): \pi^{-1}\left(\mathbf{Y}_{0}\right)\rightarrow \mathbf{Y}_{0}}$ is a $\mathrm{k}$-fibering. 

An example for the construction of $\mathrm{\widehat{\,\mathbf{X}\,}}$ as described above is given by the next example.

\begin{example}\label{Example Proper Inclusion} \textnormal{Let $\mathrm{\mathbf{X}=\mathbb{C}\mathbb{P}^{1}\times \mathbb{C}\mathbb{P}^{1}}$ with the K\"aher form $\mathrm{\omega=p_{1}^{*}\omega_{FS}+p_{2}^{*}\omega_{FS}}$ where $\mathrm{p_{i}\!:\!}$ $\mathrm{\mathbf{X}}$ $\mathrm{\rightarrow}$ $\mathrm{\mathbb{C}\mathbb{P}^{1}}$ denotes the $\mathrm{i}$-th projection, $\mathrm{i\in\left\{0,1\right\}}$ equipped with the $\mathrm{\mathbb{T}\cong\mathbb{C}^{*}}$-action given by $$\mathrm{t.\left(\left[z_{0}z_{1}\right],\left[\zeta_{0}\!:\!\zeta_{1}\right]\right)=\left(\left[t.z_{0}\!:\!z_{1}\right],\left[t.\zeta_{0}\!:\!\zeta_{1}\right]\right)}$$ and consider the moment map given by (cf.\ {\bf{Example \ref{Example First Example}}}) $$\mathrm{\mu\left(\left[z_{0}\!:\!z_{1}\right],\left[\zeta_{0}\!:\!\zeta_{1}\right]\right)=\frac{\vert z_{1}\vert^{2}}{\Vert z\Vert^{2}}-\frac{\vert\zeta_{0}\vert^{2}}{\Vert\zeta\Vert^{2}}.}$$ In particular, we have $$\mathrm{\mathbf{X}_{\xi=\frac{1}{2}}^{ss}=\mathbf{X}\setminus\left\{ \left\{ z_{1}=0\right\} \cup\left\{ \zeta_{1}=0\right\} \cup\left\{ \left(\left[0\!:\!1\right],\left[0\!:\!1\right]\right)\right\} \right\} }$$ where $\mathrm{\pi\!:\!\mathbf{X}^{ss}_{0}\rightarrow }$ $\mathrm{\mathbf{X}^{ss}_{0}/\!\!/\mathbb{T}=\mathbf{Y}\cong \mathbb{C}\mathbb{P}^{1}}$ is given by the map $\mathrm{\pi\!:\!\left(\left[z_{0}\!:\!z_{1}\right],\left[\zeta_{0}\!:\!\zeta_{1}\right]\right)\rightarrow[z_{0}}$ $\mathrm{\zeta_{1}\!:\!z_{1}\zeta_{0}]}$. Furthermore, a calculation shows that $$\mathrm{\widehat{\mathbf{\,X\,}}=cl\,\mathbf{\Gamma}_{\pi}=\left\{ z_{0}\zeta_{1}\xi_{1}-z_{1}\zeta_{0}\xi_{0}=0\right\} \subset \mathbf{X}\times\mathbb{C}\mathbb{P}^{1}}$$ if $\mathrm{\left[\xi_{0},\xi_{1}\right]}$ are the homogeneous coordinates of $\mathrm{\mathbb{C}\mathbb{P}^{1}}$. A further analysis of the geometry of $\mathrm{\widehat{\mathbf{\,X\,}}}$ and $\mathrm{\mathbf{Y}}$ reveals that $\mathrm{\mathbf{Y}_{0}=\mathbf{Y}}$. Note that we have $$\mathrm{p_{\mathbf{X}}\left(\widehat{\pi}^{-1}\left(\left[\zeta_{0},\zeta_{1}\right]\right)\right)=cl\left({\pi}^{-1}\left(\left[\zeta_{0},\zeta_{1}\right]\right)\right)\text{ for all }\left[\zeta_{0},\zeta_{1}\right]\text{ such that }\zeta_{0,1}\neq0.}$$ For $\mathrm{[\zeta_{0}]=[1\!:\!0]}$ and $\mathrm{\zeta_{1}=[0\!:\!1]}$, we have a proper inclusion $$\mathrm{p_{\mathbf{X}}\left(\widehat{\pi}^{-1}\left(\left[\zeta_{i}\right]\right)\right)\supsetneqq cl\left(\pi^{-1}\left(\left[\zeta_{i}\right]\right)\right)}$$ for all $\mathrm{i}$ $\mathrm{\in\left\{0,1\right\}}$.\,$\boldsymbol{\Box}$}
\end{example}

\subsection{The Fiber Probability Measure Sequence (Tame Case)}\label{Definition of the Tame Fiber Probability Measure Sequence} 

Let $\mathrm{\mathfrak{U}=\left\{\mathbf{U}_{j}\right\}_{j}}$ be a finite cover of $\mathrm{\mathbf{Y}}$ then, after having chosen a finite refinement $\mathrm{\mathfrak{U}^{\prime}}$ of $\mathrm{\mathfrak{U}}$ (for convenience set $\mathrm{\mathfrak{U}=\mathfrak{U}^{\prime}}$), it follows by {\bf{Section \ref{Existence of Tame Sequences}}} that there exists a tame collection $\mathrm{\left\{\left(s_{n}^{i}\right)\right\}_{i}}$ so that $\mathrm{\mathbf{U}_{j_{i}}\subset\pi\left(\mathbf{X}^{i}\right)=\mathbf{Y}^{i}}$. By changing the index set $\mathrm{J}$ of the finite cover $\mathrm{\mathfrak{U}}$ we can always assume that $\mathrm{\mathbf{U}_{j_{i}}=\mathbf{U}_{i}}$.

Let $\mathrm{\mathbf{Y}_{0}}$ be as in {\bf{Section} \ref{Definition of Y_{0}}} and set $$\mathrm{\Vert s_{n}^{i}\Vert^{2}\!:\!\mathbf{X}^{i}\cap \mathbf{Y}_{0}\rightarrow \mathbb{R}^{\geq 0},\,\Vert s_{n}^{i}\Vert^{2}\left(x\right)\coloneqq \int_{\pi\left(\pi^{-1}\left(x\right)\right)}\vert s_{n}^{i}\vert^{2}\,d\,[\pi_{y}]\label{Notation Fiber Integral Initial Sequence}}$$ where $\mathrm{\int_{\pi^{-1}\left(\pi\left(x\right)\right)}d\,[\pi_{y}]}$ denotes the fiber integral of $\mathrm{\omega^{k}}$ with respect to the $\mathrm{k}$-fibering $\mathrm{\pi\!:\!\mathbf{X}_{0}}$ $\mathrm{\rightarrow \mathbf{Y}_{0}}$ as defined in the work of {\scshape{J. King}} ({cf.\! }\cite{Kin}). 

\begin{remark} Since we have $$\mathrm{\int_{\pi^{-1}\left(y\right)}\vert s_{n}^{i}\vert^{2}\left(\omega\vert\pi^{-1}\left(y\right)\right)^{k}\leq\int_{\widehat{\pi}^{-1}\left(y\right)}\left(\zeta\circ p_{\mathbf{X}}\right)^{*}\vert s_{n}^{i}\vert^{2}\left(\omega^{\prime}\vert\widehat{\pi}^{-1}\left(y\right)\right)^{k}}$$ it follows by the compactness of $\mathrm{\widehat{\pi}^{-1}\left(y\right)}$ where $\mathrm{y\in \mathbf{Y}}$, that $$\mathrm{\Vert s_{n}^{i}\Vert^{2}<\infty \text{ for all }y\in \mathbf{Y}_{0}\subset\mathbf{Y}.}$$
\end{remark}

The next step is to define a sequence of collections of fiber probability densities attached to $\mathrm{\mathfrak{U}=\left\{\mathbf{U}_{i}\right\}_{i}}$ as follows.

\begin{definition} \label{Notation Sequence of Collections of Fiber Distribution Densities}Let $\mathrm{\mathfrak{U}=\left\{\mathbf{U}_{i}\right\}_{i}}$ and $\mathrm{\left\{\left(s_{n}^{i}\right)\right\}_{i}}$ be as above, i.e.\ $\mathrm{\mathbf{U}_{i}\subset\pi\left(\mathbf{X}^{i}\right)}$, then define a sequence of collections of fiber distribution densities on $\mathrm{\pi^{-1}\left(\mathbf{U}_{i}\right)\cap \mathbf{X}_{0}}$ by $$\mathrm{\phi^{i}_{n}\!:\!x\mapsto \phi^{i}_{n}\left(x\right)=\Vert s_{n}^{i}\Vert^{-2}\left(x\right)\vert s_{n}^{i}\vert^{2}\left(x\right).}$$ In the sequel, the terminology $\mathrm{\left\{\phi^{\mathfrak{U}}_{n}\right\}=\left\{\phi^{i}_{n}\right\}_{i}}$ will be used. Furthermore, $\mathrm{\left\{\phi^{\mathfrak{U}}_{n}\right\}}$ will be referred to as the collection of fiber distribution densities associated to $\mathrm{\left\{s_{n}^{i}\right\}_{i}}$.
\end{definition} 

Having defined $\mathrm{\left\{\phi^{\mathfrak{U}}_{n}\right\}}$, it is self-evident to introduce

\begin{definition}\label{Definition Fiber Probability Measure Tame Case} Let $\mathrm{\left\{\phi^{\mathfrak{U}}_{n}\right\}}$ be a sequence of collections of fiber distribution densities associated to a tame collection $\mathrm{\left\{s_{n}^{i}\right\}_{i}}$, then define a sequence of collections of fiber probability measures over $\mathrm{\mathbf{Y}_{0}}$ by $$\mathrm{\bm{\nu}_{n}^{i}\left(y\right)\!:\!\mathbf{A}\mapsto\bm{\nu}_{n}^{i}\left(y\right)\left(\mathbf{A}\right)\coloneqq\int_{\mathbf{A}}d{\bm{\nu}}^{i}_{n}\coloneqq \int_{\mathbf{A}}\phi^{i}_{n}\,d\,[\pi_{y}]\label{Notation Sequence Of Fiber Probability Measures}}$$ where $\mathrm{y\in \mathbf{U}_{i}\cap \mathbf{Y}_{0}}$ and $\mathrm{\mathbf{A}\subset\pi^{-1}\left(y\right)}$ measurable.

As in {\bf{Definition \ref{Notation Sequence of Collections of Fiber Distribution Densities}}}, set $\mathrm{\left\{\bm{\nu}^{\mathfrak{U}}_{n}\right\}=\left\{\bm{\nu}^{i}_{n}\right\}_{i}}$ \label{Notation Collection Of Fiber Probability Measures} and refer to $\mathrm{\left\{\bm{\nu}^{\mathfrak{U}}_{n}\right\}}$ as the collection of fiber probability measures associated to $\mathrm{\left\{s_{n}^{i}\right\}_{i}}$.
\end{definition}

We complete our definitions with

\begin{definition}\label{Definition Sequence Of Collections Of Fiber Probability Densities} Let $\mathrm{\left\{\phi^{\mathfrak{U}}_{n}\right\}}$ be a sequence of collections of fiber distribution densities associated to a tame collection $\mathrm{\left\{s_{n}^{i}\right\}_{i}}$, then define a sequence of collections of cumulative fiber probability densities over $\mathrm{\mathbf{Y}_{0}}$ by $$\mathrm{D_{n}^{i}\left(y,\cdot\right)\!:\!t\mapsto D_{n}^{i}\left(y,t\right)\coloneqq \int_{\left\{\phi^{i}_{n}\geq t\right\}\cap \pi^{-1}\left(\pi\left(x\right)\right)}d\,[\pi_{y}]\label{Notation Cumulative Distribution Densities}}$$ where $\mathrm{y\in \mathbf{U}_{i}\cap \mathbf{Y}_{0}}$.

As in the aforementioned definitions, we set $\mathrm{\left\{D^{\mathfrak{U}}_{n}\right\}=\left\{D^{i}_{n}\right\}_{i}}$ \label{Notation Collection Of Cumulative Fiber Probability Densities} and refer to $\mathrm{\left\{D^{\mathfrak{U}}_{n}\right\}}$ as the collections of cumulative fiber probability densities associated to $\mathrm{\left\{s_{n}^{i}\right\}_{i}}$.
\end{definition}

In {\bf{Section \ref{Section Uniform Convergence Theorems In The Tame Case}}} we will give a prove of the following two convergence results.

\begin{theo_3*}\label{Theorem Uniform Convergence of the Tame Measure Sequence} \textnormal{[\scshape{Uniform Convergence of the Tame Measure Sequence}]}\vspace{0.1 cm}\\
For for every tame collection $\mathrm{\left\{s_{n}^{i}\right\}_{i}}$ there exists a finite cover $\mathrm{\mathfrak{U}}$ of $\mathrm{\mathbf{Y}}$ with $\mathrm{\mathbf{U}_{i}\subset}$ $\mathrm{\pi\left(\mathbf{X}^{i}\right)}$ so that the collection of fiber probability measures $\mathrm{\left\{\bm{\nu}^{\mathfrak{U}}_{n}\right\}}$ associated to $\mathrm{\left\{s^{i}_{n}\right\}_{i}}$ converges uniformly on $\mathrm{\mathbf{Y}_{0}}$ to the fiber Dirac measure of $\mathrm{\mu^{-1}\left(\xi\right)\cap \mathbf{X}_{0}}$, i.e.\ for every $\mathrm{i\in I}$ and every $\mathrm{f\in \mathcal{C}^{0}\left(\mathbf{X}\right)}$, we have $$\begin{xy}
 \xymatrix{\mathrm{\left(y\mapsto \int_{\pi^{-1}\left(y\right)}f\, d\bm{\nu}_{n}^{i}\left(y\right)\right)}\ar[rr]^{\mathrm{\,\,\,\,\,}}&&\mathrm{f_{red}\text{ uniformly on }\mathbf{U}_{i}\cap \mathbf{Y}_{0}.\label{Notation Reduced Function}}
 }
\end{xy}$$
\end{theo_3*}

\begin{theo_4*}\label{Theorem Uniform Convergence of the Tame Distribution Sequence} \textnormal{[\scshape{Uniform Convergence of the Tame Distribution Sequence}]}\vspace{0.1 cm}\\
For every $\mathrm{t\in \mathbb{R}}$ and every tame collection $\mathrm{\left\{s_{n}^{i}\right\}_{i}}$ there exists a finite cover $\mathrm{\mathfrak{U}}$ of $\mathrm{\mathbf{Y}}$ with $\mathrm{\mathbf{U}_{i}\subset}$ $\mathrm{\pi\left(\mathbf{X}^{i}\right)}$ so that the collection of cumulative fiber probability densities $\mathrm{\left\{D^{\mathfrak{U}}_{n}\left(\cdot,t\right)\right\}}$ associated to $\mathrm{\left\{s^{i}_{n}\right\}_{i}}$ converges uniformly on $\mathrm{\mathbf{Y}_{0}}$ to the zero function on $\mathrm{\mathbf{Y}_{0}}$, i.e.\ for every $\mathrm{i\in I}$ we have $$\begin{xy}
 \xymatrix{\mathrm{\left(y\mapsto D_{n}^{i}\left(y,t\right)\right)}\ar[rr]^{\mathrm{\,\,\,\,\,}}&&\mathrm{0\text{ uniformly on }\mathbf{U}_{i}\cap \mathbf{Y}_{0}\subset \pi\left(\mathbf{X}^{i}\right).}
 }
\end{xy}$$
\end{theo_4*}

\subsection{The Fiber Probability Measure Sequence (Non-Tame Case)}\label{Definition of the Fiber Probability Measure Sequence (Non-Tame Case)} Let $\mathrm{\left(s_{n}\right)_{n}}$ be a sequence of $\mathrm{\mathbb{T}}$-eigensections so that $\mathrm{\vert \xi_{n}-n\,\xi\vert\in \mathcal{O}\left(1\right)}$. Moreover, let $\mathrm{\pi\!:\!\mathbf{X}^{ss}_{\xi}\rightarrow \mathbf{Y}=\mathbf{X}^{ss}_{\xi}/\!\!/\mathbb{T}}$ be the projection map attached to the Hilbert quotient associated to the level subset $\mathrm{\mu^{-1}\left(\xi\right)}$ and $\mathrm{\pi_{n}\!:\!\mathbf{X}^{ss}_{n^{-1}\xi_{n}}\rightarrow \mathbf{Y}_{n}=\mathbf{X}^{ss}_{n^{-1}\xi_{n}}/\!\!/\mathbb{T}}$  the corresponding Hilbert quotient associated to the level subset $\mathrm{\mu^{-1}\left(n^{-1}\xi_{n}\right)}$.

The aim of this subsection is to define a sequence $\mathrm{\left(\bm{\nu}_{n}\right)_{n}}$ of fiber probability measures over the base $\mathrm{\mathbf{Y}_{0}}$ by $$\mathrm{\bm{\nu}_{n}\left(y\right)\!:\!\mathbf{A}\mapsto\bm{\nu}_{n}\left(y\right)\left(\mathbf{A}\right)\coloneqq\int_{\mathbf{A}}\frac{\vert s_{n}\vert^{2}}{\Vert s_{n}\Vert^{2}}\, d\,[ \pi_{y}].}$$ However, since $\mathrm{\mathbf{X}\left(s_{n}\right)=\left\{x\in \mathbf{X}\!:\!\,?s_{n}\left(x\right)\neq 0\right\}}$ moves as $\mathrm{n\rightarrow \infty}$, the above measure is not well defined on all of $\mathrm{\bf{Y}}$:\ For example if $\mathrm{y\in \mathbf{Y}}$ is so that $\mathrm{\pi^{-1}\left(y\right)\subset \left\{s_{n}=0\right\}}$, then the above measure is not defined for $\mathrm{n}$ at $\mathrm{y}$. In some cases, it is however possible to circumvent this problem. For this we will first introduce the following definition.

\begin{definition}\label{Definition Removable Singularity} \textnormal{[\scshape{Removable Singularity}]} 
\\A point $\mathrm{y\in \mathbf{Y}}$ is defined to be a removable singularity of order $\mathrm{N_{0}}$ for the fiber measure $\mathrm{\bm{\nu}_{n}}$ if there exists an open neighborhood $\mathrm{\mathbf{U}_{y}\subset \mathbf{Y}}$ of $\mathrm{y}$, a sequence $\mathrm{\left(f_{y,n}\right)_{n}}$ of local holomorphic functions $\mathrm{f_{n}\in \mathcal{O}\left(\mathbf{U}_{y}\right)}$ and $\mathrm{N_{0}\in \mathbb{N}}$ so that $$\mathrm{\widehat{s}_{f_{y,n}}\coloneqq s_{n}\cdot\pi^{*}f_{y,n}^{-1}\label{Notation Local Extension Of S_{n}}}$$ defines a local holomorphic section on $\mathrm{\pi^{-1}\left(\mathbf{U}_{y}\right)}$ for all $\mathrm{n\geq N_{0}}$ which does not vanish identically on $\mathrm{\pi^{-1}\left(y^{\prime}\right)}$ for all $\mathrm{y^{\prime}\in \mathbf{U}_{y}}$. 

In the sequel, we will denote the set of all removable singularities of order $\mathrm{N_{0}}$ by $\mathrm{\mathbf{R}_{N_{0}}}$\label{Notation Removable Singularity Of Order N_0}.
\end{definition}
\begin{remark} The local extension $\mathrm{\widehat{s}_{f_{y,n}}}$ is again a $\mathrm{\xi_{n}}$-eigensection. 
\end{remark}

As a corollary of the definition, we deduce

\begin{cor}\label{Corollary Independence Of Extension} If $\mathrm{y\in\mathbf{Y}_{0}}$ is a removable singularity for $\mathrm{\bm{\nu}_{n}}$ of order $\mathrm{N_{0}\in\mathbb{N}}$, then after having shrunken $\mathrm{\mathbf{U}_{y}}$ appropriately, the quotient $\mathrm{\vert \widehat{s}_{f_{y,n}}\vert^{2}}$ $\mathrm{\Vert \widehat{s}_{f_{y,n}}\Vert^{-2}}$ is independent of the local scaling functions $\mathrm{\left(f_{y,n}\right)_{n}}$, i.e.\ we have $$\mathrm{\frac{\vert \widehat{s}_{f^{0}_{y,n}}\vert^{2}}{\Vert \widehat{s}_{f^{0}_{y,n}}\Vert^{2}}=\frac{\vert \widehat{s}_{f^{1}_{y,n}}\vert^{2}}{\Vert \widehat{s}_{f^{1}_{y,n}}\Vert^{2}}}$$ over $\mathrm{\mathbf{U}_{y}\subset \mathbf{Y}_{0}\cap \mathbf{R}_{N_{0}}}$. 
\end{cor}
\begin{proof} First of all, since $\mathrm{y}$ is contained in the open set $\mathrm{\mathbf{Y}_{0}\cap \mathbf{R}_{N_{0}}}$, we can assume that $\mathrm{\mathbf{U}_{y}\subset \mathbf{Y}_{0}}$. 

Let $\mathrm{\left(f^{i}_{y,n}\right)_{n}}$, $\mathrm{f^{i}_{y,n}\in \mathcal{O}\left(\mathbf{U}_{y}\right)}$ where $\mathrm{i\in\left\{0,1\right\}}$. As in {\bf{Definition \ref{Definition Removable Singularity}}}, we have $\mathrm{\widehat{s}_{f^{i}_{y,n}}}$ $\mathrm{=s_{n}\cdot}$ $\mathrm{\pi^{*}f^{i,-1}_{y,n}}$. Throughout this proof, we use the abbreviation $\mathrm{f^{i}_{y,n}=f_{i,n}}$ for $\mathrm{i\in\left\{0,1\right\}}$. Using this notation, it follows that 
\begin{equation}\label{Equation Quotient Of Holomorphic Functions}
\mathrm{\widehat{s}^{\phantom{-1}}_{f_{0,n}}=\pi^{*}(f_{1,n}^{\phantom{-1}}f_{0,n}^{-1})\cdot \widehat{s}_{f_{1,n}}}
\end{equation} where $\mathrm{h_{0,1,n}\coloneqq f_{1,n}^{\phantom{-1}}f_{0,n}^{-1}}$ yields a sequence of meromorphic function defined on $\mathrm{\mathbf{U}_{y}}$. We claim that after having shrunken $\mathrm{\mathbf{U}_{y}}$, each $\mathrm{h_{0,1,n}}$ is bounded from above and below away form zero on $\mathrm{\mathbf{U}_{y}}$.

For this note the following: Since $\mathrm{\widehat{s}_{f_{1},n}\vert \pi^{-1}\left(y\right)\not\equiv 0}$ there exists one $\mathrm{x\in \pi^{-1}\left(y\right)}$ so that $\mathrm{\widehat{s}_{f_{1},n}\left(x\right)\neq 0}$ and hence $\mathrm{\vert \widehat{s}_{f_{1},n}\vert^{2}\left(x\right)>0}$. In particular, we find an open neighborhood $\mathrm{\mathbf{V}\subset \mathbf{X}_{0}}$ of $\mathrm{x}$ so that 
\begin{equation}\label{Equation Simplification Of This Proof}
\mathrm{\vert \widehat{s}_{f_{1},n}\vert^{2}(x^{\prime})\geq c>0}
\end{equation} for all $\mathrm{x^{\prime}\in \mathbf{V}}$. By continuity we also have
\begin{equation}\label{Equation Simplification Of This Proof2} 
\mathrm{\vert \widehat{s}_{f_{0},n}\vert^{2}(x^{\prime})\leq C<\infty}
\end{equation} for all $\mathrm{x^{\prime}\in \mathbf{V}}$. 

Since $\mathrm{\mathbf{V}\subset \mathbf{X}_{0}}$ and since $\mathrm{\pi\vert \mathbf{X}_{0}\!:\!\mathbf{X}_{0}\rightarrow \mathbf{Y}_{0}}$ is a $\mathrm{k}$-fibering and hence an open map ({cf.\! }\cite{Loj}, {p.\ }{\bf{\oldstylenums{297} f.}}), we can assume (after having shrunken ${\mathbf{V}}$ and $\mathrm{\mathbf{U}_{y}}$ appropriately) that $\mathrm{\pi\left(\mathbf{V}\right)=\mathbf{U}_{y}}$. Combining \ref{Equation Simplification Of This Proof}, \ref{Equation Simplification Of This Proof2} and \ref{Equation Quotient Of Holomorphic Functions} we deduce $\mathrm{\vert h_{0,1,n}\vert^{2}\left(y^{\prime}\right)\leq C^{\prime}<\infty}$ for all $\mathrm{y^{\prime}\in \mathbf{U}_{y}}$. Reversing the roles of $\mathrm{f_{0}}$ and $\mathrm{f_{1}}$ shows (after having shrunken $\mathrm{\mathbf{U}_{y}}$ again) that $\mathrm{\vert h_{0,1,n}\vert^{2}\left(y^{\prime}\right)\geq c^{\prime}>0}$. Hence, the meromorphic function $\mathrm{h_{0,1,n}= f_{1,n}^{\phantom{-1}}f_{0,n}^{-1}}$ is bounded from above and below away from zero on $\mathrm{\mathbf{U}_{y}}$. As the base $\mathrm{\mathbf{Y}}$ is assumed to be normal, we can apply {\scshape{Riemann}}'s Extension Theorem (cf.\ \cite{Gr-Re}, {p.\ }{\bf{\oldstylenums{144}}}) in order to deduce that $\mathrm{h_{0,1,n}}$ yields a non-vanishing holomorphic function on $\mathrm{\mathbf{U}_{y}}$. All in all we, deduce $$\mathrm{\frac{\vert \widehat{s}_{f^{0}_{y,n}}\vert^{2}}{\Vert \widehat{s}_{f^{0}_{y,n}}\Vert^{2}}\cdot \frac{\pi^{*}\vert h_{0,1,n}\vert^{2}}{\pi^{*}\vert h_{0,1,n}\vert^{2}}=\frac{\vert \widehat{s}_{f^{1}_{y,n}}\vert^{2}}{\Vert \widehat{s}_{f^{1}_{y,n}}\Vert^{2}}}$$ over $\mathrm{\mathbf{U}_{y}}$ as claimed.

\end{proof} In this sense, the quotient $\mathrm{\vert s_{n}\vert^{2}\Vert s_{n}\Vert^{-2}}$ can be uniquely extended onto $\mathrm{\mathbf{R}_{N_{0}}\cap \mathbf{Y}_{0}}$. Assume now that $\mathrm{\mathbf{Y}_{0}\cap \mathbf{R}_{N_{0}}\neq \varnothing}$. Using the independence of $\mathrm{\vert \widehat{s}_{f_{y,n}}\vert^{2}\Vert \widehat{s}_{f_{y,n}}\Vert^{-2}}$ of the chosen sequence $\mathrm{\left(f_{y,n}\right)_{n}}$, $\mathrm{f_{y,n}\in \mathcal{O}\left(\mathbf{U}_{y}\right)}$ shown in {\bf{Corollary \ref{Corollary Independence Of Extension}}}, we can define

\begin{definition} \label{Notation Sequence of Collections of Fiber Distribution Densities Initial Sequence}Let $\mathrm{\left(s_{n}\right)_{n}}$ be sequence of $\mathrm{\xi_{n}}$-eigensections whose rescaled weights approximate $\mathrm{\xi}$, then define a sequence of fiber distribution densities over $\mathrm{\mathbf{Y}_{0}\cap \mathbf{R}_{N_{0}}}$ by $$\mathrm{\phi_{n}\!:\!x\mapsto \phi_{n}\left(x\right)\coloneqq \Vert \widehat{s}_{f_{y,n}}\Vert^{-2}\left(x\right)\vert  \widehat{s}_{f_{y,n}}\vert^{2}\left(x\right)}$$ where $\mathrm{\left(f_{y,n}\right)_{n}}$ is a sequence of holomorphic functions $\mathrm{f_{y,n}\in \mathcal{O}\left(\mathbf{U}_{y}\right)}$ as in {\bf{Definition \ref{Definition Removable Singularity}}}.
\end{definition} 

As in the tame case we introduce the following two definitions.

\begin{definition}\label{Definition Fiber Probability Measure Non Tame Case} Let $\mathrm{\left(\phi_{n}\right)_{n}}$ be the sequence of fiber distribution densities associated to $\mathrm{\left(s_{n}\right)_{n}}$ as defined in {\bf{Definition \ref{Notation Sequence of Collections of Fiber Distribution Densities Initial Sequence}}}, then we define a sequence of fiber probability measures parametrized over $\mathrm{\mathbf{Y}_{0}\cap \mathbf{R}_{N_{0}}}$ by $$\mathrm{\bm{\nu}_{n}\left(y\right)\!:\!\mathbf{A}\mapsto\bm{\nu}_{n}\left(y\right)\left(\mathbf{A}\right)\coloneqq\int_{\mathbf{A}}d{\bm{\nu}}_{n}\coloneqq \int_{\mathbf{A}}\phi_{n}\,d\,[\pi_{y}]}$$ for $\mathrm{\mathbf{A}\subset\pi^{-1}\left(y\right)}$ measurable where $\mathrm{y\in \mathbf{Y}_{0}\cap \mathbf{R}_{N_{0}}}$.
\end{definition}

\begin{definition} Let $\mathrm{\left(s_{n}\right)_{n}}$ be sequence of $\mathrm{\xi_{n}}$-eigensections whose rescaled weights approximate $\mathrm{\xi}$, then we define a sequence of cumulative fiber probability densities over $\mathrm{\mathbf{Y}_{0}\cap \mathbf{R}_{N_{0}}}$ by $$\mathrm{D_{n}\left(y,\cdot\right)\!:\!t\mapsto D_{n}\left(y,t\right)\coloneqq \int_{\left\{\phi_{n}\geq t\right\}\cap \pi^{-1}\left(\pi\left(x\right)\right)}d\,[\pi_{y}]}$$ where $\mathrm{y\in \mathbf{Y}_{0}\cap \mathbf{R}_{N_{0}}}$.
\end{definition}

The aim of {\bf{Section \ref{Section Uniform Convergence Theorems In The Non-Tame Case}}} is to give a proof of the following two convergence results.

\begin{theo_5.b*}\label{theo_5.b*}\textnormal{[\scshape{Uniform Convergence of the Initial Distribution Sequence}]}\vspace{0.1 cm}\\
For fixed $\mathrm{t\in \mathbb{R}}$ the sequence $\mathrm{\left(D_{n}\left(\cdot,t\right)\right)_{n}}$ converges uniformly on $\mathrm{\mathbf{Y}_{0}\cap\mathbf{R}_{N_{0}}}$ to the zero function.
\end{theo_5.b*}

\begin{theo_6.b*}\label{theo_6.b*}\textnormal{[\scshape{Uniform Convergence of the Initial Measure Sequence}]}\vspace{0.10 cm}\\
Let $\mathrm{f\in \mathcal{C}^{0}\!\left(\mathbf{X}\right)}$ then the sequence $$\mathrm{\left(y\mapsto \int_{\pi^{-1}\left(y\right)}f\,d\bm{\nu}_{n}\left(y\right)\right)_{n}}$$ converges uniformly over $\mathrm{\mathbf{Y}_{0}\cap \mathbf{R}_{N_{0}}}$ to the reduced function $\mathrm{f_{red}}$.
\end{theo_6.b*}

We close this section by showing that in general there exists no $\mathrm{N_{0}\in \mathbb{N}}$ so that $\mathrm{\mathbf{Y}=\mathbf{R}_{N_{0}}}$. Furthermore, the example shows that even the extreme case $\mathrm{\mathbf{R}_{N_{0}}=\varnothing}$ for all $\mathrm{N_{0}\in \mathbb{N}}$ is possible.

\begin{example}\label{Example Existence Of Removable Singularities} \textnormal{Let $\mathrm{\mathbf{X}\subset \mathbb{C}\mathbb{P}^{k}\times \mathbb{C}\mathbb{P}^{k}}$ equipped with the $\mathrm{\mathbb{T}=\mathbb{C}^{*}}$ action given by $$\mathrm{t.\left([\zeta_{0}\!:\!{\dots}\!:\!\zeta_{k}],[z_{0}\!:\!{\dots}\!:\!z_{k}]\right)=\left([\zeta_{0}\!:\!{\dots}\!:\!\zeta_{k}],[t^{-1}z_{0}\!:\!{\dots}\!:\!t^{-1}z_{k-1}\!:\!z_{k}]\right).}$$ In the sequel, we will consider the $\mathrm{\mathbb{T}}$-linearization $\mathrm{\mathbf{L}=p_{0}^{*}\mathcal{O}_{\mathbb{C}\mathbb{P}^{k}}\left(1\right)\otimes p_{1}^{*}\mathcal{O}_{\mathbb{C}\mathbb{P}^{k}}\left(1\right)}$ of the $\mathrm{\mathbb{T}}$-action on $\mathrm{\mathbf{X}}$ induced by the trivial $\mathrm{\mathbb{T}}$-action on the first factor $\mathrm{\mathcal{O}\left(1\right)_{\mathbb{C}\mathbb{P}^{k}}\rightarrow \mathbb{C}\mathbb{P}^{k}}$ given by $$\mathrm{t.[\zeta_{0}\!:\!{\dots}\!:\!\zeta_{k}\!:\!\zeta]=[\zeta_{0}\!:\!{\dots}\!:\!\zeta_{k}\!:\!\zeta]}$$ where we have used the identification $\mathrm{\mathcal{O}\left(1\right)_{\mathbb{C}\mathbb{P}^{k}}\cong \mathbb{C}\mathbb{P}^{k+1}\setminus\left\{[0\!:\!\dots\!:\!0\!:\!1]\right\}}$ and the $\mathrm{\mathbb{T}}$-action on the second factor $\mathrm{\mathcal{O}_{\mathbb{C}\mathbb{P}^{k}}\left(1\right)\rightarrow \mathbb{C}\mathbb{P}^{k}}$ given by $$\mathrm{t.[z_{0}\!:\!{\dots}\!:\!z_{k}\!:\!\zeta]=[t^{-1}z_{0}\!:\!{\dots}\!:\!t^{-1}z_{k-1}\!:\!z_{k}\!:\!\zeta].}$$ A calculation shows that the $\mathrm{\xi=0}$-level of the associated moment map is given by the set $$\mathrm{\mu^{-1}\left(0\right)=\mathbb{C}\mathbb{P}^{k}\times \left\{[0\!:\!{\dots}\!:\!0\!:\!1]\right\}}$$ and it is also direct to verify that one can identify $\mathrm{\mathbf{Y}=\mathbf{X}^{ss}_{\xi=0}/\!\!/\mathbb{T}\cong \mathbb{C}\mathbb{P}^{k}}$ where $\mathrm{\mathbf{X}^{ss}_{0}\cong}$ $\mathrm{\mathbb{C}\mathbb{P}^{k}\times \mathbb{C}^{k}}$. Furthermore, each fiber of the quotient map $\mathrm{\pi}$ is isomorphic to $\mathrm{\mathbb{C}^{k}}$ equipped with the inverse diagonal action.}

\textnormal{We will now consider the sequence $\mathrm{\left(s_{n}\right)_{n}}$ whose rescaled weights converge to $\mathrm{\xi=0}$ defined by $$\mathrm{s_{n}=\sum_{i=0}^{k-1}\zeta^{n}_{i}\,z_{i}\,z_{k}^{n-1}\in H^{0}\left(\mathbf{X},\mathbf{L}^{n}\right).}$$ Note that we have $$\mathrm{s_{n}\vert \pi^{-1}\left([0\!:\!{\dots}\!:\!0\!:\!1]\right)\equiv 0 \text{ for all }n\in \mathbb{N}.}$$} 

\textnormal{Using the homogenous standard coordinates $\mathrm{\zeta^{\prime}_{i}=\frac{\zeta_{i}}{\zeta_{k}}}$, $\mathrm{z^{\prime}_{i}=\frac{z_{i}}{z_{k}}}$, $\mathrm{0\leq i\leq k-1}$ on the open subset $$\mathrm{\mathbf{U}_{k,k}=\left\{\left([z],[\zeta]\right)\!:\!\,\zeta_{k}\neq 0,\,z_{k}\neq 0\right\}\subset \mathbf{X}^{ss}_{0}\subset \mathbb{C}\mathbb{P}^{k}\times \mathbb{C}\mathbb{P}^{k},}$$ the restriction $\mathrm{\pi\vert \mathbf{U}_{k,k}\rightarrow \mathbf{V}_{k}\cong \mathbb{C}^{k}}$ of the quotient map $\mathrm{\pi\!:\!\mathbf{X}^{ss}_{0}\rightarrow \mathbb{C}\mathbb{P}^{k}}$ is given by the projection map $\mathrm{\pi\vert \mathbf{U}_{k,k}\!:\!\mathbf{U}_{k,k}\cong\mathbb{C}^{k}\times \mathbb{C}^{k}\rightarrow \mathbb{C}^{k}}$. With respect to this trivialization the sequence $\mathrm{s_{n}\vert \mathbf{U}_{k,k}}$ is given by $$\mathrm{s_{n}\vert {\mathbf{U}_{k,k}}=\sum_{i=0}^{k-1}\zeta^{\prime,n}_{i}\,z^{\prime}_{i}}$$ where $\mathrm{\pi^{-1}\left([0\!:\!{\dots}\!:\!0\!:\!1]\right)=\left\{0\right\}\times \mathbb{C}^{k}}$. To shorten notation, we will just write $\mathrm{s_{n}\vert \mathbf{U}_{k,k}}$ $\mathrm{=s_{n}}$ throughout the rest of this example and set $\mathrm{\zeta^{\prime}_{i}=\zeta_{i}}$, $\mathrm{z^{\prime}_{i}=z_{i}}$ for all $\mathrm{0\leq i\leq k-1}$.}

\textnormal{Assume now that $\mathrm{\zeta_{0}=[0\!:\!{\dots}\!:\!0\!:\!1]}$ is a removable singularity for the fiber measure $\mathrm{\bm{\nu}_{n}}$ (we will fix $\mathrm{n}$ henceforth), i.e.\ there exists a non-vanishing holomorphic function $\mathrm{f\in \mathcal{O}\left(\mathbf{U}\right)}$ defined on an open neighborhood $\mathrm{\mathbf{U}\subset \mathbb{C}^{k}}$ of $\mathrm{\zeta_{0}}$ so that $\mathrm{\widehat{s}_{n}\coloneqq s_{n}\cdot\pi^{*}f^{-1}_{n}}$ is holomorphic and does not vanish identically on $\mathrm{\pi^{-1}\left(\zeta^{\prime}\right)}$ for all $\mathrm{y^{\prime}\in \mathbf{U}}$. Note that, after having shrunken $\mathrm{\mathbf{U}}$, we can assume that $\mathrm{\zeta_{0}}$ is an isolated zero of the function $\mathrm{f}$. In particular, the restriction $\mathrm{\widehat{s}_{n}\vert \pi^{-1}\left(\zeta^{\prime}\right)}$ defines a non-vanishing linear one form on $\mathrm{\pi^{-1}\left(\zeta^{\prime}\right)\cong\mathbb{C}^{k}}$ for all $\mathrm{\zeta^{\prime}\in \mathbf{U}}$. Using this, it follows that for each sequence $\mathrm{\left(\zeta_{m}\right)_{m}}$ in $\mathrm{\mathbf{U}}$ converging to $\mathrm{\zeta_{0}}$, the sequence of one-codimensional subspaces in $\mathrm{\mathbb{C}^{k}\cong \pi^{-1}\left(\zeta_{m}\right)}$ given by $\mathrm{{\bm{\mathfrak{H}}}\left(\zeta_{m}\right)=\left\{x\in\mathbb{C}^{k}\!:\!\,\widehat{s}_{n}\left(\zeta_{m}\right)=0\right\} }$ must converge to a uniquely defined one-codimensional subspace which is independent of the choice of $\mathrm{\left(\zeta_{m}\right)_{m}}$. However, this is a contradiction to the equation $\mathrm{\widehat{s}_{n}=s_{n}\cdot }$ $\mathrm{\pi^{*}f^{-1}}$ and the fact that $\mathrm{f}$ is non-zero on $\mathrm{\mathbf{U}\setminus\left\{\zeta_{0}\right\}}$. For example, consider the collection of sequences given by $$\mathrm{\left\{ \zeta_{m}^{i}\right\} _{i}=\left\{ \left(0,\dots,\mathrm{m^{-1}},\dots,0\right)_{m}\right\}_{i},}$$ then we have $$\mathrm{{\bm{\mathfrak{H}}}\left(\zeta_{m}^{i}\right)=\left\{ x\in\mathbb{C}^{k}\!:\!\,\widehat{s}_{n}\left(\zeta_{m}^{i}\right)=0\right\} \rightarrow\left\{ x\in\mathbb{C}^{k}\!:\!\, z_{i}=0\right\}}$$ so the limit is not independent of the chosen sequence. Hence, we deduce a contradiction and it follows that $\mathrm{\zeta_{0}}$ is a not a removable singularity for any $\mathrm{N_{0}\in \mathbb{N}}$ in the sense of {\bf{Definition \ref{Definition Removable Singularity}}}, i.e we have $\mathrm{\zeta_{0}\notin \mathbf{R}_{N_{0}}}$ for all $\mathrm{N_{0}}$.} 

\textnormal{Furthermore, by slightly changing the above sequence $\mathrm{\left(s_{n}\right)_{n}}$, we can show that $\mathrm{\mathbf{R}_{N_{0}}=\varnothing}$ for each $\mathrm{N_{0}}$:\ Choose a dense sequence $\mathrm{\left(\zeta_{n}\right)_{n}}$ in the quotient $\mathrm{\mathbf{Y}\cong \mathbb{C}\mathbb{P}^{k}}$ and let $\mathrm{\left(\Phi_{n}\right)_{n}}$, $\mathrm{\Phi_{n}\in}$ $\mathrm{Aut\left(\mathbf{Y}\right)}$ be a sequence of projective transformations so that $\mathrm{\Phi_{n}\left(\zeta_{0}\right)=\zeta_{n}}$. Define a new sequence of eigensections by $$\mathrm{s^{\prime}_{n}\left([\zeta],[z]\right)\coloneqq s_{1}\left([\Phi_{n}\left(\zeta\right)],[z]\right)}$$ and consider the sequence given by $$\mathrm{s_{n}\coloneqq \prod_{i=1}^{n}s^{\prime}_{i}.}$$ It is direct to see that $\mathrm{\mathbf{R}_{N_{0}}=\varnothing}$ because for each open neighborhood $\mathrm{\mathbf{U}}$ of any point $\mathrm{y\in \mathbf{Y}}$, the subset $\mathrm{\mathbf{U}\cap \left\{\zeta_{n}\right\}}$ is dense by construction and $\mathrm{\zeta_{n}}$ is non-removable. $\boldsymbol{\Box}$}
\end{example}

As a consequence of this example we conclude

\begin{remark} There are examples of approximating sequences $\mathrm{\left(s_{n}\right)_{n}}$ so that $\mathrm{\mathbf{R}_{N_{0}}=\varnothing}$ for all $\mathrm{N_{0}\in \mathbb{N}_{0}}$.
\end{remark}

As a further-reaching question, one could ask whether the set of all $\mathrm{\xi}$-approximating sequences of $\mathrm{\xi_{n}}$-eigensections $\mathrm{\left(s_{n}\right)_{n}}$, with the property that $\mathrm{\mathbf{R}_{N_{0}}=\varnothing}$ for all $\mathrm{N_{0}}$, is "thin" as a subset of $\mathrm{\bigoplus_{n=0}^{\infty}H^{0}\left(\mathbf{X},\mathbf{L}^{n}\right)}$. 

It turns out that there is no definitive answer to this question: In the context of {\bf{Example \ref{Example First Example}}}, one can show that each singularity of $\mathrm{s\in H^{0}\left(\mathbf{X},\mathbf{L}^{n}\right)}$ is removable and hence $\mathrm{\mathbf{Y}=\mathbf{R}_{N_{0}}}$ for all $\mathrm{N_{0}\in \mathbb{N}}$ and any choice of $\mathrm{\left(s_{n}\right)_{n}}$. On the other hand, if $\mathrm{k=2}$ in {\bf{Example \ref{Example Existence Of Removable Singularities}}}, it turns out that a $\mathrm{\xi_{n}}$-eigensection $\mathrm{s_{n}\in H^{0}\left(\mathbf{X},\mathbf{L}^{n}\right)}$, which has been randomly chosen with respect to a choice of a Lebesgue measure on $\mathrm{H^{0}\left(\mathbf{X},\mathbf{L}^{n}\right)}$ (induced by a choice of a basis), has almost surely at least one non-removable singularity $\mathrm{y_{n}\in \mathbf{Y}}$. Moreover, if one randomly chooses a $\mathrm{\xi}$-approximating sequence of $\mathrm{\xi_{n}}$-eigensections in this setting, it turns out that the set $\mathrm{\left\{y_{n}\right\}_{n\in \mathbb{N}}}$ is almost surely dense in $\mathrm{\mathbf{Y}}$.

\section{The $\mathrm{k}$-Fibering $\mathrm{\Pi\!:\!\widetilde{\bf{\,X\,}}\rightarrow \widetilde{\bf{\,Y\,}}}$}

\subsection{Construction of the $\mathrm{k}$-Fibering $\mathrm{\Pi\!:\!\widetilde{\bf{\,X\,}}\rightarrow \widetilde{\bf{\,Y\,}}}$}

Let $\mathrm{\widehat{\pi}\!:\!\widehat{\mathbf{\,X\,}}\rightarrow \mathbf{Y}}$ be the holomorphic map between the purely dimensional varieties $\mathrm{\widehat{\mathbf{\,X\,}}}$ and $\mathrm{\mathbf{Y}}$ as defined in {\bf{Section \ref{Definition of Y_{0}}}} and recall that there exits a Zariski-dense subset $\mathrm{\mathbf{Y}_{0}\subset \mathbf{Y}}$ so that all fibers $\mathrm{\widehat{\pi}^{-1}\left(y\right)}$ for $\mathrm{y\in \mathbf{Y}_{0}}$ are purely $\mathrm{k}$-dimensional, compact subvarieties not necessarily irreducible. Recall that we have assumed $\mathrm{\mathbf{X}}$ to be normal and hence, it follows (cf.\ \cite{He-Hu1}, {p.\ }{\bf{\oldstylenums{124}}}) that  $\mathrm{\mathbf{Y}}$ is normal and therefore, in particular, the open subset $\mathrm{\mathbf{Y}_{0}}$ as well.

By \cite{Bar1} the following is known in the above context: There exists a holomorphic map $\mathrm{\varphi^{\widehat{\pi}}:\mathbf{Y}_{0}\rightarrow\mathcal{C}^{k}(\mathbf{\widehat{\,\bm{X}\,}})}$ \label{Notation Regular Map Universal Property Chow Scheme} into the cycle space $\mathrm{\mathcal{C}^{k}(\mathbf{\widehat{\,\bm{X}\,}})}$ \label{Notation Cycle Space Of All k-Dim Cycles} of all compact $\mathrm{k}$-dimensional cycles $$\mathrm{{\bm{\mathfrak{C}}}=\sum_{i\in I}n_{i}\mathbf{C}_{i}, n_{i}\in \mathbb{N}, \mathbf{C}_{i}\subset \widehat{\bf{\,X\,}} \text{ globally irreducible subspaces of }\widehat{\bf{\,X\,}}\label{Notation Cycle Of Dimension K}}$$ of dimension $\mathrm{k}$ so that the support\footnote{The support $\mathrm{\vert{\bm{\mathfrak{C}}}\vert}$ of a cycle $\mathrm{{\bm{\mathfrak{C}}}=\sum_{i\in I}n_{i}\mathbf{C}_{i}}$ is defined by $\mathrm{\vert{\bm{\mathfrak{C}}}\vert=\bigcup_{i\in I}\mathbf{C}_{i}}$\label{Notation Support Of A Cycle}.} $\mathrm{\vert \varphi^{\widehat{\pi}}\left(y\right)\vert}$ of the cycle $\mathrm{{\bm{\mathfrak{C}}}_{y}\coloneqq \varphi^{\widehat{\pi}}\left(y\right)}$ \label{Notation Cycle Of Dimension K Associated To The Fiber Of} for $\mathrm{y\in {\bf{Y}}_{0}}$ is equal to the set theoretic fiber $\mathrm{\widehat{\pi}^{-1}\left(y\right)}$, i.e.\ we have 
\begin{equation}\label{Equation Equality Cycles Barlet Construction}
\mathrm{\vert {\bm{\mathfrak{C}}}_{y}\vert=\widehat{\pi}^{-1}\left(y\right)} \text{ for all }\mathrm{y\in \mathbf{Y}_{0}.}
\end{equation} Furthermore, since $\mathrm{\widehat{\bf{\,X\,}}}$ and $\mathrm{\mathbf{Y}}$ are compact, there exists (cf.\ \cite{Bar1}) a proper modification $\mathrm{\sigma\!:\!\widetilde{\,\bf{Y}\,}\rightarrow \mathbf{Y}}$ with center $\mathrm{\mathbf{Y}\setminus \mathbf{Y}_{0}}$, a proper modification $\mathrm{\Sigma\!:\!\widetilde{\,\mathbf{X}\,}\rightarrow \widehat{\,\bf{X}\,}}$\label{Notation Blow Up Space} \label{Notation Blow Up Map onto X} with center $\mathrm{\widehat{\pi}^{-1}\left(\mathbf{Y}\setminus \mathbf{Y}_{0}\right)}$ and a surjective holomorphic map $\mathrm{\Pi\!:\!\widetilde{\,\mathbf{X}\,}\rightarrow \widetilde{\,\mathbf{Y}\,}}$ \label{Notation Blown Up Projection Map} so that the following diagram commutes: $$\begin{xy}\label{Commutating Diagram}
\xymatrix{
\mathrm{\widehat{\mathbf{\,X\,}}}\ar[d]^{\mathrm{\widehat{\pi}}} &&& \mathrm{\widetilde{\bf{\,X\,}}}\ar[d]^{\mathrm{\Pi}}\ar[lll]_{\mathrm{\Sigma}}\\
\mathrm{\mathbf{Y}}&&& \mathrm{\widetilde{\bf{\,Y\,}}}\ar[lll]_{\mathrm{\sigma}}\\
}
\end{xy}
$$ The compact, complex space $\mathrm{\widetilde{\bf{\,Y\,}}}$ is given by 
\begin{equation}
\mathrm{\widetilde{\,\bf{Y}\,}\coloneqq  cl\,\left\{\left(y,{\bm{\mathfrak{C}}}\right)\in {\bf{Y}}_{0}\times\mathcal{C}^{k}(\widehat{\,\bf{X}\,})\!:\!\, \varphi^{\widehat{\pi}}\left(y\right)={\bm{\mathfrak{C}}}\right\} \subset \mathbf{Y}\times\mathcal{C}^{k}(\widehat{\,\bf{X}\,})}\label{Notation Blown Up Space Associated to Y}
\end{equation} and the holomorphic map $\mathrm{\sigma\!:\!\widetilde{\,\bf{Y}\,}\rightarrow \mathbf{Y}}$ \label{Notation Blow Up Map onto Y} is defined by $\mathrm{\sigma\coloneqq p_{\mathbf{Y}}\vert\widetilde{\,\bf{Y}\,}}$. Moreover, if $\mathrm{{\bm{\mathfrak{X}}}\subset \mathcal{C}^{k}(\widehat{\,\bf{X}\,})\times}$ $\mathrm{\widehat{\,\bf{X}\,}}$\label{Notation Universal Space} denotes the universal space defined by $\mathrm{{\bm{\mathfrak{X}}}\coloneqq \left\{\big({\bm{\mathfrak{C}}},x\right)\in \mathcal{C}^{k}(\widehat{\,\bf{X}\,})\times \widehat{\,\bf{X}\,}: x\in \vert {\bm{\mathfrak{C}}}\vert\big\}}$, then $\mathrm{\widetilde{\bf{\,X\,}}}$ is compact and given by $$\mathrm{\widetilde{\bf{\,X\,}}=\left(\mathbf{Y}\times {\bm{\mathfrak{X}}}\right)\cap \big(\widetilde{\bf{\,Y\,}}\times\widehat{\bf{\,X\,}}\big)}$$ where $\mathrm{\Sigma}$ is the restriction of the projection $\mathrm{p_{}\!:\!\mathbf{Y}\times\mathcal{C}^{k}(\widehat{\,\bf{X}\,})\times\widehat{\bf{\,X\,}}\rightarrow \widehat{\bf{\,X\,}}}$ to $\mathrm{\widetilde{\bf{\,X\,}}}$. Note that the above construction, whose details can be found in \cite{Bar1}, implies that the fiber $\mathrm{\Pi^{-1}\left(y,{\bm{\mathfrak{C}}}\right)\in}$ $\mathrm{ \widetilde{\bf{\,X\,}}}$ identifies with $\mathrm{\vert{\bm{\mathfrak{C}}}\vert\subset \widehat{\,\bf{X}\,}}$. Using this identification, we will simply write $\mathrm{\Pi^{-1}\left(y,\bm{\mathfrak{C}}\right)=\vert \bm{\mathfrak{{C}}}\vert}$ henceforth.

\begin{remark} Note that $\mathrm{\Pi\!:\!\widetilde{\bf{\,X\,}}\rightarrow \widetilde{\bf{\,Y\,}}}$ is a holomorphic map whose fibers are purely $\mathrm{k}$-dimensional by construction. Furthermore, $\mathrm{\widetilde{\bf{\,X\,}}}$ and $\widetilde{\bf{\,Y\,}}$ are both purely dimensional where $\mathrm{dim_{\mathbb{C}}\widetilde{\bf{\,X\,}}}$ $\mathrm{=dim_{\mathbb{C}}\bf{\,X\,}}$ and  $\mathrm{dim_{\mathbb{C}}\widetilde{\bf{\,Y\,}}=dim_{\mathbb{C}}\bf{\,Y\,}}$ (which is a well known fact of the theory of proper modifications, cf.\ {\textnormal{\cite{Gr-Re}}}, {p.\ }{\bf{\oldstylenums{214}}}). 
\end{remark}

Moreover, we have the following lemma.

\begin{lemma}\label{Invariance of Cycles} The support $\mathrm{\Pi^{-1}\left(y,\bm{\mathfrak{C}}\right)=\vert \bm{\mathfrak{{C}}}\vert}$ is $\mathrm{\mathbb{T}}$-invariant subset of $\mathrm{\widehat{\pi}^{-1}\left(y\right)}$.
\end{lemma}
\begin{proof} By the commutativity of the previous diagram, we deduce that $\mathrm{\vert{\bm{\mathfrak{C}}} \vert\subset \widehat{\pi}^{-1}\left(y\right)}$, hence it remains to verify that $\mathrm{\vert{\bm{\mathfrak{C}}} \vert}$ is $\mathrm{\mathbb{T}}$-invariant. In order to prove this, we can proceed as follows: Since the set $$\mathrm{\left\{\left(y,{\bm{\mathfrak{C}}}\right)\in {{\bf{Y}}}_{0}\times\mathcal{C}^{k}(\widehat{\,\bf{X}\,})\!:\!\, \varphi^{\widehat{\pi}}\left(y\right)={\bm{\mathfrak{C}}}\right\}}$$ is Euclidean dense in $\mathrm{\widetilde{\bf{\,Y\,}}}$, we can choose a sequence $\mathrm{\left(y_{n}\right)_{n}}$\label{Notation y,frak C_y} in $\mathrm{\mathbf{Y}_{0}}$ so that $\mathrm{\left(y_{n},{\bm{\mathfrak{C}}}_{y_{n}}\right)\rightarrow}$ $\mathrm{\left(y,{\bm{\mathfrak{C}}}\right)}$ where $\mathrm{{\bm{\mathfrak{C}}}_{y_{n}}}$ are the cycles whose underlying sets are equal to $\mathrm{\widehat{\pi}^{-1}\left(y_{n}\right)}$ by property \ref{Equation Equality Cycles Barlet Construction}. The $\mathrm{\mathbb{T}}$-invariance of the limit cycle $\mathrm{{\bm{\mathfrak{C}}}}$ follows by the following reasoning: Let $\mathrm{x=t.x_{0}\in }$ $\mathrm{\mathbb{T}.\mathfrak {C}}$ where $\mathrm{x_{0}\in{\bm{\mathfrak{C}}} }$. Then, since $\mathrm{{\bm{\mathfrak{C}}}_{n}\rightarrow {\bm{\mathfrak{C}}}}$ means convergence in the Hausdorff topology of the underlying support, there exists a sequence $\mathrm{x_{n}\in {\bm{\mathfrak{C}}}_{n}}$ so that $\mathrm{x_{n}\rightarrow x_{0}}$. Since $\mathrm{\vert {\bm{\mathfrak{C}}}_{y}\vert=\widehat{\pi}^{-1}\left(y\right)}$ for all $\mathrm{y\in \mathbf{Y}_{0}}$, it follows that $\mathrm{\vert {\bm{\mathfrak{C}}}_{y_{n}}\vert=\widehat{\pi}^{-1}\left(y_{n}\right)}$ for all $\mathrm{n\in \mathbb{N}}$. Using the $\mathrm{\mathbb{T}}$-invariance of $\mathrm{\widehat{\pi}^{-1}\left(y\right)}$ (cf.\ {\bf{Remark \ref{Remark T-Invariance of Fibers}}}), we deduce that $\mathrm{t.x_{n}\in\vert {\bm{\mathfrak{C}}}_{y_{n}}\vert}$. By the continuity of the action it follows that $\mathrm{t.x_{n}\rightarrow t.x_{0}}$. So $\mathrm{\left(t.x_{n}\right)_{n}}$ is a convergent sequence with limit $\mathrm{t.x_{0}}$ where $\mathrm{t.x_{n}\in {\bm{\mathfrak{C}}}_{n}}$ and $\mathrm{{\bm{\mathfrak{C}}}_{n}\rightarrow {\bm{\mathfrak{C}}}}$. By the definition of the Hausdorff topology, it then follows that $\mathrm{t.x_{0}\in{\bm{\mathfrak{C}}}}$ and hence $\mathrm{\mathbb{T}.{\bm{\mathfrak{C}}}\subset {\bm{\mathfrak{C}}}}$ which proves the claim $\mathrm{\mathbb{T}.{\bm{\mathfrak{C}}}={\bm{\mathfrak{C}}}}$. 
\end{proof}

We close this section with the following remark and example.

\begin{remark} In general ({cf.\! }{\bf{Example \ref{Example Blowing Up Example}}}) $\mathrm{\vert \bm{\mathfrak{{C}}}\vert}$ is a proper subset of $\mathrm{\widehat{\pi}^{-1}\left(y\right)}$.
\end{remark}
 
\begin{example}\label{Example Blowing Up Example} \textnormal{Let $\mathrm{\mathbf{X}=\mathbb{C}\mathbb{P}^{3}}$ equipped with the $\mathrm{\mathbb{T}=\mathbb{C}^{*}}$ action given by $$\mathrm{t.[z_{0}\!:\!z_{1}\!:\!z_{2}\!:\!z_{3}]=[t^{-1}z_{0}\!:\!t\,z_{1}\!:\!t\,z_{2}\!:\!z_{3}]}$$ and consider the Hilbert quotient $$\mathrm{\pi\!:\!\mathbf{X}^{ss}_{0}=\mathbb{C}\mathbb{P}^{3}\setminus\left(\left\{[1\!:\!0\!:\!0\!:\!0]\right\}\cup\left\{z_{0}=z_{3}=0\right\}\right)\rightarrow \mathbf{Y}\cong\mathbb{C}\mathbb{P}^{2}}$$ associated to the $\mathrm{0=\xi}$-level set of the moment map $$\mathrm{\mu\!:\![z]\mapsto\Vert z\Vert^{-2}\left(-\vert z_{0}\vert^{2}+\vert z_{1}\vert^{2}+\vert z_{2}\vert^{2}\right).}$$ The corresponding projection map $\mathrm{\pi}$ is given by $\mathrm{\pi\!:\![z]\mapsto[\zeta]=[z_{0}z_{1}\!:\!z_{0}z_{2}\!:\!z_{3}^2]}$ where $\mathrm{[\zeta]\in }$ $\mathrm{\mathbf{Y}\cong \mathbb{C}\mathbb{P}^{2}}$. Note that all fibers $\mathrm{\widehat{\pi}^{-1}\left([\zeta]\right)}$ over $\mathrm{\mathbf{Y}_{0}=\mathbb{C}\mathbb{P}^{2}\setminus \left\{[0\!:\!0\!:\!1]\right\}}$ are of pure dimension one and of degree two. Moreover, it is direct to verify that these fibers can be parameterized by $$\mathrm{\gamma_{[\zeta]}\!:\!t\mapsto [\zeta_{2}t_{1}^{2}\!:\!\zeta_{0}t_{0}^{2}\!:\!\zeta_{1}t_{0}^{2}\!:\!\zeta_{2}t_{0}t_{1}]\text{ for }\zeta\in \mathbf{Y}_{0}}$$ and that they are given as the zero set of the following system of equations\!:\! $$\mathrm{\zeta_{2}z_{0}z_{1}-\zeta_{0}z^{2}_{3}=0,\,\zeta_{2}z_{0}z_{2}-\zeta_{1}z_{3}^{2}=0\text{ where }\zeta\in \mathbf{Y}_{0}.}$$ Let $\mathrm{\mathbf{U}_{2}=\left\{[\zeta]\in \mathbb{C}\mathbb{P}^{2}\!:\!\,\zeta_{2}\neq 0\right\}}$ and set $\mathrm{c_{0}\coloneqq \zeta_{2}^{-1}\zeta_{0}}$, $\mathrm{c_{1}\coloneqq \zeta_{2}^{-1}\zeta_{1}}$ and consider $\mathrm{\mathbf{U}_{2}^{*}\coloneqq}$ $\mathrm{\mathbf{Y}_{0}\cap \mathbf{U}_{2}}$ which we can identify with $\mathrm{\mathbb{C}^{2}\setminus \left\{0\right\}}$. As mentioned before, there exists a holomorphic map $\mathrm{\varphi^{\widehat{\pi}}\!:\!\mathbf{U}^{*}_{2}\hookrightarrow \mathcal{C}^{1}(\widehat{\,\mathbf{X}\,})}$. It turns out that all fibers of $\mathrm{\widehat{\pi}}$ are compact subvarieties of degree $\mathrm{2}$ in $\mathrm{\mathbb{C}\mathbb{P}^{3}}$. Hence, it follows that the image of $\mathrm{\varphi^{\widehat{\pi}}}$ is contained in the cycle space component which can be identified with the compact connected {\scshape{Chow}} Variety $\mathrm{\mathcal{C}_{1,2}\left(\mathbb{C}\mathbb{P}^{3}\right)}$ of all $\mathrm{1}$-dimensional cycles in $\mathrm{\mathbb{C}\mathbb{P}^{3}}$ of degree $\mathrm{2}$ which itself is realized as closed variety in the projective space $\mathrm{\mathbb{C}\mathbb{P}^{\nu_{3,1,2}}}$ (for a rigorous definition cf.\ \cite{Sha}).} 

\textnormal{Recall that the Chow coordinates of a cycle $\mathrm{{\bm{\mathfrak{C}}}}$ in $\mathrm{\mathbf{X}\subset \mathbb{C}\mathbb{P}^{m}}$ of degree $\mathrm{d}$ and dimension $\mathrm{k}$ are given by the coefficients of the Chow form $\mathrm{\mathfrak{F}_{{\bm{\mathfrak{C}}},\mathbb{C}\mathbb{P}^{m}}}$ \label{Notation Chow Form}, i.e.\ by the coefficients of a polynomial homogenous in $\mathrm{k+1}$ groups $\mathrm{\xi_{0}^{(i)},\dots,\xi^{(i)}_{m}}$, $\mathrm{i\in \left\{0,\dots,k\right\}}$ of $\mathrm{m+1}$ indeterminates of degree $\mathrm{d}$ modulo multiplication with a non-vanishing complex number $\mathrm{\lambda\in \mathbb{C}^{*}}$ (cf.\ \cite{Sha}). In the sequel, let $\mathrm{{\bm{\mathfrak{C}}}_{c}}$ for $\mathrm{c\in \mathbb{C}^{2}\setminus \left\{0\right\}}$ and let $\mathrm{\mathfrak{F}_{{\bm{\mathfrak{C}}}_{c},\mathbb{C}\mathbb{P}^{3}}}$ be the corresponding Chow form.
A calculation shows that
$$\mathrm{\begin{array}{rl}
 & \mathrm{\mathfrak{F}_{{\bm{\mathfrak{C}}}_{c},\mathbb{C}\mathbb{P}^{3}}\left(\xi_{0}^{\left(0\right)},\xi_{1}^{\left(0\right)},\xi_{2}^{\left(0\right)},\xi_{3}^{\left(0\right)},\xi_{0}^{\left(1\right)},\xi_{1}^{\left(1\right)},\xi_{2}^{\left(1\right)},\xi_{3}^{\left(1\right)}\right)}\\[0.4 cm]
=& \mathrm{c_{0}^{2}\,{\xi_{1}^{(0)}}^{2}{\xi_{0}^{(1)}}^{2}+c_{1}^{2}\,{\xi_{2}^{\left(0\right)}}^{2}{\xi_{0}^{\left(1\right)}}^{2}+c_{0}^{2}\,{\xi_{0}^{\left(0\right)}}^{2}{\xi_{1}^{\left(1\right)}}^{2}+c_{1}^{2}\,{\xi_{0}^{\left(0\right)}}^{2}{\xi_{2}^{\left(1\right)}}^{2}+2\, c_{0}\,c_{1}\,{\xi_{0}^{\left(0\right)}}^{2}\xi_{1}^{\left(1\right)}\xi_{2}^{\left(1\right)}}\\[0.4 cm]
&\mathrm{+c_{0}\,\xi_{0}^{\left(0\right)}\xi_{1}^{\left(0\right)}{\xi_{3}^{\left(1\right)}}^{2}+c_{1}\,\xi_{0}^{\left(0\right)}\xi_{2}^{\left(0\right)}{\xi_{3}^{\left(1\right)}}^{2}+c_{0}\,{\xi_{3}^{\left(0\right)}}^{2}\xi_{0}^{\left(1\right)}\xi_{1}^{\left(1\right)}+c_{1}\,{\xi_{3}^{\left(0\right)}}^{2}\xi_{0}^{\left(1\right)}\xi_{2}^{\left(1\right)}}\\[0.4 cm]
 &\mathrm{+2\,c_{0}\,c_{1}\,\xi_{1}^{\left(0\right)}\xi_{2}^{\left(0\right)}{\xi_{0}^{\left(1\right)}}^{2}-c_{0}\,\xi_{0}^{\left(0\right)}\xi_{3}^{\left(0\right)}\xi_{1}^{\left(1\right)}\xi_{3}^{\left(1\right)}-c_{1}\,\xi_{0}^{\left(0\right)}\xi_{3}^{\left(0\right)}\xi_{2}^{\left(1\right)}\xi_{3}^{\left(1\right)}-c_{0}\,\xi_{1}^{\left(0\right)}\xi_{3}^{\left(0\right)}\xi_{0}^{\left(1\right)}\xi_{3}^{\left(1\right)}}\\[0.4 cm]
 &\mathrm{-c_{1}\,\xi_{2}^{\left(0\right)}\xi_{3}^{\left(0\right)}\xi_{0}^{\left(1\right)}\xi_{3}^{\left(1\right)}-2\,c_{0}^{2}\,\xi_{0}^{\left(0\right)}\xi_{1}^{\left(0\right)}\xi_{0}^{\left(1\right)}\xi_{1}^{\left(1\right)}-2\,c_{0}\,c_{1}\,\xi_{0}^{\left(0\right)}\xi_{1}^{\left(0\right)}\xi_{0}^{\left(1\right)}\xi_{2}^{\left(1\right)}}\\[0.4 cm]
 &\mathrm{-2\,c_{0}\,c_{1}\,\xi_{0}^{\left(0\right)}\xi_{2}^{\left(0\right)}\xi_{0}^{\left(1\right)}\xi_{1}^{\left(1\right)}-2\,c_{1}^{2}\,\xi_{0}^{\left(0\right)}\xi_{2}^{\left(0\right)}\xi_{0}^{\left(1\right)}\xi_{2}^{\left(1\right)}}\\
\end{array}}$$ so the map $\mathrm{\varphi^{\widehat{\pi}}\!:\!\mathbb{C}^{2}\setminus \left\{0\right\}\rightarrow \mathcal{C}_{1,2}\left(\mathbb{C}\mathbb{P}^{3}\right)}$ is given by $$\mathrm{\begin{array}{rcl}
\mathrm{\varphi^{\widehat{\pi}}\!:\!\mathbb{C}^{2}\setminus \left\{0\right\}\ni\left(c_{1},c_{2}\right)} & \mapsto & \mathrm{[c_{0}^{2}\!:\!c_{0}^{2}\!:\!\!-2\, c_{0}^{2}\!:\!c_{1}^{2}\!:\!c_{1}^{2}\!:\!\!-2\, c_{1}^{2}\!:\!2\, c_{0}c_{1}\!:\!2\, c_{0}c_{1}\!:\!\,-2\, c_{0}c_{1}\!:\!}\\[0,2 cm]
 &  & \mathrm{-2\, c_{0}c_{1}\!:\!-c_{0}\!:\!\!-c_{0}\!:\!c_{0}\!:\!c_{0}\!:\!\!-c_{1}\!:\!\!-c_{1}\!:\!c_{1}\!:\!c_{1}\!:\!0\!:\!...\!:\!0].}\end{array}}$$ The closure of the graph $$\mathrm{\left\{\left(c,{\bm{\mathfrak{C}}}\right)\!:\!\,c\in\mathbb{C}^{2}\setminus \left\{0\right\},\,\varphi^{\widehat{\pi}}\left(c\right)={\bm{\mathfrak{C}}}\right\}\subset \mathbb{C}^{2}\times \mathcal{C}_{1,2}\left(\mathbb{C}\mathbb{P}^{3}\right)}$$ of the map $\mathrm{\varphi^{\widehat{\pi}}}$ turns out to be isomorphic to the blow-up $\mathrm{ {\bm{\mathfrak{Bl}}}\left(0,\mathbb{C}^{2}\right)}$ of the origin $\mathrm{0\in \mathbb{C}^{2}}$. This can be seen in the following way: Let $$\mathrm{\varphi\!:\!\mathbb{C}^{2}\setminus\left\{0\right\}\ni\left(c_{0},c_{1}\right)\mapsto \left(\left(c_{0},c_{1}\right),[c_{0}\!:\!c_{1}]\right) \in {\bm{ {\bm{\mathfrak{Bl}}}}}\left(0,\mathbb{C}^{2}\right)}$$ where $\mathrm{ {\bm{\mathfrak{Bl}}}\left(0,\mathbb{C}^{2}\right)=\left\{\left(\left(c_{0},c_{1}\right),[c^{\prime}_{0}\!:\!c^{\prime}_{1}]\right)\!:\!\,c_{0}c^{\prime}_{1}-c_{1}c^{\prime}_{0}=0\right\}\subset \mathbb{C}^{2}\times \mathbb{C}\mathbb{P}^{1}}$. We have $$\mathrm{cl\left(\varphi\left(\mathbb{C}^{2}\setminus\{0\}\right)\right)= {\bm{\mathfrak{Bl}}}\left(0,\mathbb{C}^{2}\right).}$$ We can embed $\mathrm{ {\bm{\mathfrak{Bl}}}\left(0,\mathbb{C}^{2}\right)\subset \mathbb{C}^{2}\times \mathbb{C}\mathbb{P}^{1}\subset \mathbb{C}\mathbb{P}^{2}\times \mathbb{C}\mathbb{P}^{1}}$ into $\mathrm{\mathbb{C}\mathbb{P}^{5}}$ using the {\scshape{Segre}} map. The composition of $\mathrm{\varphi}$ with the {\scshape{Segre}} embedding yields an embedding $\mathrm{\widehat{\varphi}\!:\!\mathbb{C}^{2}\setminus}$ $\mathrm{\left\{0\right\}\hookrightarrow \mathbb{C}\mathbb{P}^{5}}$ given by $$\mathrm{\widehat{\varphi}\!:\!\mathbb{C}^{2}\setminus\left\{0\right\}\mapsto [c_{0}\!:\!c_{1}\!:\!c_{0}^2\!:\!c_{0}c_{1}\!:\!c_{0}c_{1}\!:\!c_{1}^{2}].}$$ It is direct to see that there exists a projective transformation $\mathrm{\Phi}$ of $\mathrm{\mathbb{C}\mathbb{P}^{\nu_{3,1,2}}}$ so that the following diagram commutes:
$$\begin{xy}
\xymatrix{
\mathrm{\mathbb{C}^{2}\setminus\left\{0\right\}}  \ar@{^{(}->}[d]  \ar@{^{(}->}[rrrrr]^{\zeta_{\mathbb{C}^{2}\setminus\left\{0\right\}}} &&&&&\mathrm{\mathcal{C}_{1,2}\left(\mathbb{C}\mathbb{P}^{3}\right)\subset}\hspace{-0.9 cm}& \mathrm{\mathbb{C}\mathbb{P}^{\nu_{3,1,2}}}\\
\mathrm{ {\bm{\mathfrak{Bl}}}\left(0,\mathbb{C}^{2}\right)}&**[r] \hspace{-1cm}\subset \mathbb{C}\mathbb{P}^{2}\times \mathbb{C}\mathbb{P}^{1}\ar@{^{(}->}[rrrr]^-{\text{{\scshape{Segre}} map}}&&& &\mathrm{\mathbb{C}\mathbb{P}^{5}\subset }\hspace{ -1.85 cm} & \mathrm{\mathbb{C}\mathbb{P}^{\nu_{3,1,2}}}\ar[u]_{\Phi} \\
}
\end{xy}
$$ Let $\mathrm{\left(\left(c_{n},{\bm{\mathfrak{C}}}_{c_{n}}\right)\right)_{n}}$ be the sequence in $\mathrm{\widetilde{\bf{\,Y\,}}\cap\sigma^{-1}\left(\mathbb{C}^{2}\setminus{0}\right)}$ given by $\mathrm{c_{n}=n^{-1}\left(c^{\prime}_{0},c_{1}^{\prime}\right)}$. The above formulas show that $\mathrm{\left({\bm{\mathfrak{C}}}_{c_{n}}\right)_{n}}$ converges to $$\mathrm{{\bm{\mathfrak{C}}}_{c^{\prime}}=[0\!:\!0\!:\!0\!:\!0\!:\!0\!:\!0\!:\!0\!:\!0\!:\!-c^{\prime}_{0}\!:\!-c^{\prime}_{0}\!:\!c^{\prime}_{0}\!:\!c^{\prime}_{0}\!:\!-c^{\prime}_{1}\!:\!-c^{\prime}_{1}\!:\!c^{\prime}_{1}\!:\!c^{\prime}_{1}\!:\!0\!:\!...\!:\!0]}$$ which corresponds to the point $\mathrm{\left(\left(0,0\right),[c^{\prime}_{1}\!:\!c^{\prime}_{2}]\right)\in  {\bm{\mathfrak{Bl}}}\left(0,\mathbb{C}^{2}\right)}$ under the above identification. In particular we have $\mathrm{\sigma\left(0,{\bm{\mathfrak{C}}}_{c^{\prime}}\right)=[0\!:\!0\!:\!1]\in \mathbb{C}^{2}}$. In order to determine the limit cycle $$\mathrm{{\bm{\mathfrak{C}}}_{c^{\prime}}\in \widetilde{\bf{\,Y\,}}\cap\sigma^{-1}\left(\mathbf{U}_{2}\right)}$$ we consider the Chow form $\mathrm{\mathfrak{F}_{{\bm{\mathfrak{C}}}_{c^{\prime}},\mathbb{C}\mathbb{P}^{3}}}$ 
$$\mathrm{\begin{array}{rl}
& \mathrm{\mathfrak{F}_{{\bm{\mathfrak{C}}}_{c^{\prime}},\mathbb{C}\mathbb{P}^{3}}\left(\xi_{0}^{\left(0\right)},\xi_{1}^{\left(0\right)},\xi_{2}^{\left(0\right)},\xi_{3}^{\left(0\right)},\xi_{0}^{\left(1\right)},\xi_{1}^{\left(1\right)},\xi_{2}^{\left(1\right)},\xi_{3}^{\left(1\right)}\right)}\\[0.4 cm]
=& \mathrm{c^{\prime}_{0}\,\xi_{0}^{\left(0\right)}\xi_{1}^{\left(0\right)}{\xi_{3}^{\left(1\right)}}^{2}+c^{\prime}_{1}\,\xi_{0}^{\left(0\right)}\xi_{2}^{\left(0\right)}{\xi_{3}^{\left(1\right)}}^{2}+c^{\prime}_{0}\,{\xi_{3}^{\left(0\right)}}^{2}\xi_{0}^{\left(1\right)}\xi_{1}^{\left(1\right)}+c^{\prime}_{1}\,{\xi_{3}^{\left(0\right)}}^{2}\xi_{0}^{\left(1\right)}\xi_{2}^{\left(1\right)}-c^{\prime}_{0}\,\xi_{0}^{\left(0\right)}\xi_{3}^{\left(0\right)}\xi_{1}^{\left(1\right)}\xi_{3}^{\left(1\right)}}\\[0.4 cm]
&\mathrm{-c^{\prime}_{1}\,\xi_{0}^{\left(0\right)}\xi_{3}^{\left(0\right)}\xi_{2}^{\left(1\right)}\xi_{3}^{\left(1\right)}-c^{\prime}_{0}\,\xi_{1}^{\left(0\right)}\xi_{3}^{\left(0\right)}\xi_{0}^{\left(1\right)}\xi_{3}^{\left(1\right)}-c^{\prime}_{1}\,\xi_{2}^{\left(0\right)}\xi_{3}^{\left(0\right)}\xi_{0}^{\left(1\right)}\xi_{3}^{\left(1\right)}}\end{array}}$$ which turns out to be reducible:}

\textnormal{$$\mathrm{\begin{array}{rl}
& \mathrm{\mathfrak{F}_{{\bm{\mathfrak{C}}}_{c^{\prime}},\mathbb{C}\mathbb{P}^{3}}\left(\xi_{0}^{\left(0\right)},\xi_{1}^{\left(0\right)},\xi_{2}^{\left(0\right)},\xi_{3}^{\left(0\right)},\xi_{0}^{\left(1\right)},\xi_{1}^{\left(1\right)},\xi_{2}^{\left(1\right)},\xi_{3}^{\left(1\right)}\right)}\\[0.4 cm]
=&\mathrm{\mathrm{\overset{\mathfrak{F}_{\widetilde{{\bm{\mathfrak{C}}}}_{c^{\prime}},\mathbb{C}\mathbb{P}^{3}}}{\overbrace{\left(\mathrm{c^{\prime}_{0}\,\xi_{1}^{\left(0\right)}\xi_{3}^{\left(1\right)}+c_{1}^{\prime}\,\xi_{2}^{\left(0\right)}\xi_{3}^{\left(1\right)}-c^{\prime}_{0}\,\xi_{3}^{\left(0\right)}\xi_{1}^{\left(1\right)}-c^{\prime}_{1}\,\xi_{3}^{\left(0\right)}\xi_{2}^{\left(1\right)}}\right)}}}\cdot\overset{\mathfrak{F}_{{\bm{\mathfrak{C}}}_{0},\mathbb{C}\mathbb{P}^{3}}}{\overbrace{\left(\mathrm{\xi_{0}^{\left(0\right)}\xi_{3}^{\left(1\right)}-\xi_{3}^{\left(0\right)}\xi_{0}^{\left(1\right)}}\right)}}}\\
\end{array}}$$ A direct computation shows that $\mathrm{\mathfrak{F}_{\widetilde{{\bm{\mathfrak{C}}}}_{c^{\prime}},\mathbb{C}\mathbb{P}^{3}}}$ is the Chow form associated to the line $$\mathrm{\widetilde{{\bm{\mathfrak{C}}}}_{c^{\prime}}=\left\{z_{0}=0,\,z_{4}=0, \,c^{\prime}_{2}z_{1}-c_{1}^{\prime}z_{2}=0\right\}}$$ and $\mathrm{\mathfrak{F}_{{\bm{\mathfrak{C}}}_{0},\mathbb{C}\mathbb{P}^{3}}}$ is the corresponding Chow form of the line $$\mathrm{{\bm{\mathfrak{C}}}_{0}=\left\{z_{1}=0,\,z_{2}=0,\,z_{3}=0\right\}}$$ where $\mathrm{{\bm{\mathfrak{C}}}_{c^{\prime}}=\widetilde{{\bm{\mathfrak{C}}}}_{c^{\prime}}+ {\bm{\mathfrak{C}}}_{0}}$. In particular, note that $\mathrm{{\bm{\mathfrak{C}}}_{c^{\prime}}\neq \widehat{\pi}^{-1}\left([0\!:\!0\!:\!1]\right)}$ where $$\mathrm{\widehat{\pi}^{-1}\left([0\!:\!0\!:\!1]\right)=\left\{z_{0}=0\right\}\cup\left\{z_{1}=z_{2}=z_{3}=0\right\}.\,\boldsymbol{\Box}}$$}
\end{example}

\subsection{Fiber Integral Properties of the $\mathrm{k}$-Fibering $\mathrm{\Pi\!:\!\widetilde{\bf{\,X\,}}\rightarrow \widetilde{\bf{\,Y\,}}}$}

As in {\bf{Section \ref{Uniform Localization Proposition}}}, let $\mathrm{\mathbf{X}^{i}}$ be the $\mathrm{\mathbb{T}}$-invariant Zariski open subset of $\mathrm{\mathbf{X}\left(s_{n}^{i}\right)\subset\mathbf{X}^{ss}_{\xi}}$ of {\bf{Theorem 1}}, where $\mathrm{\mathbf{X}\left(s_{n}^{i}\right)}$ is the $\mathrm{n}$-stable complement of the zero set of the tame sequence $\mathrm{\left(s^{i}_{n}\right)_{n}}$. As before, let $\mathrm{\varrho^{i}\!:\!\mathbf{X}^{i}}$ $\mathrm{\rightarrow \mathbb{R}}$ be the normalized s.p.s.h.\ limit function and recall the definition of the compact tube $\mathrm{T\left(\epsilon,\mathbf{W}^{i}\right)}$ $\mathrm{\subset \mathbf{X}^{i}}$ for $\mathrm{\mathbf{W}^{i}\subset \pi\left(\mathbf{X}^{i}\right)}$ a compact neighborhood and $\mathrm{\epsilon>0}$ given in {\bf{Section \ref{Uniform Localization Proposition}}}: $$\mathrm{T\left(\epsilon,\mathbf{W}^{i}\right)= \left(\varrho^{i}\times \pi\right)^{-1}\left([0,\epsilon]\times \mathbf{W}^{i}\right).}$$ Define the corresponding tube $\mathrm{\widetilde{\,T\,}\left(\epsilon,\mathbf{W}^{i}\right)}$ \label{Notation Compact Corresponding Neighborhood Tube} in $\mathrm{\widetilde{\bf{\,X\,}}}$ by $$\mathrm{\widetilde{\,T\,}\left(\epsilon,\mathbf{W}^{i}\right)\coloneqq \Sigma^{-1}\left(T\left(\epsilon,\mathbf{W}^{i}\right)\right)}$$ where we have used the fact that $\mathrm{\mathbf{X}^{ss}_{\xi}\supset T\left(\epsilon,\mathbf{W}^{i}\right)}$ is naturally embedded in $\mathrm{\widehat{\mathbf{\,\bm{X}\,}}}$ via $\mathrm{\zeta^{-1}\vert \mathbf{X}^{ss}_{\xi}}$ (recall that $\mathrm{\zeta\vert \zeta^{-1}\big(\mathbf{X}^{ss}_{\xi}\big)}$ is biholomorphic because $\mathrm{\mathbf{X}^{ss}_{\xi}}$ is assumed to be normal). Note that $\mathrm{\widetilde{\,T\,}\left(\epsilon,\mathbf{W}^{i}\right)}$ projects down via $\mathrm{\Pi}$ onto the compact neighborhood $\mathrm{\widetilde{\bf{W}}^{i}\coloneqq\sigma^{-1}\left(\mathbf{W}^{i}\right)}$. \label{Notation cal W^i}

The first aim of this section is to show that the fiber integral $\mathrm{vol\left(\pi^{-1}\left(y\right)\right)=\int_{\pi^{-1}\left(y\right)}d\,[\pi_{y}]}$ \label{Notation Volume of a Fiber Over Y_0} is bounded as $\mathrm{y}$ varies in $\mathrm{\mathbf{Y}_{0}}$.
\begin{lemma}\label{Lemma Equality of Volume} For all $\mathrm{y\in \mathbf{W}^{i}\cap \mathbf{Y}_{0}}$ we have $$\mathrm{vol\left(T\left(\epsilon,{\mathbf{W}}^{i}\right)\cap\pi^{-1}\left(y\right)\right)=vol\left(\widetilde{\,T\,}\left(\epsilon,{\bf{W}}^{i}\right)\cap\Pi^{-1}\left(y,{\bm{\mathfrak{C}}}_{y}\right)\right)}$$ where the right hand side is the fiber integral of the projection taken with respect to $\mathrm{\Omega^{k}}$ where $\mathrm{\Omega\coloneqq\Sigma^{*}\omega^{\prime}}$ \label{Notation Capital Omega}.

Furthermore, we have $$\mathrm{vol\left(\pi^{-1}\left(y\right)\right)\leq vol\left(\Pi^{-1}\left(y,{\bm{\mathfrak{C}}}_{y}\right)\right)}$$ for all $\mathrm{y \in\mathbf{Y}_{0}}$.
\end{lemma}
\begin{proof} This is a direct consequence of the following reasoning: Recall that $\mathrm{\Sigma}$ is a modification with center $\mathrm{\widehat{\pi}^{-1}\left(\mathbf{Y}\setminus \mathbf{Y}_{0}\right)}$ and $\mathrm{\zeta\vert\zeta^{-1}\big(\mathbf{X}^{ss}_{\xi}\big)}$ is an isomorphism (as before we consider $\mathrm{\mathbf{X}^{ss}_{\xi}}$ embedded in $\mathrm{\mathbf{\Gamma}_{\pi}}$). Hence, it follows that the open subset $$\mathrm{\Pi^{-1}\left(\zeta^{-1}\left(\mathbf{X}_{0}\right)\right), \text{ where }\mathbf{X}_{0}=\pi^{-1}\left(\mathbf{Y}_{0}\right)}$$ in $\mathrm{\widehat{\,\mathbf{X}\,}}$ is mapped isomorphically on $\mathrm{\pi^{-1}\left(\mathbf{Y}_{0}\right)}$ by $\mathrm{\zeta\circ\Sigma}$. Therefore, if $\mathrm{y\in \mathbf{Y}_{0}}$, we deduce that $\mathrm{\pi^{-1}\left(y\right)}$ is biholomorphic to $$\mathrm{\Pi^{-1}\left(y,{\bm{\mathfrak{C}}}_{y}\right)\cap \Sigma^{-1}\big(\zeta^{-1}\left(\mathbf{X}^{ss}_{\xi}\right)\big)}$$ via $\mathrm{\left(\zeta\circ \Sigma\right)^{-1}}$. Using the fact that $\mathrm{\widetilde{\,T\,}\left(\epsilon,{\bf{W}}^{i}\right)}$ is defined as the pull back of $\mathrm{T\left(\epsilon,{\bf{W}}^{i}\right)}$ and that $\mathrm{\Omega=(\zeta\circ}$ $\mathrm{\Sigma)^{*}\omega}$, it follows that the volume of $\mathrm{T\left(\epsilon,{\mathbf{W}}^{i}\right)\cap\pi^{-1}\left(y\right)}$ with respect to $\mathrm{\omega}$ is equal to the volume of $\mathrm{\widetilde{\,T\,}\left(\epsilon,{\bf{W}}^{i}\right)\cap\Pi^{-1}\left(y,{\bm{\mathfrak{C}}}_{y}\right)}$ with respect to $\mathrm{\Omega}$, which proves the first claim.

The second claim is an immediate consequence of the above argumentation: Via $\mathrm{\left(\zeta\circ \Sigma\right)^{-1}}$, the fiber $\mathrm{\pi^{-1}\left(y\right)}$, where $\mathrm{y\in \mathbf{Y}_{0}}$, is biholomorphic to $$\mathrm{\Pi^{-1}\left(y,{\bm{\mathfrak{C}}}_{y}\right)\cap \Sigma^{-1}\big(\zeta^{-1}\left(\mathbf{X}^{ss}_{\xi}\right)\big).}$$ In particular, it can be seen as a subset of $\mathrm{\Pi^{-1}\left(y,{\bm{\mathfrak{C}}}_{y}\right)}$ realized by $\mathrm{\left(\zeta\circ \Sigma\right)^{-1}}$.  
\end{proof}

\begin{remark} In general, the inequality in {\bf{Lemma \ref{Lemma Equality of Volume}}} is a strict inequality. This is exhibited in {\bf{Example \ref{Example Proper Inclusion}}}, where $$\mathrm{cl\left(\pi^{-1}\left(\left[1\!:\!0\right]\right)\right)=\left\{ z_{1}=0\right\} \text{ and }cl\left(\pi^{-1}\left(\left[0\!:\!1\right]\right)\right)=\left\{ \zeta_{1}=0\right\}}$$ on the one hand and $$\mathrm{p_{\mathbf{X}}\left(\widehat{\pi}^{-1}\left(\left[1\!:\!0\right]\right)\right)=\left\{ z_{1}=0\right\} \cup\left\{ \zeta_{0}=0\right\} \text{ resp. }}$$ $$\mathrm{p_{\mathbf{X}}\left(\widehat{\pi}^{-1}\left(\left[0\!:\!1\right]\right)\right)=\left\{ \zeta_{1}=0\right\} \cup\left\{ z_{0}=0\right\} }$$ on the other hand. Hence, $\mathrm{cl\left(\pi^{-1}\left(\left[\zeta_{i}\right]\right)\right)}$ $\mathrm{i\in\left\{0,1\right\}}$ is properly contained as an irreducible component in $\mathrm{p_{\mathbf{X}}\left(\widehat{\pi}^{-1}\left([\zeta_{i}]\right)\right)}$.
\end{remark}

As a direct consequence of {\bf{Lemma \ref{Lemma Equality of Volume}}}, we deduce the following two corollaries.

\begin{cor}\label{Boundedness of Fiber Integral} Let $\mathrm{\mathbf{Y}_{0}}$ be as in {\bf{Section \ref{Definition of Y_{0}}}}, then there exists a constant $\mathrm{C>0}$ so that $$\mathrm{vol\left(\pi^{-1}\left(y\right)\right)=\int_{\pi^{-1}\left(y\right)}\,d\,[\pi_{y}]\leq C}$$ for all $\mathrm{y\in \mathbf{Y}_{0}}$.
\end{cor}
\begin{proof} By the second claim of {\bf{Lemma \ref{Lemma Equality of Volume}}}, we have $$\mathrm{vol\left(\pi^{-1}\left(y\right)\right)\leq vol\left(\Pi^{-1}\left(y,{\bm{\mathfrak{C}}}_{y}\right)\right)\text{ for all }y\in\mathbf{Y}_{0}.}$$ Since the projection of $\mathrm{\widetilde{\bf{\,Y\,}}}$ on $\mathrm{\mathcal{C}^{k}\big(\widehat{\bf{\,X\,}}\big)}$ is a compact subset, the claim then follows by the fact that the volumes of all cycles which are contained in a compact subset of $\mathrm{\mathcal{C}^{k}\big(\widehat{\bf{\,X\,}}\big)}$ are uniformly bounded from above ({cf.\! }\cite{Bar2}).
\end{proof}

\begin{lemma}\label{Lemma Intersection Zero Measure} Let $\mathrm{\left(y,{\bm{\mathfrak{C}}}\right)\in\widetilde{\bf{\,Y\,}}}$, then $\mathrm{\Pi^{-1}\left(y,{\bm{\mathfrak{C}}}_{y}\right)\cap \widetilde{\,T\,}\left(0,\mathbf{W}^{i}\right)}$ is of measure zero concerning the measure induced by $\mathrm{\Omega}$.

Moreover, the restriction of the form $\mathrm{\Omega}$ on $\mathrm{\Pi^{-1}\left(y,{\bm{\mathfrak{C}}}\right)\cap \widetilde{\,T\,}\left(\epsilon,\mathbf{W}^{i}\right)}$ where $\mathrm{\epsilon>0}$ is non-zero.
\end{lemma}
\begin{proof} First of all note, that we can view $\mathrm{\Pi^{-1}\left(y,{\bm{\mathfrak{C}}}\right)\cap \widetilde{\,T\,}\left(\epsilon,\mathbf{W}^{i}\right)}$ as a $\mathrm{\mathbb{T}}$-invariant, closed, $\mathrm{k}$-dimensional complex subspace in $\mathrm{\pi^{-1}\left(y\right)\cap T\left(\epsilon,\mathbf{W}^{i}\right)}$. In fact, each fiber $\mathrm{\Pi^{-1}\left(y,{\bm{\mathfrak{C}}}\right)}$ is given by the $\mathrm{\mathbb{T}}$-invariant, $\mathrm{k}$-dimensional subvariety $\mathrm{\vert{\bm{\mathfrak{C}}}\vert\subset \widehat{\,\bf{X}\,}}$. The identification is then induced by $\mathrm{p_{\mathbf{X}}\vert cl\left({\mathbf{\Gamma}}_{\pi}\right)\circ\zeta}$ which is biholomorphic over $\mathrm{\mathbf{X}^{ss}_{\xi}\supset T\left(\epsilon,{\mathbf{W}}^{i}\right)}$. Now, the second claim of the lemma is an immediate consequence of this fact combined with $\mathrm{\omega^{\prime}=\left(p_{\mathbf{X}}\vert cl\left({\mathbf{\Gamma}}_{\pi}\right)\circ\zeta\right)^{*}\omega}$. 

The first claim follows from the fact that the minimal closed, $\mathrm{\mathbb{T}}$-invariant complex space of $\mathrm{\pi^{-1}\left(y\right)\cap T\left(\epsilon,\mathbf{W}^{i}\right)}$ containing $$\mathrm{\pi^{-1}\left(y\right)\cap T\left(0,\mathbf{W}^{i}\right)=\pi^{-1}\left(y\right)\cap \mu^{-1}\left(\xi\right)}$$ is given by the unique closed orbit $\mathrm{\mathbb{T}.z_{y}}$ which contains $\mathrm{\pi^{-1}\left(y\right)\cap T\left(0,\mathbf{W}^{i}\right)}$ as a total real submanifold.
\end{proof}

The rest of this section is devoted to the proof of the existence of uniform estimates concerning the fiber volume.

\begin{prop}\label{Continuity of Fiber Integral} Let $\mathrm{\Delta>0}$ and $\mathrm{\mathbf{W}^{i}\subset \pi\left(\mathbf{X}^{i}\right)}$ be as above, then there exists $\mathrm{\epsilon_{\Delta}>0}$ so that $$\mathrm{vol\left(\pi^{-1}\left(y\right)\cap T\left(\epsilon_{\Delta},\mathbf{W}^{i}\right) \right)\leq\Delta}$$ for all $\mathrm{y\in \mathbf{W}^{i}\cap \mathbf{Y}_{0}}$. 

Moreover, if $\mathrm{\epsilon>0}$, then there exists $\mathrm{\delta>0}$ so that $$\mathrm{\delta\leq vol\left(\pi^{-1}\left(y\right)\cap T\left(\epsilon,\mathbf{W}^{i}\right) \right)}$$ for all $\mathrm{y\in \mathbf{W}^{i}\cap \mathbf{Y}_{0}}$.
\end{prop}
\begin{proof} Choose a sequence $\mathrm{\epsilon_{n}\rightarrow 0}$ and a sequence $\mathrm{\left(\psi_{n}\right)_{n}}$ of smooth cut-off functions on $\mathrm{\widetilde{\bf{\,X\,}}}$ so that $$\mathrm{\psi_{n}\vert \widetilde{\,T\,}\left(\epsilon_{n},\mathbf{W}^{i}\right)\equiv 1\text{ and }supp\,\psi_{n}\cap  \widetilde{\,T\,}^{c}\left(\epsilon_{n+1},\mathbf{W}^{i}\right)=\varnothing}$$ which is possible since $\mathrm{\widetilde{\,T\,}\left(\epsilon_{n},\mathbf{W}^{i}\right)}$ is relatively compact. In particular, we have $$\mathrm{\Pi^{-1}\big(\widetilde{\bf{W}}^{i}\big)\cap supp\,\psi_{n}\subset \widetilde{\,T\,}\left(\epsilon_{n+1},\mathbf{W}^{i}\right)}$$ and therefore $$\mathrm{\Pi^{-1}\big(\widetilde{\bf{W}}^{i}\big)\cap supp\,\psi_{n}\downarrow \widetilde{\,T\,}\left(0,\mathbf{W}^{i}\right) \text{ as }n\rightarrow \infty.}$$ Hence, by the first claim of {\bf{Lemma \ref{Lemma Intersection Zero Measure}}}, the intersection of the set $\mathrm{supp\,\psi_{n}}$ and a fiber of $\mathrm{\Pi}$ over $\mathrm{\widetilde{\bf{W}}^{i}}$ converges monotonically decreasing to a set of measure zero with respect to the measure induced by $\mathrm{\Omega}$. Set $\mathrm{\Omega_{n}\coloneqq \psi_{n}\Omega^{k}}$.

To sum up, $\mathrm{\left(\Omega_{n}\right)_{n}}$ is a sequence of smooth $\mathrm{\left(k,k\right)}$-forms on $\mathrm{\widetilde{\bf{\,X\,}}}$ with compact support where $\mathrm{\Pi\!:\!\widetilde{\bf{\,X\,}}\rightarrow \widetilde{\bf{\,Y\,}}}$ is a holomorphic map between purely dimensional compact complex spaces so that each fiber $\mathrm{\Pi^{-1}\left(y\right)}$ is purely $\mathrm{k}$-dimensional. Furthermore, for each $\mathrm{\left(y,{\bm{\mathfrak{C}}}\right)\in\widetilde{\bf{W}}^{i}\subset \widetilde{\bf{\,Y\,}}}$ we know that $$\mathrm{\int_{\Pi^{-1}\left(y,{\bm{\mathfrak{C}}}\right)}\Omega_{n}\downarrow 0.}$$ 

In order to finish the proof of {\bf{Proposition \ref{Continuity of Fiber Integral}}}, we need the following lemma, which we will proved at the end of this section.

\begin{lemma}\label{Lemma Fiber Integration} Let $\mathrm{\Pi\!:\!\widetilde{\bf{\,X\,}}\rightarrow \widetilde{\bf{\,Y\,}}}$ be a $\mathrm{k}$-fibering between compact spaces and let $\mathrm{{\widetilde{\bf{W}}}\subset \widetilde{\bf{\,Y\,}}}$ be a closed subset.

\begin{enumerate}  
\item Let $\mathrm{\left(\Omega_{n}\right)_{n}}$ be a sequence of smooth, positive $\mathrm{\left(k,k\right)}$-forms on $\mathrm{\widetilde{\bf{\,X\,}}}$ and assume that $$\mathrm{\int_{\Pi^{-1}\left(y\right)}\Omega_{n}\downarrow 0}$$ for each $\mathrm{y\in{\widetilde{\bf{W}}}\subset \widetilde{\bf{\,X\,}}}$ and let $\mathrm{\Delta>0}$. Then there exists $\mathrm{n_{\Delta}\in \mathbb{N}}$ so that $$\mathrm{\mathrm{\int_{\Pi^{-1}\left(y\right)}\Omega_{n}\leq\Delta}\text{ for all }y\in {\widetilde{\bf{W}}}\text{ and all }n\geq n_{\Delta}.}$$
\item If $\mathrm{\Omega}$ is a smooth, positive $\mathrm{\left(k,k\right)}$-form on $\mathrm{\widetilde{\bf{\,X\,}}}$ so that $$\mathrm{\int_{\Pi^{-1}\left(y\right)}\Omega>0}$$ for all $\mathrm{y\in {\widetilde{\bf{W}}}}$. Then there exists $\mathrm{\delta>0}$ so that $$\mathrm{\delta\leq\int_{\Pi^{-1}\left(y\right)}\Omega\text{ for all }y\in {\widetilde{\bf{W}}}.}$$
\end{enumerate} 
\end{lemma}

Applying the first statement of the above lemma to the fiber integral yields the existence of $\mathrm{n_{\Delta}\in \mathbb{N}}$ so that $$\mathrm{\left(y\mapsto\int_{\Pi^{-1}\left(y,{\bm{\mathfrak{C}}}\right)}\Omega_{n}\right)\leq\Delta\text{ for all }n\geq n_{\Delta}\text{ and all }\left(y,{\bm{\mathfrak{C}}}\right)\in {\widetilde{\bf{W}}}^{i}.}$$ By the choice $\mathrm{\psi_{n}\vert\widetilde{\,T\,}\left(\epsilon_{n},\mathbf{W}^{i}\right)\equiv1}$, we deduce
\begin{equation*} 
\begin{split}
\mathrm{\left(y,{\bm{\mathfrak{C}}}\right)\mapsto \int_{\Pi^{-1}\left(y,{\bm{\mathfrak{C}}}\right)}\Omega_{n}}&=\mathrm{\int_{\Pi^{-1}\left(y,{\bm{\mathfrak{C}}}\right)}\psi_{n}\Omega^{k}}\\[0,2 cm]
&\geq\mathrm{vol\left(\widetilde{\,T\,}\left(\epsilon_{n},\mathbf{W}^{i}\right)\cap \Pi^{-1}\left(y,{\bm{\mathfrak{C}}}_{y}\right)\right)}\\
\end{split}
\end{equation*} for all $\mathrm{\left(y,{\bm{\mathfrak{C}}}\right)\in {\widetilde{\bf{W}}}^{i}}$. Now, set $\mathrm{n=n_{\triangle}}$, resp.\ $\mathrm{\epsilon=\epsilon_{n_{\triangle}}}$ and note that for all $\mathrm{y\in \mathbf{W}^{i}\cap \mathbf{Y}_{0}}$ we have $$\mathrm{vol\left(T\left(\epsilon,\mathbf{W}^{i}\right)\cap\pi^{-1}\left(y\right)\right)=vol\left(\widetilde{\,T\,}\left(\epsilon,\mathbf{W}^{i}\right)\cap\Pi^{-1}\left(y,{\bm{\mathfrak{C}}}_{y}\right)\right)}$$ by {\bf{Lemma \ref{Lemma Equality of Volume}}}. Hence, it follows that $$\mathrm{vol\left(T\left(\epsilon,\mathbf{W}^{i}\right)\cap\pi^{-1}\left(y\right)\right)\leq\Delta\text{ for all }y\in \mathbf{W}^{i}\cap \mathbf{Y}_{0}}$$ as claimed.

The second claim is a direct consequence of the second claim of the above lemma applied to the smooth form $\mathrm{\Omega_{n_{0}}}$, where $\mathrm{n_{0}\in \mathbb{N}}$ is chosen so that $\mathrm{\epsilon_{n_{0}+1}<\epsilon}$. In fact, we have
\begin{equation}\label{Equation wisa fwwq}
\begin{split}
\mathrm{\left(y,{\bm{\mathfrak{C}}}\right)\mapsto \int_{\Pi^{-1}\left(y,{\bm{\mathfrak{C}}}\right)}\Omega_{n_{0}}}&=\mathrm{\int_{\Pi^{-1}\left(y,{\bm{\mathfrak{C}}}\right)}\psi_{n_{0}}\Omega^{k}}\\[0,2 cm]
&\leq\mathrm{vol\left(\widetilde{\,T\,}\left(\epsilon,\mathbf{W}^{i}\right)\cap \Pi^{-1}\left(y,{\bm{\mathfrak{C}}}_{y}\right)\right)}\\
\end{split}
\end{equation} for all $\mathrm{\left(y,{\bm{\mathfrak{C}}}\right)\in {\widetilde{\bf{W}}}^{i}}$ which follows by $\mathrm{\Pi^{-1}\big(\widetilde{\bf{W}}^{i}\big)\cap supp\,\psi_{n}\subset \widetilde{\,T\,}\left(\epsilon_{n+1},\mathbf{W}^{i}\right)}$. Note that $\mathrm{\Omega_{n_{0}}}$ has compact support and that the left hand side of \ref{Equation wisa fwwq} is non-zero for all  $\mathrm{\left(y,{\bm{\mathfrak{C}}}\right)\in {\widetilde{\bf{W}}}^{i}}$ which follows by the second claim of {\bf{Lemma \ref{Lemma Intersection Zero Measure}}} applied to $\mathrm{\epsilon=\epsilon_{n_{0}}}$: $$\mathrm{0<vol\left(\widetilde{\,T\,}\left(\epsilon_{n_{0}},\mathbf{W}^{i}\right)\cap \Pi^{-1}\left(y,{\bm{\mathfrak{C}}}_{y}\right)\right)\leq\int_{\Pi^{-1}\left(y,{\bm{\mathfrak{C}}}\right)}\Omega_{n_{0}}\text{ for all }\left(y,{\bm{\mathfrak{C}}}\right)\in {\widetilde{\bf{W}}}^{i}.}$$ Applying the equality $$\mathrm{vol\left(T\left(\epsilon,\mathbf{W}^{i}\right)\cap\pi^{-1}\left(y\right)\right)=vol\left(\widetilde{\,T\,}\left(\epsilon,\mathbf{W}^{i}\right)\cap\Pi^{-1}\left(y,{\bm{\mathfrak{C}}}_{y}\right)\right)}$$ for all $\mathrm{y\in \mathbf{W}^{i}\cap \mathbf{Y}_{0}}$ and using the second claim of {\bf{Lemma \ref{Lemma Fiber Integration}}} completes the proof.
\end{proof}

It remains to prove {\bf{Lemma \ref{Lemma Fiber Integration}}}.

\begin{proof} (of {\bf{Lemma \ref{Lemma Fiber Integration}}}) Let $\mathrm{\xi\!:\!\widetilde{\bf{\,Y\,}}^{nor}\rightarrow \widetilde{\bf{\,Y\,}}}$ \label{Notation Normalization Map Xi} be the normalization of $\mathrm{\widetilde{\bf{\,Y\,}}}$ and consider the pull back space $\mathrm{\xi^{*}\widetilde{\bf{\,X\,}}}$ of $\mathrm{\widetilde{\bf{\,X\,}}}$ which is defined to be the complex space given by $$\mathrm{\xi^{*}\widetilde{\bf{\,X\,}}\coloneqq\left\{ \left(x,\widehat{y}\right)\!:\!\,\Pi\left(x\right)=\xi\left(\widehat{y}\right)\right\} \subset\widetilde{\bf{\,X\,}}\times\widetilde{\bf{\,Y\,}}^{nor}.}$$ Note that we have a projection map $\mathrm{\xi^{*}\Pi\!:\!\xi^{*}\widetilde{\bf{\,X\,}}\rightarrow \widetilde{\bf{\,Y\,}}^{nor}}$ which is given by the map $\mathrm{\Pi\times }$ $\mathrm{Id_{\widetilde{\bf{\,Y\,}}^{nor}}}$ and whose fibers are purely $\mathrm{k}$-dimensional since they are exactly given by the fibers of the projection associated to the universal space $\mathrm{\widetilde{\bf{\,X\,}}}$. Moreover, note that each smooth $\mathrm{\left(k,k\right)}$-form $\mathrm{\Omega}$ defined on $\mathrm{\widetilde{\bf{\,X\,}}}$ induces a smooth $\mathrm{\left(k,k\right)}$-form on $\mathrm{\xi^{*}\widetilde{\bf{\,X\,}}}$ via its pull back with respect to the projection map $\mathrm{p_{\widetilde{\bf{\,X\,}}}\vert \xi^{*}\widetilde{\bf{\,X\,}}\!:\!\xi^{*}\widetilde{\bf{\,X\,}}\rightarrow \widetilde{\bf{\,X\,}}}$. In the sequel, we will label this form by $\mathrm{\Omega^{nor}}$. Since $\mathrm{\xi^{*}\widetilde{\bf{\,X\,}}}$ is compact as a closed subset of the compact space $\mathrm{\widetilde{\bf{\,X\,}}\times\widetilde{\bf{\,Y\,}}^{nor}}$, it follows that $\mathrm{\Omega^{nor}}$ has compact support. Furthermore, the inverse image $\mathrm{\widetilde{\bf{W}}^{nor}\coloneqq\xi^{-1}\big(\widetilde{\bf{W}}\big)}$ of the compact subset $\mathrm{\bf{W}}$ is likewise compact for the same reason.

We will now prove the first claim of the lemma. For this, note that 
\begin{equation}\label{Equation for This Proof}
\mathrm{\int_{\Pi^{-1}\left(y\right)}\Omega_{n}=\int_{\Pi^{nor,-1}\left(\widehat{y}\right)}\Omega_{n}^{nor} \text{ for all }\widehat{y}\text{ with }\xi\left(\widehat{y}\right)=y.}
\end{equation}

Since $\mathrm{\Omega^{nor}_{n}}$ defines a sequence of smooth $\mathrm{\left(k,k\right)}$-forms with compact support and since the map fulfills all requirements of theorem in \cite{Kin}, {pp.\ }{\bf{\oldstylenums{185}-\oldstylenums{220}}}, we deduce that the fiber integral $$\mathrm{\widehat{y}\mapsto\int_{\Pi^{nor,-1}\left(\widehat{y}\right)}\Omega_{n}^{nor}}$$ defines a continuous function over the base $\mathrm{\widetilde{\bf{\,Y\,}}^{nor}}$. So the sequence given by $$\mathrm{\left(\widehat{y}\mapsto\int_{\Pi^{nor,-1}\left(\widehat{y}\right)}\Omega_{n}^{nor}\right)_{n}}$$ defines a sequence of continuous functions over $\mathrm{\widetilde{\bf{\,Y\,}}^{nor}}$ which will converge for each point by the assumption of the lemma to the zero function. By {\scshape{Dini}}'s convergence theorem of strictly decreasing sequence of continuous functions, it follows that this sequence of functions converges uniformly over the compact subset $\mathrm{\widetilde{\bf{W}}^{nor}}$ to the zero function. Hence, for each $\mathrm{\Delta>0}$ we can find $\mathrm{n_{\Delta}\in \mathbb{N}}$ so that $$\mathrm{\int_{\Pi^{nor,-1}\left(\widehat{y}\right)}\Omega_{n}^{nor}\leq\Delta\text{ for all }\widehat{y}\in \widetilde{\bf{W}}^{nor}\text{ and all }n\geq n_{\Delta}.}$$ By equation \ref{Equation for This Proof} and the fact that $\mathrm{\widetilde{\bf{W}}^{nor}=\xi^{-1}\big(\widetilde{\bf{W}}\big)}$ the first claim is shown.

The second claim is a direct consequence of the fact that $\mathrm{\Omega^{nor}}$ has compact support and by the assumption of the lemma which is given by $$\mathrm{0<\int_{\Pi^{-1}\left(y\right)}\Omega\text{ for all }y\in \widetilde{\bf{W}}.}$$ Hence, we deduce that $$\mathrm{0<\int_{\Pi^{nor,-1}\left(\widehat{y}\right)}\Omega^{nor}\text{ for all }\widehat{y}\text{ so that }\xi\left(\widehat{y}\right)=y\in \widetilde{\bf{W}}}$$ by \ref{Equation for This Proof}. Note that $\mathrm{\widehat{y}\mapsto\int_{\Pi^{nor,-1}\left(\widehat{y}\right)}\Omega^{nor}}$ is a continuous function which does not vanish over $\mathrm{\widetilde{\bf{\,W\,}}^{nor}}$ by the above inequality. Hence, there exists $\mathrm{\delta>0}$ so that $\mathrm{\delta\leq }$ $\mathrm{\int_{\Pi^{nor,-1}\left(\widehat{y}\right)}\Omega^{nor}}$ for all $\mathrm{\widehat{y}}$ in the compact subset $\mathrm{\widetilde{\bf{W}}^{nor}}$. This proves the claim by applying \ref{Equation for This Proof} again.
\end{proof}
 
For the sake of completeness, we will close this section by stating the theorem of {\scshape{J. King}}, which we have used in the proof of the preceding proposition.

\begin{theo} \textnormal{(cf.\ \cite{Kin}, {pp.\ }{\bf{\oldstylenums{185}-\oldstylenums{220}}})}\label{Theorem King}\vspace{0.15 cm}\\
Let $\mathrm{F\!:\!\mathbf{X}\rightarrow \mathbf{Y}}$ be a $\mathrm{k}$-fibering between complex purely dimensional spaces $\mathrm{\mathbf{X},\mathbf{Y}}$ where $\mathrm{m=}$ $\mathrm{dim_{\mathbb{C}}\,\mathbf{X}}$, $\mathrm{n=dim_{\mathbb{C}}\,\mathbf{Y}}$ and assume that $\mathrm{\mathbf{Y}}$ is normal. If $\mathrm{\Omega}$ is a continuous, complex valued $\mathrm{\left(k,k\right)}$-form on $\mathrm{\mathbf{X}}$ with compact support, then the fiber integral $$\mathrm{y\mapsto \int_{F^{-1}\left(y\right)}\Omega\,d\,[F_{y}]\coloneqq \int_{F^{-1}\left(y\right)}\nu_{F}\,\Omega\vert F^{-1}\left(y\right)}$$ \label{Notation Fiber Integration with Respect to a k-Fibering} defines a continuous function on $\mathrm{\mathbf{Y}}$ where $\mathrm{\nu_{F}}$ denotes the order\label{Notation Order Of F At Point}\,\footnote{For a rigorous definition of the order of a $\mathrm{k}$-fibering in the point $\mathrm{x\in \mathbf{X}}$ {cf.\! }\cite{Kin}.} of the $\mathrm{k}$-fibering $\mathrm{F\!:\!\mathbf{X}\rightarrow \mathbf{Y}}$.
\end{theo}

\section{Uniform Convergence in the Tame Case}\label{Section Uniform Convergence Theorems In The Tame Case}
\subsection{Uniform Convergence of the Fiber Probability Measures}

The aim of this section is to prove {\bf{Theorem 3}}. For this recall, that for each $\mathrm{x\in \mu^{-1}\left(\xi\right)}$ there exists a $\mathrm{\mathbb{T}}$-invariant Zariski open subset $\mathrm{\mathbf{X}^{i}\subset \mathbf{X}^{ss}_{\xi}}$ which is contained in the $\mathrm{n}$-stable complement of the zero set of the tame sequence $\mathrm{\left(s^{i}_{n}\right)_{n}}$ given by $\mathrm{\mathbf{X}\left(s_{n}^{i}\right)\subset\mathbf{X}^{ss}_{\xi}}$ ({cf.\! }{\bf{Theorem 1}}). Moreover, let $\mathrm{\varrho^{i}\!:\!\mathbf{X}^{i}\rightarrow \mathbb{R}}$ be the normalized s.p.s.h.\ limit function as defined at the beginning of {\bf{Section \ref{Uniform Localization Proposition}}} and recall that there exists a compact neighborhood $\mathrm{\mathbf{W}^{i}\subset \pi\left(\mathbf{X}^{i}\right)}$ so that $\mathrm{T\left(\epsilon,\mathbf{W}^{i}\right)\subset \mathbf{X}^{i}}$ given by $$\mathrm{T\left(\epsilon,\mathbf{W}^{i}\right)=\left(\varrho^{i}\times \pi\right)^{-1}\left([0,\epsilon]\times \mathbf{W}^{i}\right)\text{ (cf.\ {\bf{Section \ref{Uniform Localization Proposition}}})}}$$ defining a compact $\mathrm{T}$-invariant tube for all $\mathrm{\epsilon\geq 0}$.
 
We begin the proof with the following lemma.

\begin{lemma}\label{Lemma Reduced Function} Let $\mathrm{f\in \mathcal{C}^{0}\!\left(\mathbf{X}\right)}$ be a $\mathrm{T}$-invariant, continuous function. Given $\mathrm{\sigma>0}$, there exists $\mathrm{\epsilon_{\sigma}>0}$ so that $$\mathrm{\left|f\vert\left(\pi^{-1}\left(y\right)\cap T\left(\epsilon_{\sigma},\mathbf{W}^{i}\right)\right)-f_{red}\left(y\right)\right|\leq\sigma}$$ for all $\mathrm{y\in \mathbf{W}^{i}\subset \mathbf{Y}^{i}}$.
\end{lemma}
\begin{proof} Let us assume that $\mathrm{\epsilon_{\sigma}>0}$ with the above property does not exists. Then we can find a sequence $\mathrm{\left(\epsilon_{\sigma,n}\right)_{n}}$ of positive numbers with $\mathrm{\epsilon_{\sigma,n}\rightarrow 0}$, a sequence $\mathrm{\left(y_{n}\right)_{n}}$ in $\mathrm{\mathbf{W}^{i}}$ and a lifted sequence $\mathrm{\left(x_{n}\right)_{n}}$ in $\mathrm{\pi^{-1}\left(\mathbf{W}^{i}\right)}$ (i.e.\ $\mathrm{\pi\left(x_{n}\right)=y_{n}}$) where $$\mathrm{x_{n}\in T\left(\epsilon_{\sigma,n},\mathbf{W}^{i}\right)}$$ and so that $$\mathrm{\left\vert f\left(x_{n}\right)-f_{red}\left(y_{n}\right) \right\vert\geq \epsilon>0}$$ for all $\mathrm{n}$. 

By the compactness of $\mathrm{\mathbf{W}^{i}}$, we can assume that $\mathrm{y_{n}\rightarrow y\in \mathbf{W}^{i}}$. Furthermore, by the compactness of $\mathrm{T\left(\epsilon_{\sigma,n},\mathbf{W}^{i}\right)}$, we have can assume that the lifted sequence $\mathrm{\left(x_{n}\right)_{n}}$, which is contained in the sequence of the compact nested subsets given by $\mathrm{T\left(\epsilon_{\sigma,n},\mathbf{W}^{i}\right)}$, is convergent as well, i.e.\ $\mathrm{x_{n}\rightarrow x}$. Since $$\mathrm{T\left(\epsilon_{\sigma,n},\mathbf{W}^{i}\right)\downarrow \mu^{-1}\left(\xi\right)\cap \pi^{-1}\left(\mathbf{W}^{i}\right)}$$ we have $\mathrm{x_{n}\rightarrow x\in  \mu^{-1}\left(\xi\right)\cap \pi^{-1}\left(\mathbf{W}^{i}\right)}$.

So, all in all, we have found a convergent sequence $\mathrm{\left(x_{n}\right)_{n}}$ in $\mathrm{\pi^{-1}\left(\mathbf{W}^{i}\right)}$ so that $$\mathrm{x_{n}\rightarrow x\in \mu^{-1}\left(\xi\right)\cap \pi^{-1}\left(\mathbf{W}^{i}\right) }$$ and $$\mathrm{\vert f\left(x_{n}\right)-f_{red}\left(\pi\left(x_{n}\right)\right)\vert=\left\vert f\left(x_{n}\right)-f_{red}\left(y_{n}\right) \right\vert\geq \epsilon>0}$$ for all $\mathrm{n}$. From the continuity of $\mathrm{\pi, f}$ and the fact that $\mathrm{f_{red}\left(\pi\left(x\right)\right)=f_{red}\left(y\right)}$ $\mathrm{=f\left(x\right)}$ we deduce a contradiction.
\end{proof}

After this preparation, we can prove the uniform convergence of the measure sequence with respect to the weak topology. For this recall that $\mathrm{\overline{f}}$ for $\mathrm{f\in \mathcal{C}^{0}\!\left(\mathbf{X}\right)}$ is defined to be the averaged function given by $\mathrm{\overline{f}\left(x\right)=\int_{T}f\left(t.x\right)\,d\nu_{T}}$ where denotes the {\scshape{Haar}} measure. Moreover, recall the definition of the sequence $\mathrm{\left({\bm{\nu}_{n}^{i}}\right)_{n}}$ of fiber probability measures attached to the tame sequence $\mathrm{\left(s_{n}^{i}\right)_{n}}$ (cf.\,?{\bf{Definition \ref{Definition Fiber Probability Measure Tame Case}}}).

\begin{prop}\label{Proposition Uniform Convergence} Let $\mathrm{f\in \mathcal{C}^{0}\!\left(\mathbf{X}\right)}$ and $\mathrm{\mathbf{W}^{i}\subset}$ $\mathrm{\mathbf{Y}^{i}=}$ $\mathrm{\pi\left(\mathbf{X}^{i}\right)}$ as before. Then the sequence of functions on $\mathrm{\mathbf{W}^{i}\cap \mathbf{Y}_{0}}$ given by $$\mathrm{y\mapsto \int_{\pi^{-1}\left(y\right)}f\,d\bm{\nu}_{n}^{i}\left(y\right)}$$ converges uniformly on $\mathrm{\mathbf{W}^{i}\cap \mathbf{Y}_{0}}$ to the reduced function $\mathrm{f_{red}\vert \mathbf{W}^{i}\cap \mathbf{Y}_{0}\!:\!\mathbf{W}^{i}\cap \mathbf{Y}_{0} \rightarrow \mathbb{R}.}$
\end{prop}
\begin{proof} First, let us show that it is enough to prove the claim for continuous $\mathrm{T}$-invariant functions: Let us assume that the claim is valid for all continuous functions which are $\mathrm{T}$-invariant and assume that $\mathrm{f\in \mathcal{C}^{0}\!\left(\mathbf{X}\right)}$ is arbitrary. Using the $\mathrm{T}$-invariance of $\mathrm{\bm{\nu}_{n}}$ and $\mathrm{\nu_{T}}$ and the fact $\mathrm{\int_{T}d\nu_{T}=1}$ it follows that: $$\mathrm{\int_{t\in T}\left[\int_{\pi^{-1}\left(y\right)}t^{*}f\,d\bm{\nu}_{n}\left(y\right)\right]\,d\nu_{T}=\int_{\pi^{-1}\left(y\right)}f\,d\bm{\nu}_{n}\left(y\right).}$$ Interchanging the order of integration yields 
\begin{equation}\label{Equation Interchanging Integration}
\mathrm{\int_{\pi^{-1}\left(y\right)}f\,d\bm{\nu}_{n}\left(y\right)=\int_{\pi^{-1}\left(y\right)}\left[\int_{t\in T}t^{*}f\,d\nu_{T}\right]\,d\bm{\nu}_{n}\left(y\right)=\int_{\pi^{-1}\left(y\right)}\overline{f}\,d\bm{\nu}_{n}\left(y\right).}
\end{equation} Since $\mathrm{\overline{f}}$ is $\mathrm{T}$-invariant and continuous and since we have assumed that the claim is true for those functions, we deduce by \ref{Equation Interchanging Integration} that $$\mathrm{\left(\int_{\pi^{-1}\left(y\right)}f\,d\bm{\nu}_{n}\left(y\right)\right)=\left(y\mapsto \int_{\pi^{-1}\left(y\right)}\overline{f}\,d\bm{\nu}_{n}\left(y\right)\right)}$$ converges uniformly on $\mathrm{\mathbf{W}^{i}\cap \mathbf{Y}_{0}}$ to the reduced function $\mathrm{f_{red}=\vert \mathbf{W}^{i}\cap \mathbf{Y}_{0}\!:\!\mathbf{W}^{i}\cap \mathbf{Y}_{0} \rightarrow \mathbb{R}}$. Hence, it remains to verify the claim for continuous, $\mathrm{T}$-invariant functions.

If $\mathrm{\sigma>0}$ and $\mathrm{f\in\mathcal{C}^{0}\!\left(\mathbf{X}\right)}$, then by {\bf{Lemma \ref{Lemma Reduced Function}}}, there exists $\mathrm{\epsilon_{\sigma}>0}$ so that $$\mathrm{\left|f\vert\left(\pi^{-1}\left(y\right)\cap T\left(\epsilon_{\sigma},\mathbf{W}^{i}\right)\right)-f_{red}\left(y\right)\right|\leq\frac{\sigma}{2}}$$ for all $\mathrm{y\in \mathbf{W}^{i}\subset \mathbf{Y}^{i}}$. By {\bf{Theorem 2}} we can find $\mathrm{N_{0}\in \mathbb{N}}$ so that for all $\mathrm{n\geq N_{0}}$ we have $$\mathrm{\varrho^{i}_{n}\left(x\right)\geq \frac{\epsilon_{\sigma}}{2}}$$ for all $\mathrm{x\in  T^{c}\left(\epsilon_{\sigma},\mathbf{W}^{i}\right)}$. Therefore, we deduce that
\begin{equation*}
\begin{split}
\mathrm{\left\vert \int_{\pi^{-1}\left(y\right)}f\,d\bm{\nu}^{i}_{n}\left(y\right)-f_{red}\left(y\right)\right\vert} &\leq \mathrm{\frac{\sigma}{2}\frac{\int_{\pi^{-1}\left(y\right)\cap T\left(\epsilon_{\sigma},\mathbf{W}^{i}\right)}\left|s_{n}^{i}\right|^{2}\,d\,[\pi_{y}]}{\left\Vert s_{n}^{i}\right\Vert ^{2}\left(y\right)}}\\[0,2 cm]
&+ \mathrm{e^{-n\,\frac{\epsilon_{\sigma}}{2}}\,C\left(f\right)\,\frac{\int_{\pi^{-1}\left(y\right)\cap T^{c}\left(\epsilon_{\sigma},\mathbf{W}^{i}\right)}\,d\,[\pi_{y}]}{\left\Vert s_{n}^{i}\right\Vert ^{2}\left(y\right)}}\end{split}
\end{equation*} for all $\mathrm{y\in \mathbf{W}^{i}\cap \mathbf{Y}_{0}}$ where $\mathrm{C\left(f\right)\coloneqq \underset{x\in \mathbf{X}}{max}\left|f-\pi^{*}f_{red}\right|<\infty}$.

By {\bf{Proposition \ref{Proposition Uniform Convergence Potential Functions in the Abelian Case}}} the sequence $\mathrm{\varrho^{i}_{n}}$ converges uniformly on compact subsets of $\mathrm{\pi^{-1}\left(\mathbf{Y}^{i}\right)=\mathbf{X}^{i}}$ to $\mathrm{\varrho^{i}}$ and hence, for each $\mathrm{\delta>0}$, there exists $\mathrm{N^{\prime}_{0}\in \mathbb{N}}$ so that $$\mathrm{\varrho^{i}_{n}\leq \epsilon^{\prime}+\delta}$$ for all $\mathrm{x\in T\left(\epsilon^{\prime},\mathbf{W}^{i}\right)}$. Here we will choose $\mathrm{\delta>0}$ and $\mathrm{\epsilon^{\prime}>0}$ so that  $\mathrm{\epsilon^{\prime}+\delta<\frac{\epsilon_{\sigma}}{2}}$. It then follows for all $\mathrm{y\in \mathbf{W}^{i}\cap \mathbf{Y}_{0}}$
\begin{equation*}
\begin{split}
\mathrm{\left\Vert s_{n}^{i}\right\Vert^{2}\left(y\right)} &\geq \mathrm{\int_{\pi^{-1}\left(y\right)\cap T\left(\epsilon^{\prime},\mathbf{W}^{i}\right)}\left|s_{n}^{i}\right|^{2}\,d\,[\pi_{y}]}\\[0,2 cm]
&\geq \mathrm{e^{-n\,\left(\epsilon^{\prime}+\delta\right)}\,\int_{\pi^{-1}\left(y\right)\cap T\left(\epsilon^{\prime},\mathbf{W}^{i}\right)}\,d\,[\pi_{y}].}\end{split}
\end{equation*}
So we deduce \begin{equation*}
\begin{split}
\mathrm{\left\vert \int_{\pi^{-1}\left(y\right)}f\,d\bm{\nu}^{i}_{n}\left(y\right)-f_{red}\left(y\right)\right\vert} &\leq \mathrm{\frac{\sigma}{2}\frac{\int_{\pi^{-1}\left(y\right)\cap T\left(\epsilon_{\sigma},\mathbf{W}^{i}\right)}\left|s_{n}^{i}\right|^{2}\,d\,[\pi_{y}]}{\left\Vert s_{n}^{i}\right\Vert ^{2}\left(y\right)}}\\[0,2 cm]
&+ \mathrm{e^{-n\,\left(\frac{\epsilon_{\sigma}}{2}-\epsilon^{\prime}-\delta\right)}\,C\left(f\right)\,\frac{\int_{\pi^{-1}\left(y\right)\cap T^{c}\left(\epsilon_{\sigma},\mathbf{W}^{i}\right)}\,d\,[\pi_{y}]}{\int_{\pi^{-1}\left(y\right)\cap T\left(\epsilon^{\prime},\mathbf{W}^{i}\right)}\,d\,[\pi_{y}]}}\end{split} 
\end{equation*} for all $\mathrm{y\in \mathbf{W}^{i}\cap \mathbf{Y}_{0}}$.
Using {\bf{Corollary \ref{Boundedness of Fiber Integral}}}, we know that the dominator of the last term is bounded for all $\mathrm{y\in \mathbf{Y}_{0}}$ and by the second claim of {\bf{Proposition \ref{Continuity of Fiber Integral}}}, we know that the denominator is bounded away form zero as $\mathrm{y}$ varies in $\mathrm{\mathbf{W}^{i}\cap \mathbf{Y}_{0}}$. Therefore, we find a constant $\mathrm{\Gamma<\infty}$ so that 
\begin{equation*}
\begin{split}
\mathrm{\left\vert \int_{\pi^{-1}\left(y\right)}f\,d\bm{\nu}^{i}_{n}\left(y\right)-f_{red}\left(y\right)\right\vert} &\leq \mathrm{\frac{\sigma}{2}\frac{\int_{\pi^{-1}\left(y\right)\cap T\left(\epsilon_{\sigma},\mathbf{W}^{i}\right)}\left|s_{n}^{i}\right|^{2}\,d\,[\pi_{y}]}{\left\Vert s_{n}^{i}\right\Vert ^{2}\left(y\right)}}\\[0,2 cm]
&+ \mathrm{e^{-n\,\left(\frac{\epsilon_{\sigma}}{2}-\epsilon^{\prime}-\delta\right)}\,C\left(f\right)\,\Gamma.}\end{split} 
\end{equation*} Since $$\mathrm{\frac{\int_{\pi^{-1}\left(y\right)\cap T\left(\epsilon_{\sigma},\mathbf{W}^{i}\right)}\left|s_{n}^{i}\right|^{2}\,d\,[\pi_{y}]}{\left\Vert s_{n}^{i}\right\Vert ^{2}\left(y\right)}\leq 1}$$ for all $\mathrm{n\in \mathbb{N}}$ and all $\mathrm{y\in \mathbf{W}^{i}\cap \mathbf{Y}_{0}}$, we find $$\mathrm{\left\vert \int_{\pi^{-1}\left(y\right)}f\,d\bm{\nu}^{i}_{n}\left(y\right)-f_{red}\left(y\right)\right\vert \leq \frac{\sigma}{2}+e^{-n\,\left(\frac{\epsilon_{\sigma}}{2}-\epsilon^{\prime}-\delta\right)}\,C\left(f\right)\,\Gamma.}$$ Since $\mathrm{\frac{\epsilon_{\sigma}}{2}-\epsilon^{\prime}-\delta>0}$ there exists $\mathrm{N_{0}^{\prime\prime}\geq N^{\prime}_{0}}$ so that $$\mathrm{\left\vert \int_{\pi^{-1}\left(y\right)}f\,d\bm{\nu}^{i}_{n}\left(y\right)-f_{red}\left(y\right)\right\vert \leq \sigma}$$ for all $\mathrm{n\geq N^{\prime\prime}_{0}}$ as claimed.
\end{proof}

We sum up the results in

\begin{theo_3*}\label{Theorem Uniform Convergence of the Tame Measure Sequence} \textnormal{[\scshape{Uniform Convergence of the Tame Measure Sequence}]}\vspace{0.1 cm}\\
For for every tame collection $\mathrm{\left\{s_{n}^{i}\right\}_{i}}$ there exists a finite cover $\mathrm{\mathfrak{U}}$ of $\mathrm{\mathbf{Y}}$ with $\mathrm{\mathbf{U}_{i}\subset}$ $\mathrm{\pi\left(\mathbf{X}^{i}\right)}$ so that the collection of fiber probability measures $\mathrm{\left\{\bm{\nu}^{\mathfrak{U}}_{n}\right\}}$ associated to $\mathrm{\left\{s^{i}_{n}\right\}_{i}}$ converges uniformly on $\mathrm{\mathbf{Y}_{0}}$ to the fiber Dirac measure of $\mathrm{\mu^{-1}\left(\xi\right)\cap \mathbf{X}_{0}}$, i.e.\ for every $\mathrm{i\in I}$ and every $\mathrm{f\in \mathcal{C}^{0}\left(\mathbf{X}\right)}$, we have $$\begin{xy}
 \xymatrix{\mathrm{\left(y\mapsto \int_{\pi^{-1}\left(y\right)}f\, d\bm{\nu}_{n}^{i}\left(y\right)\right)}\ar[rr]^{\mathrm{\,\,\,\,\,}}&&\mathrm{f_{red}\text{ uniformly on }\mathbf{U}_{i}\cap \mathbf{Y}_{0}.\label{Notation Reduced Function}}
 }
\end{xy}$$
\end{theo_3*}
\begin{proof} Let $\mathrm{\left\{s^{i}_{n}\right\}_{i}}$ be a tame collection and $\mathrm{\mathfrak{U}}$ be a finite open cover of $\mathrm{\mathbf{Y}}$ so that $$\mathrm{\mathbf{U}^{i} \subset \mathbf{W}^{i}\subset \mathbf{Y}^{i}=\pi\left(\mathbf{X}^{i}\right)}$$ where $\mathrm{\mathbf{W}^{i}}$ is a compact neighborhood as defined at the beginning of {\bf{Section \ref{Uniform Localization Proposition}}}.

By {\bf{Proposition \ref{Proposition Uniform Convergence}}} there exists $\mathrm{N^{i}_{0}\in \mathbb{N}}$ so that $$\mathrm{\left|\int_{\pi^{-1}\left(y\right)}f\, d\bm{\nu}_{n}^{i}\left(y\right)-f_{red}\left(y\right)\right|\leq\epsilon}$$ for all $\mathrm{y\in \mathbf{Y}_{0}\cap \mathbf{W}^{i}}$ and all $\mathrm{n\geq N^{i}_{0}}$ which proves the claim.
\end{proof}

\subsection{Uniform Convergence of the Fiber Distribution Densities}

In this section we will give a proof of {\bf{Theorem 4}}. For this, let $\mathrm{\left(D_{n}^{i}\left(\cdot,t\right)\right)_{n}}$ be the sequence of cumulative fiber distribution functions on $\mathrm{\mathbf{W}^{i}\cap \mathbf{Y}_{0}}$ as defined in {\bf{Definition \ref{Definition Sequence Of Collections Of Fiber Probability Densities}}} associated to a tame sequence  $\mathrm{\left(s_{n}^{i}\right)_{n}}$ and let $\mathrm{\mathbf{W}^{i}\subset \pi\left(\mathbf{X}^{i}\right)}$ be a compact neighborhood so that the $\mathrm{T}$-invariant tube $\mathrm{T\left(\epsilon,\mathbf{W}^{i}\right)\subset \mathbf{X}^{i}}$ as defined in {\bf{Section \ref{Uniform Localization Proposition}}} is compact.
 
\begin{prop}\label{Proposition Uniform Convergence of The Distribution Sequence} The sequence $\mathrm{\left(D_{n}^{i}\left(\cdot,t\right)\right)_{n}}$ of distribution functions on $\mathrm{\mathbf{W}^{i}\cap \mathbf{Y}_{0}}$ converges uniformly on $\mathrm{\mathbf{W}^{i}\cap \mathbf{Y}_{0}}$ to the zero function for all $\mathrm{t\geq 0}$.
\end{prop}
\begin{proof} First of all fix $\mathrm{\sigma>0}$. Then by the first part of {\bf{Proposition \ref {Continuity of Fiber Integral}}} there exists $\mathrm{\epsilon_{\sigma}>0}$ so that $$\mathrm{vol\left(\pi^{-1}\left(y\right)\cap T\left(\epsilon_{\sigma},\mathbf{W}^{i}\right)\right)\leq \sigma}$$ for all $\mathrm{y\in \mathbf{W}^{i}\cap \mathbf{Y}_{0}}$. Hence, we are finished as soon as we have proved that there exists $\mathrm{N_{0}\in \mathbb{N}}$ so that $$\mathrm{\left\{x\in \pi^{-1}\left(\mathbf{W}^{i}\cap \mathbf{Y}_{0}\right)\!:\!\,\frac{\vert s_{n}\vert^{2}}{\Vert s_{n}\Vert^{2}}\left(x\right)\geq t\right\}\subset  T\left(\epsilon_{\sigma},\mathbf{W}^{i}\right)}$$ for all $\mathrm{n\geq N_{0}}$.

By {\bf{Theorem 2}} we find $\mathrm{N_{0}\in \mathbb{N}}$, so that $\mathrm{\varrho_{n}^{i}\left(x\right)\geq \frac{\epsilon_{\sigma}}{2}}$ for all $\mathrm{x\in T^{c}(\epsilon_{\sigma},\mathbf{W}^{i})}$ and all $\mathrm{n\geq N_{0}}$. Therefore, we deduce \begin{equation}\label{Equation Proof}
\mathrm{\frac{\left|s_{n}\right|^{2}}{\left\Vert s_{n}\right\Vert ^{2}}\leq\frac{e^{-n\,\frac{\epsilon_{\sigma}}{2}}}{\left\Vert s_{n}\right\Vert ^{2}}}
\end{equation}
for all $\mathrm{n\geq N_{0}}$ and all all $\mathrm{x\in T^{c}\left(\epsilon_{\sigma},\mathbf{W}^{i}\right)\cap \pi^{-1}\left(\mathbf{Y}_{0}\right)}$.
Moreover, since $\mathrm{\varrho^{i}_{n}}$ converges uniformly to $\mathrm{\varrho^{i}}$ on the compact subsets like $\mathrm{T\left(\frac{\epsilon_{\sigma}}{5},\mathbf{W}^{i}\right)}$ ({cf.\! }{\bf{Proposition \ref{Proposition Uniform Convergence Potential Functions in the Abelian Case}}}), there exists $\mathrm{N^{\prime}_{0}\geq N_{0}}$ so that $$\mathrm{\varrho^{i}_{n}\left(z\right)\leq \varrho^{i}\left(z\right)+\frac{\epsilon_{\sigma}}{5}}$$ for all $\mathrm{z\in T\left(\frac{\epsilon_{\sigma}}{5},\mathbf{W}^{i}\right)}$ and all $\mathrm{n\geq N_{0}^{\prime}}$, i.e.\ $$\mathrm{\varrho^{i}_{n}\left(z\right)\leq \frac{2}{5} \,\epsilon_{\sigma}\text{ on }T\left(\frac{\epsilon_{\sigma}}{5},\mathbf{W}^{i}\right).}$$ So we deduce $$\mathrm{\Vert s_{n}\Vert^{2}\left(y\right)\geq e^{-n\,\frac{2}{5}\,\epsilon_{\sigma}} \int_{\pi^{-1}\left(y\right)\cap T\left(\frac{\epsilon_{\sigma}}{5},\mathbf{W}^{i}\right)}\,d\,[\pi_{y}]}$$ for all $\mathrm{y\in \mathbf{W}^{i}\cap \mathbf{Y}_{0}}$. If we substitute this into equation \ref{Equation Proof}, it follows that $$\mathrm{\mathrm{\frac{\left|s_{n}\right|^{2}}{\left\Vert s_{n}\right\Vert ^{2}}\left(x\right)\leq\frac{e^{-n\,\frac{1}{10}\,\epsilon_{\sigma}}}{\int_{\pi^{-1}\left(\pi\left(x\right)\right)\cap T\left(\frac{\epsilon_{\sigma}}{5},\mathbf{W}^{i}\right)}\,d\,[\pi_{y}]}}}$$ for all $\mathrm{x\in T^{c}\left(\epsilon_{\sigma},\mathbf{W}^{i}\right)\cap \pi^{-1}\left(\mathbf{Y}_{0}\right)}$ and all $\mathrm{n\geq N_{0}^{\prime}}$. By the second part of {\bf{Proposition \ref{Continuity of Fiber Integral}}}, we know that the denominator is bounded away form zero as $\mathrm{x}$ varies in $\mathrm{\pi^{-1}\left(\mathbf{W}^{i}\cap \mathbf{Y}_{0}\right)}$ and therefore we find $\mathrm{N_{0}\in \mathbb{N}}$, $\mathrm{N_{0}\geq N^{\prime}_{0}}$ so that $$\mathrm{\mathrm{\frac{\left|s_{n}\right|^{2}}{\left\Vert s_{n}\right\Vert ^{2}}\left(x\right)\leq t}}$$ for all $\mathrm{x\in T^{c}\left(\epsilon_{\sigma},\mathbf{W}^{i}\right)\cap \pi^{-1}\left(\mathbf{Y}_{0}\right)}$ and all $\mathrm{n\geq N_{0}}$. So it follows that  $$\mathrm{\left\{x\in \pi^{-1}\left(\mathbf{W}^{i}\cap \mathbf{Y}_{0}\right)\!:\!\,\frac{\vert s_{n}\vert^{2}}{\Vert s_{n}\Vert^{2}}\left(x\right)\geq t\right\}\subset  T\left(\epsilon_{\sigma},\mathbf{W}^{i}\right)}$$ for all $\mathrm{n\geq N_{0}}$  and therefore $$\mathrm{D_{n}\left(y,t\right)\leq vol\left\{\pi^{-1}\left(y\right)\cap T\left(\epsilon_{\sigma},\mathbf{W}^{i}\right)\right\} \leq \sigma .}$$ for all $\mathrm{y\in \mathbf{W}^{i}\cap \mathbf{Y}_{0}}$ and all $\mathrm{n\geq N_{0}}$.
\end{proof}

To sum up, we have proved

\begin{theo_4*}\label{Theorem Uniform Convergence of the Tame Distribution Sequence} \textnormal{[\scshape{Uniform Convergence of the Tame Distribution Sequence}]}\vspace{0.1 cm}\\
For every $\mathrm{t\in \mathbb{R}}$ and every tame collection $\mathrm{\left\{s_{n}^{i}\right\}_{i}}$ there exists a finite cover $\mathrm{\mathfrak{U}}$ of $\mathrm{\mathbf{Y}}$ with $\mathrm{\mathbf{U}_{i}\subset}$ $\mathrm{\pi\left(\mathbf{X}^{i}\right)}$ so that the collection of cumulative fiber probability densities $\mathrm{\left\{D^{\mathfrak{U}}_{n}\left(\cdot,t\right)\right\}}$ associated to $\mathrm{\left\{s^{i}_{n}\right\}_{i}}$ converges uniformly on $\mathrm{\mathbf{Y}_{0}}$ to the zero function on $\mathrm{\mathbf{Y}_{0}}$, i.e.\ for every $\mathrm{i\in I}$ we have $$\begin{xy}
 \xymatrix{\mathrm{\left(y\mapsto D_{n}^{i}\left(y,t\right)\right)}\ar[rr]^{\mathrm{\,\,\,\,\,}}&&\mathrm{0\text{ uniformly on }\mathbf{U}_{i}\cap \mathbf{Y}_{0}\subset \pi\left(\mathbf{X}^{i}\right).}
 }
\end{xy}$$
\end{theo_4*}
\begin{proof} As in the proof of {\bf{Theorem 3}}, using the compactness of $\mathrm{\mathbf{Y}}$, we can find by means of {\bf{Theorem 1}} a tame collection $\mathrm{\left\{s^{i}_{n}\right\}_{i}}$ and an finite open cover $\mathrm{\mathfrak{U}}$ of $\mathrm{\mathbf{Y}}$ so that $$\mathrm{\mathbf{U}^{i} \subset \mathbf{W}^{i}\subset \mathbf{Y}^{i}=\pi\left(\mathbf{X}^{i}\right)}$$ where $\mathrm{\mathbf{W}^{i}}$ is a compact neighborhood as defined at the beginning of {\bf{Section \ref{Uniform Localization Proposition}}}. The claim then follows by applying {\bf{Proposition \ref{Proposition Uniform Convergence}}} to the sequence $\mathrm{\left(D^{i}_{n}\left(\cdot,t\right)\right)_{n}}$.
\end{proof}

\section{Uniform Convergence in the Non-Tame Case}\label{Section Uniform Convergence Theorems In The Non-Tame Case}

Throughout this section we make the following assumption.
\begin{agree}\label{General Agreement}
There exists $\mathrm{N_{0}\in \mathbb{N}}$ so that $\mathrm{\mathbf{R}_{N_{0}}\subset \mathbf{Y}}$ is non-empty.
\end{agree}

Recall that $\mathrm{\mathbf{R}_{N_{0}}\cap \mathbf{Y}_{0}}$ is open. Furthermore, if $\mathrm{y\in\mathbf{R}_{N_{0}}\cap \mathbf{Y}_{0}}$, there exists an open neighborhood $\mathrm{{\mathbf{U}}_{y}\subset\mathbf{R}_{N_{0}}\cap \mathbf{Y}_{0}}$ and a sequence of local, holomorphic $\mathrm{\xi_{n}}$-eigensections $\mathrm{\widehat{s}_{f_{y},n}}$ defined on the open $\mathrm{\pi}$-saturated subset $\mathrm{\pi^{-1}\left(\mathbf{U}_{y}\right)}$ so that $\mathrm{\widehat{s}_{f_{y},n}\vert \pi^{-1}\left(y^{\prime}\right)\not \equiv0}$ for all $\mathrm{n\geq N_{0}}$ and all $\mathrm{y^{\prime}\in \mathbf{U}_{y}}$. Here $\mathrm{\widehat{s}_{f_{y,n}}}$ is given by (cf.\ {\bf{Section \ref{Definition of the Fiber Probability Measure Sequence (Non-Tame Case)}}}) $$\mathrm{\widehat{s}_{f_{y,n}}=s_{n}\cdot\pi^{*}f^{-1}_{y,n}\text{ for all }y^{\prime}\in \mathbf{U}_{y}\subset \mathbf{Y}}$$ where $\mathrm{\left(f_{y,n}\right)_{n}}$ is a sequence of holomorphic functions $\mathrm{f_{y,n}\in \mathcal{O}\left(\mathbf{U}_{y}\right)}$.

Using {\bf{Theorem 2}}, there exists a tame sequence $\mathrm{\left(s^{i}_{n}\right)_{n}}$ so that $\mathrm{y\in \mathbf{Y}^{i}=\pi\left(\mathbf{X}^{i}\right)}$. Note that after having shrunken $\mathrm{\mathbf{U}_{y}}$, we can always assume that $\mathrm{\mathbf{U}_{y}\subset \mathbf{Y}^{i}}$. We now define a sequence $\mathrm{\big(\triangle_{f_{y,n}}^{i}\big)_{n}}$ of holomorphic $\mathrm{\xi_{n}-\xi^{i}_{n}}$-eigenfunctions on $\mathrm{\pi^{-1}\left(\mathbf{U}_{y}\right)}$ by $$\mathrm{s^{i}_{n}\cdot\triangle_{f_{y,n}}^{i}=\widehat{s}_{f_{y,n}}\label{Notation Difference Function}.}$$ Note that $\mathrm{\triangle^{i}_{f_{y,n}}\vert \pi^{-1}\left(y^{\prime}\right)\not \equiv0}$ for all $\mathrm{n\geq N_{0}}$ and all $\mathrm{y^{\prime}\in \mathbf{U}_{y}}$. 

Moreover, recall that the sequence $$\mathrm{\phi_{n}\!:\!x\mapsto \phi_{n}\left(x\right)\coloneqq \Vert \widehat{s}_{f_{y,n}}\Vert^{-2}\left(x\right)\vert  \widehat{s}_{f_{y,n}}\vert^{2}\left(x\right)}$$ is independent of the choice of $\mathrm{\left(f_{y,n}\right)_{n}}$, $\mathrm{f_{y,n}\in \mathcal{O}\left(\mathbf{U}_{y}\right)}$, over $\mathrm{\mathbf{U}_{y}\subset \mathbf{R}_{N_{0}}\cap \mathbf{Y}_{0}}$ ({cf.\! }{\bf{Corollary \ref{Corollary Independence Of Extension}}}). Therefore, we will not longer specify $\mathrm{\left(f_{y,n}\right)_{n}}$ and just write $\mathrm{\triangle^{i}_{f_{y,n}}=\triangle^{i}_{y,n}}$ and $\mathrm{s_{n}=\widehat{s}_{f_{y,n}}}$.

\subsection{Analysis of $\mathrm{\triangle^{i}_{n}\vert \pi^{-1}\left(y\right)}$}

The decomposition $\mathrm{s_{n}=s^{i}_{n}\cdot \triangle^{i}_{y,n}}$ introduced above will play a crucial role in the proof of $\mathrm{\mathbf{Theorem\, 5.a,\, b}}$ and $\mathrm{\mathbf{Theorem\, 6.a,\,b}}$. Since the sequence of local functions given by $\mathrm{\triangle^{i}_{y,n}}$ can be seen as measure of the difference between the initial sequence of eigensections $\mathrm{s_{n}}$ and the tame sequence $\mathrm{\left(s^{i}_{n}\right)_{n}}$, it is desirable to control the growth of $\mathrm{\triangle^{i}_{y,n}}$. However, in general the restrictions of the sequence of functions given by $\mathrm{\triangle^{i}_{y,n}}$ to the fibers of the quotient map $\mathrm{\pi}$ turns out to be unbounded. The aim of this section is to prove (cf. {\bf{Proposition \ref{Theorem Estimates Tales of Triangle}}}) that there always exists $\mathrm{m_{0}\in \mathbb{N}}$ so that $\mathrm{s^{i}_{m_{0}}\cdot\triangle^{i}_{y,n}}$ is uniformly bounded over the quotient and takes on its maximum on the fibers in a given neighborhood of $\mathrm{\mu^{-1}\left(\xi\right)}$ for all $\mathrm{n}$ big enough.

\begin{prop} \label{Theorem Estimates Tales of Triangle} Let $\mathrm{\mathbf{W}\subset\mathbf{Y}}$ be a compact neighborhood (we can assume that $\mathrm{\mathbf{W}\subset \mathbf{Y}^{i}}$), then there exists $\mathrm{m_{0}\in \mathbb{N}}$ so that the restriction of $\mathrm{\vert s^{i}_{m_{0}}\cdot\triangle_{y,n}^{i}\vert^{2}}$ on $\mathrm{\pi^{-1}\left(y\right)}$ takes on its maximum in $\mathrm{\pi^{-1}\left(y\right)\cap}$ $\mathrm{ T\left(\epsilon,\mathbf{W}\right)}$ for all $\mathrm{n}$ big enough and all $\mathrm{y\in \mathbf{W}\cap \mathbf{R}_{N_{0}}}$.
\end{prop}

Before we prove the above claim, we first have to show the following lemma.

\begin{lemma} \label{Lemma Fibers at Stage n Are Completely Contained In} If $\mathrm{\left(\eta_{m}\right)_{m}}$ is an arbitrary sequence in $\mathrm{\mathfrak{t}^{*}}$ so that $\mathrm{\frac{1}{m}\eta_{m}\rightarrow \xi}$, then there exists $\mathrm{N_{0}\in \mathbb{N}}$ so that $$\mathrm{\mathbf{X}^{ss}_{m^{-1}\eta_{m}}\subset \mathbf{X}^{ss}_{\xi},}$$ for all $\mathrm{m\geq N_{0}}$.

Furthermore, for all $\mathrm{m\geq N_{0}}$, each fiber of $\mathrm{\pi_{m}}$ is entirely contained in a fiber of $\mathrm{\pi}$.
\end{lemma}
\begin{proof} Since $\mathrm{m^{-1}\eta_{m}\rightarrow \xi}$ it follows that the sequence of compact sets given by $\mathrm{\mu^{-1}(m^{-1}}$ $\mathrm{\eta_{m})}$ converges to $\mathrm{\mu^{-1}\left(\xi\right)}$. From the fact $\mathrm{\mathbf{X}^{ss}_{\xi}}$ is an open neighborhood of $\mathrm{\mu^{-1}\left(\xi\right)}$ we deduce the existence of $\mathrm{N_{0}\in \mathbb{N}}$ so that $\mathrm{\mu^{-1}\left(m^{-1}\eta_{m}\right)\subset \mathbf{X}^{ss}_{\xi}}$ for all $\mathrm{m\geq N_{0}}$. As a consequence, we conclude $$\mathrm{\mathbf{X}^{ss}_{m^{-1}\eta_{m}}\subset \mathbf{X}^{ss}_{\xi}}$$ for all $\mathrm{m\geq N_{0}}$: Let $\mathrm{x\in \mathbf{X}^{ss}_{m^{-1}\eta_{m}}}$ (in the sequel fix $\mathrm{m\geq N_{0}}$). Then by the definition of set of semistable points, we can find $\mathrm{y}$ so that $$\mathrm{y\in cl\left(\mathbb{T}.x\right)\cap \mu^{-1}\left(m^{-1}\eta_{m}\right).}$$ Since $\mathrm{y\in \mu^{-1}\left(m^{-1}\eta_{m}\right)\subset {\mathbf{X}}^{ss}_{\xi}}$, it follows that $$\mathrm{cl\left(\mathbb{T}.y\right)\cap \mu^{-1}\left(\xi\right)\neq \varnothing.}$$ However, since $\mathrm{cl\left(\mathbb{T}.y\right)\subset cl\left(\mathbb{T}.x\right)}$, it follows that $\mathrm{cl\left(\mathbb{T}.x\right)\cap \mu^{-1}\left(\xi\right)\neq \varnothing}$ which proves that $\mathrm{x\in \mathbf{X}^{ss}_{\xi}}$ and hence $\mathrm{\mathbf{X}^{ss}_{m^{-1}\eta_{m}}\subset \mathbf{X}^{ss}_{\xi}}$ as claimed.

Using the fact that $\mathrm{\mathbf{X}^{ss}_{m^{-1}\eta_{m}}\subset \mathbf{X}^{ss}_{\xi}}$ and the universality of the Hilbert quotient map $\mathrm{\pi_{m}}$, there exists an algebraic map $\mathrm{\varphi_{m}\!:\!\mathbf{Y}_{m}\coloneqq \mathbf{X}^{ss}_{m^{-1}\eta_{m}}/\!\!/\mathbb{T}\rightarrow \mathbf{Y}}$ for each $\mathrm{m\geq N_{0}}$ so that the following diagram commutes:
$$\begin{xy}
\xymatrix{
\mathrm{\mathbf{X}^{ss}_{m^{-1}\eta_{m}}}  \ar@{^{(}->}[d]  \ar[r]^{\mathrm{\pi_{m}}} &\mathrm{\mathbf{Y}_{m}}\ar[dd]^{\mathrm{\varphi_{m}}}\\
\mathrm{\mathbf{X}^{ss}_{\xi}}\ar[rd]^{\mathrm{\pi}}& & \\
&\mathrm{\mathbf{Y}}\\
}
\end{xy}
$$ This proves the second claim.
\end{proof}

Before we proceed with the proof of {\bf{Proposition \ref{Theorem Estimates Tales of Triangle}}}, we note the following remark.

\begin{remark}\label{Remark Minimum Of Potential Functions} Let $\mathrm{\mathbf{U}\subset \mathbf{Y}_{m}}$ be an open subset, $\mathrm{\mathbf{V}\coloneqq \pi_{m}^{-1}\left(\mathbf{U}\right)}$ and let $$\mathrm{\sigma\in H^{0}\left(\mathbf{V},\mathbf{L}^{m}\vert \mathbf{V}\right)}$$ be a local $\mathrm{\eta_{m}}$-eigensection over $\mathrm{\mathbf{V}}$. In this situation, it follows for $\mathrm{y\in{\bf{U}}}$ that the restriction of $\mathrm{\vert \sigma\vert^{2}}$ to $\mathrm{\pi^{-1}_{m}\left(y\right)}$ takes on its maximum on $\mathrm{\mu^{-1}\left(m^{-1}\eta_{m}\right)\cap \pi^{-1}_{m}\left(y\right)}$. To see this, recall that by the construction of the algebraic Hilbert quotient, we can always find a global $\mathrm{N\cdot\eta_{m}}$-eigensection $\mathrm{\sigma^{\prime}\in H^{0}\left(\mathbf{X},\mathbf{L}^{N\cdot m}\right)}$ for $\mathrm{N}$ big enough so that $\mathrm{\sigma^{\prime}\left(x\right)\neq 0}$ for all $\mathrm{x\in \pi_{m}^{-1}\left(y\right)}$. Therefore, $\mathrm{\sigma^{\prime,-1} \cdot\sigma^{N}}$ defines a holomorphic $\mathrm{\mathbb{T}}$-invariant function on $\mathrm{\pi^{-1}_{m}\left(y\right)}$. 

The next step is to consider two possibilities: If $\mathrm{\sigma^{\prime,-1}\cdot \sigma^{N}\equiv 0}$ then it follows that $\mathrm{\sigma\equiv 0}$ and hence the claim is true. Otherwise, it follows that $\mathrm{\sigma\left(x\right)\neq 0}$ for all $\mathrm{x\in \pi_{m}^{-1}\left(y\right)}$. In this case, the claim follows by the theory of the Hilbert quotient because $\mathrm{-\mathbf{log}\,\vert \sigma\vert^{2}}$ defines a smooth plurisubharmonic potential on $\mathrm{\pi^{-1}\left(y\right)}$ of the shifted moment map data and in this case the claim is known.
\end{remark}

\begin{proof} (of {\bf{Proposition \ref{Theorem Estimates Tales of Triangle}}}) Let $\mathrm{y\in \mathbf{W}\cap \mathbf{R}_{N_{0}}}$. First of all, note that $\mathrm{s^{i}_{m}\cdot \triangle^{i}_{y,n}}$ defines a local holomorphic $\mathrm{\eta_{m,n}\coloneqq \xi^{i}_{m}+\left(\xi_{n}-\xi^{i}_{n}\right)}$ eigensection over $\mathrm{\pi^{-1}\left(\mathbf{U}_{y}\right)\supset\pi^{-1}\left(y\right)}$ for all $\mathrm{n}$. As $\mathrm{\vert \xi_{n}-\xi^{i}_{n}\vert\in \mathcal{O}\left(1\right)}$, it follows that the set $\mathrm{\left\{\xi_{n}-\xi_{n}^{i}\right\}_{n}\subset \mathfrak{t}^{*}_{\mathbb{Z}}}$ is finite. Hence, it is enough to prove the claim under the assumption that $\mathrm{\xi_{n}-\xi^{i}_{n}}$ is a constant weight $\mathrm{\xi_{0}\in \mathfrak{t}^{*}}$. In the sequel, set $\mathrm{\eta_{m}\coloneqq \xi^{i}_{m}+\xi_{0}}$. Since $\mathrm{m^{-1}\eta_{m}\rightarrow \xi}$ we can apply {\bf{Lemma \ref{Lemma Fibers at Stage n Are Completely Contained In}}}, in order to find $\mathrm{m_{0}\in \mathbb{N}}$ so that $\mathrm{\mathbf{X}^{ss}_{m_{0}^{-1}\eta_{m_{0}}}\subset \mathbf{X}^{ss}_{\xi}}$. Set $$\mathrm{\triangle^{i}_{y,m_{0},n}\coloneqq s^{i}_{m_{0}}\cdot \triangle^{i}_{y,n}}$$ and note that $\mathrm{\triangle^{i}_{y,m_{0},n}}$ induces a local holomorphic $\mathrm{\xi^{i}_{m}+\xi_{0}}$-eigensection over the open, $\mathrm{\pi_{m_{0}}}$-saturated subset $$\mathrm{\mathbf{V}_{m_{0},y}\coloneqq \left(\varphi_{m_{0}}\circ\pi_{m_{0}}\right)^{-1}\left(\mathbf{U}_{y}\right)\subset \mathbf{X}^{ss}_{m_{0}^{-1}\eta_{m_{0}}}}$$ for all $\mathrm{n}$ big enough. In fact: By the above assumption, $\mathrm{\triangle^{i}_{y,n}}$ is of fixed weight $\mathrm{\xi_{0}}$ for all $\mathrm{n\in \mathbb{N}}$.

Note that by {\bf{Remark \ref{Remark Minimum Of Potential Functions}}}, it follows that the strictly plurisubharmonic function given by $\mathrm{-\mathbf{log}\,\vert \triangle^{i}_{y,m_{0},n}\vert^{2}}$ takes on its uniquely defined minimum on $\mathrm{\pi^{-1}_{m_{0}}\left(\widetilde{y}\right)\cap \mu^{-1}\left(m_{0}^{-1}\eta_{m_{0}}\right)}$ for all $\mathrm{n}$ big enough and all $\mathrm{\widetilde{y}\in \varphi_{m_{0}}^{-1}\left(\mathbf{U}_{y}\right)}$. Equivalently, the restriction of $\mathrm{\vert \triangle^{i}_{y,m_{0},n}\vert^{2}}$ to on $\mathrm{\pi^{-1}_{m_{0}}\left(\widetilde{y}\right)}$ takes on its uniquely defined maximum on $\mathrm{\pi^{-1}_{m_{0}}\left(\widetilde{y}\right)\cap \mu^{-1}\left(m_{0}^{-1}\eta_{m_{0}}\right)}$ for all $\mathrm{n}$ big enough and all $\mathrm{\widetilde{y}\in \varphi_{m_{0}}^{-1}\left(\mathbf{U}_{y}\right)}$.  

By the commutative diagram of {\bf{Lemma \ref{Lemma Fibers at Stage n Are Completely Contained In}}}, it then follows that the restriction of $\mathrm{\vert \triangle^{i}_{y,m_{0},n}\vert^{2}}$ to $\mathrm{\pi^{-1}\left(y^{\prime}\right)\cap \mathbf{X}^{ss}_{m_{0}^{-1}\eta_{m_{0}}}}$ takes on its maximum in $\mathrm{\mu^{-1}\left(m_{0}^{-1}\eta_{m_{0}}\right)\cap \pi^{-1}\left(y^{\prime}\right)}$ $\mathrm{\cap \mathbf{X}^{ss}_{m_{0}^{-1}\eta_{m_{0}}}}$ for all $\mathrm{y^{\prime}\in \mathbf{U}_{y}}$ and all $\mathrm{n}$ big enough. Since $\mathrm{\mathbf{X}^{ss}_{m_{0}^{-1}\eta_{m_{0}}}}$ is Zariski dense in $\mathrm{\mathbf{X}^{ss}_{\xi}}$ it follows by continuity that the restriction of $\mathrm{\vert \triangle^{i}_{y,m_{0},n}\vert^{2}}$ to $\mathrm{\pi^{-1}\left(y^{\prime}\right)}$ takes on its maximum on $$\mathrm{\mu^{-1}\left(m_{0}^{-1}\eta_{m_{0}}\right)\cap \pi^{-1}\left(y^{\prime}\right)}$$ for all $\mathrm{y^{\prime}\in \mathbf{U}_{y}}$ and all $\mathrm{n}$ big enough - in particular this holds for $\mathrm{y\in \mathbf{U}_{y}}$ itself. The claim then follows by the fact that we can always assume that $$\mathrm{\mu^{-1}\left(m_{0}^{-1}\eta_{m_{0}}\right)\cap \pi^{-1}\left(\mathbf{W}^{i}\right)\subset T\left(\epsilon,\mathbf{W}^{i}\right)}$$ for $\mathrm{m_{0}}$ big enough.
\end{proof}

We close this section with the proof of

\begin{lemma}\label{Lemma The Maximum Of Triangle Is Fiberwisely Not Zero} Let $\mathrm{\mathbf{W}\subset\mathbf{Y}}$ be a compact neighborhood of $\mathrm{y_{0}\in \mathbf{Y}}$ (we can assume that $\mathrm{\mathbf{W}\subset \mathbf{Y}^{i}}$), and $\mathrm{\epsilon>0}$, then there exists $\mathrm{N_{0}\in \mathbb{N}}$ so that $$\mathrm{\underset{x\in\pi^{-1}\left(y\right)\cap T\left(\epsilon,\mathbf{W}\right)}{max}\,\left|\triangle_{y,n}^{i}\right|^{2}\left(x\right)>0}$$ for all $\mathrm{y\in \mathbf{W}\cap \mathbf{R}_{N_{0}}}$ and all $\mathrm{n}$ big enough.

In particular, it follows that $$\mathrm{\underset{x\in\pi^{-1}\left(y\right)\cap T\left(\epsilon,\mathbf{W}\right)}{max}\,\left|s^{i}_{m_{0}}\cdot\triangle_{y,n}^{i}\right|^{2}\left(x\right)>0}$$ for all $\mathrm{y\in \mathbf{W}\cap \mathbf{R}_{N_{0}}}$ and all $\mathrm{n}$ big enough.
\end{lemma}
\begin{proof} First of all note that the second claim is a direct consequence of the first because $\mathrm{\vert s^{i}_{m_{0}}\vert^{2}\left(x\right)>0}$ for all $\mathrm{x\in \mathbf{X}^{i}}$ and all $\mathrm{m_{0}\in \mathbb{N}}$.

We already know that $\mathrm{\triangle^{i}_{y,n}\vert \pi^{-1}\left(y\right)\not\equiv 0}$ for all $\mathrm{n\in \mathbb{N}}$ and all $\mathrm{y\in \mathbf{W}\cap \mathbf{R}_{N_{0}}}$. Throughout the proof we will fix $\mathrm{y}$ and let $$\mathrm{\pi^{-1}\left(y\right)=\bigcup_{j}\mathbf{C}_{j,y}}$$ be the decomposition of $\mathrm{\pi^{-1}\left(y\right)}$ in its global irreducible components $\mathrm{\mathbf{C}_{j,y}}$. It is direct to see that these components are $\mathrm{\mathbb{T}}$-invariant and hence each $\mathrm{\mathbf{C}_{j,y}}$ intersects $\mathrm{\mu^{-1}\left(\xi\right)}$ non-trivially:\ In fact, let $\mathrm{z\in \mathbf{C}_{j,y}}$ and choose a one parameter subgroup $\mathrm{\gamma\!:\!\mathbb{C}^{*}\rightarrow \mathbb{T}}$ so that $$\mathrm{z_{0}=\underset{t\rightarrow0}{lim}\,\gamma\left(t\right).z\in \mathbb{T}.x_{y}}$$ where $\mathrm{\mathbb{T}.x_{y}}$ is the unique closed orbit in the fiber $\mathrm{\pi^{-1}\left(y\right)}$. Since $\mathrm{\mathbf{C}_{j,y}}$ is closed and $\mathrm{\mathbb{T}}$ invariant it follows that $\mathrm{z_{0}\in \mathbb{T}.x_{y}\cap \mathbf{C}_{j,y}}$. Again by the invariance of $\mathrm{\mathbf{C}_{j,y}}$ and the fact that $\mathrm{\mu^{-1}\left(\xi\right)\cap \mathbb{T}.x_{y}=\mu^{-1}\left(\xi\right)\cap \pi^{-1}\left(y\right)}$ it follows that $\mathrm{\mu^{-1}\left(\xi\right)\cap \mathbf{C}_{j,y}\neq \varnothing}$. In particular, the open set $$\mathrm{T\left(\epsilon,\mathbf{W}\right)\cap \mathbf{C}_{j,y}}$$ is always non-empty in $\mathrm{\mathbf{C}_{j,y}}$. So if $$\mathrm{\underset{x\in\pi^{-1}\left(y\right)\cap T\left(\epsilon,\mathbf{W}\right)}{max}\,\left|\triangle_{y,n}^{i}\right|^{2}\left(x\right)=0}$$ for $\mathrm{y\in \mathbf{W}\cap\mathbf{R}_{N_{0}}}$, it would follow that $\mathrm{\triangle^{i}_{y,n}}$ would vanish identically on each non-empty open subset $\mathrm{T\left(\epsilon,\mathbf{W}\right)\cap \mathbf{C}_{j,y}}$ of the irreducible component $\mathrm{\mathbf{C}_{j,y}}$. Hence, it would follow by the Identity Principle that $\mathrm{\triangle^{i}_{y,n}\vert \mathbf{C}_{j,y}\equiv 0}$ for all $\mathrm{j}$ and therefore $\mathrm{\triangle^{i}_{y,n}\vert  \pi^{-1}\left(y\right)\equiv 0}$ in contradiction to $\mathrm{\triangle^{i}_{y,n}\vert \pi^{-1}\left(y\right)\not\equiv 0}$.
\end{proof}

\subsection{A Local Proposition Concerning Fiber Integration}\label{A Local Proposition Concerning Fiber Integration}
In this section we will prove a technical proposition ({cf.\ }{\bf{Proposition \ref{Proposition Local Version Of Proposition}}}) which will be of crucial importance when proving {\bf{Theorem 5.a,\ b}} and {\bf{Theorem 6.a,\ b}}.

A first step towards {\bf{Proposition \ref{Proposition Local Version Of Proposition}}} is the following technical lemma.
 
\begin{lemma}\label{Lemma Bound Of Norm Squares Of Roots} Let $$\mathrm{z^{d}+\alpha_{d-1}z^{d-1}+\dots+\alpha_{0},\,\text{ where }\alpha_{i}\in \mathbb{C}\text{ for }0\leq i\leq d-1}$$ be a monic polynomial of degree $\mathrm{d}$ and let $\mathrm{\zeta_{i}}$, $\mathrm{1\leq i\leq d}$ the corresponding roots, then there exists a constant $\mathrm{c_{d}>0}$ which only depends of the degree $\mathrm{d}$, so that $$\mathrm{\sum_{i=1}^{d}\vert \zeta_{i}\vert^{2}\geq c_{d}\left(\sum_{i=0}^{d-1}\vert \alpha_{i}\vert^{2}\right)^{\frac{1}{d}}.}$$
\end{lemma}
\begin{proof} Consider the holomorphic map $\mathrm{F\!:\!\mathbb{C}^{d}\rightarrow \mathbb{C}^{d}}$ defined by $$\mathrm{F\!:\!\left(\zeta_{1},\dots,\zeta_{d}\right)\mapsto \left(\mathfrak{P}_{0}\left(\zeta_{1},\dots,\zeta_{d}\right),\dots,\mathfrak{P}_{d-1}\left(\zeta_{1},\dots,\zeta_{d}\right)\right)}$$ where $\mathrm{\left\{\mathfrak{P}_{\ell}\right\}_{0\leq \ell\leq d-1}}$ is the set of all elementary symmetric polynomials, i.e.\ $$\mathrm{\begin{array}{rcl}
\mathrm{\mathfrak{P}_{d-1}\left(\zeta_{1},\dots,\zeta_{d}\right)} & \mathrm{=} & \mathrm{\left(-1\right)\,\sum_{1\leq j\leq d}\zeta_{j}}\\[0.2 cm]
\mathrm{\mathfrak{P}_{d-2}\left(\zeta_{1},\dots,\zeta_{d}\right)} & \mathrm{=} & \mathrm{\left(+1\right)\,\sum_{1\leq j_{1}<j_{2}\leq d}\zeta_{j_{1}}\zeta_{j_{2}}}\\[0.2 cm]
 & \mathrm{\vdots}\\[0.2 cm]
\mathrm{\mathfrak{P}_{d-\ell}\left(\zeta_{1},\dots,\zeta_{d}\right)} & \mathrm{=} & \mathrm{\left(-1\right)^{\ell}\sum_{1\leq j_{1}<\dots<j_{\ell}\leq d}\zeta_{j_{1}\leq d}\dots\zeta_{j_{\ell}}}\\[0.2 cm]
 & \mathrm{\vdots}\\[0.2 cm]
\mathrm{\mathfrak{P}_{0}\left(\zeta_{1},\dots,\zeta_{d}\right)} & = & \mathrm{\left(-1\right)^{d}\zeta_{1}\dots\zeta_{m}}\end{array}}$$ If $\mathrm{\alpha=\left(\alpha_{0},\dots,\alpha_{d-1}\right)\in \mathbb{C}^{d}}$, then by {\scshape{Vieta}}'s Formula we deduce that $$\mathrm{F^{-1}\left(\alpha\right)=\left\{\zeta\in \mathbb{C}\!:\!\, \zeta^{d}+\alpha_{d-1}\,\zeta^{d-1}+\alpha_{d-2}\,\zeta^{d-2}+\dots +\alpha_{0}=0 \right\}}$$ so the inverse image of $\mathrm{\alpha\in \mathbb{C}^{d}}$ contains exactly all roots of the polynomial with coefficients $\mathrm{\alpha_{\ell}}$ for $\mathrm{0\leq\ell\leq d-1}$. In particular, it follows that $\mathrm{F^{-1}\left(0\right)=\left\{0\right\}}$ and by continuity there exists $\mathrm{c_{d}> 0}$ so that $\mathrm{{\bm{\Delta}}_{c_{d}}\subset Int\,F^{-1}\left({\bm{\Delta}}_{1}\right)}$ where $\mathrm{{\bm{\Delta}}_{\lambda}\subset \mathbb{C}^{d}}$ is the closed ball of radius $\mathrm{\lambda>0}$ in $\mathrm{\mathbb{C}^{d}}$. It is direct to check that $$\mathrm{F\left(\lambda\,\zeta_{1},\dots,\lambda\,\zeta_{d}\right)=\left(\lambda^{d}\,\mathfrak{P}_{0}\left(\zeta_{1},\dots,\zeta_{d}\right),\dots,\lambda\,\mathfrak{P}_{d-1}\left(\zeta_{1},\dots,\zeta_{d}\right)\right)\text{ for all }\lambda\in \mathbb{C}}$$ and hence, we deduce that $$\mathrm{{\bm{\Delta}}_{\lambda^{\frac{1}{d}}\,c_{d}}\subset Int\,F^{-1}\left({\bm{\Delta}}_{\lambda}\right)}$$ for all $\mathrm{\lambda\geq 0}$. This can be reformulated as follows: If $\mathrm{z^{d}+\alpha_{d-1}z^{d-1}+\dots+\alpha_{0}}$ is a monic polynomial of degree $\mathrm{d}$ so that $\mathrm{\sum_{i=0}^{d-1}\vert \alpha_{i}\vert^{2}=\lambda}$, {i.e.\! }$\mathrm{\alpha\in bd\, {\bm{\Delta}}_{\lambda}}$, then $\mathrm{F^{-1}\left(\alpha\right)\notin{\bm{\Delta}}_{\lambda^{\frac{1}{d}}\,c_{d}}}$, i.e. $$\mathrm{\sum_{i=1}^{d}\vert \zeta_{i}\vert^{2}\geq \lambda^{\frac{1}{d}}\cdot c_{d}}$$ where $\mathrm{F^{-1}\left(\alpha\right)=\left\{\zeta_{i}\right\}_{1\leq j\leq d}}$ is the set of the corresponding roots and therefore $$\mathrm{\sum_{i=1}^{d}\vert \zeta_{i}\vert^{2}\geq c_{d}\left(\sum_{i=0}^{d-1}\vert \alpha_{i}\vert^{2}\right)^{\frac{1}{d}}}$$ as claimed.
\end{proof}

Let $\mathrm{F\!:\!\mathbf{X}\rightarrow \mathbf{Y}}$\label{Notation D-Sheeted Covering} be $\mathrm{k}$-fibering, i.e.\ a holomorphic map between purely dimensional complex spaces where $\mathrm{m=dim_{\mathbb{C}}\,\mathbf{X}}$ and $\mathrm{n=dim_{\mathbb{C}}\,\mathbf{Y}}$ so that $\mathrm{F^{-1}\left(y\right)}$ is a purely dimensional complex space of dimension $\mathrm{k=m-n}$ for all $\mathrm{y\in \mathbf{Y}}$. The relevant examples of $\mathrm{k}$-fiberings, which we have in mind, are given by $\mathrm{\widehat{\pi}\!:\!\widehat{\mathbf{\,X\,}}\rightarrow \mathbf{Y}}$, $\mathrm{\pi\vert \mathbf{X}_{0}\!:\!\mathbf{X}_{0}\rightarrow \mathbf{Y}_{0}}$ where $\mathrm{\mathbf{X}_{0}=\pi^{-1}\left(\mathbf{Y}_{0}\right)}$ and $\mathrm{\Pi\!:\!\widetilde{\bf{\,X\,}}\rightarrow \widetilde{\bf{\,Y\,}}}$.

If $\mathrm{x_{0}\in \mathbf{X}}$ and $\mathrm{F\!:\!\mathbf{X}\rightarrow \mathbf{Y}}$ a $\mathrm{k}$-fibering, then there exists ({cf.\ }\cite{Kin}, {p.\ }{\bf{\oldstylenums{205}}}) an open neighborhood $\mathrm{\mathbf{U}\subset \mathbf{X}}$ of $\mathrm{x_{0}}$ which can be realized via an isomorphism $\mathrm{\Phi}$ as a closed analytic subset $\mathrm{\mathbf{Z}\subset\mathbf{Q}= \mathbf{Q}_{0}\times \mathbf{Q}_{1}\subset \mathbb{C}^{\upkappa}\times \mathbb{C}^{k}}$ of a relatively compact, open product set $\mathrm{\mathbf{Q}_{0}\times \mathbf{Q}_{1}}$ and an open neighborhood $\mathrm{\mathbf{B}\subset \mathbf{Y}}$ of $\mathrm{y_{0}=F\left(x_{0}\right)\in \mathbf{Y}}$ so that the following holds: If $$\mathrm{z^{\left(0\right)}=\big(z^{\left(0\right)}_{1},\dots,z^{\left(0\right)}_{\upkappa}\big),\text{ resp. }z^{\left(1\right)}=\big(z^{\left(1\right)}_{1},\dots,z^{\left(1\right)}_{k}\big).}$$ are the standard coordinates on $\mathrm{\mathbb{C}^{\upkappa}}$, resp.\ on $\mathrm{\mathbb{C}^{k}}$, then the restriction of the projection map $\mathrm{p\!:\!\mathbb{C}^{\upkappa}\times \mathbb{C}^{k}\rightarrow}$ $\mathrm{\mathbb{C}^{k}}$ to the closed $\mathrm{k}$-dimensional space $\mathrm{\mathbf{Z}_{y}\coloneqq \Phi\left(F^{-1}\left(y\right)\cap \mathbf{U}\right)}$ of $\mathrm{\mathbf{Q}=\mathbf{Q}_{0}\times \mathbf{Q}_{1}}$ induces a $\mathrm{d}$-sheeted covering map $\mathrm{p\vert\mathbf{\mathbf{Z}}_{y} \!:\!\mathbf{\mathbf{Z}}_{y}\rightarrow \mathbf{Q}_{1}}$ onto $\mathrm{\mathbf{Q}_{1}}$ for all $\mathrm{y\in \mathbf{B}\subset \mathbf{Y}}$. Moreover, by the theory of finite $\mathrm{d}$-sheeted coverings the following is known ({cf.\ }\cite{Gr-Re}, {pp.\ }{\bf{\oldstylenums{133}-\oldstylenums{146}}}): For each $\mathrm{d}$-sheeted covering map $\mathrm{p\vert \mathbf{Z}_{y}\rightarrow \mathbf{Q}_{1}}$ where $\mathrm{y\in \mathbf{B}}$, the inclusion $$\mathrm{\mathcal{O}\left(\mathbf{Q}_{1}\right)\subset \left(p\vert \mathbf{Z}_{y}\right)_{*}\mathcal{O}\left(\mathbf{Z}_{y}\right)}$$ is a finite, integral ring extension so that  for all for each $\mathrm{f\in \mathcal{O}\left(\mathbf{Z}_{y}\right)}$, there exists  $\mathrm{\alpha_{f,j}\in \mathcal{O}\left(\mathbf{Q}_{1}\right)}$, $\mathrm{0\leq j\leq d-1}$ with
\begin{equation}\label{Equation Polynomial Equation Induces By Finite Covers}
\mathrm{f^{d}+\left(p\vert \mathbf{Z}_{y}\right)^{*}\alpha_{f,d-1}\,f^{d-1}+\dots+\left(p\vert \mathbf{Z}_{y}\right)^{*}\alpha_{f,0}=0}
\end{equation} on $\mathrm{\mathbf{Z}_{y}}$. 

\begin{example} \textnormal{Let $\mathrm{\widehat{\mathbf{\,X\,}}=\left\{ z_{0}\zeta_{1}\xi_{1}-z_{1}\zeta_{0}\xi_{0}=0\right\} \subset \mathbf{X}\times\mathbb{C}\mathbb{P}^{1}}$ be as in example {\bf{Example \ref{Example Proper Inclusion}}} where $\mathrm{\widehat{\mathbf{\,Y\,}}\cong \mathbf{Y}\cong \mathbb{C}\mathbb{P}^{1}}$ and $\mathrm{\Pi=\widehat{\pi}=p_{\mathbb{C}\mathbb{P}^{1}}}$ and consider the open neighborhood $\mathrm{\mathbf{U}\subset\widehat{\mathbf{\,X\,}}}$ of $\mathrm{\left([1\!:\!0],[0\!:\!1],[1\!:\!0]\right)}$ which is isomorphic to the affine variety $\mathrm{\left\{\xi-z\zeta=0\right\}\subset \mathbb{C}^{3}}$ where $\mathrm{z=z_{0}^{-1}z_{1}}$, $\mathrm{\zeta=\zeta_{1}^{-1}\zeta_{0}}$, $\mathrm{\xi =\xi_{0}^{-1}\xi_{1}}$. After a linear change of coordinates it follows that $\mathrm{\widehat{\mathbf{\,X\,}}\cap \mathbf{U}\cong\left\{\xi-z^{\prime,2}-\zeta^{\prime,2}=0\right\}}$ and we can choose $\mathrm{\mathbb{C}^{2}\times \mathbb{C}=\mathbf{B}_{0}\times \mathbf{B}_{1}}$ where $\mathrm{\zeta^{\prime}, \xi}$ are the coordinates of the first factor and $\mathrm{z^{\prime}}$ of the second one in order to deduce a two sheeted, global projection $\mathrm{p\vert \mathbf{Z}_{\xi}:\mathbf{Z}_{\xi}\rightarrow \mathbf{Q}_{1}}$ onto $\mathrm{\mathbf{Q}_{1}=\mathbb{C}}$ for all $\mathrm{\xi\in \mathbf{Y}\setminus \left\{[0\!:\!1]\right\}}$.}

\textnormal{If $\mathrm{f\in \mathcal{O}\left(\mathbf{Z}_{\xi}\right)}$ is given by the restriction of the polynomial $\mathrm{p\left(z^{\prime}\right)=\sum_{j=0}^{m}c_{j}z^{\prime,j}}$ to $\mathrm{\mathbf{Z}_{\xi}}$ one can verify that $$\mathrm{\alpha_{f,1}\left(z^{\prime}\right)=-2\sum_{\substack{\textnormal{j=0}\\[0.05 cm] \textnormal{even}}}^{m}c_{j}\left(z^{\prime,2}-\xi\right)^{\frac{j}{2}}}$$ and $$\mathrm{\alpha_{f,0}\left(z^{\prime}\right)=\sum_{j=0}^{m}\left(-1\right)^{j}c_{j}^{2}\left(z^{\prime,2}-\xi\right)^{j}+2\sum^{m}_{\substack{\textnormal{0=j}<\textnormal{k}\\[0.05cm] \textnormal{j,k even}}}c_{j}c_{k}\left(z^{\prime,2}-\xi\right)^{\frac{1}{2}\left(j+k\right)}. \,\boldsymbol{\Box}}$$}
\end{example}

Now let $\mathrm{\mathbf{V}\subset \mathbb{C}^{\upkappa}\times \mathbb{C}^{k}}$ be an open neighborhood of $\mathrm{x=0\in \mathbb{C}^{\upkappa}\times \mathbb{C}^{k}}$ containing $\mathrm{\mathbf{Q}}$ and equip $\mathrm{\mathbf{V}}$ with a smooth K\"ahler form $\mathrm{\omega}$.

With the help of {\bf{Lemma \ref{Lemma Bound Of Norm Squares Of Roots}}} and the existence of \ref{Equation Polynomial Equation Induces By Finite Covers}, we deduce the following lemma.

\begin{lemma} \label{Lemma Lower Bound Of Integral} Let $\mathrm{y\in\mathbf{B}}$ and $\mathrm{f\in \mathcal{O}\left(\mathbf{Z}_{y}\right)}$ where $\mathrm{\alpha_{f,j}}$ $\mathrm{\in \mathcal{O}\left(\mathbf{Q}_{1}\right)}$, $\mathrm{1\leq j\leq d-1}$ are as above. Then it follows $$\mathrm{\int_{\mathbf{Z}_{y}}\vert f\vert^{2} \,d\,[\mathbf{Z}_{y}]\geq d\cdot c_{k}\int_{\mathbf{Q}_{1}}\left(\sum_{j=0}^{d-1}\vert \alpha_{f,j}\vert^{2}\right)^{\frac{1}{d}}\,\omega_{0}^{q}}$$ where $\mathrm{\omega_{0}}$ denotes the standard K\"ahler form, $\mathrm{c_{d}}$ the constant of {\bf{Lemma \ref{Lemma Bound Of Norm Squares Of Roots}}} and $\mathrm{d}$ the degree of the covering map.
\end{lemma}
\begin{proof} It is known (cf.\ \cite{Kin}, {pp.\ }{\bf{\oldstylenums{185}-\oldstylenums{220}}}) that the right hand side of the above inequality can be be bounded from below by $$\mathrm{d\,\int_{\mathbf{Q}_{1}}\left(z^{\left(1\right)}\mapsto \sum_{j=1}^{d}\left\vert \zeta_{f,j}\big(z^{\left(1\right)}\big)\right\vert^{2}\right)\,\omega^{q}_{0}}$$ where $\mathrm{\zeta_{f,j}\left(z^{\left(1\right)}\right)}$ are the $\mathrm{d}$ roots of the polynomial equation $$\mathrm{z^{d}+\alpha_{f,d-1}\big(z^{\left(1\right)}\big)\,?z^{d-1}+\dots+\alpha_{f,0}\big(z^{\left(1\right)}\big)=0.}$$ The claim then follows by using the inequality proved in {\bf{Lemma \ref{Lemma Bound Of Norm Squares Of Roots}}}.
\end{proof}

After this preparation, we can now prove the announced proposition.

\begin{prop} \label{Proposition Local Version Of Proposition} Let $\mathrm{\left(y_{n}\right)_{n}}$ be a sequence in $\mathrm{\mathbf{B}}$ converging to $\mathrm{y_{0}\in \mathbf{B}}$ and let $\mathrm{\left(f_{n}\right)_{n}}$, $\mathrm{f_{n}\in \mathcal{O}\left(\mathbf{Z}_{y_{n}}\right)}$ be a sequence of uniformly bounded holomorphic functions so that $$\mathrm{\int_{\mathbf{Z}_{n}}\vert f_{n}\vert^{2}\,d\,[\mathbf{Z}_{n}]\rightarrow 0.}$$

Then it follows that for each compact subset $\mathrm{\mathbf{K}\subset \mathbf{Q}_{1}}$ and each $\mathrm{\epsilon>0}$ there exists $\mathrm{N_{\epsilon}\left(\mathbf{K}\right)}$ $\mathrm{\in\mathbb{N}}$ so that $$\mathrm{\vert f_{n}\vert^{2}\leq \epsilon\text{ on }p^{-1}\left(\mathbf{K}\right)\cap \mathbf{Z}_{n}\text{ for all }n\geq N_{\epsilon}\left(\mathbf{K}\right).}$$
\end{prop}
\begin{proof} First of all set $\mathrm{\alpha_{n,j}\coloneqq\alpha_{f_{n},j}\in \mathcal{O}\left(\mathbf{Z}_{y_{n}}\right)}$ for $\mathrm{0\leq j\leq d-1}$. By the assumption combined with {\bf{Lemma \ref{Lemma Lower Bound Of Integral}}} we deduce that
\begin{equation}\label{Equation Sum Converges Against Zero}
\mathrm{d\cdot c_{d}\int_{\mathbf{Q}_{1}}\left(\sum_{j=0}^{d-1}\vert \alpha_{n,j}\vert^{2}\right)^{\frac{1}{d}}\,\omega_{0}^{q}\rightarrow 0}
\end{equation}
where
\begin{equation}\label{Equation F As Zero Of Polynomial}
\mathrm{f_{n}^{d}+\left(p\vert \mathbf{Z}_{n}\right)^{*}\alpha_{n,d-1}\,f_{n}^{d-1}+\dots+\left(p\vert \mathbf{Z}_{n}\right)^{*}\alpha_{n,0}=0.}
\end{equation} Since the sequence $\mathrm{\left(f_{n}\right)_{n}}$ is uniformly bounded, i.e.\ there exists $\mathrm{C>0}$ so that $\mathrm{\vert f_{n}\vert^{2}\leq C}$ on $\mathrm{\mathbf{Z}_{n}}$ for all $\mathrm{n}$, it follows by equation \ref{Equation F As Zero Of Polynomial} in combination with {\bf{Lemma \ref{Lemma Bound Of Norm Squares Of Roots}}} that $$\mathrm{d\cdot C\geq c_{d}\,\left(\sum_{j=0}^{d-1}\vert \alpha_{n,j}\vert^{2}\right)^{\frac{1}{d}}\text{ on }\mathbf{Q}_{1}\text{ for all }n.}$$ Hence, the sequences $\mathrm{\left(\alpha_{n,j}\right)_{n}}$, where $\mathrm{\alpha_{n,j}\in\mathcal{O}\left(\mathbf{Q}_{1}\right)}$, are uniformly bounded as well for all $\mathrm{1\leq j\leq d-1}$. By the theorem of {\scshape{Montel}} (after having chosen a subsequence) it follows that $$\mathrm{\alpha_{n,j}\rightarrow \alpha_{j}\in \mathcal{O}\left(\mathbf{K}\right)\text{ uniformly on the compact subset }\mathbf{K}\subset \mathbf{Q}_{1}.}$$ The next step is to show that $\mathrm{\alpha_{j}\equiv 0}$ for all $\mathrm{j}$. Let us assume that this is false, i.e.\ there exist at least one $\mathrm{\alpha_{\ell}\not\equiv 0}$ for $\mathrm{0\leq \ell\leq d-1}$. Since $\mathrm{\alpha_{n,j}\rightarrow \alpha_{j}}$ converges uniformly on $\mathrm{\mathbf{K}}$, it follows by \ref{Equation Sum Converges Against Zero} that $$\mathrm{0=\underset{n\rightarrow\infty}{lim}\,c_{d}\int_{\mathbf{K}}\left(\sum_{j=0}^{d-1}\vert \alpha_{n,j}\vert^{2}\right)^{\frac{1}{d}}\,\omega_{0}^{q}=c_{d}\int_{\mathbf{K}}\left(\sum_{j=0}^{d-1}\vert \alpha_{j}\vert^{2}\right)^{\frac{1}{d}}\,\omega_{0}^{q}}$$ which yields a contradiction to the assumption that  $\mathrm{\alpha_{\ell}\not\equiv 0}$ on $\mathrm{\mathbf{K}}$ for at least one $\mathrm{\ell}$. Hence, we deduce that $\mathrm{\alpha_{j}\equiv 0}$ on $\mathrm{\mathbf{K}}$ for all $\mathrm{0\leq j\leq d-1}$. As the convergence is uniform, we find for an arbitrary $\mathrm{\Gamma>0}$ an integer $\mathrm{N\left(\Gamma\right)\in \mathbb{N}}$ so that $\mathrm{\vert \alpha_{n,j}\vert^{2}\leq \Gamma}$ for all $\mathrm{n\geq N\left(\Gamma\right)}$ on $\mathrm{\mathbf{K}}$. By equation \ref{Equation F As Zero Of Polynomial} combined with the general fact that if $\mathrm{\zeta\in \mathbb{C}}$ is a root of $\mathrm{z^{d}+\alpha_{d-1}\,z^{d-1}+\dots+}$ $\mathrm{\alpha_{0}=0}$ then $\mathrm{\vert \zeta\vert\leq 2\,\underset{1\leq j\leq d}{max}\,\vert \alpha_{d-j}\vert^{\frac{1}{j}} }$ (cf.\ \cite{Dem}), we deduce $$\mathrm{\vert f_{n}\vert\leq 2\,\underset{1\leq j\leq d}{max}\,\Gamma^{\frac{1}{j}} \text{ on }p^{-1}\left(\mathbf{K}\right)\cap \mathbf{Z}_{n}\text{ for all }n\geq N\left(\Gamma\right)}$$ which proves the claim:\ Choose $\mathrm{\Gamma_{\epsilon}>0}$ so that $\mathrm{2\,\underset{1\leq j\leq d}{max}\,\Gamma_{\epsilon}^{\frac{1}{j}}=\epsilon}$ and set $\mathrm{N_{\epsilon}\left(\mathbf{K}\right)\coloneqq N(\Gamma_{\epsilon})}$. 
\end{proof}

\subsection{Local Uniform Convergence on $\mathrm{\mathbf{R}_{N_{0}}\cap\mathbf{Y}_{0}}$}

Recall that the set of all removable singularities $\mathrm{\mathbf{R}_{N_{0}}}$ of the measure sequence $\mathrm{\left(\bm{\nu}_{n}\right)_{n}}$ induced by the $\mathrm{\xi}$-approximating sequence $\mathrm{\left(s_{n}\right)_{n}}$ of $\mathrm{\xi_{n}}$-eigensections $\mathrm{s_{n}\in H^{0}\left(\mathbf{X},\mathbf{L}^{n}\right)}$ is assumed to be non-empty throughout this section (cf.\ {\bf{General Agreement}} at the beginning of page \pageref{General Agreement}). The first step towards the proof of {\bf{Theorem 5.a}} and {\bf{Theorem 6.a}} is the following proposition.

\begin{prop}\label{Proposition Tales Are Uniformly Bounded} Let $\mathrm{\mathbf{W}\subset\mathbf{Y}_{0}\cap \mathbf{R}_{N_{0}}}$ be a compact neighborhood of $\mathrm{y_{0}\in \mathbf{Y}_{0}\cap \mathbf{R}_{N_{0}}}$ (we can assume that $\mathrm{\mathbf{W}\subset \mathbf{Y}^{i}}$) and $\mathrm{\epsilon>0}$, then after having shrunken $\mathrm{\mathbf{W}}$, there exists a constant $\mathrm{C>0}$ and $\mathrm{m_{0}\in\mathbb{N}}$ so that $$\mathrm{\underset{x\in\pi^{-1}\left(y\right)}{max}\,\,\,\frac{\left|s_{m_{0}}^{i}\cdot\triangle_{y,n}^{i}\right|^{2}}{\int_{\pi^{-1}\left(y\right)\cap T\left(\epsilon,\mathbf{W}\right)}\left|s_{m_{0}}^{i}\cdot\triangle_{y,n}^{i}\right|^{2}\,d\,[\pi_{y}]}<C}$$ for all $\mathrm{n\geq m_{0}}$ and all $\mathrm{y\in{\mathbf{W}}}$.
\end{prop} 
\begin{proof} First of all by {\bf{Proposition \ref{Theorem Estimates Tales of Triangle}}} we know that there exists $\mathrm{m_{0}}$ so that 
\begin{equation}\label{Equation Maximum Is Contained In}
\mathrm{\underset{x\in\pi^{-1}\left(y\right)}{max}\,\left|s_{m_{0}}^{i}\cdot\triangle_{y,n}^{i}\right|^{2}=\underset{x\in\pi^{-1}\left(y\right)\cap T\left(\frac{\epsilon}{2},\mathbf{W}^{i}\right)}{max}\,\left|s_{m_{0}}^{i}\cdot\triangle_{y,n}^{i}\right|^{2}}
\end{equation} for all $\mathrm{n}$ big enough and all $\mathrm{y\in \mathbf{W}\cap \mathbf{R}_{N_{0}}=\mathbf{W}}$ (by the above assumption, $\mathrm{\mathbf{W}}$ is contained in $\mathrm{\mathbf{R}_{N_{0}}\cap \mathbf{Y}_{0}}$). Furthermore, by lemma {\bf{Lemma \ref{Lemma The Maximum Of Triangle Is Fiberwisely Not Zero}}} we deduce that $$\mathrm{\underset{x\in\pi^{-1}\left(y\right)\cap T\left(\epsilon,\mathbf{W}^{i}\right)}{max}\,\left|s_{m_{0}}^{i}\cdot\triangle_{y,n}^{i}\right|^{2}>0}$$ for all $\mathrm{n}$ big enough and all $\mathrm{y\in \mathbf{W}\cap \mathbf{R}_{N_{0}}=\mathbf{W}}$. Hence, $$\mathrm{\int_{\pi^{-1}\left(y\right)\cap T\left(\epsilon,\mathbf{W}\right)}\frac{\left|s_{m_{0}}^{i}\cdot\triangle_{y,n}^{i}\right|^{2}}{\underset{x\in\pi^{-1}\left(y\right)\cap T\left(\epsilon,\mathbf{W}\right)}{max}\,\left|s_{m_{0}}^{i}\cdot\triangle_{y,n}^{i}\right|^{2}}\,d\,[\pi_{y}]}$$ is well defined for all $\mathrm{n}$ big enough and all $\mathrm{y\in \mathbf{W}}$. Note that the claim is shown, as soon as we have shown that, after having shrunken $\mathrm{\mathbf{W}}$, there exists $\mathrm{c>0}$ (set $\mathrm{C\coloneqq c^{-1}}$) so that $$\mathrm{\int_{\pi^{-1}\left(y\right)\cap T\left(\epsilon,\mathbf{W}\right)}\frac{\left|s_{m_{0}}^{i}\cdot\triangle_{y,n}^{i}\right|^{2}}{\underset{x\in\pi^{-1}\left(y\right)\cap T\left(\epsilon\mathbf{W}\right)}{max}\,\left|s_{m_{0}}^{i}\cdot\triangle_{y,n}^{i}\right|^{2}}\,d\,[\pi_{y}]>c}$$ for all $\mathrm{n}$ big enough and all $\mathrm{y\in \mathbf{W}}$.

Let us assume that this is not the case, then there exists a sequence $\mathrm{y_{n}\in \mathbf{W}}$ converging to $\mathrm{y_{0}}$ so that 
\begin{equation}\label{Equation Integral Converges To Zero}
\mathrm{\int_{\pi^{-1}\left(y_{n}\right)\cap T\left(\epsilon,\mathbf{W}\right)}\frac{\left|s_{m_{0}}^{i}\cdot\triangle_{y_{n},n}^{i}\right|^{2}}{\underset{x\in\pi^{-1}\left(y_{n}\right)\cap T\left(\epsilon,\mathbf{W}\right)}{max}\,\left|s_{m_{0}}^{i}\cdot\triangle_{y_{n},n}^{i}\right|^{2}}\,d\,[\pi_{y_{n}}]\rightarrow 0.}
\end{equation} To shorten notation we will set $$\mathrm{f_{n}\coloneqq \frac{s^{i}_{m_{0}}\cdot \triangle^{i}_{y_{n},n}}{\underset{x\in\pi^{-1}\left(y_{n}\right)\cap T\left(\epsilon,\mathbf{W}\right)}{max}\,\left|s_{m_{0}}^{i}\cdot\triangle_{y_{n},n}^{i}\right|}\Bigg\vert\pi^{-1}\left(y_{n}\right)}$$ throughout the rest of this proof where $\mathrm{f_{n}\in H^{0}\left(\pi^{-1}\left(y_{n}\right),\mathbf{L}^{m_{0}}\vert \pi^{-1}\left(y_{n}\right)\right)}$. By equation \ref{Equation Maximum Is Contained In} it follows that $$\mathrm{\underset{x\in\pi^{-1}\left(y_{n}\right)}{max}\,\left|f_{n}\right|^{2}=\underset{x\in\pi^{-1}\left(y_{n}\right)\cap T\left(\frac{\epsilon}{2},\mathbf{W}\right)}{max}\,\left|f_{n}\right|^{2}=1}$$ for all $\mathrm{n}$ big enough where $\mathrm{y_{n}\rightarrow y_{0}\in Int\,\mathbf{W}}$. Therefore we can find a lift of $\mathrm{\left(y_{n}\right)_{n}}$, i.e.\ a sequence $\mathrm{\left(x_{n}\right)_{n}}$ in $\mathrm{T\left(\frac{\epsilon}{2},\mathbf{W}\right)}$ so that $\mathrm{\pi\left(x_{n}\right)=y_{n}}$ with the property
\begin{equation}\label{Equation Norm Is Equal To One}
\mathrm{\vert f_{n}\vert^{2}\left(x_{n}\right)=1=\underset{x\in\pi^{-1}\left(y_{n}\right)}{max}\,\left|f_{n}\right|^{2}}.
\end{equation} As $\mathrm{x_{n}\in T\left(\frac{\epsilon}{2},\mathbf{W}\right)}$ where $\mathrm{y_{n}=\pi\left(x_{n}\right)\rightarrow y_{0}\in Int\,\mathbf{W}}$, we can assume using the compactness of $\mathrm{T\left(\frac{\epsilon}{2},\mathbf{W}\right)}$ (after having chosen a subsequence) that $\mathrm{x_{n}\rightarrow x_{0}\in }$ $\mathrm{Int\,T\left(\epsilon,\mathbf{W}\right)}$. Hence, there exists an open neighborhood $\mathrm{\mathbf{U}}$ of $\mathrm{x_{0}}$ which is contained in the interior of $\mathrm{T\left(\epsilon,\mathbf{W}\right)}$ so that $\mathrm{x_{n}\in \mathbf{U}}$ for all $\mathrm{n}$ big enough.

Since we have $\mathrm{x_{0} \in \pi^{-1}\left(y_{0}\right)\subset\pi^{-1}\left(\mathbf{Y}_{0}\right)}$, we can assume that the open neighborhood $\mathrm{\mathbf{U}}$ (after having shrunken it) is isomorphic via an holomorphic embedding $\mathrm{\Phi\!:\!\mathbf{U}\rightarrow \mathbb{C}^{\upkappa}\times \mathbb{C}^{k}}$ (let $\mathrm{\Phi\left(x_{0}\right)=0=\left(0,0\right)\in \mathbf{Q}_{0}\times \mathbf{Q}_{1}}$) to a closed analytic subset in a relatively compact product neighborhood $\mathrm{\mathbf{Q}=\mathbf{Q}_{0}\times \mathbf{Q}_{1}\subset}$ $\mathrm{\mathbb{C}^{\upkappa}\times \mathbb{C}^{k}}$ so that each $\mathrm{\Phi\left(\pi^{-1}\left(y\right)\cap \mathbf{U}\right)=\mathbf{Z}_{y}}$, where $\mathrm{y}$ varies in an open neighborhood $\mathrm{\mathbf{B}\subset \mathbf{Y}_{0}}$ of $\mathrm{\pi\left(x_{0}\right)}$, is a closed analytic subset which yields a surjective, $\mathrm{d}$-sheeted covering of $\mathrm{\mathbf{Q}_{1}}$ given by $\mathrm{p\vert \mathbf{Z}_{y}\rightarrow \mathbf{Q}_{1}}$. The existence of such a neighborhood $\mathrm{\mathbf{U}}$ has been treated in {\bf{Section \ref{A Local Proposition Concerning Fiber Integration}}}, page {\bf{\oldstylenums{\pageref{Notation D-Sheeted Covering}}}}.

After having shrunken $\mathrm{\mathbf{U}}$ again, we can also assume that $\mathrm{\mathbf{L}^{m_{0}}\vert \mathbf{U} \cong  \mathbf{U}\times \mathbb{C}}$. In particular, the norm $\mathrm{\vert s_{m_{0}}^{i}\vert^{2}}$ of the holomorphic sections $\mathrm{s_{m_{0}}^{i}}$ over $\mathrm{\mathbf{U}}$ is then given by a smooth, strictly positive function $\mathrm{\alpha\in\mathcal{C}^{\infty}\left(\mathbf{U}\right)}$, independent of $\mathrm{n\in\mathbb{N}}$, so that $\mathrm{\vert f_{n} \vert^{2}=\alpha\cdot \vert \triangle^{i}_{y_{n},n}\vert^{2}}$. It is important to note that $\mathrm{\vert f_{n} \vert^{2}}$ denotes the norm of the restricted local section $\mathrm{f_{n}}$ with respect to the hermitian bundle metric $\mathrm{h}$, whereas $\mathrm{\vert \triangle^{i}_{y_{n},n}\vert^{2}}$ is the norm induced by the absolute value of the complex numbers. In the sequel, we will use the abbreviation $\mathrm{g_{n}\coloneqq \triangle^{i}_{y_{n},n}\in \mathcal{O}\left(\mathbf{Z}_{y_{n}}\right)}$, i.e.\ we have $\mathrm{\vert f_{n}\vert^{2}=\alpha\cdot \vert g_{n}\vert^{2}}$ on $\mathrm{{\bf{Z}}_{y_{n}}}$. Note that we have $$\mathrm{\underset{x\in \mathbf{Z}_{y_{n}}}{max}\vert f_{n}\vert \mathbf{Z}_{y_{n}}\vert^{2}=\vert f_{n}\left(x_{n}\right)\vert^{2}=1}$$ because $\mathrm{x_{n}\in \mathbf{U}}$ for all $\mathrm{n}$ big enough where $\mathrm{\Phi\left(\pi^{-1}\left(y_{n}\right)\cap \mathbf{U}\right)=\mathbf{Z}_{y_{n}}}$. In other words, if $\mathrm{A>0}$, resp.\ $\mathrm{a>0}$ denotes the maximum, resp.\ the minimum of $\mathrm{\alpha}$ on $\mathrm{\mathbf{U}}$ we deduce that 
\begin{equation}\label{Equation Bla bla bla bli blub}
\mathrm{a^{-1}\geq \underset{x\in \mathbf{Z}_{y_{n}}}{max}\vert g_{n}\vert \mathbf{Z}_{y_{n}}\vert^{2}\text{ and } \vert g_{n}\vert^{2}\left(x_{n}\right)\geq A^{-1}>0}
\end{equation} for all $\mathrm{n}$ big enough. Furthermore, by assumption \ref{Equation Integral Converges To Zero}, it follows that $$\mathrm{\int_{\mathbf{Z}_{n}}\vert f_{n}\vert^{2}d\,[\mathbf{Z}_{n}]\geq a\,\int_{\mathbf{Z}_{n}}\vert g_{n}\vert ^{2}d\,[\mathbf{Z}_{y_{n}}]\rightarrow 0\text{ and hence }\int_{\mathbf{Z}_{n}}\vert g_{n}\vert ^{2}d\,[\mathbf{Z}_{y_{n}}]\rightarrow 0.}$$

We are now in the situation of {\bf{Proposition \ref{Proposition Local Version Of Proposition}}}: We have a sequence $\mathrm{\mathbf{Z}_{n}=\mathbf{Z}_{y_{n}}}$ and a sequence $\mathrm{\left(g_{n}\right)_{n}}$ of uniformly bounded holomorphic functions on $\mathrm{\mathbf{Z}_{n}}$ given by $\mathrm{g_{n}}$ so that $$\mathrm{\int_{\mathbf{Z}_{n}}\vert g_{n}\vert^{2}\,d\,[\mathbf{Z}_{n}]\rightarrow 0.}$$ In particular, if $\mathrm{\mathbf{K}\subset \mathbf{Q}_{1}}$ is a compact neighborhood of $\mathrm{0\in \mathbf{K}}$ and if $\mathrm{\epsilon=\frac{1}{2}A^{-1}}$ we deduce, by {\bf{Proposition \ref{Proposition Local Version Of Proposition}}}, that there exists $\mathrm{N_{\frac{1}{2}A^{-1}}\left(\mathbf{K}\right)}$ so that
\begin{equation}\label{Equation sd<vh dsb}
\mathrm{\vert g_{n}\vert^{2}\leq \frac{1}{2}A^{-1}\text{ on }p^{-1}\left(\mathbf{K}\right)\cap \mathbf{Z}_{n}\text{ for all }n\geq N_{\frac{1}{2}A^{-1}}\left(\mathbf{K}\right).}
\end{equation} However, since $\mathrm{x_{n}\rightarrow x_{0}=0\in p^{-1}\left(\mathbf{K}\right)}$, it follows that $\mathrm{x_{n}\in p^{-1}\left(\mathbf{K}\right)\cap \mathbf{Z}_{n}}$ for all $\mathrm{n}$ big enough. According to the second inequality of \ref{Equation Bla bla bla bli blub}, we have $\mathrm{\vert g_{n}\vert^{2}}$ $\mathrm{\left(x_{n} \right)\geq A^{-1}}$ and hence a contradiction to \ref{Equation sd<vh dsb}. Therefore, the assumption \ref{Equation Integral Converges To Zero} is false and the claim is proven.
\end{proof}

\begin{theo_5.a*}\label{theo_5.a*}\textnormal{[\scshape{Locally Uniform Convergence of the Initial Distribution Sequence}]}\vspace{0.10 cm}\\ 
Let $\mathrm{y\in \mathbf{Y}_{0}\cap \mathbf{R}_{N_{0}}}$, $\mathrm{t\in \mathbb{R}}$ and let $\mathrm{\mathbf{W}\subset \mathbf{Y}_{0}\cap \mathbf{R}_{N_{0}}}$ be a compact neighborhood\,\footnote{From now on, we will always assume that $\mathrm{\mathbf{W}\subset \mathbf{Y}^{i}}$ which is possible without restriction of generality.}of $\mathrm{y_{0}}$. Then after having shrunken $\mathrm{\mathbf{W}}$, the sequence $\mathrm{\left(D_{n}\left(\cdot,t\right)\right)_{n}}$ converges uniformly to the zero function over $\mathrm{\mathbf{W}}$.
\end{theo_5.a*}
\begin{proof} Let $\mathrm{\epsilon>0}$, then by the first part of {\bf{Proposition \ref {Continuity of Fiber Integral}}} there exists $\mathrm{\sigma_{\epsilon}>0}$ so that 
\begin{equation}\label{Equation 6}
\mathrm{vol\left(\pi^{-1}\left(y\right)\cap T\left(\sigma_{\epsilon},\mathbf{W}\right)\right)\leq \epsilon}
\end{equation}
for all $\mathrm{y\in \mathbf{W}\cap \mathbf{Y}_{0}=\mathbf{W}}$. Hence, it is enough to show that
\begin{equation}\label{Equation Distribution Functions Is Small Outside Local Case} 
\mathrm{\phi_{n}=\frac{\vert s_{n}\vert^{2}}{\Vert s_{n}\Vert^{2}}\leq t}
\end{equation} on $\mathrm{T^{c}\left(\sigma_{\epsilon},\mathbf{W}\right)}$ for all $\mathrm{n}$ big enough. Recall that $\mathrm{\vert s_{n}\vert^{2}\Vert s_{n}\Vert^{-2}}$ is an abbreviation for the local description given by $\mathrm{\vert s_{y,n}\vert^{2}\Vert s_{y,n}\Vert^{-2}}$ on $\mathrm{\pi^{-1}\left(\mathbf{U}_{y}\right)}$. 

The first step is to write $$\mathrm{\phi_{n}=\frac{\vert s_{n}^{i}\cdot s_{m}^{i,-1}\vert^{2}\cdot \vert s_{m}^{i}\cdot \triangle^{i}_{y,n}\vert^{2}}{\int_{\pi^{-1}\left(y\right)}\vert s_{n}^{i}\cdot s_{m}^{i,-1}\vert^{2}\cdot \vert s_{m}^{i}\cdot \triangle^{i}_{y,n}\vert^{2}\,d\,[\pi_{y}]}}$$ for arbitrary $\mathrm{m\in \mathbb{N}}$ which is possible because $\mathrm{s^{i}_{m}\left(x\right)\neq 0}$ for all $\mathrm{x\in \mathbf{X}^{i}}$ and all $\mathrm{m\in \mathbb{N}}$. Using {\bf{Lemma \ref{Lemma The Maximum Of Triangle Is Fiberwisely Not Zero}}} we deduce that 
\begin{equation}\label{Equation Intergal In The Denominator Will Not Be Zero}
\mathrm{\underset{x\in\pi^{-1}\left(y\right)\cap T\left(\epsilon_{\frac{\sigma}{2}},\mathbf{W}\right)}{max}\,\left|s_{m}^{i}\cdot\triangle_{y,n}^{i}\right|^{2}>0}
\end{equation} for all $\mathrm{n}$ big enough and all $\mathrm{y\in \mathbf{W}}$ and hence we can estimate
\begin{equation}\label{Equation 2}
\mathrm{\phi_{n}\leq\frac{\vert s_{n}^{i}\cdot s_{m}^{i,-1}\vert^{2}\cdot \vert s_{m}^{i}\cdot \triangle^{i}_{y,n}\vert^{2}}{\int_{\pi^{-1}\left(y\right)\cap T\left(\frac{\sigma_{\epsilon}}{2},\mathbf{W}\right)}\vert s_{n}^{i}\cdot s_{m}^{i,-1}\vert^{2}\cdot \vert s_{m}^{i}\cdot \triangle^{i}_{y,n}\vert^{2}\,d\,[\pi_{y}]}}\end{equation}
for all $\mathrm{n,m\in \mathbb{N}}$ and all $\mathrm{x\in\pi^{-1}\left( \mathbf{W}\right)}$.   

Note that $\mathrm{\varrho^{m}_{n}\coloneqq -\frac{1}{n}\mathbf{log}\,\vert s^{i}_{n}\cdot s^{i,-1}_{m} \vert^{2}=\varrho^{i}_{n}-\frac{1}{n}\mathbf{log}\,\vert s^{i,-1}_{m}\vert^{2}}$ defines for a fixed $\mathrm{m}$ and for all $\mathrm{n\geq m }$ a strictly plurisubharmonic function on $\mathrm{\pi^{-1}\left(\mathbf{W}\right)}$ which converges uniformly on compact subsets to $\mathrm{\varrho^{i}}$. Therefore, for each $\mathrm{m}$, there exists $\mathrm{N_{m}\in \mathbb{N}}$ so that $\mathrm{\varrho^{m}_{n}\left(x\right)\leq \frac{5}{8}\sigma_{\epsilon}}$ for all $\mathrm{x\in T\left(\frac{\sigma_{\epsilon}}{2},\mathbf{W}\right)}$ and for all $\mathrm{n\geq N_{m}}$ or equivalently
\begin{equation}\label{Equation 1}
\mathrm{\vert s^{i}_{n}\cdot s^{i,-1}_{m}\vert^{2}\left(x\right)\geq e^{-n\,\frac{5}{8}\,\sigma_{\epsilon}}} 
\end{equation}
for all $\mathrm{x\in T\left(\frac{\sigma_{\epsilon}}{2},\mathbf{W}\right)}$ and for all $\mathrm{n\geq N_{m}}$. Using inequality \ref{Equation 1} and \ref{Equation 2}, we deduce 
\begin{equation}\label{Equation 3}
\mathrm{\phi\leq e^{n\,\frac{5}{8}\,\sigma_{\epsilon}}\,\frac{\vert s_{n}^{i}\cdot s_{m}^{i,-1}\vert^{2}\cdot \vert s_{m}^{i}\cdot \triangle^{i}_{y,n}\vert^{2}}{\int_{\pi^{-1}\left(y\right)\cap T\left(\frac{\sigma_{\epsilon}}{2},\mathbf{W}\right)}\vert s^{i}_{m}\cdot\triangle_{y,n}\vert^{2}\,d\,[\pi_{y}]}}\end{equation} for all $\mathrm{n\geq N_{m}}$ an all $\mathrm{x\in\pi^{-1}\left( \mathbf{W}\right)}$. 

Note that by {\bf{Corollary \ref{Remark Convergence Theorem For The Shifted Sequence}}} we know that $\mathrm{\varrho^{m}_{n}\left(x\right)\geq \frac{7}{8}\, \sigma_{\epsilon}}$ or equivalently 
\begin{equation}\label{Equation 4}
\mathrm{\vert s^{i}_{n}\cdot s^{i,-1}_{m}\vert^{2}\left(x\right)\leq e^{-\frac{7}{8}\,n\,\sigma_{\epsilon}}} 
\end{equation} for all $\mathrm{x\in T^{c}\left(\sigma_{\epsilon},\mathbf{W}\right)}$ and all $\mathrm{n\geq N_{m}^{\prime}}$. So if we combine \ref{Equation 4} and \ref{Equation 3} we deduce
\begin{equation}\label{Equation 5}
\mathrm{\phi\leq e^{-n\,\frac{1}{4}\,\sigma_{\epsilon}}\,\frac{\vert s_{m}^{i}\cdot \triangle^{i}_{y,n}\vert^{2}}{\int_{\pi^{-1}\left(y\right)\cap T\left(\frac{\sigma_{\epsilon}}{2},\mathbf{W}\right)}\vert s^{i}_{m}\cdot\triangle^{i}_{y,n}\vert^{2}\,d\,[\pi_{y}]}}\end{equation} for all $\mathrm{x\in T^{c}\left(\sigma_{\epsilon},\mathbf{W}\right)}$ and all $\mathrm{n\geq max\,\left\{N^{\phantom{\prime}}_{m},\, N_{m}^{\prime}\right\}}$. After having shrunken $\mathrm{\mathbf{W}}$ there exists (cf.\ {\bf{Proposition \ref{Proposition Tales Are Uniformly Bounded}}}) $\mathrm{C>0}$ and $\mathrm{m_{0}\in \mathbb{N}}$ so that $$\mathrm{\underset{x\in\pi^{-1}\left(y\right)}{max}\,\,\,\frac{\left|s_{m_{0}}^{i}\cdot\triangle_{y,n}^{i}\right|^{2}}{\int_{\pi^{-1}\left(y\right)\cap T\left(\frac{\sigma_{\epsilon}}{2},\mathbf{W}\right)}\left|s_{m_{0}}^{i}\cdot\triangle_{y,n}^{i}\right|^{2}\,d\,[\pi_{y}]}<C}$$ for all $\mathrm{n}$ big enough and all $\mathrm{y\in \mathbf{W}}$. In particular combining this with inequality \ref{Equation 5} we deduce 
\begin{equation}\label{Equation Which Is Necessary}
\mathrm{\phi_{n}\leq e^{-n\,\frac{1}{4}\,\sigma_{\epsilon}}\cdot C}
\end{equation} for all $\mathrm{n}$ big enough (i.e.\ at least $\mathrm{n\geq max\,\left\{N^{\phantom{\prime}}_{m_{0}},\, N_{m_{0}}^{\prime}\right\}}$) and all $$\mathrm{x\in T^{c}\left(\sigma_{\epsilon},\mathbf{W}\right)\cap \pi^{-1}(\mathbf{W})=T^{c}(\sigma_{\epsilon},\mathbf{W}).}$$ In other words we have
\begin{equation}\label{Equation 6}
\mathrm{\phi_{n}\leq t}
\end{equation} on $\mathrm{T^{c}(\sigma_{\epsilon},\mathbf{W})}$ for all $\mathrm{n}$ big enough which proves the claim because of equation \ref{Equation Distribution Functions Is Small Outside Local Case}.
\end{proof}

As a direct consequence of the above theorem we deduce {\bf{Theorem 6.a}}.
\begin{theo_6.a*}\label{theo_6.a*}\textnormal{[\scshape{Locally Uniform Convergence of the Initial Measure Sequence}]}\vspace{0.1 cm}\\ Let $\mathrm{y_{0}\in \mathbf{Y}_{0}\cap \mathbf{R}_{N_{0}}}$, $\mathrm{t\in \mathbb{R}}$ and let $\mathrm{\mathbf{W}\subset \mathbf{Y}_{0}\cap \mathbf{R}_{N_{0}}}$ be a compact neighborhood of $\mathrm{y_{0}}$. Moreover, let $\mathrm{f\in \mathcal{C}^{0}\!\left(\mathbf{X}\right)}$. Then after having shrunken $\mathrm{ \mathbf{W}}$, the sequence $$\mathrm{\left(y\mapsto \int_{\pi^{-1}\left(y\right)}f\,d\bm{\nu}_{n}\left(y\right)\right)_{n}}$$ converges uniformly on $\mathrm{\mathbf{W}}$ to the reduced function $\mathrm{f_{red}}$.
\end{theo_6.a*}
\begin{proof} By the same argumentation as in the proof of {\bf{Proposition \ref{Proposition Uniform Convergence}}}, it is enough to prove the claim for all $\mathrm{f\in\mathcal{C}^{0}\!\left(\mathbf{X}\right)}$ which are $\mathrm{T}$-invariant. Hence, in the sequel let $\mathrm{f}$ be continuous, $\mathrm{T}$-invariant function on $\mathrm{\mathbf{X}}$ and let $\mathrm{\epsilon>0}$. Then by {\bf{Lemma \ref{Lemma Reduced Function}}} there exists $\mathrm{\sigma_{\epsilon}>0}$ so that 
\begin{equation}\label{Equation Integral Of F With Respect To}
\mathrm{\left|f\vert\left(\pi^{-1}\left(y\right)\cap T\left(\sigma_{\epsilon},\mathbf{W}\right)\right)-f_{red}\left(y\right)\right|\leq\frac{\epsilon}{2}}
\end{equation} for all $\mathrm{y\in \mathbf{W}}$. Furthermore, note that we have \begin{equation*}
\begin{split}
&\mathrm{\int_{\pi^{-1}\left(y\right)}\left(f-\pi^{*}f_{red}\right)\,d\bm{\nu}_{n}\left(y\right)}\\[0,2 cm]
=&\mathrm{\int_{\pi^{-1}\left(y\right)\cap T\left(\sigma_{\epsilon},\mathbf{W}\right)}\left(f-\pi^{*}f_{red}\right)\,\phi_{n}\,d\,[\pi_{y}]}\\[0,2 cm]
+&\mathrm{\int_{\pi^{-1}\left(y\right)\cap T^{c}\left(\sigma_{\epsilon},\mathbf{W}\right)}\left(f-\pi^{*}f_{red}\right)\,\phi_{n}\,d\,[\pi_{y}].}\\
\end{split}
\end{equation*} Since $\mathrm{f}$ is bounded as continuous function on a compact space, we find $\mathrm{\Gamma\coloneqq max\, \{ f-}$ $\mathrm{\pi^{*}f_{red}\}<\infty}$. Moreover, we can assume that $\mathrm{f-\pi^{*}f_{red}\not\equiv 0}$ and therefore $\mathrm{\Gamma^{-1}<\infty}$ is defined. Furthermore, using {\bf{Corollary \ref{Boundedness of Fiber Integral}}}, there exists a constant $\mathrm{C>0}$ so that $$\mathrm{\int_{\pi^{-1}\left(y\right)\cap T^{c}\left(\sigma_{\epsilon},\mathbf{W}\right)}\,d\,[\pi_{y}]\leq C}$$ for all $\mathrm{y\in \mathbf{W}\subset \mathbf{Y}_{0}\cap \mathbf{R}_{N_{0}}}$. According to the proof of {\bf{Theorem 5.a}}, we know by \ref{Equation Which Is Necessary} that, after having replaced $\mathrm{\mathbf{W}}$ by a smaller compact neighborhood, $$\mathrm{\phi_{n}\leq C^{-1}\cdot\Gamma^{-1}\cdot\frac{\epsilon}{4}\text{ on }T^{c}(\sigma_{\epsilon},\mathbf{W})\text{ for all }n\text{ big enough.}}$$ Hence, it follows that 
 \begin{equation*}
\begin{split}
&\mathrm{\left\vert\int_{\pi^{-1}\left(y\right)}\left(f-\pi^{*}f_{red}\right)\,d\bm{\nu}_{n}\left(y\right)\right\vert\leq \int_{\pi^{-1}\left(y\right)}\vert f-\pi^{*}f_{red}\vert\,d\bm{\nu}_{n}\left(y\right)}\\[0,2 cm]
\leq&\mathrm{\int_{\pi^{-1}\left(y\right)\cap T\left(\sigma_{\epsilon},\mathbf{W}\right)}\left\vert f-\pi^{*}f_{red}\right\vert\,\phi_{n}\,d\,[\pi_{y}]+\frac{\epsilon}{4}}\\
\end{split}
\end{equation*} for all $\mathrm{n}$ big enough and all $\mathrm{y\in \mathbf{W}}$. Using inequality \ref{Equation Integral Of F With Respect To}, we deduce $$\mathrm{\left|\int_{\pi^{-1}\left(y\right)}f\, d\bm{\nu}_{n}\left(y\right)-f_{red}\left(y\right)\right|\leq\frac{\epsilon}{2}\int_{\pi^{-1}\left(y\right)\cap T\left(\sigma_{\epsilon},\mathbf{W}\right)}\phi\,d\,[\pi_{y}]+\frac{\epsilon}{4}}$$ for all $\mathrm{y\in \mathbf{W}}$ and all $\mathrm{n}$ big enough. Since $\mathrm{\int_{\pi^{-1}\left(y\right)\cap T\left(\sigma_{\epsilon},\mathbf{W}\right)}\phi_{n}\leq 1}$ for all $\mathrm{n}$ we find $$\mathrm{\left|\int_{\pi^{-1}\left(y\right)}f\, d\bm{\nu}_{n}\left(y\right)-f_{red}\left(y\right)\right|\leq\frac{3}{4}\,\epsilon}$$ for all $\mathrm{y\in\mathbf{W}}$ and all $\mathrm{n}$ big enough and therefore the theorem is proven.
\end{proof}

\subsection{Global Uniform Convergence on $\mathrm{\mathbf{R}_{N_{0}}\cap\mathbf{Y}_{0}}$}

The strategy of the proof of the globally uniform convergence theorems on $\mathrm{\mathbf{R}_{N_{0}}\cap\mathbf{Y}_{0}}$ runs along the following lines: If $\mathrm{y_{0}\in cl\,\left(\mathbf{Y}_{0}\cap \mathbf{R}_{N_{0}}\right)}$\footnote{Recall that we have assumed that $\mathrm{\mathbf{R}_{n_{0}}\neq \varnothing}$; in particular $\mathrm{\mathbf{Y}_{0}\cap \mathbf{R}_{N_{0}}}$ is non-empty and euclidean open.}, then we find an open neighborhood $\mathrm{\mathbf{W}\subset \mathbf{Y}}$ of $\mathrm{y_{0}}$ so that we have uniform convergence over $\mathrm{\mathbf{W}\cap \mathbf{Y}_{0}\cap \mathbf{R}_{N_{0}}}$. Since $\mathrm{cl\left(\mathbf{Y}_{0}\cap \mathbf{R}_{N_{0}}\right)}$ is compact, we can cover $\mathrm{\mathbf{Y}_{0}\cap \mathbf{R}_{N_{0}}}$ by finitely many of such compact neighborhoods and the claim follows.

We begin this section by proving an extended version {\bf{Proposition \ref{Proposition Tales Are Uniformly Bounded}}}.

\begin{prop}\label{Proposition Tales Are Uniformly Bounded On Y} Let $\mathrm{\mathbf{W}\subset\mathbf{Y}}$ be a compact neighborhood of $\mathrm{y_{0}\in cl\,\left(\mathbf{Y}_{0}\cap \mathbf{R}_{N_{0}}\right)}$ (recall: we have assumed that $\mathrm{\mathbf{W}\subset \mathbf{Y}^{i}}$) and $\mathrm{\epsilon>0}$, then after having shrunken $\mathrm{\mathbf{W}}$, there exists a constant $\mathrm{C>0}$ and $\mathrm{m_{0}\in\mathbb{N}}$ so that $$\mathrm{\underset{x\in\pi^{-1}\left(y\right)}{max}\,\,\,\frac{\left|s_{m_{0}}^{i}\cdot\triangle_{y,n}^{i}\right|^{2}}{\int_{\pi^{-1}\left(y\right)\cap T\left(\epsilon,\mathbf{W}\right)}\left|s_{m_{0}}^{i}\cdot\triangle_{y,n}^{i}\right|^{2}\,d\,[\pi_{y}]}<C}$$ for all $\mathrm{n}$ big enough and all $\mathrm{y\in \mathbf{W}\cap \mathbf{Y}_{0}\cap \mathbf{R}_{N_{0}}}$.
\end{prop} 
\begin{proof} Let us assume that this is not the case, then as in the proof of {\bf{Proposition \ref{Proposition Tales Are Uniformly Bounded}}}, we can find a sequence $\mathrm{\left(y_{n}\right)_{n}}$ in $\mathrm{y_{n}\in \mathbf{Y}_{0}\cap {\mathbf{R}}_{N_{0}}}$ converging to $\mathrm{y_{0}\in cl\left(\mathbf{Y}_{0}\cap {\bf{R}}_{N_{0}}\right)}$ so that \begin{equation}\label{Equation Integral Converges To Zero Global Case}
\mathrm{\int_{\pi^{-1}\left(y_{n}\right)\cap T\left(\epsilon,\mathbf{W}\right)}\frac{\left|s_{m_{0}}^{i}\cdot\triangle_{y_{n},n}^{i}\right|^{2}}{\underset{x\in\pi^{-1}\left(y_{n}\right)\cap T\left(\epsilon,\mathbf{W}\right)}{max}\,\left|s_{m_{0}}^{i}\cdot\triangle_{y_{n},n}^{i}\right|^{2}}\,d\,[\pi_{y_{n}}]\rightarrow 0.}
\end{equation} Using the same notation as in the proof of {\bf{Proposition \ref{Proposition Tales Are Uniformly Bounded}}}, we can write \begin{equation}\label{Equation Integral Converges To Zero Global Case}
\mathrm{\int_{\pi^{-1}\left(y_{n}\right)\cap T\left(\epsilon,\mathbf{W}\right)}\vert f_{n}\vert^{2}\,d\,[\pi_{y_{n}}]\rightarrow 0.}
\end{equation} Furthermore, as as in the proof of {\bf{Proposition \ref{Proposition Tales Are Uniformly Bounded}}}, we can assume that $$\mathrm{\underset{x\in\pi^{-1}\left(y_{n}\right)}{max}\,\left|f_{n}\right|^{2}=\underset{x\in\pi^{-1}\left(y_{n}\right)\cap T\left(\frac{\epsilon}{2},\mathbf{W}\right)}{max}\,\left|f_{n}\right|^{2}=1}$$ for all $\mathrm{n}$ big enough where $\mathrm{y_{n}\rightarrow y_{0}\in Int\,\mathbf{W}}$. Moreover, as before we find a lift $\mathrm{\left(x_{n}\right)_{n}}$ of the sequence $\mathrm{\left(y_{n}\right)_{n}}$, i.e.\ a sequence $\mathrm{\left(x_{n}\right)_{n}}$ in $\mathrm{T\left(\frac{\epsilon}{2},\mathbf{W}\right)\cap \pi^{-1}\left(\mathbf{Y}_{0}\cap \mathbf{R}_{N_{0}}\right)}$ so that $\mathrm{\pi\left(x_{n}\right)=y_{n}}$ where \begin{equation}\label{Equation Norm Is Equal To One}
\mathrm{\vert f_{n}\vert^{2}\left(x_{n}\right)=1=\underset{x\in\pi^{-1}\left(y_{n}\right)}{max}\,\left|f_{n}\right|^{2}=\underset{x\in\pi^{-1}\left(y_{n}\right)\cap T\left(\epsilon,\mathbf{W}\right)}{max}\,\left|f_{n}\right|^{2}}.
\end{equation} 

Recall that we have the following commutative diagram: 

$$\begin{xy}
\xymatrix{
\mathrm{\mathbf{X}^{ss}_{\xi}\subset \mathbf{X}}&&\mathrm{\widehat{\mathbf{\,X\,}}}\ar[d]_{\mathrm{\widehat{\pi}}}\ar[ll]^{\mathrm{p_{\mathbf{X}}\vert cl{\mathbf{\Gamma}}_{\pi}\circ \zeta}} &&& \mathrm{\widetilde{\bf{\,X\,}}}\ar[d]^{\mathrm{\Pi}}\ar[lll]_{\mathrm{\Sigma}}\\
&&\mathrm{\mathbf{Y}}&&& \mathrm{\widetilde{\bf{\,Y\,}}}\ar[lll]_{\mathrm{\sigma}}\\
}
\end{xy}
$$ Set $\mathrm{P\coloneqq p_{\mathbf{X}}\vert cl {\mathbf{\Gamma}}_{\pi}\circ\zeta\circ\Sigma}$ and define $\mathrm{\widetilde{\mathbf{L}}\coloneqq P^{*}\mathbf{L}}$. Since $\mathrm{\Sigma}$ is surjective and the above diagram commutes, we can find a lift $\mathrm{\left(\widetilde{x}_{n}\right)_{n}}$ in $\mathrm{\widetilde{\bf{\,X\,}}}$ of $\mathrm{\left(x_{n}\right)_{n}}$ under $\mathrm{P}$, i.e.\ we have $\mathrm{P\left(\widetilde{x}_{n}\right)=x_{n}}$. In the sequel, let $\mathrm{\left(\widetilde{y}_{n}\right)_{n}}$ be the sequence in $\mathrm{\widetilde{\bf{\,Y\,}}}$ given by $\mathrm{\widetilde{y}_{n}=\Pi\left(\widetilde{x}_{n}\right)}$ and note that $\mathrm{\sigma\left(\widetilde{y}_{n}\right)=y_{n}}$. Moreover, we have $\mathrm{\widetilde{x}_{n}\in \widetilde{\,T\,}\big(\frac{\epsilon}{2},\bf{W}\big) \subset \widetilde{\,T\,}\big(\epsilon,\bf{W}\big)}$ (recall:\ $\mathrm{\widetilde{\bf{W}}\coloneqq \sigma^{-1}\left(\mathbf{W}\right)}$) for all $\mathrm{n}$. Arguing as before, since $\mathrm{\widetilde{x}_{n}\in \widetilde{\,T\,}\big(\frac{\epsilon}{2},\bf{W}\big)}$ for all $\mathrm{n}$, we can assume that (after having chosen a subsequence) $$\mathrm{\widetilde{x}_{n}\rightarrow\widetilde{x}_{0}\in Int\,\widetilde{\,T\,}\big(\epsilon,\bf{W}\big).}$$ Hence, there exists an open neighborhood $\mathrm{\mathbf{U} \subset \widetilde{\,T\,}\big(\epsilon,\bf{W}\big)}$ of $\mathrm{\widetilde{x}_{0}}$ so that $\mathrm{\widetilde{x}_{n}\in \mathbf{U}}$ for all $\mathrm{n}$ big enough. In the sequel, let $\mathrm{\widetilde{y}_{0}\coloneqq \Pi\left(\widetilde{x}_{0}\right)\in \widetilde{\bf{\,Y\,}}}$. 

We will now define $\mathrm{\widetilde{f}_{n}}$ on $\mathrm{\widetilde{\,T\,}\big(\epsilon,\bf{W}\big)\cap}$ $\mathrm{\Pi^{-1}\left(\widetilde{y}_{n}\right)}$ by the pull back of the restriction of $\mathrm{f_{n}}$ to $\mathrm{T\left(\epsilon,\mathbf{W}\right)\cap \pi^{-1}\left(y_{n}\right)}$. Note that we have $\mathrm{\vert \widetilde{f}_{n}?\vert^{2}\left(\widetilde{x}_{n}\right)=1}$ for all $\mathrm{n}$ and more precisely
\begin{equation}\label{Equation 12}
\mathrm{\underset{x\in\widetilde{\,T\,}(\epsilon,{\bf{W}})\cap\Pi^{-1}\left(\widetilde{y}_{n}\right)}{max}\vert \widetilde{f}_{n}\vert^{2}=\vert \widetilde{f}_{n}\vert^{2}\left(x_{n}\right)=1}
\end{equation} which is a direct consequence of \ref{Equation Norm Is Equal To One}.

Since $\mathrm{\Pi\!:\!\widetilde{\bf{\,X\,}}\rightarrow \widetilde{\bf{\,Y\,}}}$ is a $\mathrm{k}$-fibering, we can proceed as in the proof of {\bf{Proposition \ref{Proposition Tales Are Uniformly Bounded}}}: After having shrunken $\mathrm{\mathbf{U}}$ we can assume that the open neighborhood $\mathrm{\mathbf{U}}$ is isomorphic to a closed analytic subset in a relatively compact product neighborhood $\mathrm{\mathbf{Q}=\mathbf{Q}_{0}\times \mathbf{Q}_{1}\subset \mathbb{C}^{\upkappa}\times}$ $\mathrm{\mathbb{C}^{k}}$, i.e.\ there exists an isomorphism $\mathrm{\Phi\!:\!\mathbf{U}\rightarrow \mathbb{C}^{\upkappa}\times \mathbb{C}^{k}}$ (let $\mathrm{\Phi\left(\widetilde{x}_{0}\right)=0}$) so that for each $\mathrm{\widetilde{y}}$ in an open neighborhood $\mathrm{\mathbf{B}\subset \widetilde{\bf{\,Y\,}}}$ of $\mathrm{\Pi\left(\widetilde{x}_{0}\right)=\widetilde{y}_{0}}$ each $\mathrm{\Phi\left(\Pi^{-1}(\widetilde{y}\right)}$ $\mathrm{\cap \mathbf{U})=\mathbf{Z}_{\widetilde{y}}}$ is a closed analytic subset which yields a $\mathrm{d}$-sheeted covering onto $\mathrm{\mathbf{Q}_{1}}$ given by the restriction $\mathrm{p\vert \mathbf{Z}_{\widetilde{y}_{n}}\rightarrow \mathbf{Q}_{1}}$. We can now finish the proof exactly like the proof of {\bf{Proposition \ref{Proposition Tales Are Uniformly Bounded}}}. First of all after having shrunken $\mathrm{\mathbf{U}}$ we can assume that $\mathrm{\widetilde{f}_{n}}$ is represented by $\mathrm{\widetilde{f}_{n}=\widetilde{\alpha} \cdot \widetilde{g}_{n}}$ where $\mathrm{\widetilde{\alpha}\in \mathcal{C}^{\infty}\left(\mathbf{U}\right)}$ and $\mathrm{\widetilde{g}_{n}\in \mathcal{O}\left(\mathbf{Z}_{y_{n}}\right)}$. Furthermore, we also have  
\begin{equation}\label{Equation Bla bla bla bli blub asf}
\mathrm{a^{-1}\geq \underset{x\in \mathbf{Z}_{\widetilde{y}_{n}}}{max}\vert \widetilde{g}_{n}\vert \mathbf{Z}_{\widetilde{y}_{n}}\vert^{2}\text{ and }\vert \widetilde{g}_{n}\vert^{2}\left(\widetilde{x}_{n}\right)\geq A^{-1}>0}
\end{equation} for all $\mathrm{n}$ big enough, where $\mathrm{A>0}$, resp.\ $\mathrm{a>0}$ denotes the maximum, resp.\ the minimum of $\mathrm{\widetilde{\alpha}}$ on $\mathrm{\mathbf{U}}$. We can now apply {\bf{Proposition \ref{Proposition Local Version Of Proposition}}} to the sequence $\mathrm{\widetilde{g}_{n}}$: Each $\mathrm{\widetilde{g}_{n}}$ is a holomorphic function on $\mathrm{\widetilde{\mathbf{Z}}_{n}\coloneqq\mathbf{Z}_{\widetilde{y}_{n}}}$. Furthermore, this sequence is uniformly bounded by $\mathrm{a^{-1}}$ and by assumption \ref{Equation Integral Converges To Zero Global Case}, we have $$\mathrm{\int_{\widetilde{\mathbf{Z}}_{n}}\vert \widetilde{g}_{n}\vert^{2}\,d\,[\widetilde{\mathbf{Z}}_{n}]\rightarrow 0.}$$ Using {\bf{Proposition \ref{Proposition Local Version Of Proposition}}}, it follows that for a compact subset $\mathbf{K}\subset \mathbf{Q}_{1}$ and for $\mathrm{\epsilon=\frac{1}{2}A^{-1}}$ we have 
\begin{equation}\label{Equation wlisac?kns}
\mathrm{\vert \widetilde{g}_{n}\vert^{2}\leq \frac{1}{2}A^{-1}\text{ on }p^{-1}\left(\mathbf{K}\right)\cap \widetilde{\mathbf{Z}}_{n}\text{ for all }n\geq N_{\frac{1}{2}A^{-1}}\left(\mathbf{K}\right).}
\end{equation} However, as $\mathrm{\widetilde{x}_{n}\rightarrow x_{0}=0\in p^{-1}\left(\mathbf{K}\right)}$ we know that $\mathrm{\widetilde{x}_{n}\in p^{-1}\left(\mathbf{K}\right)\cap \widetilde{\mathbf{Z}}_{n}}$ for all $\mathrm{n}$ big enough. By the second inequality of \ref{Equation Bla bla bla bli blub asf}, we know that $\mathrm{\vert g_{n}\vert^{2}}$ $\mathrm{\left(x_{n} \right)\geq A^{-1}}$ which yields a contradiction to \ref{Equation wlisac?kns}. Consequently, the assumption \ref{Equation Integral Converges To Zero Global Case} is false and the claim of the proposition holds.
\end{proof} 

As a consequence of {\bf{Proposition \ref{Proposition Tales Are Uniformly Bounded On Y}}} we deduce a generalization of {\bf{Theorem 5.a}}.
\begin{theo_5.b*}\label{theo_5.b*}\textnormal{[\scshape{Uniform Convergence of the Initial Distribution Sequence}]}\vspace{0.1 cm}\\
For fixed $\mathrm{t\in \mathbb{R}}$ the sequence $\mathrm{\left(D_{n}\left(\cdot,t\right)\right)_{n}}$ converges uniformly on $\mathrm{\mathbf{Y}_{0}\cap\mathbf{R}_{N_{0}}}$ to the zero function.
\end{theo_5.b*}
\begin{proof} Note that by the compactness of $\mathrm{cl\left(\mathbf{Y}_{0}\cap\mathbf{R}_{N_{0}}\right)}$ the claim follows as soon as we have shown that for each $\mathrm{y_{0}\in cl\left(\mathbf{Y}_{0}\cap\mathbf{R}_{N_{0}}\right)}$, there exists a compact neighborhood $\mathrm{\mathbf{W}\subset \mathbf{Y}}$ of $\mathrm{y_{0}}$ (after having shrunken $\mathrm{\mathbf{W}}$ we can assume that $\mathrm{\mathbf{W}\subset \mathbf{Y}^{i}}$) so that $\mathrm{D_{n}\left(\cdot,t\right)}$ converges uniformly on $\mathrm{\mathbf{W}\cap \mathbf{Y}_{0}\cap\mathbf{R}_{N_{0}}}$ to the zero function. The proof of the latter claim is similar to the proof of {\bf{Theorem 5.a}}: First of all fix $\mathrm{y_{0}\in cl\left(\mathbf{Y}_{0}\cap\mathbf{R}_{N_{0}}\right)}$ and $\mathrm{\epsilon>0}$. As in the proof of {\bf{Theorem 5.a}}, we apply {\bf{Proposition \ref {Continuity of Fiber Integral}}} in order to find $\mathrm{\sigma_{\epsilon}>0}$ so that 
\begin{equation}\label{Equation 6}
\mathrm{vol\left(\pi^{-1}\left(y\right)\cap T\left(\sigma_{\epsilon},\mathbf{W}\right)\right)\leq \epsilon}
\end{equation}
for all $\mathrm{y\in \mathbf{W}\cap \mathbf{Y}_{0}}$. Hence, it is enough to show that after having shrunken $\mathrm{\mathbf{W}}$ we have
\begin{equation}\label{Equation Distribution Functions Is Small Outside Global Case} 
\mathrm{\phi_{n}=\frac{\vert s_{n}\vert^{2}}{\Vert s_{n}\Vert^{2}}\leq t}
\end{equation} on $\mathrm{T^{c}\big(\sigma_{\epsilon},\mathbf{W}\big)\cap \pi^{-1}\left(\mathbf{Y}_{0}\cap \mathbf{R}_{N_{0}}\right)}$ for all $\mathrm{n}$ big enough. Note that we can now perform the same steps as in the proof {\bf{Theorem 5.a}} over the set $\mathrm{\mathbf{W}\cap \mathbf{Y}_{0}\cap \mathbf{R}_{N_{0}}}$.

First of all, we write $$\mathrm{\phi_{n}=\frac{\vert s_{n}^{i}\cdot s_{m}^{i,-1}\vert^{2}\cdot \vert s_{m}^{i}\cdot \triangle^{i}_{y,n}\vert^{2}}{\int_{\pi^{-1}\left(y\right)}\vert s_{n}^{i}\cdot s_{m}^{i,-1}\vert^{2}\cdot \vert s_{m}^{i}\cdot \triangle^{i}_{y,n}\vert^{2}\,d\,[\pi_{y}]}}$$ where $\mathrm{m\in \mathbb{N}}$ is arbitrary and by  {\bf{Lemma \ref{Lemma The Maximum Of Triangle Is Fiberwisely Not Zero}}}, we know that
\begin{equation}\label{Equation Intergal In The Denominator Will Not Be Zero Global Case}
\mathrm{\underset{x\in\pi^{-1}\left(y\right)\cap T\left({\frac{\sigma_{\epsilon}}{2}},\mathbf{W}\right)}{max}\,\left|s_{m}^{i}\cdot\triangle_{y,n}^{i}\right|^{2}>0}
\end{equation} for all $\mathrm{n}$ big enough and all $\mathrm{y\in \mathbf{W}\cap \mathbf{R}_{N_{0}}}$. Therefore we deduce the estimate \begin{equation}\label{Equation 2 Global Case}
\mathrm{\phi_{n}\leq\frac{\vert s_{n}^{i}\cdot s_{m}^{i,-1}\vert^{2}\cdot \vert s_{m}^{i}\cdot \triangle^{i}_{y,n}\vert^{2}}{\int_{\pi^{-1}\left(y\right)\cap T\left(\frac{\sigma_{\epsilon}}{2},\mathbf{W}\right)}\vert s_{n}^{i}\cdot s_{m}^{i,-1}\vert^{2}\cdot \vert s_{m}^{i}\cdot \triangle^{i}_{y,n}\vert^{2}\,d\,[\pi_{y}]}}\end{equation}
for all $\mathrm{n,m\in \mathbb{N}}$ and all $\mathrm{x\in\pi^{-1}\left( \mathbf{W\cap \mathbf{R}_{N_{0}}}\cap \mathbf{Y}_{0}\right)}$. Note that the right hand side is well defined because the denominator of the term on the right hand side of the above inequality is not zero according to equation \ref{Equation Intergal In The Denominator Will Not Be Zero Global Case}.   

As before we know that $\mathrm{\varrho^{m}_{n}\coloneqq -\frac{1}{n}\mathbf{log}\,\vert s^{i}_{n}\cdot s^{i,-1}_{m} \vert^{2}=\varrho^{i}_{n}-\frac{1}{n}\mathbf{log}\,\vert s^{i,-1}_{m}\vert^{2}}$ is a strictly plurisubharmonic function on $\mathrm{\pi^{-1}\left(\mathbf{W}\right)}$ for $\mathrm{m}$ fixed and for all $\mathrm{n\geq m }$ which converges uniformly on compact subsets to $\mathrm{\varrho^{i}}$. In particular we deduce that $\mathrm{\varrho^{m}_{n}\left(x\right)\leq \frac{5}{8}\sigma_{\epsilon}}$ for all $\mathrm{x\in T\left(\frac{\sigma_{\epsilon}}{2},\mathbf{W}\right)}$ and all $\mathrm{n\geq N_{m}\in \mathbb{N}}$ big enough where $\mathrm{m}$ is fixed. Combined this with inequality \ref{Equation 2 Global Case}, it follows that \begin{equation}\label{Equation 3 Global Case}
\mathrm{\phi_{n}\leq e^{n\,\frac{5}{8}\,\sigma_{\epsilon}}\,\frac{\vert s_{n}^{i}\cdot s_{m}^{i,-1}\vert^{2}\cdot \vert s_{m}^{i}\cdot \triangle^{i}_{y,n}\vert^{2}}{\int_{\pi^{-1}\left(y\right)\cap T\left(\frac{\sigma_{\epsilon}}{2},\mathbf{W}\right)}\vert s^{i}_{m}\cdot\triangle_{y,n}\vert^{2}\,d\,[\pi_{y}]}}\end{equation} on $\mathrm{\pi^{-1}\left(\mathbf{R}_{N_{0}}\cap \mathbf{Y}_{0}\cap \mathbf{W}\right)}$ for all $\mathrm{n\geq N_{m}}$. 

On the other hand, using {\bf{Corollary \ref{Remark Convergence Theorem For The Shifted Sequence}}}, we have $\mathrm{\varrho^{m}_{n}\left(x\right)\geq \frac{7}{8}\, \sigma_{\epsilon}}$ or equivalently 
\begin{equation}\label{Equation 4 Global Case}
\mathrm{\vert s^{i}_{n}\cdot s^{i,-1}_{m}\vert^{2}\left(x\right)\leq e^{-\frac{7}{8}\,n\,\sigma_{\epsilon}}} 
\end{equation} for all $\mathrm{x\in T^{c}\left(\sigma_{\epsilon},\mathbf{W}\right)}$ and all $\mathrm{n\geq N_{m}^{\prime}}$. Combining \ref{Equation 4 Global Case} and \ref{Equation 3 Global Case} we deduce
\begin{equation}\label{Equation 5 Global Case}
\mathrm{\phi_{n}\leq e^{-n\,\frac{1}{4}\,\sigma_{\epsilon}}\,\frac{\vert s_{m}^{i}\cdot \triangle^{i}_{y,n}\vert^{2}}{\int_{\pi^{-1}\left(y\right)\cap T\left(\frac{\sigma_{\epsilon}}{2},\mathbf{W}\right)}\vert s^{i}_{m}\cdot\triangle^{i}_{y,n}\vert^{2}\,d\,[\pi_{y}]}}\end{equation} for all $\mathrm{x\in T^{c}\left(\sigma_{\epsilon},\mathbf{W}\right)\cap \pi^{-1}\left(\mathbf{Y}_{0}\cap \mathbf{R}_{N_{0}}\right)}$ and all $\mathrm{n\geq max\,\left\{N_{m},\, N_{m}^{\prime}\right\}}$. Since $\mathrm{y_{0}\in cl\,(\mathbf{Y}_{0}\cap }$ $\mathrm{\mathbf{R}_{N_{0}})}$, we can apply {\bf{Proposition \ref{Proposition Tales Are Uniformly Bounded On Y}}} in order to deduce, after having shrunken $\mathrm{\mathbf{W}}$, the existence of $\mathrm{C>0}$ and $\mathrm{m_{0}\in \mathbb{N}}$ so that $$\mathrm{\underset{x\in\pi^{-1}\left(y\right)}{max}\,\,\,\frac{\left|s_{m_{0}}^{i}\cdot\triangle_{y,n}^{i}\right|^{2}}{\int_{\pi^{-1}\left(y\right)\cap T\left(\frac{\sigma_{\epsilon}}{2},\mathbf{W}\right)}\left|s_{m_{0}}^{i}\cdot\triangle_{y,n}^{i}\right|^{2}\,d\,[\pi_{y}]}<C}$$ for all $\mathrm{n}$ big enough and all $\mathrm{y\in \mathbf{W}\cap \mathbf{Y}_{0}\cap \mathbf{R}_{N_{0}}}$. In combination with inequality \ref{Equation 5 Global Case}, it then follows 
\begin{equation}\label{Equation Which Is Necessary Global Case}
\mathrm{\phi_{n}\leq e^{-n\,\frac{1}{4}\,\sigma_{\epsilon}}\cdot C}
\end{equation} for all $\mathrm{n}$ big enough (i.e.\ at least $\mathrm{n\geq max\,\left\{N_{m},\, N_{m}^{\prime}\right\}}$) and all $$\mathrm{x\in T^{c}\left(\sigma_{\epsilon},\mathbf{W}\right)\cap \pi^{-1}(\mathbf{W}\cap \mathbf{Y}_{0}\cap \mathbf{R}_{N_{0}})=T^{c}(\sigma_{\epsilon},\mathbf{W})\cap \pi^{-1}\left(\mathbf{Y}_{0}\cap\mathbf{R}_{N_{0}}\right).}$$ To sum up we have found that \begin{equation}\label{Equation 6}
\mathrm{\phi_{n}\leq t}
\end{equation} on $\mathrm{T^{c}(\sigma_{\epsilon},\mathbf{W})\cap \pi^{-1}\left(\mathbf{Y}_{0}\cap\mathbf{R}_{N_{0}}\right)}$ for all $\mathrm{n}$ big enough which proves the claim because of equation \ref{Equation Distribution Functions Is Small Outside Global Case}.
\end{proof}

As a direct consequence of the proof of {\bf{Theorem 5.b}} we can prove {\bf{Theorem 6.b}}.

\begin{theo_6.b*}\label{theo_6.b*}\textnormal{[\scshape{Uniform Convergence of the Initial Measure Sequence}]}\vspace{0.10 cm}\\ 
Let $\mathrm{f\in \mathcal{C}^{0}\!\left(\mathbf{X}\right)}$ then sequence $$\mathrm{\left(y\mapsto \int_{\pi^{-1}\left(y\right)}f\,d\bm{\nu}_{n}\left(y\right)\right)_{n}}$$ converges uniformly over $\mathrm{\mathbf{Y}_{0}\cap \mathbf{R}_{N_{0}}}$ to the reduced function $\mathrm{f_{red}}$.
\end{theo_6.b*}
\begin{proof} As before (cf.\ proof of {\bf{Theorem 6.a}}) it is enough to consider the case where $\mathrm{f\in\mathcal{C}^{0}\!\left(\mathbf{X}\right)}$ is $\mathrm{T}$-invariant. 

Let $\mathrm{\epsilon>0}$. Again, by the compactness of $\mathrm{cl\left(\mathbf{Y}_{0}\cap \mathbf{R}_{N_{0}}\right)}$, the proof of the claim reduces to the following statement: If $\mathrm{y_{0}\in cl\left(\mathbf{Y}_{0}\cap \mathbf{R}_{N_{0}}\right)}$, then there exists a compact neighborhood $\mathrm{\mathbf{W}}$ so that the claim is true over the set $\mathrm{\mathbf{W}\cap \mathbf{Y}_{0}\cap \mathbf{R}_{N_{0}}}$. In order to show this, we proceed similar as in the proof of {\bf{Theorem 6.a}}: First of all, by {\bf{Lemma \ref{Lemma Reduced Function}}} there exists $\mathrm{\sigma_{\epsilon}>0}$ so that 
\begin{equation}\label{Equation Integral Of F With Respect To Global Version}
\mathrm{\left|f\vert\left(\pi^{-1}\left(y\right)\cap T\left(\sigma_{\epsilon},\mathbf{W}\right)\right)-f_{red}\left(y\right)\right|\leq\frac{\epsilon}{2}}
\end{equation} for all $\mathrm{y\in \mathbf{W}}$. Furthermore, as before we have the decomposition \begin{equation*}
\begin{split}
&\mathrm{\int_{\pi^{-1}\left(y\right)}\left(f-\pi^{*}f_{red}\right)\,d\bm{\nu}_{n}\left(y\right)}\\[0,2 cm]
=&\mathrm{\int_{\pi^{-1}\left(y\right)\cap T\left(\sigma_{\epsilon},\mathbf{W}\right)}\left(f-\pi^{*}f_{red}\right)\,\phi_{n}\,d\,[\pi_{y}]}\\[0,2 cm]
+&\mathrm{\int_{\pi^{-1}\left(y\right)\cap T^{c}\left(\sigma_{\epsilon},\mathbf{W}\right)}\left(f-\pi^{*}f_{red}\right)\,\phi_{n}\,d\,[\pi_{y}].}\\
\end{split}
\end{equation*} Again we have $\mathrm{\Gamma\coloneqq max\,\left\{ f-\pi^{*}f_{red}\right\}<\infty}$ because $\mathrm{f}$ is continuous and without restriction of generality we can assume that $\mathrm{f-\pi^{*}f_{red}\not\equiv 0}$ so that $\mathrm{\Gamma^{-1}<\infty}$ exists. Moreover, as in the proof of {\bf{Theorem 5.b}}, using {\bf{Corollary \ref{Boundedness of Fiber Integral}}}, there exists $\mathrm{C>0}$ so that $$\mathrm{\int_{\pi^{-1}\left(y\right)\cap T^{c}\left(\sigma_{\epsilon},\mathbf{W}\right)}\,d\,[\pi_{y}]\leq C}$$ for all $\mathrm{y\in \mathbf{W}\cap \mathbf{Y}_{0}}$. By inequality \ref{Equation Which Is Necessary Global Case} in the proof of {\bf{Theorem 5.b}} we know, after having shrunken $\mathrm{\mathbf{W}}$, that $\mathrm{\phi_{n}\leq C^{-1}\cdot\Gamma^{-1}\cdot\frac{\epsilon}{4}\text{ on } T^{c}\big(\sigma_{\epsilon},\mathbf{W}\big)\cap\pi^{-1}(\mathbf{Y}_{0}\cap\mathbf{R}_{N_{0}})}$ for all $\mathrm{n\in \mathbb{N}}$ big enough and hence
 \begin{equation*}
\begin{split}
&\mathrm{\left\vert\int_{\pi^{-1}\left(y\right)}\left(f-\pi^{*}f_{red}\right)\,d\bm{\nu}_{n}\left(y\right)\right\vert\leq \int_{\pi^{-1}\left(y\right)}\vert f-\pi^{*}f_{red}\vert\,d\bm{\nu}_{n}\left(y\right)}\\[0,2 cm]
\leq&\mathrm{\int_{\pi^{-1}\left(y\right)\cap T\left(\sigma_{\epsilon},\mathbf{W}\right)}\left\vert f-\pi^{*}f_{red}\right\vert\,\phi_{n}\,d\,[\pi_{y}]+\frac{\epsilon}{4}}\\
\end{split}
\end{equation*} for all $\mathrm{n}$ big enough and all $\mathrm{y\in \mathbf{W}\cap \mathbf{Y}_{0}\cap \mathbf{R}_{N_{0}}}$. By \ref{Equation Integral Of F With Respect To Global Version} we have $$\mathrm{\left|\int_{\pi^{-1}\left(y\right)}f\, d\bm{\nu}_{n}\left(y\right)-f_{red}\left(y\right)\right|\leq\frac{\epsilon}{2}\int_{\pi^{-1}\left(y\right)\cap T\left(\sigma_{\epsilon},\mathbf{W}\right)}\phi\,d\,[\pi_{y}]+\frac{\epsilon}{4}}$$ for all $\mathrm{y\in \mathbf{W}\cap \mathbf{Y}_{0}\cap \mathbf{R}_{N_{0}}}$ and all $\mathrm{n}$ big enough. As $\mathrm{\int_{\pi^{-1}\left(y\right)\cap T\left(\sigma_{\epsilon},\mathbf{W}\right)}\phi_{n}\leq 1}$ for all $\mathrm{n}$ we deduce $$\mathrm{\left|\int_{\pi^{-1}\left(y\right)}f\, d\bm{\nu}_{n}\left(y\right)-f_{red}\left(y\right)\right|\leq\frac{3}{4}\,\epsilon}$$ for all $\mathrm{y\in \mathbf{W}\cap \mathbf{Y}_{0}\cap \mathbf{R}_{N_{0}}}$ and all $\mathrm{n}$ big enough as claimed. So with $\mathrm{\mathbf{W}\subset \mathbf{Y}}$ we have found a compact neighborhood of $\mathrm{y_{0}\in cl\left(\mathbf{Y}_{0}\cap \mathbf{R}_{N_{0}}\right)}$ so that the claim is true over $\mathrm{\mathbf{W}\cap \mathbf{Y}_{0}\cap \mathbf{R}_{N_{0}}}$ and hence the claim follows by the introducing statement of this proof.
\end{proof}

We close this section by giving an alternative formulation of our results in the language of operators. For this let $\mathrm{f\in \mathcal{C}^{0}\!\left(\mathbf{X}\right)}$ and consider the continuous, $\mathrm{T}$-invariant function $\mathrm{\overline{f}}$ on $\mathrm{\mathbf{X}}$ defined by $$\mathrm{\overline{f}\left(x\right)\coloneqq \int_{T}f\left(t.x\right)d\nu_{T}}$$ where $\mathrm{\nu_{T}}$ denotes the {\scshape{Haar}} measure on $\mathrm{T}$. Furthermore, we introduce $$\mathrm{\Phi_{n}\left(f\right)\left(y\right)\coloneqq \int_{\pi^{-1}\left(y\right)}\overline{f}\,d\bm{\nu}_{n}\left(y\right)\text{for all }f\in \mathcal{C}^{0}\!\left(\mathbf{X}\right)\text{ and all }y\in \mathbf{R}_{N_{0}}\cap \mathbf{Y}_{0}.\label{Notation Phi_n}}$$ If $\mathrm{\mathcal{C}^{0}_{bd}\!\left(\mathbf{R}_{N_{0}}\cap \mathbf{Y}_{0}\right)}$ denotes the space of all bounded continuous functions on $\mathrm{\mathbf{R}_{N_{0}}\cap \mathbf{Y}_{0}}$ then we have the following lemma.

\begin{lemma} Each $\mathrm{\Phi_{n}}$ (for $\mathrm{n}$ big enough) defines an operator from $\mathrm{\mathcal{C}^{0}\left(\mathbf{X}\right)}$ to $\mathrm{\mathcal{C}^{0}_{bd}(\mathbf{R}_{N_{0}}}$ $\mathrm{\cap \mathbf{Y}_{0})}$ for all $\mathrm{n}$ big enough.
\end{lemma}
\begin{proof} Let $\mathrm{f\in\mathcal{C}^{0}\!\left(\mathbf{X}\right)}$, $\mathrm{K\coloneqq max\, \left\{\vert f\vert\right\}=max\, \left\{\vert \overline{f}\vert\right\}}$ and $\mathrm{y\in \mathbf{R}_{N_{0}}\cap \mathbf{Y}_{0}}$. Choose a compact neighborhood $\mathrm{\mathbf{W}\subset \mathbf{R}_{N_{0}}\cap \mathbf{Y}_{0}}$ of $\mathrm{y}$ and let $\mathrm{\epsilon_{m}\in \mathbb{R}^{\geq 0}}$ be a strictly increasing sequence converging to $\mathrm{\infty}$. Furthermore, choose a sequence $\mathrm{\left(\psi_{m}\right)_{m}}$ of smooth, $\mathrm{T}$-invariant cut-off functions defined on $\mathrm{\mathbf{X}}$ so that $$\mathrm{\psi_{m}\vert T\left(\epsilon_{m},\mathbf{W}\right)\equiv 1,\text{ } supp\,\psi_{m}\cap T^{c}\left(\epsilon_{m+1},\mathbf{W}\right)=\varnothing\text{ and } supp\,\psi_{m}\subset \pi^{-1}\left(\mathbf{Y}_{0}\cap \mathbf{R}_{N_{0}}\right).}$$ Since $\mathrm{\pi\vert \mathbf{X}_{0}\!:\!\mathbf{X}_{0}\rightarrow \mathbf{Y}_{0}}$ is a $\mathrm{k}$-fibering, it follows by {\bf{Theorem \ref{Theorem King}}} that $\mathrm{\Phi_{n}\left(\psi_{m}\,f\right)\in \mathcal{C}^{0}(\mathbf{W})}$ for all $\mathrm{m\in \mathbb{N}}$. We calculate \begin{equation*}
\begin{split}
\mathrm{\vert \Phi_{n}\left(\psi_{m}\,f\right)-\Phi_{n}\left(f\right)\vert\left(y\right)}=&\mathrm{\left|\int_{\pi^{-1}\left(y\right)}\left(\psi_{m}\,\overline{f}-\overline{f}\right)\,d\bm{\nu}_{n}\left(y\right)\right|}\\[0,2 cm]
\leq&\mathrm{2\,K\int_{\pi^{-1}\left(y\right)\cap T^{c}\left(\epsilon_{m},\mathbf{W}\right)}\left|\psi_{m}-1\right|\,\phi_{n}\,d\,[\pi_{y}]}\\[0,2 cm]
\leq&\mathrm{2\,K\int_{\pi^{-1}\left(y\right)\cap T^{c}\left(\epsilon_{m},\mathbf{W}\right)}\phi_{n}\,d\,[\pi_{y}].}\\
\end{split}
\end{equation*} Using equation \ref{Equation Which Is Necessary} (for $\mathrm{\sigma_{\epsilon}=\sigma_{m}}$) in the proof of {\bf{Theorem 5.a}}, we deduce that $$\mathrm{\vert \Phi_{n}\left(\psi_{m}\,f\right)-\Phi_{n}\left(f\right)\vert \leq 2\cdot e^{-n\,\frac{1}{4}\,\epsilon_{m}}\cdot C\cdot K}$$ for all $\mathrm{n}$ big enough and all $\mathrm{y\in \mathbf{W}}$. Hence, for a fixed $\mathrm{n\in \mathbb{N}}$ big enough, the sequence of continuous functions on $\mathrm{\mathbf{W}}$ given by $\mathrm{\left(\Phi_{n}\left(\psi_{m}\,f\right)\vert \mathbf{W}\right)_{m}}$ converges uniformly to $\mathrm{\Phi_{n}\left(f\right)\vert \mathbf{W}}$ as $\mathrm{m\rightarrow \infty}$ and consequently $\mathrm{\Phi_{n}\left(f\right)\vert \mathbf{W}\in \mathcal{C}^{0}\!\left(\mathbf{W}\right)}$. Since $\mathrm{y\in  \mathbf{R}_{N_{0}}\cap \mathbf{Y}_{0}}$ was chosen arbitrarily, it follows that $\mathrm{\Phi_{n}\left(f\right)\in \mathcal{C}^{0}\!\left(\mathbf{R}_{N_{0}}\cap \mathbf{Y}_{0}\right)}$ for all $\mathrm{f\in \mathcal{C}^{0}\!\left(\mathbf{X}\right)}$ and all $?\mathrm{n\in \mathbb{N}}$ big enough as claimed.
\end{proof}

\begin{remark} Even if $\mathrm{\complement\left(\mathbf{R}_{N_{0}}\cap \mathbf{Y}_{0}\right)}$ is a proper analytic subset of $\mathrm{\mathbf{Y}}$, it is in general not possible to extend $\mathrm{\Phi_{n}\!:\!\mathcal{C}^{0}\!\left(\mathbf{X}\right)\rightarrow \mathcal{C}^{0}\!\left(\mathbf{R}_{N_{0}}\cap \mathbf{Y}\right)}$ to an operator $\mathrm{\widehat{\Phi}_{n}\!:\!\mathcal{C}^{0}\!\left(\mathbf{X}\right)\rightarrow }$ $\mathrm{\mathcal{C}^{0}\!\left(\mathbf{Y}\right)}$: Using {\bf{Example \ref{Example Existence Of Removable Singularities}}}, one can find a continuous function $\mathrm{f\in\mathcal{C}^{0}\!\left(\mathbf{X}\right)}$ so that $\mathrm{\Phi_{n}\left(f\right)}$ has no continuous extension in $\mathrm{[0\!:\!{\dots}\!:\!0\!:\!1]\in \mathbf{Y}=\mathbb{C}\mathbb{P}^{k}}$.
\end{remark}

If $\mathrm{\mathfrak{Red}\!:\!\mathcal{C}^{0}\!\left(\mathbf{X}\right)\rightarrow \mathcal{C}^{0}_{bd}\!\left(\mathbf{R}_{N_{0}}\cap \mathbf{Y}_{0}\right)}$ denotes the operator given by $\mathrm{\mathfrak{Red}\left(f\right)\coloneqq \overline{f}_{red}\vert \mathbf{R}_{N_{0}}\cap \mathbf{Y}_{0}}$, then we deduce the following corollary.

\begin{cor} The operator sequence $\mathrm{\left(\Phi_{n}\right)_{n}}$ converges to  $\mathrm{\mathfrak{Red}}$ with respect to the topology induced by the supremum norm on $\mathrm{\mathcal{C}^{0}\!\left(\mathbf{X}\right)}$ and $\mathrm{\mathcal{C}^{0}_{bd}\!\left(\mathbf{R}_{N_{0}}\cap \mathbf{Y}\right)}$.
\end{cor}
\begin{proof} Apply {\bf{Theorem 6.b}}.
\end{proof}

\newpage
\addcontentsline{toc}{section}{Index of Notation}
\hspace{-0.6 cm}{\Large{\bf{Index of Notation}}}
\vspace{0.55 cm}

\noindent\begin{tabular}{@{}p{2.5 cm}p{12.35cm}@{}}
$\mathrm{\mathcal{C}^{k}(\widehat{\,\mathbf{X}\,})}$ & The complex space of all $\mathrm{k}$-dimensional cycles $\mathrm{{\bm{\mathfrak{C}}}}$ in $\mathrm{\widehat{\,\mathbf{X}\,}}$, \textbf{S. \oldstylenums{\pageref{Notation Cycle Space Of All k-Dim Cycles}}}\\[0.05 cm]

$\mathrm{{\bm{\mathfrak{C}}}}$ &  $\mathrm{k}$-dimensional cycle $\mathrm{\sum_{i}n_{i}\mathbf{C}_{i}}$, $\mathrm{n_{i}\in \mathbb{N}}$, $\mathrm{dim_{\mathbb{C}}\,\mathbf{C}_{i}=k}$ in $\mathrm{\widehat{\,\mathbf{X}\,}}$, \textbf{S. \oldstylenums{\pageref{Notation Cycle Of Dimension K}}}\\[0.05 cm]

$\mathrm{{\bm{\mathfrak{C}}}_{y}}$ & $\mathrm{k}$-dimensional cycle associated to the fiber $\mathrm{\widehat{\pi}^{-1}\left(y\right)}$ for $\mathrm{y\in \mathbf{Y}_{0}}$, \textbf{S. \oldstylenums{\pageref{Notation Cycle Of Dimension K Associated To The Fiber Of}}}\\[0.05 cm]

$\mathrm{\vert {\bm{\mathfrak{C}}}\vert}$&Support $\mathrm{\vert{\bm{\mathfrak{C}}}\vert=\bigcup_{i\in I}\mathbf{C}_{i}}$ of the cycle $\mathrm{{\bm{\mathfrak{C}}}=\sum_{i}n_{i}\mathbf{C}_{i}}$, \textbf{S. \oldstylenums{\pageref{Notation Support Of A Cycle}}}\\[0.05 cm]

$\mathrm{cl\left(\mathbb{T}.x\right)}$&Zariski closure of the $\mathrm{\mathbb{T}}$-orbit $\mathrm{\mathbb{T}.x}$ where $\mathrm{x\in \mathbf{X}}$, \textbf{S. \oldstylenums{\pageref{Notation T-Orbit Closure}}}\\[0.05 cm]

$\mathrm{\mathfrak{Conv}\,\mathbf{A}}$&Convex hull of a subset $\mathrm{\mathbf{A}\subset \mathfrak{t}^{*}}$, \textbf{S. \oldstylenums{\pageref{Notation Convex Hull Of A Subset}}}\\[0.05 cm]

$\mathrm{D_{n}^{i}\left(\cdot,t\right)}$& Sequence of cumulative distribution densities associated to the tame sequence $\mathrm{\left(s^{i}_{n}\right)_{n}}$, \textbf{S. \oldstylenums{\pageref{Notation Cumulative Distribution Densities}}}\\[0.05 cm]

$\mathrm{D_{n}^{\mathfrak{U}}\left(\cdot,t\right)}$& Sequence of distribution functions associated to the open cover $\mathrm{\mathfrak{U}}$, \textbf{S. \oldstylenums{\pageref{Notation Collection Of Cumulative Fiber Probability Densities}}}\\[0.05 cm]

$\mathrm{\triangle^{i}_{f_{y,n}}\left(=\triangle^{i}_{y,n}\right)}$& Holomorphic $\mathrm{\xi_{n}-\xi^{i}_{n}}$-eigenfunction on $\mathrm{\pi^{-1}\left(\mathbf{U}_{y}\right)}$ where $\mathrm{y\in \mathbf{R}_{N_{0}}}$ defined by $\mathrm{s^{i}_{n}\cdot\triangle_{f_{y,n}}^{i}=\widehat{s}_{f_{y,n}}}$, \textbf{S. \oldstylenums{\pageref{Notation Difference Function}}}\\[0.05 cm]

$\mathrm{{\bf{Fix}}^{\mathbb{T}}}$& Set of all $\mathrm{\mathbb{T}}$-fixed points in $\mathrm{\mathbf{X}}$, \textbf{S. \oldstylenums{\pageref{Notation Fix Point Set Of T-Action}}}\\[0.05 cm]

$\mathrm{\overline{f}}$& Averaged function defined by $\mathrm{\overline{f}\left(x\right)=\int_{T}f\left(t.x\right)\,d\nu_{T}}$, \textbf{S. \oldstylenums{\pageref{Notation Averaged Function}}}\\[0.05 cm]

$\mathrm{f_{red}}$& Function on the quotient $\mathrm{\mathbf{Y}=\mathbf{X}^{ss}_{\xi}/\!\!/\mathbb{T}}$ induced by the restriction $\mathrm{\overline{f}\vert \mu^{-1}\left(\xi\right)}$ of the averaged function $\mathrm{\overline{f}}$, \textbf{S. \oldstylenums{\pageref{Notation Reduced Function}}}\\[0.05 cm]

$\mathrm{\mathfrak{F}_{{\bm{\mathfrak{C}}},\mathbb{C}\mathbb{P}^{m}}}$& Chow form associated to a cycle $\mathrm{{\bm{\mathfrak{C}}}}$ in $\mathrm{\mathbb{C}\mathbb{P}^{m}}$, \textbf{S. \oldstylenums{\pageref{Notation Chow Form}}}\\[0.05 cm]

${\bm{\Gamma}_{\pi}}$& Graph of the quotient map $\mathrm{\pi\!:\!\mathbf{X}^{ss}_{\xi}\rightarrow \mathbf{X}^{ss}_{\xi}/\!\!/\mathbb{T}=\mathbf{Y}}$ in $\mathrm{\mathbf{X}\times \mathbf{Y}}$, \textbf{S. \oldstylenums{\pageref{Notation Graph of Quotient Map}}}\\[0.05 cm]

$\mathrm{h}$& Hermitian, positive, $\mathrm{T}$-invariant bundle metric on $\mathrm{\mathbf{L}}$, \textbf{S. \oldstylenums{\pageref{Hermitian Bundle Metric}}}\\[0.05 cm]

$\mathrm{\int_{F^{-1}\left(y\right)}f\,d\,[F_y]}$& Fiber integral of $\mathrm{f}$ with respect to a $\mathrm{k}$-fibering $\mathrm{F\!:\!\mathbf{X}\rightarrow \mathbf{Y}}$, \textbf{S. \oldstylenums{\pageref{Notation Fiber Integration with Respect to a k-Fibering}}}\\[0.05 cm]

$\mathrm{\mathbf{M}_{J}}$& Subset of all points $\mathrm{x\in\mu^{-1}\left(\xi\right)}$ so that $\mathrm{s_{j}\left(x\right)\neq 0}$ for all $\mathrm{j\in J}$ and $\mathrm{s_{j}\left(x\right)\neq 0}$ for all $\mathrm{j\in J^{c}}$, \textbf{S. \pageref{Notation Definition of M_J}}\\[0.05 cm]

$\mathrm{\mu}$& Moment map associated to the hermitian, positive, $\mathrm{T}$-invariant bundle metric $\mathrm{h}$, \textbf{S. \oldstylenums{\pageref{Notation Moment Map}}}\\[0.05 cm]

$\mathrm{\bm{\nu}^{i}_{n}}$& Sequence of fiber measures associated to the tame sequence $\mathrm{\left(s^{i}_{n}\right)_{n}}$, \textbf{S. \oldstylenums{\pageref{Notation Sequence Of Fiber Probability Measures}}}\\[0.05 cm]

$\mathrm{\bm{\nu}^{\mathfrak{U}}_{n}}$& Sequence of fiber measures associated to the open cover $\mathrm{\mathfrak{U}}$, \textbf{S. \oldstylenums{\pageref{Notation Collection Of Fiber Probability Measures}}}\\[0.05 cm]

$\mathrm{\nu_{F}\left(x\right)}$& Order of a $\mathrm{k}$-fibering $\mathrm{F\!:\!\mathbf{X}\rightarrow \mathbf{Y}}$ at a point $\mathrm{x\in \mathbf{X}}$, \textbf{S. \oldstylenums{\pageref{Notation Order Of F At Point}}}\\[0.05 cm]

$\mathrm{\omega}$& Naturally associated K{\"a}hler form given by $\mathrm{\omega=-\frac{\sqrt{-1}}{\,\,2}\partial\overline{\partial}\,\mathbf{log}\,\vert \cdot\vert_{h}^{2}}$, \textbf{S. \oldstylenums{\pageref{Notation Kaehler form}}}\\[0.05 cm]

$\mathrm{\omega^{\prime}}$ & Smooth (2,2,)-form on $\mathrm{\widehat{\mathbf{\,X\,}}}$ given by $\mathrm{\omega^{\prime}= \left(p_{\mathbf{X}}\vert cl\left({\mathbf{\Gamma}}_{\pi}\right)\circ\zeta\right)^{*}\omega}$, \textbf{S. \oldstylenums{\pageref{Notation Smooth (2,2,)-form}}}\\[0.05 cm]

$\mathrm{\Omega}$ & Smooth (2,2,)-form on the compact variety $\mathrm{\widetilde{\bf{\,X\,}}}$ given by $\mathrm{\Omega\coloneqq\Sigma^{*}\omega^{\prime}}$, \textbf{S. \oldstylenums{\pageref{Notation Capital Omega}}}\\[0.05 cm]

$\mathrm{\pi}$& Algebraic projection map $\mathrm{\pi\!:\!\mathbf{X}^{ss}_{\xi}\rightarrow \mathbf{X}^{ss}_{\xi}/\!\!/\mathbb{T}}$ of the Hilbert Quotient, \textbf{S. \oldstylenums{\pageref{Notation Hilbert Quotient Map}}}\\[0.05 cm]

$\mathrm{\widehat{\pi}}$& Algebraic projection map from the compact variety $\mathrm{\widehat{\bf{\,X\,}}}$ to $\mathrm{\mathbf{Y}}$ given by $\mathrm{\widehat{\pi}= p_{\mathbf{Y}}\vert cl\,\mathbf{\Gamma}_{\pi}\circ \zeta}$, \textbf{S. \oldstylenums{\pageref{Notation Regular Map between cal X and Y}}}\\[0.05 cm]

$\mathrm{\Pi}$& Surjective, holomorphic $\mathrm{k}$-fibering from $\mathrm{\mathbf{\widetilde{\,\bf{X}\,}}}$ to $\mathrm{\widetilde{\,\bf{Y}\,}}$, \textbf{S. \oldstylenums{\pageref{Notation Blown Up Projection Map}}}\\[0.05 cm] 

$\mathrm{\varphi^{\widehat{\pi}}}$& Holomorphic map from $\mathrm{\mathbf{Y}_{0}}$ into the cycle space $\mathrm{\mathcal{C}^{k}(\widehat{\,\mathbf{X}\,})}$ induced by the $\mathrm{k}$-fibering $\mathrm{\widehat{\pi}\vert \widehat{\pi}^{-1}\left(\mathbf{Y}_{0}\right)\rightarrow \mathbf{Y}_{0}}$, \textbf{S. \oldstylenums{\pageref{Notation Regular Map Universal Property Chow Scheme}}}\\[0.05 cm]
\end{tabular}

\noindent\begin{tabular}{@{}p{2.5 cm}p{12.35cm}@{}}

$\mathrm{\Phi_{n}}$& Continuous operator from $\mathrm{\mathcal{C}^{0}\left(\mathbf{X}\right)}$ to $\mathrm{\mathcal{C}^{0}_{bd}\left(\mathbf{R}_{N_{0}}\cap \mathbf{Y}_{0}\right)}$ induced by $\mathrm{\bf{\nu}_{n}}$, \textbf{S. \oldstylenums{\pageref{Notation Phi_n}}}\\[0.05 cm]

$\mathrm{Relint\,\mathbf{A}}$& Relative interior of a convex subset $\mathrm{\mathbf{A}\subset \mathfrak{t}^{*}}$, \textbf{S. \oldstylenums{\pageref{Notation Relative Interior}}}\\[0.05 cm]

$\mathrm{\mathbf{R}_{N_{0}}}$& The set of all removable singularities of order $\mathrm{N_{0}\in \mathbb{N}}$ of the fiber measure sequence $\mathrm{\left(\bm{\nu}_{n}\right)_{n}}$, \textbf{S. \oldstylenums{\pageref{Notation Removable Singularity Of Order N_0}}}\\[0.05 cm]

$\mathrm{\varrho^{i}}$& S.p.s.h limit function associated to the tame sequence $\mathrm{\left(s^{i}_{n}\right)_{n}}$, \textbf{S. \oldstylenums{\pageref{Notation S.p.s.h. Limit Function Associated to the Tame Sequence}}}\\[0.05 cm]

$\mathrm{\varrho^{i}_{n}}$& Sequence of s.p.s.h functions associated to the tame sequence $\mathrm{\left(s^{i}_{n}\right)_{n}}$, \textbf{S. \oldstylenums{\pageref{Notation Sequence of S.p.s.h. Functions Associated to the Tame Sequence}}}\\[0.05 cm]

$\mathrm{s^{i}_{n}}$& $\mathrm{i}$-th tame sequence associated to the ray $\mathrm{\mathbb{R}^{\geq 0}\xi}$ where $\mathrm{\xi\in \mathfrak{t}^{*}}$, \textbf{S. \oldstylenums{\pageref{Notation i-th tame sequence}}}\\[0.05 cm]

$\mathrm{\widehat{s}_{f_{y,n}}}$& Extension of $\mathrm{s_{n}}$ given by $\mathrm{\widehat{s}_{f_{y,n}}=s_{n}\cdot\pi^{*}f^{-1}_{y,n}}$ for all $\mathrm{n\geq  N_{0}\in \mathbb{N}}$ where $\mathrm{f_{y,n}\in \mathcal{O}\left(\mathbf{U}_{y}\right)}$ and $\mathrm{y\in \mathbf{R}_{N_{0}}}$, \textbf{S. \oldstylenums{\pageref{Notation Local Extension Of S_{n}}}}\\[0.05 cm]

$\mathrm{\sigma}$& Proper, surjective holomorphic map from $\mathrm{\widetilde{\bf{\,Y\,}}}$ to $\mathrm{\mathbf{Y}}$ given by $\mathrm{p_{\mathbf{Y}}\vert \widetilde{\bf{\,Y\,}}}$, \textbf{S. \oldstylenums{\pageref{Notation Blow Up Map onto Y}}}\\[0.05 cm]

$\mathrm{\Sigma}$& Proper, surjective holomorphic map from $\mathrm{\widetilde{\,\bf{X}\,}}$ to $\mathrm{\widehat{\,\bf{X}\,}}$, \textbf{S. \oldstylenums{\pageref{Notation Blow Up Map onto X}}}\\[0.05 cm]

$\mathrm{\vert s\vert^{2}}$& Norm function of a section $\mathrm{s\in H^{0}\left(\mathbf{X},\mathbf{L}\right)}$ with respect to the bundle metric $\mathrm{h}$, \textbf{S. \oldstylenums{\pageref{Notation Norm of Section with Respect to H}}}\\[0.05 cm]

$\mathrm{\Vert s^{i}_{n}\Vert^{2}}$& Fiber integral of $\mathrm{\vert s^{i}_{n}\vert ^{2}}$ over the subset $\mathrm{\mathbf{Y}_{0}\subset \mathbf{Y}}$ with respect to $\mathrm{\pi\vert \pi^{-1}\left(\mathbf{Y}_{0}\right)}$, \textbf{S. \oldstylenums{\pageref{Notation Fiber Integral Initial Sequence}}}\\[0.05 cm]

$\mathrm{{\bf{S}}_{x}}$& Set of all $\mathrm{\mathbb{T}}$-fixed points contained in the Zariski closure $\mathrm{\mathbb{T}.x}$ where $\mathrm{x\in\mathbf{X}}$, \textbf{S. \oldstylenums{\pageref{Notation T-Fixed Points Contained In The Closure Of T.x}}}\\[0.05 cm] 

$\mathrm{\mathcal{S}_{J}}$& Set of all $\mathrm{\mathbb{T}}$-eigensections indexed by $\mathrm{J\subset \left\{1,\dots,m\right\}}$ where $\mathrm{m=dim_{\mathbb{C}}}$ $\mathrm{H^{0}\left(\mathbf{X},\mathbf{L}\right)}$, \textbf{S. \oldstylenums{\pageref{Notation Set Of T-Eigensections Indexed By J}}}\\[0.05 cm]

$\mathrm{{\bm{\mathfrak{S}}}_{J}}$& Set of all characters $\mathrm{\xi_{j}\in \mathfrak{t}^{*}_{\mathbb{Z}}}$ attached to the collection of all $\mathrm{\mathbb{T}}$-eigensections given by $\mathrm{\mathcal{S}_{J}}$, \textbf{S. \oldstylenums{\pageref{Notation Set Of All Characters Corresponding To S_x}}}\\[0.05 cm]

$\mathrm{T\left(\epsilon,\mathbf{W}^{i}\right)}$& Compact $\mathrm{\epsilon}$-neighborhood tube around $\mathrm{\mu^{-1}\left(\xi\right)}$ over $\mathrm{\mathbf{W}^{i}\subset \mathbf{Y}}$, \textbf{S. \oldstylenums{\pageref{Notation Compact Neighborhood Tube}}}\\[0.05 cm]

$\mathrm{T^{c}\left(\epsilon,\mathbf{W}^{i}\right)}$& Complement of $\mathrm{T\left(\epsilon,\mathbf{W}^{i}\right)}$ in $\mathrm{\pi^{-1}\left(\mathbf{W}^{i}\right)}$, \textbf{S. \oldstylenums{\pageref{Notation Complement of Compact Neighborhood Tube}}}\\[0.05 cm]

$\mathrm{\widetilde{\,T\,}\left(\epsilon,\mathbf{W}^{i}\right)}$& Compact $\mathrm{\epsilon}$-neighborhood tube in $\mathrm{\widetilde{\bf{\,X\,}}}$ induced by $\mathrm{T\left(\epsilon,\mathbf{W}^{i}\right)}$, \textbf{S. \oldstylenums{\pageref{Notation Compact Corresponding Neighborhood Tube}}}\\[0.05 cm]

$\mathrm{\mathbf{V}^{i}}$& Inverse image of the compact neighborhood $\mathrm{\mathbf{W}^{i}\subset \mathbf{Y}}$ under the projection $\mathrm{\pi}$, \textbf{S. \oldstylenums{\pageref{Notation V^i}}}\\[0.05 cm]

$\mathrm{vol\left(\pi^{-1}\left(y\right)\right)}$& Volume of the fiber $\mathrm{\pi^{-1}\left(y\right)}$ for $\mathrm{y\in \mathbf{Y}_{0}}$ with respect to the form $\mathrm{\omega^{k}\vert \pi^{-1}\left(y\right)}$ counted with multiplicities, \textbf{S. \oldstylenums{\pageref{Notation Volume of a Fiber Over Y_0}}}\\[0.05 cm]

$\mathrm{\mathbf{W}^{i}}$& Compact neighborhood contained in $\mathrm{\mathbf{Y}^{i}}$, \textbf{S. \oldstylenums{\pageref{Notation W^i}}}\\[0.05 cm]

$\mathrm{\widetilde{\bf{W}}^{i}}$& Compact neighborhood in $\mathrm{\widetilde{\bf{\,Y\,}}}$ induced by $\mathrm{\mathbf{W}^{i}}$ via $\mathrm{\widetilde{\bf{W}}^{i}=\sigma^{-1}\left(\mathbf{W}^{i}\right)}$, \textbf{S. \oldstylenums{\pageref{Notation cal W^i}}}\\[0.05 cm]

$\mathrm{\mathbf{\widehat{\,X\,}}}$& Compact variety defined as the normalization of the compact variety subvariety $\mathrm{cl\,\bm{\Gamma}_{\pi}\subset\mathbf{X}\times \mathbf{Y}}$, \textbf{S. \oldstylenums{\pageref{Notation X Hat}}}\\[0.05 cm]

$\mathrm{\widetilde{\bf{\,X\,}}}$& Compact variety given by $\mathrm{\widetilde{\bf{\,X\,}}=\left(\mathbf{Y}\times {\bm{\mathfrak{X}}}\right)\cap \big(\widetilde{\bf{\,Y\,}}\times\widehat{\bf{\,X\,}}\big)}$, \textbf{S. \oldstylenums{\pageref{Notation Blow Up Space}}}\\[0.05 cm]

$\mathrm{\mathbf{X}_{0}}$& Inverse image of $\mathrm{\mathbf{Y}_{0}}$ with respect to the projection map $\mathrm{\pi\!:\!\mathbf{X}^{ss}_{\xi}\rightarrow \mathbf{Y}}$, \textbf{S. \oldstylenums{\pageref{Notation Inverse Image of Y_0 Under The Projection Map}}}\\[0.05 cm]

$\mathrm{\mathbf{X}^{i}}$& Open, $\mathrm{\pi}$-saturated subset of $\mathrm{\mathbf{X}^{ss}_{\xi}\cap \mathbf{X}(s^{i}_{n})}$ given by $\mathrm{\mathbf{X}^{i}\coloneqq\pi^{-1}\left(\pi\left(\mathbf{M}_{i}\right)\right)}$, \textbf{S. \oldstylenums{\pageref{Notation i-th set}}}\\[0.05 cm]

$\mathrm{\mathbf{X}\left(s_{n}^{i}\right)}$& n-stable, Zariski open subset of $\mathrm{\mathbf{X}}$ given by $\mathrm{\left\{x\in \mathbf{X}:\, s_{n}^{i}\left(x\right)\neq 0\right\}}$, \textbf{S. \oldstylenums{\pageref{Notation GTZU}}}\\[0.05 cm]

$\mathrm{\mathbf{X}^{ss}_{\xi}}$& Set of semistable points with respect to the level subset $\mathrm{\mu^{-1}\left(\xi\right)}$, \textbf{S. \oldstylenums{\pageref{Notation Set of Semistable Points}}}\\[0.05 cm]

\end{tabular}

\noindent\begin{tabular}{@{}p{2.5 cm}p{12.35cm}@{}}
$\mathrm{\mathbf{X}^{ss}_{\xi}/\!\!/\mathbb{T}}$& Hilbert Quotient with respect to the level subset $\mathrm{\mu^{-1}\left(\xi\right)}$, \textbf{S. \oldstylenums{\pageref{Notation Hilbert Quotient}}}\\[0.05 cm]

$\mathrm{\widehat{X}_{\xi}}$& Fundamental vector field on $\mathrm{\mathbf{L}}$ generated by the flow induced by $\mathrm{\xi\in \mathfrak{k}}$, \textbf{S. \oldstylenums{\pageref{Notation Fundamental Vector Field on L}}}\\[0.05 cm]

$\mathrm{{\bm{\mathfrak{X}}}}$& Universal space defined by $\mathrm{{\bm{\mathfrak{X}}}\coloneqq \left\{\big({\bm{\mathfrak{C}}},x\right)\in \mathcal{C}^{k}(\widehat{\,\bf{X}\,})\times \widehat{\,\bf{X}\,}: x\in \vert {\bm{\mathfrak{C}}}\vert\big\}}$, \textbf{S. \oldstylenums{\pageref{Notation Universal Space}}}\\[0.05 cm]

${\xi}$& Normalization map $\mathrm{\xi\!:\!\widetilde{\bf{\,Y\,}}^{nor}\rightarrow \widetilde{\bf{\,Y\,}}}$ of $\mathrm{\widetilde{\bf{\,Y\,}}}$, \textbf{S. \oldstylenums{\pageref{Notation Normalization Map Xi}}}\\[0.05 cm]

$\mathrm{\xi^{i}_{n}}$& Sequence of weights associated to a tame sequence $\mathrm{\left(s^{i}_{n}\right)_{n}}$, \textbf{S. \oldstylenums{\pageref{Notation Approximating Sequence of Weight Vectors}}}\\[0.05 cm]

$\mathrm{\left(y,{\bm{\mathfrak{C}}}_{y}\right)}$& Point in $\mathrm{\sigma^{-1}\left(\mathbf{Y}_{0}\right)\subset \widetilde{\bf{\,Y\,}}\subset \mathbf{Y}\times \mathcal{C}^{k}(\widehat{\,\bf{X}\,})}$ where $\mathrm{\varphi^{\widehat{\pi}}\left(y\right)={\bm{\mathfrak{C}}}_{y}}$ for $\mathrm{y\in \mathbf{Y}_{0}}$, \textbf{S. \oldstylenums{\pageref{Notation y,frak C_y}}}\\[0.05 cm]

$\mathrm{\mathbf{Y}}$& Abbreviation for the Hilbert Quotient $\mathrm{\mathbf{X}^{ss}_{\xi}/\!\!/\mathbb{T}}$, \textbf{S. \oldstylenums{\pageref{Notation Abbreviation for Hilbert Quotient}}}\\[0.05 cm]

$\mathrm{\mathbf{Y}^{i}}$& Image of $\mathrm{\mathbf{X}^{i}}$ under the projection map $\mathrm{\pi\!:\!\mathbf{X}^{ss}_{\xi}\rightarrow \mathbf{Y}}$, \textbf{S. \oldstylenums{\pageref{Notation Image of X_i Under Projection Map}}}\\[0.05 cm]

$\mathrm{\mathbf{Y}_{0}}$& Subset of all points in $\mathrm{\mathbf{Y}}$ so that $\mathrm{\widehat{\pi}^{-1}\left(y\right)}$ is a $\mathrm{k}$-dimensional subvariety of $\mathrm{\mathbf{\widehat{\,X\,}}}$, \textbf{S. \oldstylenums{\pageref{Notation Subset Y_0}}}\\[0.05 cm]

$\widetilde{\bf{\,Y\,}}$& Complex subspace of $\mathrm{\mathbf{Y}\times \mathcal{C}^{k}(\widehat{\,\mathbf{X}\,})}$ given by the closure of the graph of $\mathrm{\varphi^{\widehat{\pi}}\!:\!\mathbf{Y}_{0}\rightarrow \mathcal{C}^{k}(\widehat{\,\mathbf{X}\,})}$, \textbf{S. \oldstylenums{\pageref{Notation Blown Up Space Associated to Y}}}\\[0.05 cm]

${\zeta}$& Normalization map $\mathrm{\zeta\!:\!\widehat{\,\mathbf{X}\,}=\left(cl\,\mathbf{\Gamma}_{\pi}\right)^{nor}\rightarrow cl\,\mathbf{\Gamma}_{\pi}}$ of the compact variety $\mathrm{cl\,\mathbf{\Gamma}_{\pi}\subset}$ $\mathrm{ \mathbf{X}\times \mathbf{Y}}$, \textbf{S. \oldstylenums{\pageref{Notation Normalization Map Zeta}}}

\end{tabular}

\end{spacing}
\end{document}